%% file: A_main.tex
\documentclass[12pt]{amsart}
    \usepackage{amsmath,amssymb,amsthm,amsfonts,amsrefs}
\usepackage[margin=1in,marginparwidth=0.8in, marginparsep=0.1in]{geometry}

\DeclareRobustCommand{\SkipTocEntry}[5]{}
\usepackage{etoolbox}

\makeatletter
\let\old@tocline\@tocline
\let\section@tocline\@tocline
\newcommand{\subsection@dotsep}{4.5}
\newcommand{\subsubsection@dotsep}{4.5}
\patchcmd{\@tocline}
  {\hfil}
  {\nobreak
     \leaders\hbox{$\m@th
        \mkern \subsection@dotsep mu\hbox{.}\mkern \subsection@dotsep mu$}\hfill
     \nobreak}{}{}
\let\subsection@tocline\@tocline
\let\@tocline\old@tocline

\patchcmd{\@tocline}
  {\hfil}
  {\nobreak
     \leaders\hbox{$\m@th
        \mkern \subsubsection@dotsep mu\hbox{.}\mkern \subsubsection@dotsep mu$}\hfill
     \nobreak}{}{}
\let\subsubsection@tocline\@tocline
\let\@tocline\old@tocline

\let\old@l@subsection\l@subsection
\let\old@l@subsubsection\l@subsubsection

\def\@tocwriteb#1#2#3{%
  \begingroup
    \@xp\def\csname #2@tocline\endcsname##1##2##3##4##5##6{%
      \ifnum##1>\c@tocdepth
      \else \sbox\z@{##5\let\indentlabel\@tochangmeasure##6}\fi}%
    \csname l@#2\endcsname{#1{\csname#2name\endcsname}{\@secnumber}{}}%
  \endgroup
  \addcontentsline{toc}{#2}%
    {\protect#1{\csname#2name\endcsname}{\@secnumber}{#3}}}%

\newlength{\@tocsectionindent}
\newlength{\@tocsubsectionindent}
\newlength{\@tocsubsubsectionindent}
\newlength{\@tocsectionnumwidth}
\newlength{\@tocsubsectionnumwidth}
\newlength{\@tocsubsubsectionnumwidth}
\newcommand{\settocsectionnumwidth}[1]{\setlength{\@tocsectionnumwidth}{#1}}
\newcommand{\settocsubsectionnumwidth}[1]{\setlength{\@tocsubsectionnumwidth}{#1}}
\newcommand{\settocsubsubsectionnumwidth}[1]{\setlength{\@tocsubsubsectionnumwidth}{#1}}
\newcommand{\settocsectionindent}[1]{\setlength{\@tocsectionindent}{#1}}
\newcommand{\settocsubsectionindent}[1]{\setlength{\@tocsubsectionindent}{#1}}
\newcommand{\settocsubsubsectionindent}[1]{\setlength{\@tocsubsubsectionindent}{#1}}

\renewcommand{\l@section}{\section@tocline{1}{\@tocsectionvskip}{\@tocsectionindent}{}{\@tocsectionformat}}%
\renewcommand{\l@subsection}{\subsection@tocline{1}{\@tocsubsectionvskip}{\@tocsubsectionindent}{}{\@tocsubsectionformat}}%
\renewcommand{\l@subsubsection}{\subsubsection@tocline{1}{\@tocsubsubsectionvskip}{\@tocsubsubsectionindent}{}{\@tocsubsubsectionformat}}%
\newcommand{\@tocsectionformat}{}
\newcommand{\@tocsubsectionformat}{}
\newcommand{\@tocsubsubsectionformat}{}
\expandafter\def\csname toc@1format\endcsname{\@tocsectionformat}
\expandafter\def\csname toc@2format\endcsname{\@tocsubsectionformat}
\expandafter\def\csname toc@3format\endcsname{\@tocsubsubsectionformat}
\newcommand{\settocsectionformat}[1]{\renewcommand{\@tocsectionformat}{#1}}
\newcommand{\settocsubsectionformat}[1]{\renewcommand{\@tocsubsectionformat}{#1}}
\newcommand{\settocsubsubsectionformat}[1]{\renewcommand{\@tocsubsubsectionformat}{#1}}
\newlength{\@tocsectionvskip}
\newcommand{\settocsectionvskip}[1]{\setlength{\@tocsectionvskip}{#1}}
\newlength{\@tocsubsectionvskip}
\newcommand{\settocsubsectionvskip}[1]{\setlength{\@tocsubsectionvskip}{#1}}
\newlength{\@tocsubsubsectionvskip}
\newcommand{\settocsubsubsectionvskip}[1]{\setlength{\@tocsubsubsectionvskip}{#1}}

\patchcmd{\tocsection}{\indentlabel}{\makebox[\@tocsectionnumwidth][l]}{}{}
\patchcmd{\tocsubsection}{\indentlabel}{\makebox[\@tocsubsectionnumwidth][l]}{}{}
\patchcmd{\tocsubsubsection}{\indentlabel}{\makebox[\@tocsubsubsectionnumwidth][l]}{}{}

\newcommand{\@sectypepnumformat}{}
\renewcommand{\contentsline}[1]{%
  \expandafter\let\expandafter\@sectypepnumformat\csname @toc#1pnumformat\endcsname%
  \csname l@#1\endcsname}
\newcommand{\@tocsectionpnumformat}{}
\newcommand{\@tocsubsectionpnumformat}{}
\newcommand{\@tocsubsubsectionpnumformat}{}
\newcommand{\setsectionpnumformat}[1]{\renewcommand{\@tocsectionpnumformat}{#1}}
\newcommand{\setsubsectionpnumformat}[1]{\renewcommand{\@tocsubsectionpnumformat}{#1}}
\newcommand{\setsubsubsectionpnumformat}[1]{\renewcommand{\@tocsubsubsectionpnumformat}{#1}}
\renewcommand{\@tocpagenum}[1]{%
  \hfill {\mdseries\@sectypepnumformat #1}}

\let\oldappendix\appendix
\renewcommand{\appendix}{%
  \leavevmode\oldappendix%
  \addtocontents{toc}{%
    \protect\settowidth{\protect\@tocsectionnumwidth}{\protect\@tocsectionformat\sectionname\space}%
    \protect\addtolength{\protect\@tocsectionnumwidth}{2em}}%
}
\makeatother

\makeatletter
\settocsectionnumwidth{2em}
\settocsubsectionnumwidth{2.5em}
\settocsubsubsectionnumwidth{3em}
\settocsectionindent{1pc}%
\settocsubsectionindent{\dimexpr\@tocsectionindent+\@tocsectionnumwidth}%
\settocsubsubsectionindent{\dimexpr\@tocsubsectionindent+\@tocsubsectionnumwidth}%
\makeatother

\settocsectionvskip{10pt}
\settocsubsectionvskip{0pt}
\settocsubsubsectionvskip{0pt}

\settocsectionformat{\scshape}
\settocsubsectionformat{\mdseries}
\settocsubsubsectionformat{\mdseries}
\setsectionpnumformat{\scshape}
\setsubsectionpnumformat{\mdseries}
\setsubsubsectionpnumformat{\mdseries}

\let\oldtableofcontents\tableofcontents
\renewcommand{\tableofcontents}{%
  \vspace*{-\linespacing}
  \oldtableofcontents}

    \usepackage{graphicx}
    \usepackage[all]{xy}
    \usepackage[colorlinks, linkcolor=blue, anchorcolor=blue,
            citecolor=blue]{hyperref}

    \usepackage{enumerate}
    \usepackage{tikz}
    \usepackage{tikz-cd}
    \usepackage{xcolor,float}
    \usepackage{mathrsfs}
    \usepackage{caption}
    \usepackage[mathscr]{eucal} 

    \usepackage{fancyhdr}
    \usepackage{mathrsfs}

    \theoremstyle{definition}
    \newtheorem{definition}{Definition}[subsection]
    \newtheorem{example}[definition]{Example}
    \newtheorem{remark}[definition]{Remark}
    
    \theoremstyle{plain}
    \newtheorem{theorem}[definition]{Theorem}
    \newtheorem{proposition}[definition]{Proposition}
    \newtheorem{definition-proposition}[definition]{Definition--Proposition}
    \newtheorem{lemma}[definition]{Lemma}
    \newtheorem{corollary}[definition]{Corollary}
    
    \newtheorem{innercustomthm}{Theorem}
\newenvironment{customthm}[1]
  {\renewcommand\theinnercustomthm{#1}\innercustomthm}
  {\endinnercustomthm}

\newcommand{\R}{\mathbb{R}}

\newcommand{\SH}{\mathscr{H}}

\newcommand{\SC}{\mathcal{C}}

\newcommand{\SD}{\mathscr{D}}
\newcommand{\SE}{\mathscr{E}}
\newcommand{\SF}{\mathscr{F}}
\newcommand{\SG}{\mathscr{G}}
\newcommand{\ST}{\mathscr{T}}
\newcommand{\SK}{\mathscr{K}}
\newcommand{\SL}{\mathscr{L}}

\newcommand{\lr}{\longrightarrow}

\newcommand{\Aut}{\operatorname{Aut}}

\newcommand{\Sh}{\operatorname{Sh}}

\newcommand{\WFuk}{\operatorname{WFuk}}

\newcommand{\ol}{\overline}

\newcommand{\bR}{\mathbb{R}}
\newcommand{\bZ}{\mathbb{Z}}

\newcommand{\cC}{\mathcal{C}}

\newcommand{\cF}{\mathcal{F}}

\newcommand{\cV}{\mathcal{V}}

\newcommand\Cat{ {\operatorname{Cat}} }

\newcommand\Fun{ {\operatorname{Fun}} }
\newcommand\id{\operatorname{id}}

\newcommand\Hom{\operatorname{Hom}}

\newcommand\PrL{  {  \operatorname{Pr}^{\operatorname{L} }  } }

\newcommand\PrLst{  {  \operatorname{Pr}^{\operatorname{L} }_{st} } }

\newcommand\PrR{ {\operatorname{Pr}^{\operatorname{R} } } }

\newcommand\PrRst{  {  \operatorname{Pr}^{\operatorname{R} }_{st} } }

\newcommand\colim{\operatorname{colim}}

\newcommand\Mod{\operatorname{Mod}}

\newcommand\Loc{ {\operatorname{Loc}} }

\newcommand\msh{\operatorname{\mu sh}}
\renewcommand\SS{\operatorname{ss}}
\renewcommand\ss{\operatorname{ss}}

\newcommand\sHom{\mathscr{H}om}

\newcommand\supp{ {\operatorname{supp}} }

\newcommand\Fuk{\operatorname{Fuk}}

\newcommand\PrLV[1][\cV]{  {  \operatorname{Pr}^{\operatorname{L} }_{\cV,st} } }

\AtBeginDocument{%
   \def\MR#1{}
}

\begin{document}

\title{\hspace{-4mm} Microsheaf composition of Lagrangian correspondences}
\date{}
\author{Wenyuan Li}
\author{David Nadler}
\author{Vivek Shende}

\begin{abstract}
    In exact symplectic manifolds whose Liouville flow is gradientlike for a proper Morse function, one can associate conic microsheaves to eventually conic exact Lagrangians.  
    Here we study how this `microsheaf quantization' interacts with composition of Lagrangian correspondences.  In particular: these operations commute when the composition is embedded.   As an illustration, we show that Lie groups of exact symplectomorphisms act on microsheaf categories. 
    
    The key technical advance is a version `in families' of the gappedness criterion for commuting nearby cycles past tensor or Hom. 
\end{abstract}

\maketitle

\thispagestyle{empty}

\renewcommand{\contentsname}{}
\tableofcontents
\thispagestyle{empty}

\newpage

\input{intro}

\input{applications}

\input{pdfl}

\input{nearbycyclereview}

\input{microsheafreview}

\input{compatibilities}

\input{gappedsheafcomposition}

\input{gappedmicrosheafcomposition}

\input{contactproduct}

\input{lagrangiancomposition}

\appendix

\input{oh-tanaka-comparison}

\input{extras}

\input{tamarkin}

\bibliographystyle{amsplain}
\bibliography{ref_LNS}

\end{document}

%% file: intro.tex

\section{Introduction}

As Weinstein explained, symplectic geometry is organized by a `category' whose objects are symplectic manifolds, whose morphisms are Lagrangian correspondences, and in which composition is set-theoretic composition of correspondences \cite{Weinstein-symplectic}.  This picture was partly motivated by the idea that geometric quantization  \cite{Hormander,Kostant, Souriau} should be a `functor' from some version of this symplectic `category' to the category of Hilbert spaces. 

More recently, the association from a symplectic manifold to its Fukaya category has begun to be extended similarly to a `functor' from the symplectic category \cite{Wehrheim-Woodward, LekiliLipyanskiy, MauWehrheimWoodward, Fukaya-correspondences} to some category of categories.  This has had many applications, though foundational work remains to be done to establish the expected 2-categorical properties (see for instance  \cite{Abouzaid-Bottman}).    

There is another,  rather older, such functor, albeit defined on a rather limited class of symplectic manifolds and an unusual class of Lagrangians: to a cotangent bundle $T^*M$ of a complex manifold $M$, associate the category of $\mathcal D$-modules $\mathcal{D}(M)$; from a conic, generally singular, Lagrangian $L \subset T^*M \times T^*N$ and the data of a $\mathcal D$-module with characteristic cycle $L$, one obtains an integral transform functor $\mathcal{D}(M) \to \mathcal{D}(N)$,  and the characteristic cycle of the composition of such functors is the set-theoretic composition of the characteristic cycles \cite{SKK, Kashiwara-thesis,Bernstein}.  Shortly thereafter, by replacing the notion of characteristic cycle with that of microsupport of an arbitrary sheaf, Kashiwara and Schapira gave a similar structure on cotangent bundles of real manifolds \cite{KS}.

\vspace{2mm}
Our purpose here is to extend the Kashiwara--Schapira theory from cotangent bundles to the class of  exact symplectic manifolds -- also named for Weinstein -- admitting a Morse function for which the Liouville vector field is the gradient flow.  We explained already in \cite{Shende-h-principle, Nadler-Shende} how to associate categories of microsheaves to such Weinstein manifolds. These microsheaf categories are now known to be equivalent to the wrapped Fukaya categories of the same manifolds \cite{Ganatra-Pardon-Shende3}. 
In the present article, we will construct functors  from exact immersed Lagrangian correspondences, and study their compositions.

As illustrations, we construct Lagrangian correspondence incarnations of the Viterbo restriction, of the contactomorphism action on microsheaves, and of the Lie group actions by exact symplectomorphisms.  

Let us mention two advantages that the present methods have over
presently available Floer theoretic methods.  First, we may consider singular Lagrangians and Lagrangian correspondences.   Second, various categorical desiderata, such as 2-categorical coherences and the ability to work over the sphere spectrum, are essentially automatic given modern sheaf theoretical foundations \cite{Lurie-HTT,Gaitsgory-Rozenblyum,Volpe-six-operations,LoubatonRuit,Scholze-sixfunctor}. 

\vspace{2mm}

We proceed to describe our setup and state our results. 
Fix a compactly generated\footnote{The non-characteristic deformation lemma   \cite[Proposition 2.7.2]{KS}, fundamental in the microlocal theory of sheaves, holds for sheaves valued in compactly generated stable categories \cite{Robalo-Schapira}, but fails more generally for sheaves in presentable stable categories \cite[Remark 4.26]{EfimovK}.}
symmetric monoidal stable $(\infty, 1)$-category of coefficients $\mathcal{C}$, such as 
the category of chain complexes of modules over a commutative ring, localized along quasi-isomorphisms; or such as the category of spectra. Henceforth all categories will be infinity categories, all functors will mean their derived version, etc. 
For a manifold $M$, we write $\Sh(M)$ for the category of sheaves on $M$ valued in $\mathcal{C}$. For $\SF \in \Sh(M)$, we write $\ss(\SF) \subset T^*M$ for the microsupport as defined by Kashiwara and Schapira;  it is a conic co-isotropic subset which records the failure of propagation of sections in the corresponding codirections \cite{KS}.  

Let $(W, \lambda)$ be an exact symplectic manifold; the Liouville vector field $Z$ defined by $d\lambda(Z, -) = \lambda$ gives it a conic structure. In \cite{Shende-h-principle, Nadler-Shende}, we explained that, given a null homotopy $\sigma$ of a certain canonical map $W \to B^2 {Pic}(\mathcal{C})$ -- we call it `Maslov data' -- which we will fix and mostly omit from the notation, there is a canonically associated conic sheaf of categories $\msh_W$ on $W$, generalizing the categories of microsheaves on cotangent bundles introduced by Kashiwara and Schapira \cite{KS}. 

For a conic subset $K$ of $W$, we write $\mu sh_{K} \subset \mu sh_W$ for the  subsheaf of full subcategories on objects supported in $K$. When $W$ is compact and the Liouville flow is outwardly transverse to its boundary (such $W$ are termed `Liouville domains'), we write $\mathfrak{c}_W$ for the locus of points in $W$ which do not escape under the Liouville flow.  More generally,  
for $\Lambda \subset \partial W$, we write $\mathfrak{c}_{W, \Lambda} := \mathfrak{c}_W \cup Cone(\Lambda)$ for the locus of points which do not escape to the complement of $\Lambda$.  When $\mathfrak{c}_{W, \Lambda}$ is `sufficiently Lagrangian', it is known that  
$\Gamma(\msh_{\mathfrak{c}_{W, \Lambda}})$ is equivalent to the Fukaya category of $W$ with wrapping stopped by $\Lambda$
\cite{Ganatra-Pardon-Shende3}. 

An apparent tension between microsheaf methods and symplectic topology is that subsets immediately visible to the former are conic (because microsupports are conic) whereas the subsets of most interest in symplectic topology are typically not conic.  However, as Tamarkin borrowed from semiclassical analysis, (exact) nonconic Lagrangians can be expressed in terms of conic Lagrangians in one dimension more \cite{Tamarkin1}.

Let $L \subset W \times \R$ be a smooth Legendrian whose projection to $\R$ lies in $(-N, \infty)$; we assume always its Lagrangian projection $\bar{L} \subset W$ is (possibly not embedded but) eventually conic along an embedded Legendrian $\partial L$. 
Consider the locus
\begin{equation} \label{defining diagram of sheaf quantization}
\mathfrak{c}_W^L := (\mathfrak{c}_{W,\partial L} \times (-N, \infty)) \cup Cone(L) \subset W \times T^* \R
\end{equation}
There are restriction functors: 
\begin{equation} \label{sheaf quantization diagram} \Gamma(\msh_L) \leftarrow \Gamma(\msh_{\mathfrak{c}_W^L}) \rightarrow \Gamma(\msh_{\mathfrak{c}_{W,\partial L}}).\end{equation}
The left arrow is the restriction near $L$,
and its image consists of (Maslov twisted) local systems on $L$; in other words, the only input from symplectic geometry to this category $\Gamma(\msh_L)$ is through the Maslov obstruction \cite{Guillermou-survey,Jin2, Nadler-Shende}.

We call objects in the middle category $\Gamma(\msh_{\mathfrak{c}_W^L})$  {\em conic microsheaf quantizations} of $L$, and objects in the image of the right arrow {\em conic microsheaf quantizations of $\overline{L}$}. 
The term `quantization' can be understood as follows: for cotangent bundles $W = T^*M$ carrying the canonical 1-form, the core $\mathfrak{c}_{T^*M}$ is the zero section $M \subset T^*M$; one has
$\mathfrak{c}_{T^*M, \partial L} = M \cup Cone(\partial L)$, and the (nearly tautological) identification $\Gamma(\msh_{M \cup Cone(\partial L)}) = \Sh_{M \cup Cone(\partial L)}(M)$.
Thus a conic (micro)sheaf quantization of a Lagrangian $\bar{L} \looparrowright T^*M$ is a  sheaf on the base $M$, 
just as a geometric quantization would be a function (or distribution) on the base.

There is a fundamental existence result for conic sheaf quantizations of {\em embedded} Lagrangians, 
which asserts that the left arrow of Formula \eqref{sheaf quantization diagram} is an equivalence, and the right arrow is fully faithful; or in other words, that conic sheaf quantizations are characterized by the corresponding Maslov data.\footnote{The same is true by definition for objects in the Fukaya category of $W$ arising from $L$.  The sheaf/Fukaya correspondence of \cite{Ganatra-Pardon-Shende3} intertwines these assertions; see \cite[Proposition 11]{Shende-monodromy}.}  
This is 
due to Guillermou for $W = T^*M$ and smooth $L$ \cite{Guillermou-survey}, and was generalized to Weinstein manifolds $W$ and a class of singular Legendrians $L$ in \cite{Nadler-Shende}.  (More precisely, the main result of  \cite{Nadler-Shende} can be reformulated as Formula \eqref{sheaf quantization diagram}; we explain in Theorem \ref{thm: nearby = restriction at infty}.)


We turn to composition of Lagrangian correspondences.  Kashiwara and Schapira showed that, under appropriate compactness hypotheses, the  `composition of integral kernels'  $$\circ: \Sh(M_1 \times M_2) \times \Sh(M_2 \times M_3) \to \Sh(M_1 \times M_3)$$ satisfies the microsupport estimation
$\ss(\SF_{23} \circ \SF_{12}) \subset \ss(\SF_{23}) \circ \ss(\SF_{12})$.  
They further considered descending the composition of kernels to the setting of microsheaves, which however required certain strong hypotheses \cite[Chapter 7.3]{KS}, not suitable for our desired applications to symplectic topology.  

In the present article, we will construct -- using nontrivial geometric arguments and requiring various isotropicity and displaceability hypotheses -- a composition of microsheaf kernels in more general settings. Our functor is given in Definition \ref{def: gap composition weinstein}; we show it is associative in Proposition \ref{prop: gapped composition associative weinstein}. 
A key feature of our composition is that it is compatible with the sheaf quantization diagram \eqref{sheaf quantization diagram} in the following sense:

\begin{customthm}{I}[Theorem \ref{thm: main composition}]
 \label{sheaf quantization commutes with composition}
    Fix Weinstein manifolds $W_1, W_2, W_3$ and eventually conic Legendrians $L_{12} \subset W_{1}^- \times W_2 \times \R$ and $L_{23} \subset W_2^- \times W_{3} \times \R$ conic at infinity.  $L_{13} \subset W_{1}^- \times W_{3} \times \R$ is a certain `contact composition' of the subsets $L_{12}$ and $L_{23}$,  given in Definition \ref{def: contact composition final}.
    We require the $L_{ij}$ are `sufficiently Legendrian', defined in Definition \ref{def: universal legendrian}.
    
    Then there are associative composition functors $\circ$ which fit into a commutative diagram 
    
    \[\begin{tikzcd}
    \Gamma(\msh_{L_{12}}) \otimes \Gamma(\msh_{L_{23}})  \ar[r, "\circ"]
    & 
    \Gamma(\msh_{L_{13}})
    \\
    \Gamma( \msh_{\mathfrak{c}_{W_1^- \times W_2}^{L_{12}}})
    \otimes 
    \Gamma(
    \msh_{\mathfrak{c}_{W_2^- \times W_3}^{L_{23}}})
    \ar{u} \ar{d} \ar[r, "\circ"] 
    & 
    \Gamma(\msh_{\mathfrak{c}_{W_1^- \times W_3}^{L_{13}}}) \ar[u] \ar[d]
    \\ 
    \Gamma(\msh_{\mathfrak{c}_{W_1^- \times W_2, \partial L_{12}}})
    \otimes 
    \Gamma(\msh_{\mathfrak{c}_{W_2^- \times W_3, \partial L_{23}}}) \ar[r,"\circ"] & \Gamma(\msh_{\mathfrak{c}_{W_1^- \times W_3,\partial L_{13}}})
    \end{tikzcd}\]
    When $L_{13}$ is a smooth Legendrian, the top horizontal arrow is a pullback of (twisted) local systems along the map $L_{13} \hookrightarrow L_{12} \times L_{23}$. 

    A corresponding result holds after enlarging $\mathfrak{c}_{W_i}$ to any sufficiently Lagrangian $\mathfrak{c}_{W_i, \Lambda_i}$ so long as the $L_{ij}$ are disjoint at infinity from the products $\mathfrak{c}_{W_i, \Lambda_i} \times \mathfrak{c}_{W_j, \Lambda_j}$.
\end{customthm}


In words: the commuting of the lower square says that if $\SF_{12}$ is a conic microsheaf quantization of $L_{12}$ and $\SF_{23}$ is a conic microsheaf quantization of $L_{23}$, then $\SF_{23} \circ \SF_{12}$ is a conic microsheaf quantization of $L_{23} \circ L_{12}$. Recall that if the the projections $L_{ij} \to W_i^- \times W_j$ are all embeddings, 
then the up arrows are equivalences and the down arrows are fully faithful.  Thus, in this case, the commuting of the upper square implies that the compositions can be computed by the (much simpler) top horizontal arrow.  In particular, when the $L_{ij}$ are all smooth, the top horizontal arrow is just the pullback along the appropriate inclusion of a tensor product of local systems. 

The deepest assertion in Theorem \ref{sheaf quantization commutes with composition} is the commutativity of the lower square, which is new even when all the $W_i$ are cotangent bundles.  The key technical result needed here is a criterion to commute  composition of sheaf theoretic correspondences with nearby cycles, proven in Theorem \ref{thm:gapped composition}.  This is then transported to a corresponding result in Theorem \ref{thm:gapped composition microsheaf} in the microsheaf-in-cotangent-bundle setting using the doubling construction.  Finally we pass by an appropriate high codimension embedding (as in \cite{Shende-h-principle, Nadler-Shende}) to the setting of Weinstein manifolds.  We also prove a parallel version for the `Hom' (rather than tensor) version of composition. 

Theorem \ref{sheaf quantization commutes with composition}  is  analogous to Fukaya's result constructing a composition of bounding cochains for a geometric composition of Lagrangians \cite{Fukaya-correspondences}.
It becomes a precise counterpart if one replaces the Akaho-Joyce \cite{Akaho-Joyce} (curved) $A_\infty$ algebra associated to an immersed Lagrangian $\bar{L}$ by the $A_\infty$ algebra generated by positive self Reeb chords for the associated Legendrian $L$.

\vspace{2mm}

In typical applications of Theorem \ref{sheaf quantization commutes with composition} (we give several in Section \ref{sec:example}), one wants to use $L_{ij}$ to define a morphism, i.e. an element of 
$$\Fun^\mathrm{L}(\Gamma(\msh_{\mathfrak{c}_{W_i}}), \Gamma(\msh_{\mathfrak{c}_{W_j}}) ) = \Gamma(\msh_{\mathfrak{c}_{W_i}})^\vee \otimes \Gamma(\msh_{\mathfrak{c}_{W_j}}) = \Gamma(\msh_{\mathfrak{c}_{W_i^- \times W_j}}).$$
(These identifications are formal consequences of various results on dualizability of categories of sheaves with sufficiently Legendrian microsupport; see Theorem \ref{thm: weinstein kunneth duality maslov}.)
There is a map:
\begin{equation} \label{stop removal of kernel}
    \Gamma(\msh_{\mathfrak{c}_{W_i^- \times W_j, \partial L_{ij}}}) \to \Gamma(\msh_{\mathfrak{c}_{W_i^- \times W_j}})
\end{equation} 
namely the left adjoint of the inclusion in the opposite direction.  
(The sheaf/Fukaya correspondence of \cite{Ganatra-Pardon-Shende3} identifies such maps with stop removals.)  

\begin{customthm}{II}
[Theorem \ref{thm: main composition stop removal}] \label{stop removal commutes with composition} 
    Preserving the notation and hypotheses of Theorem \ref{sheaf quantization commutes with composition}, assume in addition that
    \begin{equation} \label{stop removal hypothesis}\mathfrak{c}_{W_i \times W_j, \partial L_{ij}} \circ \mathfrak{c}_{W_i} \subset \mathfrak{c}_{W_j}.
    \end{equation} Then there is a commutative diagram:
    \[\begin{tikzcd}
    \Gamma(\msh_{\mathfrak{c}_{W_1^- \times W_2,\partial L_{12}}})
    \otimes 
    \Gamma(\msh_{\mathfrak{c}_{W_2^- \times W_3, \partial L_{23}}}) \ar[r, "\circ"] \ar[d] & \Gamma(\msh_{\mathfrak{c}_{W_1^- \times W_3,\partial L_{13}}}) \ar[d] \\
    \Gamma(\msh_{\mathfrak{c}_{W_1^- \times W_2}})
    \otimes 
    \Gamma(\msh_{\mathfrak{c}_{W_2^- \times W_3}}) \ar[r, "\circ"] & \Gamma(\msh_{\mathfrak{c}_{W_1^- \times W_3}})
    \end{tikzcd}\]
 Moreover, the bottom horizontal arrow can be identified with the natural functor
    $$ \Gamma(\msh_{\mathfrak{c}_{W_1}})^\vee \otimes \Gamma(\msh_{\mathfrak{c}_{W_2}}) \otimes \Gamma(\msh_{\mathfrak{c}_{W_2}})^\vee \otimes \Gamma(\msh_{\mathfrak{c}_{W_3}}) \to \Gamma(\msh_{\mathfrak{c}_{W_1}})^\vee \otimes \Gamma(\msh_{\mathfrak{c}_{W_3}})$$
\end{customthm}

We note that the condition in Formula \eqref{stop removal hypothesis} can always be satisfied after  enlarging $\mathfrak{c}_{W_j}$ to $\mathfrak{c}_{W_i\times W_j,\partial L_{ij}} \circ \mathfrak{c}_{W_i}$, and replacing $L_{ij}$ by a small negative pushoff.

The main point in the proof of Theorem \ref{stop removal commutes with composition} is to show that, for the bottom arrow, the geometric definition of composition agrees with the formal one.  This is established in Theorem \ref{thm: gap composition wrap kunneth}, using our criterion for commuting nearby cycles and compositions (Theorem \ref{thm:gapped composition}).  

Finally, we mention that in the main body of the paper, the higher coherence of the diagrams in Theorems \ref{sheaf quantization commutes with composition} and \ref{stop removal commutes with composition} are also obtained (in Corollaries \ref{thm: main composition higher} and \ref{thm: main composition stop removal higher}).

\addtocontents{toc}{\SkipTocEntry}
\subsection*{Acknowledgments}
    We would like to thank Sheel Ganatra, Peter Haine, Chris Kuo, Tatsuki Kuwagaki, Nick Rozenblyum, Harold Williams, Bingyu Zhang and Peng Zhou for helpful discussions. WL is partially supported by the AMS-Simons Travel Grant.
    DN is supported by NSF DMS grant 2401178.  VS is supported by the 
    VILLUM FONDEN grant VILLUM Investigator 37814,  Novo Nordisk Foundation 
grant NNF20OC0066298, and Danish National Research Foundation grant DNRF157.  This work began over a conversation at the  Centre de Recherches
Math\'ematiques, where V.S. was supported by the Simons-CRM scholar-in-residence program. 

\addtocontents{toc}{\SkipTocEntry}
\subsection*{Conventions}

    When we write formulae stepping through sheaf identities, we use $=$ to mean that we have done something trivial or applied standard six-functor type identities including proper or smooth base change. We assume throughout all sheaves are compactly supported; as usual, this may often be replaced by requiring ``tameness at infinity''.   
    
    We write $0_M \subset T^*M$ for the zero section, $\dot T^*M = T^*M \setminus 0_M$ for the punctured cotangent bundle, and $S^*M = \dot T^*M / \R_{>0}$. We write $\Sh(M)$ for the category of sheaves on $M$ valued in a compactly generated symmetric monoidal stable category $\SC$.  In this setting, there is a 6-functor formalism \cite{Volpe-six-operations, Scholze-sixfunctor}, and the noncharacteristic deformation lemma holds \cite{Robalo-Schapira}.     
    For $\SF \in \Sh(M)$, we write $\ss(\SF)$ for the microsupport and $\dot{\ss}(\SF) := \ss(\SF) \cap \dot T^*M$; as this locus is conic, its projection to  $S^*M = \dot T^*M / \R_{>0}$ carries the same information, so we use $\dot{\ss}(\SF)$ also for this projection.

    We write $\Cat$ for the category of all (not necessarily small) categories, $\PrL$ for the category of presentable categories with colimit preserving functors, denoted by $\Fun^\mathrm{L}(-,-)$, and $\PrR$ for the category of presentable categories with limit preserving functors, denoted by $\Fun^\mathrm{R}(-,-)$. Respectively, $\Cat_{st}$ is the category of all stable categories with exact functors, and $\PrLst$ and $\PrRst$ are the category of presentable stable categories. Finally, $\operatorname{Pr}^\mathrm{L}_\omega$ (resp.~$\operatorname{Pr}^\mathrm{R}_\omega$) consists of compactly generated presentable categories with colimit preserving functors that preserve compact objects (resp.~limit and colimit preserving functors).

%% file: applications.tex

\section{Sample applications}\label{sec:example}

    In this section, we illustrate our main theorems by studying functors associated to  embeddings of Weinstein domains, contact isotopies on the contact boundaries of Weinstein domains, and Hamiltonian group actions.

    The Lagrangians (and Lagrangian correspondences) we consider in this section are all exact, in the sense of being the bijective images of Legendrians. 
    In this case, the microsheaf quantization theorem  of \cite{Nadler-Shende} (recalled and reformulated as Theorem \ref{thm: nearby = restriction at infty}  below) asserts 
    \begin{equation} \label{quantization setup}
    \msh_{L}(L) \xleftarrow{\sim} \msh_{\mathfrak{c}_W^L}(\mathfrak{c}_W^L) \hookrightarrow \msh_{\mathfrak{c}_{W,\partial L}}(\mathfrak{c}_{W,\partial L}).    
    \end{equation}
    We denote the composition as
    \begin{equation}\label{eq: quantization}
    \psi: \msh_{L}(L) \hookrightarrow \msh_{\mathfrak{c}_{W,\partial L}}(\mathfrak{c}_{W,\partial L}).
    \end{equation}
    Moreover, given a (compactly supported) contact Hamiltonian isotopy in $W \times \bR$, we have the following commutative diagram (formulated in Corollary \ref{cor: Hamilton invariance quantization} below): 
    \begin{equation}\label{eq: Hamilton invariance quantization}
    \begin{tikzcd}
        \msh_{L_{0}}(L_0) & \msh_{\mathfrak{c}_{W,\sigma}^{L_0}}(\mathfrak{c}_W^{L_0}) \ar[l, "i_{L_0}^*" above] \ar[r, "i_\infty^*"] & \msh_{\mathfrak{c}_{W,\partial L_0}}(\mathfrak{c}_{W,\partial L_0}) \\
        \msh_{L_H}(L_H) \ar[u, "\rotatebox{90}{$\sim$}"] \ar[d, "\rotatebox{90}{$\sim$}" left] & \msh_{\mathfrak{c}_{W\times T^*I}^{L_H}}(\mathfrak{c}_{W\times T^*I}^{L_H}) \ar[l, "i_{L_H}^*" above] \ar[r, "i_\infty^*"] \ar[d, "\rotatebox{90}{$\sim$}"] \ar[u, "\rotatebox{90}{$\sim$}" right] & \msh_{\mathfrak{c}_{W \times T^*I,\partial L_H}}(\mathfrak{c}_{W \times T^*I,\partial L_H}) \ar[d, "\rotatebox{90}{$\sim$}"] \ar[u, "\rotatebox{90}{$\sim$}" right] \\
        \msh_{L_{1}}(L_1) & \msh_{\mathfrak{c}_{W}^{L_1}}(\mathfrak{c}_W^{L_1}) \ar[l, "i_{L_1}^*" above] \ar[r, "i_\infty^*"] & \msh_{\mathfrak{c}_{W,\partial L_1}}(\mathfrak{c}_{W,\partial L_1}).
    \end{tikzcd}
    \end{equation}


\subsection{Viterbo functor}

    We recall that an exact symplectic manifold-with-boundary $(W, \lambda)$ is said to be a Liouville domain if the corresponding Liouville vector field is outward 
    pointing along $\partial W$; in this case $\lambda|_{\partial W}$ is a contact form and we call $(\partial W, \lambda|_{\partial W})$ the contact boundary.  
    Given Liouville domains $(W_0, \lambda_0)$ and $(W_1, \lambda_1)$, a Liouville domain inclusion is a map $i_{01}: W_0 \hookrightarrow W_1$ such that $i_{01}^*\lambda_1 = \lambda_0 + df_{0}$. 
    Said differently, $i_{01} \times f: W_0 \to W_1 \times \R$ pulls back the contact form on the contactization to Liouville form $\lambda_0$.   We fix Maslov data on $W_1$ and, when appropriate, pull it back to give Maslov data on $W_0$. 
    
    When $W_0$ and $W_1$ are Weinstein, we get an instance of Formula \eqref{quantization setup} where $W:= W_1$ and $L := (i_{01} \times f)(\mathfrak{c}_{W_0})$.   We denote the corresponding instance of Formula \eqref{eq: quantization} 
    by
    $$i_{01*}: \msh_{\mathfrak{c}_{W_0}}(\mathfrak{c}_{W_0})  \hookrightarrow \msh_{\mathfrak{c}_{W_1}}(\mathfrak{c}_{W_1}).$$

    The above construction of $i_{01*}$ corresponds under the sheaf/Fukaya dictionary to (pullback along) Sylvan's  Viterbo transfer (see \cite[Section 8.3]{Sylvan2,Ganatra-Pardon-Shende2}). 
    When $W_0 = W_1$ and $\lambda_0$ and $\lambda_1$ are related by an exact deformation, the above map is used to compare the corresponding microsheaf categories in \cite[Section 9.2]{Nadler-Shende}. 

    Here we would like to show that the Viterbo transfer is implemented by a Lagrangian correspondence. Thus consider the 
    graph of the embedding 
    $\Gamma_{01} = \{(x_0, i_{01}(x_0)) \mid x_0 \in W_0\}$.  It is an exact Lagrangian, but it is not cylindrical with respect to the product Liouville form $-\lambda_0 \times \lambda_1$ on the Weinstein manifold $W_0^- \times W_1$. (It is cylindrical with respect to the Liouville form $-\lambda_1 \times \lambda_1$.)
    We may repair this if we further restrict attention to embeddings such that $f_0$ is a smooth function such that $Z_{\lambda_0}f_0 = 0$ near the contact boundary $\partial W_0$. 
    Under this assumption, we consider the Hamiltonian flow $\varphi_{f_0}^t$ associated to $f_0$ and define 
    $$\Gamma_{01}' = \{(\varphi_{f_0}^1(x_0), i_{01}(x_0)) \mid x_0 \in W_0\}.$$
    This is an exact Lagrangian with cylindrical ends with respect to the Liouville form $-\lambda_0 \times \lambda_1$.

\begin{proposition}\label{prop: subdomain}
    Given a Liouville subdomain embedding of Weinstein manifolds $i_{01}: W_0 \hookrightarrow W_1$, the image of the constant rank one local system under 
    \begin{align*}
    \msh_{\Gamma_{01}'}(\Gamma_{01}') &\hookrightarrow \msh_{\mathfrak{c}_{W_0^-} \times \mathfrak{c}_{W_1}, \partial \Gamma_{01}' }(\mathfrak{c}_{W_0^-} \times \mathfrak{c}_{W_1}) \to \msh_{\mathfrak{c}_{W_0^-} \times \mathfrak{c}_{W_1}}(\mathfrak{c}_{W_0^-} \times \mathfrak{c}_{W_1}) \\
    &= \Fun^\mathrm{L}\big(\msh_{\mathfrak{c}_{W_0}}(\mathfrak{c}_{W_0}), \msh_{\mathfrak{c}_{W_1}}(\mathfrak{c}_{W_1})\big)
    \end{align*}
    is the functor $i_{01*}: \msh_{\mathfrak{c}_{W_0}}(\mathfrak{c}_{W_0}) \to \msh_{\mathfrak{c}_{W_1}}(\mathfrak{c}_{W_1})$.
\end{proposition}
\begin{proof}
    Let $\lambda_t = \lambda_0 + t df_0$ for $t\in I$. Consider the Liouville form $-\lambda_t \times \lambda_1 \times \lambda_{T^*I}$ on $W_0^- \times W_1 \times T^*\bR$. By \cite[Proposition 2.42]{LazarevSylvanTanaka1}, there is a compactly supported deformation of the Liouville structure $-\lambda_t \times \lambda_1 \times \lambda_{T^*I}$ to a Weinstein structure $\lambda_I$ on $W_0^- \times W_1 \times T^*\bR$. We fix this Weinstein structure $\lambda_I$ on $W_0^- \times W_1 \times T^*\bR$ and denote the core by $\mathfrak{c}_{W_0^- \times W_1 \times T^*I,\lambda_I}$. Let $\varphi_{f_0}^t: W_0 \to W_0$ be the Hamiltonian flow associated to $f_0$. Consider the following Legendrian that is cylindrical at infinity with respect to the Weinstein structure:
    $$\Gamma_{01,I} = \{(\varphi_{f_0}^t(x_0), i_{01}(x_0), t, -tf_0(x_0)) \mid x_0 \in W_0, t \in I\} \subset W_0^- \times W_1 \times T^*\bR.$$
    Denote the symplectic reduction at $t = 0$ by $\Gamma_{01}$ and its reduction at $t = 1$ by $\Gamma_{01}'$.
    Then we have a commutative diagram as a special case of Formula \eqref{eq: Hamilton invariance quantization} 
    \[\begin{tikzcd}
    \Gamma(\msh_{\Gamma_{01}'}) \ar[r] & \Gamma(\msh_{\mathfrak{c}_{W_0^-\times W_1,\partial \Gamma_{01}'}}) \ar[r] & \Gamma(\msh_{\mathfrak{c}_{W_0^-\times W_1,-\lambda_0\times\lambda_1}}) \\
    \Gamma(\msh_{\Gamma_{01,I}}) \ar[r] \ar[u, "i_1^*"] \ar[d, "i_0^*" left] & \Gamma(\msh_{\mathfrak{c}_{W_0^-\times W_1 \times T^*I,\partial \Gamma_{01,I}}}) \ar[r] \ar[u] \ar[d] & \Gamma(\msh_{\mathfrak{c}_{W_0^- \times W_1 \times T^*I,\lambda_I}}) \ar[u, "i_1^*" right] \ar[d, "i_0^*"] \\
    \Gamma(\msh_{\Gamma_{01}}) \ar[r] & \Gamma(\msh_{\mathfrak{c}_{W_0^-\times W_1,\partial \Gamma_{01}}}) \ar[r] & \Gamma(\msh_{\mathfrak{c}_{W_0^-\times W_1,-\lambda_1\times \lambda_1}}).
    \end{tikzcd}\]
    The composition of the two right vertical arrows is the microsheaf quantization functor, which, by dualizability and K\"unneth formula, is realized by $i_{01*}^\vee \otimes \id: \msh_{\mathfrak{c}_{W_0,\lambda_1}}(\mathfrak{c}_{W_0,\lambda_1})^\vee \otimes  \msh_{\mathfrak{c}_{W_1,\lambda_1}}(\mathfrak{c}_{W_1,\lambda_1}) \to \msh_{\mathfrak{c}_{W_0,\lambda_0}} (\mathfrak{c}_{W_0,\lambda_0})^\vee \otimes \msh_{\mathfrak{c}_{W_1,\lambda_1}}(\mathfrak{c}_{W_1,\lambda_1})$. Therefore, it sends a functor $\varphi_{01}: \msh_{\mathfrak{c}_{W_0,\lambda_1}}(\mathfrak{c}_{W_0,\lambda_1}) \to \msh_{\mathfrak{c}_{W_1,\lambda_1}}(\mathfrak{c}_{W_1,\lambda_1})$ to the composition
    $$\msh_{\mathfrak{c}_{W_0,\lambda_0}}(\mathfrak{c}_{W_0,\lambda_0}) \xrightarrow{i_{01*}} \msh_{\mathfrak{c}_{W_0,\lambda_1}}(\mathfrak{c}_{W_0,\lambda_1}) \xrightarrow{\varphi_{01}} \msh_{\mathfrak{c}_{W_1,\lambda_1}}(\mathfrak{c}_{W_1,\lambda_1}).$$
    Then, since the constant rank 1 local system on $\Gamma_{01}$ is the diagonal in $W_0^- \times W_0 \subset W_0^- \times W_1$, the image under the bottom horizontal arrow is the naive inclusion functor. Therefore, by commutativity of the diagram, the image of the constant rank 1 local system on $\Gamma_{01}'$ under the top horizontal arrow is the composition of $i_{01*}$ with the natural inclusion.
\end{proof}

In particular, we recover the functoriality of the embedding functor (which was not explicitly shown in \cite{Ganatra-Pardon-Shende2} but can be established using the methods there, then translated by the sheaf/Fukaya comparison \cite{Ganatra-Pardon-Shende3}; the result in the sheaf context was also established by a manipulation of vanishing cycles in  \cite[Theorem 1.1]{LiFunctoriality}). 

\begin{corollary} \label{functoriality of embedding}
      Given Liouville subdomain embeddings 
      $W_0 \hookrightarrow W_1 \hookrightarrow W_2$ of Weinstein domains with embeddings denoted by $i_{01}, i_{12}$ and $i_{02}$, we have $i_{12*} \circ i_{01*} \simeq i_{02*}$. 
\end{corollary}
\begin{proof}
    We can understand the assertion as an identity happening at the level of the bottom row of Theorem \ref{stop removal commutes with composition}.  The identifications of Proposition \ref{prop: subdomain} (and the fact that in the case at hand, the Lagrangian correspondences are embedded), along with Theorems \ref{sheaf quantization commutes with composition} and \ref{stop removal commutes with composition} mean that said identity can be deduced from an identity at the level of the top row of Theorem \ref{sheaf quantization commutes with composition}. In the case at hand, Proposition \ref{prop: subdomain} asserts that the required identity is between the constant rank one local system on $\Gamma_{02}'$ and a pullback of an external product of constant rank one local systems $\Gamma_{12}' \times \Gamma_{01}'$ under
    $$\Gamma_{02}' \hookrightarrow \Gamma_{12}' \times \Gamma_{01}',$$
    and hence they are obviously isomorphic. Moreover, we can choose the isomorphism canonically to be the identity between constant rank one local systems and in particular we have functorially with respect to further composition.
\end{proof}



    

\subsection{Sheaf quantization of contact isotopy}
    Let $W$ be a Weinstein domain with contact boundary $\partial W$. Consider a contact Hamiltonian isotopy $\varphi_{H_\infty}^t: \partial W \to \partial W$; extend it to a symplectic Hamiltonian isotopy $\varphi_H^t: W \to W$ by a cut-off function supported in a cylindrical neighborhood of $\partial W$. Consider the Lagrangian graph of the Hamiltonian isotopy
    $$\Gamma_{H} = \{(x, \varphi_H^t(x), t, -H \circ \varphi_H^t(x)) \mid x \in W, t\in I\} \subset W^- \times W \times I.$$
    which is an exact and cylindrical near boundary. Similarly, for $\Lambda \subset \partial W$, we can define
    $$\Lambda_H = \{(x, t, -H \circ \varphi_H^t(x)) \mid x \in \Lambda, t\in I\} \subset \partial W \times T^*I.$$

    We fix Maslov data on $W$, and consider the opposite Maslov data on $W^-$.  
    Then by considering the natural secondary Maslov data on the diagonal, there is a natural choice of secondary Maslov data on $\Gamma_{H}$, similar to \cite[Equation 11]{CKNS2}. This induces an equivalence between microsheaves and local systems on $\Gamma_{H}$. We define $\SK_{H} \in \Gamma(\msh_{\Gamma_{H}})$ to be the microsheaf associated to the constant rank one local system.  

\begin{proposition}\label{prop: GKS}
    Let $W$ be a Weinstein domain with contact boundary $\partial W$ and $\Lambda \subset \partial W$ be a Whitney stratified Legendrian. Let $\varphi_{H}^t: W \to W$ be a Hamiltonian isotopy that restricts to a contact isotopy $\varphi_{H_\infty}^t: \partial W \to \partial W$.  Then the image of the constant rank one local system under
    \begin{align*}\msh_{\Gamma_H}(\Gamma_H) &\hookrightarrow \msh_{\mathfrak{c}_{W^- \times W, \partial \Gamma_{H}}}(\mathfrak{c}_{W^- \times W}) \to \msh_{\mathfrak{c}_{W^-,\Lambda} \times \mathfrak{c}_{W \times T^*I,\Lambda_H}}(\mathfrak{c}_{W^-,\Lambda} \times \mathfrak{c}_{W \times T^*I,\Lambda_H}) \\
    &= \Fun^\mathrm{L}\big(\msh_{\mathfrak{c}_{W,\Lambda}}(\mathfrak{c}_{W,\Lambda}), \msh_{\mathfrak{c}_{W \times T^*I,\Lambda_H}}(\mathfrak{c}_{W \times T^*I,\Lambda_H})\big)
    \end{align*}
    is the identity when restricted to $t = 0$. Moreover, for any $J$-paramatric family of Hamiltonian isotopies that restrict to $\varphi_{H_\infty}^t: \partial W \to \partial W$, the functors assigned at $t\in I$ are canonically isomorphic.
\end{proposition}
\begin{proof}
    First, consider the restriction to $t = 0$. Then we have a diagram
    $$\msh_\Delta(\Delta) \hookrightarrow \msh_{\mathfrak{c}_{W^- \times W,\partial \Delta}}(\mathfrak{c}_{W^- \times W}) \to \msh_{\mathfrak{c}_{W^-,\Lambda} \times \mathfrak{c}_{W,\Lambda}}(\mathfrak{c}_{W^-} \times \mathfrak{c}_{W}).$$
    The constant rank 1 local system on the diagonal defines the identity functor.
    
    Consider any smooth $J$-family of extensions $H_J: W \times J \to \bR$ that are equal to $H_\infty$ on a neighborhood of $\partial W$. Assume that $J$ contains $0$ and $1$. Then we have a commutative diagram as a special case of Formula \eqref{eq: Hamilton invariance quantization}
    \[\begin{tikzcd}
     \Gamma(\msh_{\Gamma_{H_0}}) \ar[r]  & \Gamma(\msh_{\mathfrak{c}_{W^- \times W \times T^*I,\partial \Gamma_{H_0}}}) \ar[r] & \Gamma(\msh_{\mathfrak{c}_{W^-,\Lambda} \times \mathfrak{c}_{W \times T^*I,\Lambda_{H_0}}}) \\
    \Gamma(\msh_{\Gamma_{H_J}}) \ar[r] \ar[u, "i_0^*"] \ar[d, "i_1^*" left] & \Gamma(\msh_{\mathfrak{c}_{W^- \times W \times T^*I \times T^*J,\partial \Gamma_{H} \times J}}) \ar[r] \ar[u] \ar[d] & \Gamma(\msh_{\mathfrak{c}_{W^-,\Lambda} \times \mathfrak{c}_{W \times T^*I \times T^*J,\Lambda_{H} \times J}}) \ar[u, "i_0^*" right] \ar[d, "i_1^*"] \\
    \Gamma(\msh_{\Gamma_{H_1}}) \ar[r] & \Gamma(\msh_{\mathfrak{c}_{W^-} \times \mathfrak{c}_{W \times T^*I},\partial \Gamma_{H_1}}) \ar[r] & \Gamma(\msh_{\mathfrak{c}_{W^-,\Lambda} \times \mathfrak{c}_{W \times T^*I,\Lambda_{H_1}}}).
    \end{tikzcd}\]
    Under the identification that $\Gamma(\msh_{\Gamma_{H}}) \simeq \Loc(\Gamma_H)$, we know that the composition of the left vertical arrows is the identity. The composition of the right vertical arrows is also the identity. Therefore, we have a canonical isomorphism between the associated functors.
\end{proof}

\begin{remark}
    For an exact symplectic manifold $W$ with contact boundary $\partial W$, consider the group of compactly supported Hamiltonian diffeomorphisms $\mathrm{Ham}(W, \partial W)$, the group of Hamiltonian diffeomorphisms that restrict to contactomorphisms on the boundary $\mathrm{Ham}(W)$, and the group of contactomorphisms in the identity component $\mathrm{Cont}_0(\partial W)$. It is well known that $\mathrm{Ham}(W, \partial W) \to \mathrm{Ham}(W) \to \mathrm{Cont}_0(\partial W)$ is a Serre fibration (the homotopy lifting is defined by extending the contact isotopies on $\partial W$ to the interior $W$ by cut-off functions), and the above proposition shows that $\mathrm{Ham}(W, \partial W)$ acts trivially on the microsheaf category $\msh_{\mathfrak{c}_{W,\Lambda}}(\mathfrak{c}_{W,\Lambda})$ for any Legendrian stop $\Lambda$.
\end{remark}

\begin{definition}\label{def: GKS}
    Let $W$ be a Weinstein domain with contact boundary $\partial W$ and $\varphi_{H_\infty}^t: \partial W \to \partial W$ be a contact isotopy. Then we define the microsheaf quantization $\SK_{H_\infty}$ to be the image of the (constant) rank one microsheaf $\SK_H \in \msh_{\Gamma_H}(\Gamma_H)$ under
    $$\msh_{\Gamma_H}(\Gamma_H) \hookrightarrow \msh_{\mathfrak{c}_{W^- \times W \times T^*I, \partial \Gamma_{H}}}(\mathfrak{c}_{W^- \times W \times T^*I}),$$
    where $\varphi_H^t$ is any extension of $\varphi_{H_\infty}^t$ to a Hamiltonian isotopy.
\end{definition}

    Consider the Hamiltonian $G_\infty \# H_\infty$ whose contact isotopy is 
    $$\varphi_{G_\infty \# H_\infty}^t = \varphi_{G_\infty}^t \circ \varphi_{H_\infty}^t.$$
    We show functoriality of the sheaf quantizations of contact isotopies.

\begin{corollary}\label{cor: GKS kernel invertible} 
    Let $W$ be a Weinstein domain with contact boundary $\partial W$, and $\varphi_{G_\infty}^t, \varphi_{H_\infty}^t: \partial W \to \partial W$ be contact Hamiltonian isotopies. Then there is a  canonical isomorphism
    $$\SK_{G_\infty \# H_\infty} \simeq \SK_{G_\infty} \circ_I \SK_{H_\infty}.$$
    In particular, taking $G = -H \circ \varphi_H^1$, we see that such kernels are invertible. 
\end{corollary}
\begin{proof}
    We understand the assertion as an identity at the level of the bottom row of Theorem \ref{sheaf quantization commutes with composition}. The identification of Proposition \ref{prop: GKS}, along with Theorem \ref{sheaf quantization commutes with composition}, mean that said identity can be deduced from the corresponding identity at the level of the top row of Theorem \ref{sheaf quantization commutes with composition} that 
    $$\SK_{G} \circ_I \SK_{H} \simeq \SK_{G \# H}.$$
    By Proposition \ref{prop: GKS}, the above identity is between the constant rank one local system on $\Gamma_{G\#H}$ and the pullback of the external product of constant rank one local systems on $\Gamma_G \times \Gamma_H$ via
    $$\Gamma_{G\#H} \hookrightarrow \Gamma_G \times \Gamma_H,$$
    and hence they are isomorphic. Moreover, we can choose the isomorphism canonically to be the identity between constant rank one local systems and in particular we have functorially with respect to further composition.
\end{proof}

\begin{corollary}
     For any sufficiently Legendrian $\Lambda$ and $\Lambda_t = \varphi_{H_\infty}^t(\Lambda)$ in $\partial W$, denote the image of $\SK_{H_\infty}^t$ in $\Fun^L(\Gamma(\msh_{\mathfrak{c}_{W,\Lambda}}), \Gamma(\msh_{\mathfrak{c}_{W,\Lambda_t}}))$ by $\Phi_{H_\infty}^t$. Then, under the notation of Corollary \ref{cor: GKS kernel invertible}, we have a canonical isomorphism of functors $\Phi_{G_\infty}^t \circ \Phi_{H_\infty}^t \simeq \Phi_{G_\infty \# H_\infty}^t$.
\end{corollary}
\begin{proof}
    Follows from Corollary \ref{cor: GKS kernel invertible} and Theorem \ref{stop removal commutes with composition}.
\end{proof}

\subsection{Hamiltonian action and equivariant microsheaves}
    Recall that a Hamiltonian $G$-action on a symplectic manifold $W$ is equivalent data to a Lagrangian correspondence in $W^- \times W \times T^*G$ which is appropriately `group-like' for the set-theoretic composition given by the natural composition of correspondences in the $W^- \times W$ factor and the convolution in the $T^*G$ factor.  Explicitly, if the moment map is $\mu: W \to \mathfrak{g}^*$, then the Lagrangian is 
    $$\Gamma_\mu = \{(x, \varphi_\mu^g(x), g, -\mu(\varphi_\mu^g(x))) \mid x \in W, g \in G\} \subset W^- \times W \times G \times \mathfrak{g}^*.$$
    
    When $W$ is Liouville, we will say that the action is eventually conic if the Hamiltonian diffeomorphism $\varphi_\mu^g: W \to W$ is eventually conic for any $g \in G$. When the Hamiltonian action is eventually conic, the graph $\Gamma_\mu$ is an exact Lagrangian that is eventually conic.

    Let us discuss Maslov data.  
    If $G$ acts on $W$, then $G$ acts on the space of 
    null homotopies of the canonical map $W \to BU \to B^2Pic(\SC)$.  By $G$-equivariant Maslov data, we mean a $G$-equivariant null-homotopy of this map.  
    
    There is a corresponding notion of secondary $G$-equivariant Maslov data.  If we fix $G$-equivariant Maslov data $\sigma$ on $W$ and the corresponding opposite Maslov data on $W^-$, 
    similarly to \cite[Equation 11]{CKNS2}, there is canonical secondary $G$-equivariant Maslov data on the diagonal $\Delta$ and hence on $\Gamma_{\mu}$.
    This induces an equivalence 
    between microsheaves and local systems on $\Gamma_{\mu}$. We define $\SK_{\mu} \in \Gamma(\msh_{\Gamma_{\mu}})$ to be the microsheaf associated to the constant rank one local system.

\begin{proposition}\label{prop: group action}
    Let $W$ be a Weinstein domain with contact boundary $\partial W$. Let $G$ be a (finite dimensional) Lie group and $\mu: W \to \mathfrak{g}^*$ be an eventually conical exact Hamiltonian $G$-action.  Fix $G$-equivariant Maslov data. Then the image of the constant rank one local system under
    \begin{align*}\msh_{\Gamma_\mu}(\Gamma_\mu) &\hookrightarrow \msh_{\mathfrak{c}_{W^- \times W \times T^*G, \partial \Gamma_{\mu}}}(\mathfrak{c}_{W^- \times W \times T^*G}) \to \msh_{\mathfrak{c}_{W^-} \times \mathfrak{c}_{W \times T^*G}}(\mathfrak{c}_{W^-} \times \mathfrak{c}_{W \times T^*G}) \\
    &= \Fun^\mathrm{L}\big(\msh_{\mathfrak{c}_{W}}(\mathfrak{c}_{W}), \msh_{\mathfrak{c}_{W \times T^*G}}(\mathfrak{c}_{W \times T^*G})\big)
    \end{align*}
    is the identity when restricted to $g = 1$. Moreover, for any $J$-paramatric family of Hamiltonian $G$-actions that restricts to the same contact $G$-action near boundary, the functors assigned are canonically isomorphic.
\end{proposition}
\begin{proof}
    The proof is identical to Proposition \ref{prop: GKS}. Consider any smooth $J$-family of extensions $\mu_J: W \times J \to \bR$ that are equal to $\mu_\infty$ on a neighborhood of $\partial W$. Assume that $J$ contains $0$ and $1$. Then since the microsheaf quantization functor is compatible with respect to Hamiltonian isotopies, we have a commutative diagram as a special case of Formula \eqref{eq: Hamilton invariance quantization}
    \[\begin{tikzcd}
    \Gamma(\msh_{\Gamma_{\mu_0}}) \ar[r] & \Gamma(\msh_{\mathfrak{c}_{W^- \times W \times T^*G,\partial \Gamma_{\mu}}}) \ar[r] & \Gamma(\msh_{\mathfrak{c}_{W^-} \times \mathfrak{c}_{W \times T^*G}}) \\
    \Gamma(\msh_{\Gamma_{\mu_J}}) \ar[r] \ar[u, "i_0^*"] \ar[d, "i_1^*" left] & \Gamma(\msh_{\mathfrak{c}_{W^- \times W \times T^*G \times T^*J},\partial \Gamma_\mu \times J}) \ar[r] \ar[u] \ar[d] & \Gamma(\msh_{\mathfrak{c}_{W^-} \times \mathfrak{c}_{W \times T^*G \times T^*J,\Lambda_{\mu} \times J}}) \ar[u, "i_0^*" right] \ar[d, "i_1^*"] \\
    \Gamma(\msh_{\Gamma_{\mu_1}}) \ar[r] & \Gamma(\msh_{\mathfrak{c}_{W^-} \times \mathfrak{c}_{W \times T^*G},\partial \Gamma_{\mu}}) \ar[r] & \Gamma(\msh_{\mathfrak{c}_{W^-} \times \mathfrak{c}_{W \times T^*G,\Lambda_\mu}}).
    \end{tikzcd}\]
    Under the identification that $\msh_{\Gamma_{\mu}}(\Gamma_\mu) \simeq \Loc(\Gamma_\mu)$, we know that the composition of the left vertical arrows is the identity. The composition of the right vertical arrows is also the identity. Therefore, we have a canonical isomorphism between the associated functors.
\end{proof}

    For a contact moment map $\mu_\infty: \partial W \to \mathfrak{g}^*$, we can define a contact $G$-action on $\partial W$. The above result allows us to define the sheaf quantization for contact Hamiltonian $G$-actions.

\begin{definition}
    Let $W$ be a Weinstein domain with contact boundary $\partial W$. Let $G$ be a (finite dimensional) Lie group and $\mu_\infty: \partial W \to \mathfrak{g}^*$ be a contact Hamiltonian $G$-action.  Fix $G$-equivariant Maslov data. Then we define the microsheaf quantization $\SK_{\mu_\infty}$ to be the image of the (constant) rank one microsheaf $\SK_\mu \in \Gamma(\msh_{\Gamma_\mu})$ under
    $$\msh_{\Gamma_\mu}(\Gamma_\mu) \hookrightarrow \msh_{\mathfrak{c}_{W^- \times W \times T^*G, \partial \Gamma_\mu}}(\mathfrak{c}_{W^- \times W \times T^*G}),$$
    where $\mu$ is any extension of $\mu_\infty$ to an exact Hamiltonian group action.
\end{definition}

\begin{theorem}\label{thm: group action} 
    Let $W$ be a Weinstein domain. Let $G$ be a (finite dimensional) Lie group with a contact Hamiltonian $G$-action $\mu_\infty$ on the contact boundary $\partial W$. Fix $G$-equivariant Maslov data.  
    (Note that we do not require the $G$ action to preserve $\mathfrak{c}_W$.)
    Then under the multiplication map $m: G \times G \to G$, the inverse map $a: G \to G$ and the projection map $\pi_i: G \times G \to G$, there are canonical equivalences 
    $$m^*\SK_{\mu_\infty} \simeq \pi_1^*\SK_{\mu_\infty} \circ \pi_2^*\SK_{\mu_\infty}, \quad \SK_{\mu_\infty} \circ a^*\SK_{\mu_\infty} = 1_{\Delta \times G}.$$
    which assemble into 
    a topological group homomorphism
    $$G \longrightarrow \Aut(\msh_{\mathfrak{c}_{W}}(\mathfrak{c}_{W})).$$
\end{theorem}
\begin{proof}
    We can understand the assertion as an identity at the level of the bottom row of Theorem \ref{sheaf quantization commutes with composition} or, after stop removal, the bottom row of Theorem \ref{stop removal commutes with composition}. The identification of Proposition \ref{prop: group action}, along with Theorems \ref{sheaf quantization commutes with composition} and \ref{stop removal commutes with composition} mean that said identity can be deduced from the corresponding identity at the level of the top row of Theorem \ref{sheaf quantization commutes with composition} that
    $$m^*\SK_\mu \simeq \pi_1^*\SK_\mu \circ \pi_2^*\SK_\mu, \quad \SK_\mu \circ a^*\SK_\mu = 1_{\Delta \times G}.$$
    For the identity that $m^*\SK_\mu \simeq \pi_1^*\SK_\mu \circ \pi_2^*\SK_\mu$, by Proposition \ref{prop: group action}, the left hand side is the pullback of the constant rank one local system on $\Gamma_\mu$ via the multiplication and the pullback of the external product of constant rank one local systems on $\Gamma_\mu \times \Gamma_\mu$, via the diagram
    $$\Gamma_{\mu} \xleftarrow{m} \Gamma_{\mu \times \mu} \hookrightarrow \Gamma_\mu \times \Gamma_\mu.$$
    For the identity $\SK_\mu \circ a^*\SK_\mu = 1_{\Delta \times G}$, the identity we need is between the constant rank one local system on $\Delta \times G$ and the pullback of the external product of constant rank one local systems on $\Gamma_\mu \times \Gamma_\mu$ via
    $$\Delta \times G \hookrightarrow \Gamma_\mu \times \Gamma_{-\mu}.$$
    Therefore, we can conclude that they are canonically isomorphic by choosing the isomorphism to be the identity between the constant rank one local systems.

    For the notion of ($\infty$-)group objects and ($\infty$-)group actions, we follow \cite[Definition 6.1.2.7 \& 7.2.2.1]{Lurie-HTT} or \cite[Definition 4.1.2.5 \& 4.2.2.2]{Lurie-HA}.
    The structure of being a ($\infty$-)group object in spaces is encoded by the following simplicial diagram:
    $$\cdots \,\substack{\longrightarrow \\[-0.8em] \longrightarrow \\[-0.8em] \longrightarrow \\[-0.8em] \longrightarrow}\, G \times G \times G \,\substack{\longrightarrow \\[-0.8em] \longrightarrow \\[-0.8em] \longrightarrow}\, G \times G \,\substack{\longrightarrow \\[-0.8em] \longrightarrow}\, G \rightarrow *.$$
    Here, the face morphisms are given by group multiplications and projections (and all the degeneration maps defined by units are omitted). The geometric realization of the diagram is classifying space $BG$ \cite{Segal}. This gives rise to a group object in $\PrLst$
    $$\cdots \,\substack{\longrightarrow \\[-0.8em] \longrightarrow \\[-0.8em] \longrightarrow \\[-0.8em] \longrightarrow}\, \Loc(G) \otimes \Loc(G) \otimes \Loc(G) \,\substack{\longrightarrow \\[-0.8em] \longrightarrow \\[-0.8em] \longrightarrow}\, \Loc(G) \otimes \Loc(G) \,\substack{\longrightarrow \\[-0.8em] \longrightarrow}\, \Loc(G) \rightarrow \SC,$$
    and there is a fully faithful embedding $G \hookrightarrow \Loc(G)$. By definition, a topological $G$-action on a presentable category $\SE$ is a $\Loc(G)$-action on $\SE$ encoded by the simplicial diagram
    $$\cdots \,\substack{\longrightarrow \\[-0.8em] \longrightarrow \\[-0.8em] \longrightarrow \\[-0.8em] \longrightarrow \\[-0.8em] \longrightarrow}\, \SE \otimes \Loc(G)^{\otimes 3} \,\substack{\longrightarrow \\[-0.8em] \longrightarrow \\[-0.8em] \longrightarrow\\[-0.8em] \longrightarrow}\, \SE \otimes \Loc(G)^{\otimes 2} \,\substack{\longrightarrow \\[-0.8em] \longrightarrow \\[-0.8em] \longrightarrow}\, \SE \otimes \Loc(G) \substack{\longrightarrow \\[-0.8em] \longrightarrow} \, \SE.$$

    We have shown already $m^*\SK_{\mu_\infty} \simeq \pi_1^*\SK_{\mu_\infty} \circ \pi_2^*\SK_{\mu_\infty}$ and $\SK_{\mu_\infty} \circ a^*\SK_{\mu_\infty} = 1_{\Delta \times G}$. By induction, the same argument implies that any commutative diagram in the ($\infty$-)group induces isomorphisms of microsheaves, and in particular, 
    \begin{gather*}
    (\id^{i-1} \times m \times \id^j)^*(\pi_1^*\SK_{\mu_\infty} \circ \dots \pi_{i+j}^*\SK_{\mu_\infty}) \simeq \pi_1^*\SK_{\mu_\infty} \circ \dots (\pi_{i}^*\SK_{\mu_\infty} \circ \pi_{i+1}^*\SK_{\mu_\infty}) \dots \circ \pi_{i+j+1}^*\SK_{\mu_\infty}, \\
    (\id^{i-1} \times a \times \id^j)^*(\pi_1^*\SK_{\mu_\infty} \circ \dots \circ \pi_{i+j}^*\SK_{\mu_\infty}) \simeq \pi_1^*\SK_{\mu_\infty} \circ \dots \widehat{\pi_{i}^*\SK_{\mu_\infty}} \dots \circ \pi_{i+j}^*\SK_{\mu_\infty}.
    \end{gather*}
    Hence we have a monoidal functor $\Loc(G) \to \Fun^\mathrm{L}(\msh_{\mathfrak{c}_W}(\mathfrak{c}_W), \msh_{\mathfrak{c}_W}(\mathfrak{c}_W))$ such that the functor $\Loc(G)^{\otimes j} \to \Fun^\mathrm{L}(\msh_{\mathfrak{c}_W}(\mathfrak{c}_W), \msh_{\mathfrak{c}_W}(\mathfrak{c}_W))^{\otimes j}$ in the simplicial diagram is defined by $\pi_1^*\SK_{\mu_\infty} \otimes \dots \otimes \pi_j^*\SK_{\mu_\infty}$. Therefore, we can define the $G$-action through the simplicial diagram
    $$\cdots 
    \,\substack{\longrightarrow \\[-0.8em] \longrightarrow \\[-0.8em] \longrightarrow \\[-0.8em] \longrightarrow \\[-0.8em] \longrightarrow}\, \msh_{\mathfrak{c}_{W}}(\mathfrak{c}_{W}) \otimes \Loc(G)^{\otimes 3} 
    \,\substack{\longrightarrow \\[-0.8em] \longrightarrow \\[-0.8em] \longrightarrow \\[-0.8em] \longrightarrow}\, \msh_{\mathfrak{c}_{W}}(\mathfrak{c}_{W}) \otimes \Loc(G)^{\otimes 2} \,\substack{\longrightarrow \\[-0.8em] \longrightarrow \\[-0.8em] \longrightarrow}\, \msh_{\mathfrak{c}_{W}}(\mathfrak{c}_{W}) \otimes \Loc(G) \,\substack{\longrightarrow \\[-0.8em] \longrightarrow}\, \msh_{\mathfrak{c}_W}(\mathfrak{c}_W)$$
    where the action is induced by the monoidal functor $\Loc(G) \to \Fun^\mathrm{L}(\msh_{\mathfrak{c}_W}(\mathfrak{c}_W), \msh_{\mathfrak{c}_W}(\mathfrak{c}_W))$. More precisely, the face morphisms are given by the action 
    $\SF \mapsto \SK_\mu \circ \SF$ for $\SF \in \msh_{\mathfrak{c}_W}(\mathfrak{c}_W)$ and group multiplication $\SL \mapsto m^*\SL$ for $\SL \in \Loc(G)$, and the degeneracy map is given by restrictions along the unit. This completes the construction.
\end{proof}

    With microsheaves replaced by Fukaya categories, such a construction was anticipated by Teleman \cite[Conjecture 2.9]{Teleman} (and played a role in his related conjectures \cite[Conjecture 4.2]{Teleman}), but has yet to appear due to difficulties with composition of correspondences in that setting.  For actions on compact symplectic $X$, a morphism $C_*(\Omega G) = \WFuk(T^*G) \to \Fuk(X^- \times X)$
    was constructed via Lagrangian correspondence by Evans and Lekili \cite{LekiliEvans}, but the coherences needed to identify this as a topological $G$-action have not appeared.  Meanwhile, Oh and Tanaka have given a rather different (not using Lagrangian correspondences) construction of topological $G$-actions on Liouville manifolds \cite{Oh-Tanaka1}; we sketch in Appendix \ref{comparison of G-actions} why our action agrees with theirs.

    As also noted by Teleman, one important consequence of the existence of such a topological $G$ action is that, by taking endomorphisms of the identity, one obtains a morphism 
    $$C_*\Omega G \longrightarrow  \operatorname{HC}^{*}(\msh_{\mathfrak{c}_{W}}(\mathfrak{c}_W)),$$
    which is an $E_2$-algebra homomorphism for abstract reasons, e.g. by 
    Tamarkin's proof of the Deligne conjecture via Dunn additivity \cite{TamarkinDG} (or the $\infty$-categorical version \cite{Lurie-FMP}). 
    Here, $\operatorname{HC}^*$ is the Hochschild cochains of a stable category.

    
    

    Finally, we note that the construction allows us to define the equivariant microsheaf category of Weinstein manifolds with respect to a Hamiltonian $G$-action that restricts to a contact $G$-action on the contact boundary.

\begin{definition}
    Let $W$ be Weinstein sector with Maslov data.
    Let $G$ be a Lie group with an exact and eventually conical Hamiltonian $G$-action on $W$ that respects Maslov data. Then we define the $G$-equivariant microsheaf category to be the homotopy quotient $\msh_{\mathfrak{c}_{W}}(\mathfrak{c}_W)_{hG}$, defined as the geometric realization
    $$\operatorname{colim}\Big(\cdots  \,\substack{\longrightarrow \\[-0.8em] \longrightarrow \\[-0.8em] \longrightarrow}\, \msh_{\mathfrak{c}_{W \times T^*G^2}}(\mathfrak{c}_{W \times T^*G^2}) \,\substack{\longrightarrow \\[-0.8em] \longrightarrow}\, \msh_{\mathfrak{c}_{W \times T^*G}}(\mathfrak{c}_{W \times T^*G}) \rightarrow \msh_{\mathfrak{c}_W}(\mathfrak{c}_W)\Big).$$
\end{definition}

    

%% file: pdfl.tex

\section{Sufficiently Legendrian subsets}

Recall that a sheaf is said to be  cohomologically constructible when the filtered colimits computing the stalks/costalks can be taken over cofinal sequences which are eventually constant.  This condition ensures various commutativities which would not ordinarily be satisfied for stalks.
    
It is classical that constructibility with respect to a Whitney stratification ensures  cohomological constructibility; in \cite{KS} it was shown that said Whitney constructibility can be ensured by requiring subanalytic Lagrangian microsupport. However, for purposes of symplectic topology, it is inconvenient to impose subanalyticity hypotheses.  In \cite{Nadler-Shende}, a weaker condition was introduced, motivated by the non-characteristic deformation lemma of Kashiwara and Schapira, which we recall:

\begin{lemma}[Non-characteristic deformation {\cite[Proposition 2.7.2]{KS}}]\label{lem: non-char}
    Let $\SF \in \Sh(M)$ and $\{U_t\}_{t \in \bR}$ be a family 
    of open subsets. Suppose
    \begin{enumerate}
        \item for any $t \in \bR$, $U_t = \bigcup_{s<t}U_s$;
        \item for any $t \in \bR$, $\ol{U}_t \cap \mathrm{supp}(F)$ is compact;
        \item for any $t \in \bR$, $\dot N^*_{out}U_t \cap \dot\SS(\SF) = \varnothing$.
    \end{enumerate}
    Then for any $t\in \bR$, the restriction map induces an isomorphism
    $$\Gamma\big(\bigcup\nolimits_{t\in \bR}U_t, \SF\big) \xrightarrow{\sim} \Gamma(U_t, \SF).$$

    Alternatively, let $\{Z_t\}_{t \in \bR}$ be a family of closed subsets. Suppose
    \begin{enumerate}
        \item for any $t \in \bR$, $Z_t = \bigcap_{s<t}Z_s$;
        \item for any $t \in \bR$, $Z_t \cap \mathrm{supp}(F)$ is compact;
        \item for any $t \in \bR$, $\dot N^*_{in}Z_t \cap \dot\SS(\SF) = \varnothing$.
    \end{enumerate}
    Then for any $t\in \bR$, the restriction map induces an isomorphism
    $$\Gamma_{\bigcap\nolimits_{t\in \bR} Z_t}\big(M, \SF\big) \xrightarrow{\sim} \Gamma_{Z_t}(M, \SF).$$
\end{lemma}

\subsection{Contact isotopies}
    One important way to obtain the family of subsets to apply non-characteristic propagation is by considering contact isotopies defined by positive (non-negative) Hamiltonian functions, known as {\em positive (non-negative) contact isotopies}.

    Given a contact isotopy on $S^*M$, Guillermou--Kashiwara--Schapira constructed sheaves in the product manifold $M \times M$ whose microsupport is the graph of the contactomorphism, which is known as the `sheaf quantization for contact isotopies':

\begin{theorem}[{\cite[Proposition 3.12]{Guillermou-Kashiwara-Schapira}}]\label{thm: GKS sheaf}
    Let $\varphi_H: S^*M \times I \to S^*M$ be a contact isotopy defined by the Hamiltonian $H: S^*M \times I \to \mathbb{R}$. Let $\Lambda$ be a closed subset and $\Lambda_t = \varphi_H^t(\Lambda)$. Then there is a unique sheaf $\SK_\varphi \in \Sh(M \times M \times I)$ such that the restrictions $\SK_{\varphi}^t = i_t^*\SK_\varphi$ induces natural equivalences
    $$\SK_\varphi^t \circ -: \Sh_\Lambda(M) \to \Sh_{\Lambda_t}(M).$$
\end{theorem}

    Positive contact isotopies induce natural continuation maps:

\begin{proposition}[{\cite[Proposition 4.8]{Guillermou-Kashiwara-Schapira} \cite[Proposition 3.2]{Kuo-wrapped-sheaves}}]\label{prop: positive continue shv}
    Let $\Lambda \subset S^*M \times T^*_{\tau \leq 0}I$ be any closed subset. Then for any $\SF \in \Sh_\Lambda(M \times I)$, there is a natural continuation morphism 
    $i_0^*\SF \to i_t^*\SF$, such that for any $U \subset M$, its induced map on sections is 
    $$\Gamma(U, i_0^*\SF) \xleftarrow{\sim} \Gamma(U \times I, \SF) \rightarrow \Gamma(U, i_1^*\SF).$$
    Hence, for a positive contact isotopy $\varphi$, there is a natural continuation map
    $\SF \to \SF_t := \SK_{\varphi}^t \circ \SF$.
\end{proposition}

\begin{proposition}[{\cite[Corollary 3.4 \& 3.8]{Kuo-wrapped-sheaves}}]\label{prop: perturb-lim}
    Consider a non-negative contact isotopy. Then for any $\SF, \SG \in \Sh(M)$, the natural map is an isomorphism
    $$\Hom(\SF, \SG) \xrightarrow{\sim} \lim_{t \to 0^+}\Hom(\SF, \SG_t).$$
\end{proposition}

\subsection{Positive displaceable from fibers}

\begin{definition}[{\cite[Definition 2.2]{Nadler-Shende}}]
    A closed subset $\Lambda \subset S^*M$ is positively displaceable from Legendrians (pdff) 
    if $\Lambda$ is contact isotopic to $\Lambda'$ such that for any $x\in M$, there is a positive Legendrian isotopy $(S^*_xM)_t$ for $t \in (-\delta, \delta)$, such that $(S^*_xM)_t$ is disjoint from $\Lambda'$ when $t \neq 0$. 
\end{definition}

    Whitney stratified Legendrians are always positively displaceable from fibers, and this is how the above condition is usually checked:

\begin{lemma}\label{lem: pdff whitney}
    Let $\Lambda \subset S^*M$ be a Whitney stratified isotropic subset. Then $\Lambda$ is positively displaceable from fibers.
\end{lemma}
\begin{proof}
    When $\Lambda \subset S^*M$ is a Whitney stratified isotropic subset, there is a small contact perturbation to $\Lambda'$ such that $\Lambda'$ is positively displaceable from cotangent fibers. Consider singularities of the projection $\pi: \Lambda \to M$ with the Thom--Boardman decomposition in the jet bundle \cite{MatherThom}. Since the Thom--Boardman decomposition admits a Whitney stratified refinement \cite[Theorem 3]{Ryabichev,MatherThom}, we know that there is a small contact perturbation of $\Lambda$ to $\Lambda'$ that tranverse intersection with the refined Thom--Boardman stratification by the stratified Thom transversality theorem of Goresky and Trotman \cite[Part 1 Section 1.3]{Goresky-MacPherson}. Then the image of the projection $\pi: \Lambda' \to M$ admits a refinement to a Whitney stratified immersion, and thus it is positive displaceable from cotangent fibers by the Whitney condition B \cite[Part I Section 1.4]{Goresky-MacPherson}.
\end{proof}
\begin{remark}
    When $\Lambda$ is contained in the conormal bundle of a $C^1$-Whitney stratification of $M$, then $\Lambda$ is positively displaceable from cotangent fibers (without small contact perturbation) by the Whitney condition B.
\end{remark}
    
\begin{lemma}\label{lem: pdff analytic}
    Let $\Lambda \subset S^*M$ be a subanalytic isotropic subset with respect to some analytic structure on $M$. Then $\Lambda$ is positively displaceable from fibers.
\end{lemma}
\begin{proof}
    When $\Lambda \subset S^*M$ be a subanalytic isotropic subset, it is contained in a Whitney stratified Legendrian \cite[Theorem 4.15 \& 5.6]{Czapla}. Hence it is positively displaceable from fibers by Lemma \ref{lem: pdff whitney}. Alternatively, by microlocal Bertini--Sard theorem \cite[Proposition 8.3.12]{KS}, we can also conclude that $\Lambda$ is positively displaceable from fibers.
\end{proof}

    The positive Legendrian isotopy for the cotangent fibers determines choices of open balls near the given point:

\begin{definition}  \label{pdff ball}  
    From now on, for any pdff subset $\Lambda \subset S^*M$, we will fix a positive isotopy of $S^*_xM$ that displaces $S^*_xM$ from $\Lambda$.  Said isotopy distinguishes a class of small balls around $x$, these being the locus bounded by the projection of the time $\epsilon$ flow of $S^*_x M$. We write 
    $S_\epsilon(x)$, $B_\epsilon(x)$ and $\overline{B}_{\epsilon}({x})$ be respectively the sphere, open ball and closed ball thus determined for sufficiently small $\epsilon > 0$.
\end{definition}

\begin{lemma}\label{stalk simple}
    Let ${\SF} \in \Sh(M \times \bR_{>0})$ be a sheaf such that $\dot{\ss}({\SF})$ is pdff. Then for any $x \in M$, and $\epsilon > 0$ sufficiently small,
    $$\Gamma(\overline{B}_{\epsilon}({x}), \SF) \simeq \Gamma({B}_{\epsilon}({x}), \SF) \simeq \SF_x.$$
\end{lemma}
\begin{proof}
    By noncharacteristic propagation (Lemma \ref{lem: non-char}) with respect to the flow of the cosphere over $x$ by the positive isotopy displacing it from $\dot{\ss}({\SF})$. 
\end{proof}

    The second Legendrian condition we need is the finite position assumption, which will be used in Section \ref{sec:microsheaf recollections} to identify the stalks of microsheaves:

\begin{definition}\label{def: ptfp}
    A subset $\Lambda \subset S^*M$ is in {\em finite position} if the projection of its closure to $M$ is finite to one. We say a subset of $S^*M$ is {\em perturbable to finite position} if its image under a contact isotopy is in finite position.
\end{definition} 

    Whitney stratified Legendrians and subanalytic Legendrians are always perturbable to finite positions:

\begin{lemma}\label{lem: ptfp whitney}
    Let $\Lambda \subset S^*M$ be a compact Whitney stratified isotropic subset. Then $\Lambda$ is perturbable to finite position.
\end{lemma}
\begin{proof}
    When $\Lambda$ is a compact Whitney stratified isotropic subset, consider singularities of the projection $\pi: \Lambda \to M$ with the Thom--Boardman decomposition in the jet bundle of $M$ \cite{MatherThom}. Then there is a small contact perturbation to $\Lambda'$ whose projection to $M$ intersect transversely with the Thom--Boardman singularities by the Thom transversality theorem of Goresky and Trotman \cite[Part I Section 1.3]{Goresky-MacPherson}. Then it follows that the projection of $\Lambda'$ to $M$ is in finite position.
\end{proof}

\begin{remark}
    When $\Lambda \subset S^*M$ is contained in the conormal bundle of a $C^1$-Whitney stratification of $M$, then it is perturbable to finite position.
\end{remark}

\begin{lemma}\label{lem: ptfp analytic}
    Let $\Lambda \subset S^*M$ be a compact subanalytic isotropic subset with respect to some analytic structure on $M$. Then $\Lambda$ is perturbable to finite position.
\end{lemma}
\begin{proof}
    When $\Lambda \subset S^*M$ be a subanalytic isotropic subset, it is contained in a Whitney stratified Legendrian \cite[Theorem 4.15 \& 5.6]{Czapla}. Hence it is positively displaceable from fibers by Lemma \ref{lem: ptfp whitney}.
\end{proof}

\subsection{Self-displaceability}
    The third `Legendrian' condition we need is a self displaceability assumption.  We will use this in Section \ref{sec:microsheaf recollections} to construct the doubling functor, but it can also be motivated by the following fact:

\begin{proposition}[{\cite[Proposition 2.9]{Zhou} \cite[Proposition 3.4]{Kuo-wrapped-sheaves} \cite[Theorem 4.46]{KuoLi-spherical}\footnote{These references assume stronger constructibility hypotheses, but the proofs require only what we have assumed.}}]\label{prop: perturb compact} 
    Let $\Lambda \subset S^*M$ be any subset and $\Omega \subset \Lambda$ be a precompact open subset. Assume that there is a non-negative compactly supported contact isotopy $\phi_t: S^*M \to S^*M$, supported in the closure of some $U \subset S^*M$ where $\Omega = \Lambda \cap U$, such that
    $$\Lambda \cap \phi_t(\Omega) = \varnothing$$
    for any small $t \neq 0$. Then for any $\SF, \SG \in \Sh_\Lambda(M)$ and $t \geq 0$, continuation induces an isomorphism:
    $$\Hom(\SF, \SG) \xrightarrow{\sim} \Hom(\SF, \SG_t).$$
\end{proposition}

We axiomatize the hypotheses of the proposition:

\begin{definition}\label{def: self displaceable} 
    A closed subset $\Lambda \subset S^*M$ is {\em positively self displaceable} if there is a topological basis of precompact open subsets in $\Lambda$, such that for any such $\Omega \subset \Lambda$ in the topological basis, there exists a positive contact isotopy for $t \in (-\delta, \delta)$, supported in the closure of some open subset $U \subset S^*M$  with $\Omega = \Lambda \cap U$, such that $\Omega_t$ is disjoint from $\Lambda$ when $t \neq 0$.
\end{definition}

\begin{lemma}\label{lem: self displace hypersurface}
    Let $\Lambda \subset S^*M$ be contained in an exact symplectic hypersurface. Then $\Lambda$ is always positively self displaceable. 
\end{lemma}
\begin{proof}
    For an exact symplectic hypersurface $(W, \lambda)$, we can consider an open neighbourhood of $W$ that is contactomorphic to $(W \times (-\epsilon, \epsilon), dt - \lambda)$. Consider the basis of open subsets of $W$. Then for any open subset $\Omega \subset W$, the subset $\Lambda \cap \Omega$ is positively self displaceable via some smooth cut-off of the standard Reeb flow supported in the closure of the contact neighbourhood $\Omega \times (-\epsilon, \epsilon)$.
\end{proof}

\begin{lemma}\label{lem: self displace morse smale}
    Let $\Lambda \subset S^*M$ be contained in the conormal bundle of a $C^1$-stratification by the stable submanifolds of a Morse--Smale function on $M$. Then $\Lambda$ is self displaceable.
\end{lemma}
\begin{remark}
    For a $C^1$-stratification by the stable submanifolds of a $C^2$-Morse--Smale function on $M$, we know that it is a $C^1$-Whitney stratification \cite[Theorem 4.33]{Nicolaescu}. Moreover, it is known that the conormal of the stratification is contained in an exact symplectic hypersurface (defined by a symplectic Lefschetz fibration) \cite[Theorem 1 \& Lemma 5]{GirouxLefschetz}. Therefore, if $\Lambda \subset S^*M$ is contained in the conormal bundle of a Whitney stratification of a Morse--Smale function on $M$, it is self displaceable by Lemma \ref{lem: self displace hypersurface}.
\end{remark}

\begin{lemma}\label{lem: self displace analytic}
    Let $\Lambda \subset S^*M$ be a subanalytic isotropic subset with respect to some analytic structure of $M$. Then $\Lambda$ is positively self displaceable.
\end{lemma}
\begin{proof}
    When $\Lambda$ is a subanalytic isotropic subset with respect to some analytic structure of $M$, we can consider the basis of open subsets of $S^*M$ with (sub)analytic boundaries. For each subset $\Lambda \cap \Omega$ with subanalytic boundary, consider an analytic contact isotopy that vanish on the boundary of $\Omega$ up to order $k \geq 2$ and then apply the microlocal Bertini--Sard theorem \cite[Proposition 8.3.12]{KS}. We can conclude that it is positively self displaceable.
\end{proof}

\begin{example}
    We warn the reader that it is not true that any Whitney stratified isotropic subset in $S^*M$ is positively self displaceable. For example, consider a smooth function $f: \bR \to \bR$ such that the critical values are dense at $0$. Then $\Lambda = \{(x, y; \xi, \eta) \mid y = 0, \xi = 0 \text{ or } y = f(x), \xi = \eta\, df(x)\}$ is a Whitney stratified Legendrian but it is not positively self displaceable. The distinction between this non-example and the above examples (in Lemma \ref{lem: self displace hypersurface}) is reflected through the fact that $\Lambda$ in this non-example cannot be the Lagrangian skeleton of any exact symplectic manifold \cite[Remark 1.2]{Eliashberg-revisit}.
\end{example}

    We will find the following terminology convenient: 

\begin{definition}\label{def: sufficient Legendrian}
    We say a subset is {\em sufficiently Legendrian} if it is pdff, perturbable to finite position and positively self displaceable. 
\end{definition} 

\begin{lemma}\label{lem: sufficient legendrian whitney}
    Let $\Lambda \subset S^*M$ be any Whitney stratified Legendrian subset that is contained in an exact symplectic hypersurface. Then $\Lambda$ is a sufficiently Legendrian subset.
\end{lemma}
\begin{proof}
    This follows from Lemmas \ref{lem: pdff whitney}, \ref{lem: ptfp whitney} and \ref{lem: self displace hypersurface}.
\end{proof}

\begin{lemma}\label{lem: sufficient legendrian morse}
    Let $\Lambda \subset S^*M$ be any subset that is contained in the union of conormals to strata in the stratification by stable submanifolds of a Morse--Smale function. Then $\Lambda$ is a sufficiently Legendrian subset.
\end{lemma}
\begin{proof}
    This follows from Lemmas \ref{lem: pdff whitney}, \ref{lem: ptfp whitney} and \ref{lem: self displace morse smale}.
\end{proof}

\begin{lemma}\label{lem: sufficient legendrian analytic}
    Let $\Lambda \subset S^*M$ be any subanalytic Legendrian subset. Then $\Lambda$ is a sufficiently Legendrian susbet.
\end{lemma}
\begin{proof}
    This follows from Lemmas \ref{lem: pdff analytic}, \ref{lem: ptfp analytic} and \ref{lem: self displace analytic}.
\end{proof}

%% file: nearbycyclereview.tex

\section{Recollections on nearby cycles}
\label{appen: nearby}

    Let $M$ be a manifold. Consider the maps:
    $$M \times 0 \xrightarrow{i} M \times \bR \xleftarrow{j} M \times \bR_{>0}.$$
    Consider the nearby cycle functor:
    \begin{equation}
    \psi = i^* j_* : \Sh(M \times \bR_{>0}) \to \Sh(M).
    \end{equation}

    In this section we will recall notions and results from \cite{Nadler-Shende, LiCobordism} around microsupport conditions which ensure the commutativity of the nearby cycle functor with various operations.  Also we will give the `$\otimes$' version of certain results which were established in \cite{Nadler-Shende} for Hom.

\subsection{Microsupport estimates}

    For a smooth map $f: M \rightarrow N$ and a subset $\Lambda \subset T^*N$, recall the subset $i^\#\Lambda \subset T^*M$ defined in \cite[Proposition 6.2.4]{KS} to be the set of points $(x, \xi) \in T^*M$ such that there are sequences $x_n \in M$ and $(y_n, \eta_n) \in \Lambda$ with 
    $$x_n \to x, \quad y_n \to f(x), \quad f^*\eta_n \to \xi, \quad |y_n - f(x_n)||\eta_n| \to 0.$$
    This is used in the following estimate \cite[Corollary 6.4.4]{KS}: 
\begin{equation}\label{lem: ss-pull back}
    \SS(f^*\SF) \subset f^\#\SS(\SF), \quad \SS(f^!\SF) \subset f^\#\SS(\SF).
\end{equation}
    In particular, when $f: M \to N$ is a submersion, by considering the maps of cotangent bundles $T^*M \xleftarrow{f^*} f^*T^*N \xrightarrow{f_\pi} T^*N$, we have \cite[Proposition 5.4.5]{KS}:
    $$\SS(f^*\SF) = \SS(f^!\SF) = f^*f_\pi^{-1}\SS(F).$$

    For two subsets $\Lambda, \Lambda' \subset T^*M$, we define the subset $\Lambda \,\widehat+\,\Lambda' \subset T^*M$ as in \cite[Proposition 6.2.4]{KS} to be the set of points $(x, \xi) \in T^*M$ such that there are sequences $(x_n, \xi_n) \in T^*M$ and $(y_n, \eta_n) \in T^*M$ with
    $$x_n\to x, \quad y_n \to x, \quad \xi_n - \eta_n \to \xi, \quad |y_n - x_n||\xi_n| \to 0.$$

    There are the standard estimates \cite[Corollary 6.4.5]{KS}: 
\begin{equation}\label{lem: ss-hom}
    \SS(\SF \otimes \SG) \subset \SS(\SF) \,\widehat +\, \SS(\SG), \quad \SS(\sHom(\SF, \SG)) \subset -\SS(\SF) \,\widehat +\, \SS(\SG).
\end{equation}
    In particular, as in \cite[Proposition 5.4.14]{KS}, if $\SS(\SF) \cap \SS(\SG) \subset 0_M$, then
    $$\SS(\sHom(\SF, \SG)) \subset -\SS(\SF) + \SS(\SG).$$
    If $-\SS(\SF) \cap \SS(\SG) \subset 0_M$, then
    $$\SS(\SF \otimes \SG) \subset \SS(\SF) + \SS(\SG).$$

    Let $j: U \hookrightarrow M$ be an open embedding and $\SF \in \Sh(U)$. Then \cite[Theorem 6.3.1 \& Proposition 6.3.2]{KS}
\begin{equation}\label{lem: ss-push forward}
    \SS(j_*\SF) \subset \SS(\SF) \,\widehat+\, N^*_{in}U.
\end{equation}
    In particular, if $N = M \setminus U$ is a closed submanifold, then
    $$\SS(j_*\SF) \cap T^*M|_N \subset \SS(\SF) \,\widehat+\, T^*_NM.$$

    Let $B^* \subset B$ be an open submanifold whose complement is a point. Write $j: M \times B^* \hookrightarrow M \times B$ and $i: M \times B \setminus B^* \hookrightarrow M \times B$. Let $\SF \in \Sh(M \times B^*)$. Then we define the nearby cycle functor to be
    $$\psi\SF = i^*j_*\SF.$$
    For a subset $\Lambda \subset T^*(M \times B)$, define the projection $\Pi: T^*(M \times B) \to T^*M \times B$. We define
    \begin{equation}\label{eq: nearby subset}
    \psi(\Lambda) = \ol{\Pi(\Lambda)} \cap T^*M.
    \end{equation}
    Then by Formulae \eqref{lem: ss-pull back} and \eqref{lem: ss-push forward}, we have:

\begin{lemma}[{\cite[Lemma 3.16]{Nadler-Shende}}]\label{lem: ss-nearby-cycle}
    Let $\SF \in \Sh(M \times B^*)$. Then we have 
    $$\SS(\psi\SF) \subset \psi(\SS(\SF)) = \ol{\Pi(\SS(F))} \cap T^*M.$$
\end{lemma}

We say $\Lambda \subset T^*(M \times B)$ is $B$-noncharacteristic if $\Lambda \cap (0_M \times T^*B) \subset 0_{M \times B}$; this ensures non-characteristic propagation with respect to families of open sets pulled back from $B$.

When the above limiting microsupport is pdff, we may compute the stalk of the nearby cycles as sections over a nearby ball:

\begin{lemma}[{\cite[Corollary 4.4]{Nadler-Shende}}]\label{stalk}
    Let ${\SF} \in \Sh(M \times \bR_{>0})$ be a sheaf such that $\dot{\ss}({\SF})$ is $\bR_{>0}$-non-characterstic, and $\psi(\dot{\ss}_{\pi}({\SF}))$ is pdff. Then for any $x \in M$, and $\epsilon > 0$, $\epsilon' > 0$ sufficiently small,
    $$\Gamma\big(\overline{B}_{\epsilon'}({x}) \times (0, \epsilon], {\SF}\big) \simeq \Gamma\big({B}_{\epsilon'}({x}) \times (0, \epsilon), {\SF}\big) \simeq \psi\SF_x,$$
    where $\overline{B}_{\epsilon'}({x})$ (resp.~${B}_{\epsilon'}({x})$) is a closed (resp.~open) ball of radius $\epsilon'$ around $x$.
\end{lemma}

\subsection{Commuting nearby cycles}

    We recall some general considerations around the commutativity of nearby cycles for multi-parametric families.
    Consider a sheaf $\SF \in \Sh(M \times \bR_{>0}^2)$. We write
    $$ M \times 0 \times 0 \xrightarrow{i} M \times \bR^2 \xleftarrow{j} M \times \bR_{>0}^2.$$
    Consider now 
    \begin{gather*}
    M \times 0 \times 0 \xrightarrow{i_1} M \times \bR \times 0 \xleftarrow{j_1} M \times \bR_{>0} \times 0,\\
    M \times \bR_{>0} \times 0 \xrightarrow{\overline{i}_1} M \times \bR_{>0} \times \bR \xleftarrow{\overline{j}_1} M \times \bR_{>0} \times \bR_{>0},
    \end{gather*}
    and by abuse of notations
    \begin{gather*}
    M \times \bR \times 0 \xrightarrow{\overline{i}_1} M \times \bR \times \bR \xleftarrow{\overline{j}_1} M \times \bR \times \bR_{>0},
    \end{gather*}
    We have the following micro-support condition for a base change: 

\begin{lemma}[{\cite[Proposition 3.3]{LiCobordism}, {\cite[Proposition 1.2]{LiFunctoriality}}}]\label{nearby basechange}
    Let ${\SF} \in \Sh(M \times \bR_{\geq 0} \times \bR_{>0})$ be a sheaf such that $\dot{\ss}(\SF)$ is $\bR_{>0}^2$-non-characteristic, $\overline{i}_1^\# \dot{\ss}(\SF)$ is $\bR_{>0}$-non-characterstic, and $\psi(\dot{\ss}_\pi(\SF))$ is pdff. Then there is a natural isomorphism of sheaves 
    $$\overline{i}^*_1\overline{j}_{2*}{\SF} \simeq \overline{j}_{2*}\overline{i}^*_1{\SF}.$$
\end{lemma}

    Consider now also 
    \begin{gather*}
    M \times 0 \times 0 \xrightarrow{i_1} M \times \bR \times t \xleftarrow{j_1} M \times \bR_{>0} \times 0,\\
    M \times \bR_{>0} \times 0 \xrightarrow{\overline{i}_1} M \times \bR_{>0} \times \bR \xleftarrow{\overline{j}_1} M \times \bR_{>0} \times \bR_{>0},
    \end{gather*}
    We denote the nearby cycle functors as: 
    $$\psi_{12} = i^*j_*, \; \psi_k = i_k^*j_{k *}, \; \overline{\psi}_k = \overline{i}_k^* \overline{j}_{k *}.$$
    We deduce the following variant (weakening the constructibility hypotheses, which we only impose on the limiting microsupport; see \cite{LiCobordism,LiFunctoriality}) of a lemma on commutativity of nearby cycles \cite{Kochersperger-Nearby,Maisonobe-Nearby, Nadler-Nearby}.

\begin{lemma}[{\cite[Theorem 4]{Maisonobe-Nearby}, \cite[Theorem 1.1.1]{Nadler-Nearby}, \cite[Section 3.2.2]{LiCobordism}}] \label{nearby commute Nadler}
    Let $\SF \in \Sh(M \times \bR_{>0}^2)$ be a sheaf such that $\dot{\ss}(\SF)$ is $\bR_{>0}^2$-non-characteristic, $\overline{\psi}_1(\dot{\ss}(\SF))$, $\overline{\psi}_2(\dot{\ss}(\SF))$ are $\bR_{>0}$-non-characterstic, and $\psi(\dot{\ss}_\pi(\SF))$ is pdff. Then
    $$\psi_{12}\SF \simeq \psi_1 \overline{\psi}_2 \SF \simeq \psi_2 \overline{\psi}_1 \SF.$$
\end{lemma}
\begin{proof}
    We have the following isomorphisms
    \begin{equation*}
    i_1^*j_{1*}\ol{i}_2^*\ol{j}_{2*}\SF \xrightarrow{\ref{nearby basechange}} i_1^*\ol{i}_2^*j_{1*}\ol{j}_{2*}\SF \xrightarrow{\sim} i^*j_*\SF. \qedhere
    \end{equation*}
\end{proof}

    More generally, let $\mathrm{Flag}_k$ be the set of partial flags in the simplex $[k]$ with partial orders (which encodes all the different ways to take the $k$-parametric nearby cycle functors). Then:

\begin{corollary}\label{rem: nearby commute higher}
    Consider $\SF \in \Sh(M \times \bR_{>0}^k)$ such that $\dot\SS(\SF)$ is $\bR_{>0}^k$-non-characteristic, $\ol{\psi}_{j_1\dots j_l}(\SS(\SF))$ are $\bR_{>0}^{k-l}$-non-characteristic, and such that $\psi(\dot\SS_\pi(\SF))$ is pdff. Then for any permutation $j_1, j_2, \dots, j_k$ any $1< l_1 < \dots < l_r < k$, we have natural isomorphisms 
    $$\psi_{12\cdots k}\SF \simeq \psi_{j_1j_2\cdots j_{l_1}}\ol{\psi}_{j_{l_1+1}\cdots j_{l_2}} \dots \ol{\psi}_{j_{l_r+1}\cdots j_k}\SF,$$
    that fit into a coherent diagram $\mathrm{Flag}_k \to \Fun^{ex}(\Sh(M \times \bR_{>0}^k), \Sh(M))$, where the partial flag given by $j_1, j_2, \dots, j_k$ and $1< l_1 < \dots < l_r < k$ is sent to the right hand side of the equation.
\end{corollary}
\begin{proof}
    We have a diagram of functors with (non-proper) base changes induced by adjunctions as in Figure \ref{fig:nearby_commutative}, which defines a diagram of functors with (op)lax natural transformations by \cite[Corollary F]{HaugsengHebestreitLinskensNuiten}; see also \cite[Appendix A Section 12.3]{Gaitsgory-Rozenblyum}.
    Then by applying Lemma \ref{nearby basechange} iteratively as in \cite[Section 3.2.2]{Nadler-Nearby}, we obtain a commutative diagram of functors where all the natural transformations are invertible. 
\end{proof}

\begin{figure}
    \centering
    \includegraphics[width=0.95\linewidth]{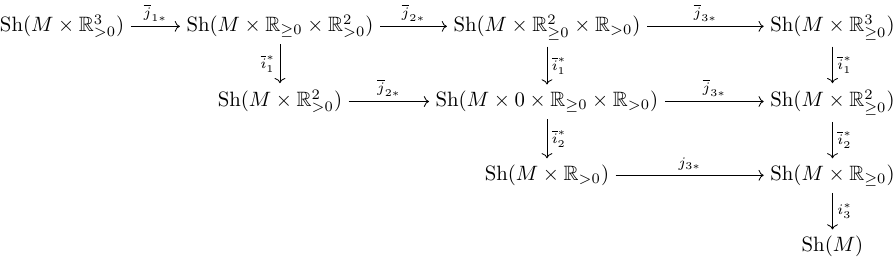}
    \caption{Part of the diagram of non-proper base changes that appears in the commutativity of the parametric nearby cycle functors, where the natural transformation in each square is induced by adjunction of six functors.}
    \label{fig:nearby_commutative}
\end{figure}
    
    We also have the following lemma, where 
    for $t \in \bR_{>0}$, we write
    \begin{gather*}
    M \times \bR_{>0} \times t \xrightarrow{i_{1,t}} M \times \bR \times t \xleftarrow{j_{1,t}} M \times \bR_{>0} \times t.
    \end{gather*}

\begin{lemma}[{\cite[Section 3.2.2]{LiCobordism}}]\label{nearby commute with restrict}
    Let $\SF \in \Sh(M \times \bR_{>0}^2)$ be such that $\dot \ss(\SF)$ is $\bR_{>0}^2$-non-characteristic, and $\overline{\psi}_1(\dot \ss_\pi(\SF))$ is $\bR_{>0}$-non-characterstic and pdff. 
    Then
    $$i_{2,t}^*\overline{\psi}_1 \SF = \psi_1(i_{2,t}^*\SF).$$
\end{lemma}
\begin{proof}
    We have the following isomorphisms
    \begin{equation*}
    i_1^*j_{1*}{i}_{2,t}^*\SF \xrightarrow{\ref{nearby basechange}} i_1^*\ol{i}_{2,t}^*\ol{j}_{1*}\SF \xrightarrow{\sim} i_{2,t}^*\ol{i}_1\ol{j}_{1*}\SF. \qedhere
    \end{equation*}
\end{proof}

\subsection{Nearby cycles and exterior tensor products}
    For $\SF_1 \in \Sh(M_1 \times \bR_{>0})$ and $\SF_2 \in \Sh(M_2 \times \bR_{>0})$, we write
    $$\pi_1: M_1 \times M_2 \times \bR_{>0} \to M_1 \times \bR_{>0}, \quad \pi_2: M_1 \times M_2 \times \bR_{>0} \to M_2 \times \bR_{>0}$$
    and define the exterior tensor product relative to $\bR_{>0}$ by
    $$\SF_1 \boxtimes_{\bR_{>0}} \SF_2 = \pi_1^*\SF_1 \otimes \pi_2^*\SF_2.$$

    A similar notion for Hom was considered in \cite{Nadler-Shende}, where the following was established: 

\begin{proposition}[{\cite[Theorem 4.2]{Nadler-Shende}}]\label{nearby cycle hom product}
    Let $\SF_1 \in \Sh(M_1 \times \bR_{>0})$ and $\SF_2 \in \Sh(M_2 \times \bR_{>0})$. Suppose $\dot{\ss}(\SF_1)$ and $\dot{\ss}(\SF_2)$ are $\bR_{>0}$-non-characteristic and $\psi_1(\dot{\ss}_\pi(\SF_1))$ and $\psi_2(\dot{\ss}_\pi(\SF_2))$ pdff. Then there is an isomorphism of sheaves on $M_1 \times M_2$:
    $$\psi_{12}\sHom^{\boxtimes/\bR_{>0}}(\SF_1, \SF_2) \simeq \sHom^\boxtimes(\psi_1 \SF_1 , \psi_2\SF_2).$$
\end{proposition}

    The purpose of this subsection is to establish the corresponding result for $\boxtimes$, by a similar argument. We first treat the case when one of the sheaves is constant along the $\bR$-direction. 

\begin{lemma} \label{remove irrelevant factors}
    Let $\SF_1 \in \Sh(M_1 \times \bR_{>0})$ and $\SG_2 \in \Sh(M_2)$. Suppose that $\dot{\ss}(\SF_1)$ is $\bR_{>0}$-non-characteristic and $\psi_1(\dot{\ss}_\pi(\SF_1))$ is pdff, and suppose that $\dot{\ss}(\SG_2)$ is pdff. Then there is an isomorphism of sheaves on $M_1 \times M_2$:
    $$\psi_{1}(\SF_1 \boxtimes \SG_2) = \psi_1 \SF_1 \boxtimes \SG_2.$$
\end{lemma}
\begin{proof}
    Consider $(x_1, x_2) \in M_1 \times M_2$, it suffices to show that the stalks at $x$ are isomorphic. Since $\psi_1(\ss_\pi(\SF_1))$ and $\ss(\SG_2)$ are pdff, we know by Lemma \ref{stalk} that 
    $$\psi_1(\SF_1 \boxtimes \SG_2)_{(x_1, x_2)} = \Gamma(\ol{B}_{\epsilon}({x_1}) \times \ol{B}_{\epsilon'}({x_2}) \times (0, \epsilon], \SF_1 \boxtimes \SG_2).$$
    On the other hand, since $\psi_1(\dot{\ss}_\pi(\SF_1))$ is pdff, by Lemma \ref{stalk},
    $$(\psi_1 \SF_1 \boxtimes \SG_2)_{(x_1, x_2)} = (\psi_1\SF_1)_{x_1} \otimes (\SG_2)_{x_2} = \Gamma(\ol{B}_\epsilon({x_1}) \times (0, \epsilon], \SF_1) \otimes (\SG_2)_{x_2}.$$
    We may assume that $\epsilon > 0$ is sufficiently small so that 
    $$\dot{\ss}(\SG_2) \cap \dot{N}^*_{in/out}B_{\epsilon'}(x_2) = \varnothing.$$
    This immediately implies that 
    $$\dot{\ss}_\pi(\SF_1) \times \dot{\ss}(\SG_2) \cap \dot{N}^*B_\epsilon(x_1) \times \dot{N}^*B_{\epsilon'}(x_2) \times (0, \epsilon) = \varnothing.$$
    By noncharacteristic deformation lemma applied to the family $\ol{B}_{\epsilon}(x_1) \times \ol{B}_{\epsilon'}(x_2) \times (0, \epsilon]$ as $\epsilon' \to 0$, we get
    \begin{align*}
    (\psi_1 \SF_1 \boxtimes \SG_2)_{(x_1, x_2)} &= \Gamma (\ol{B}_{\epsilon}({x_1}) \times \ol{B}_\epsilon({x_2}) \times (0, \epsilon], \SF_1 \boxtimes \SG_2) \\
    &\simeq  \Gamma(\ol{B}_{\epsilon}({x_1}) \times x_2 \times (0, \epsilon], \SF_1 \boxtimes \SG_2) \\
    &=  \Gamma(\ol{B}_\epsilon({x_1}) \times (0, \epsilon], \SF_1) \otimes (\SG_2)_{x_2} = (\psi_1 \SF_1 \boxtimes \SG_2)_{(x_1, x_2)}.
    \end{align*}
    This completes the proof.
\end{proof}

\begin{proposition}\label{nearby cycle product}
    Let $\SF_1 \in \Sh(M_1 \times \bR_{>0})$ and $\SF_2 \in \Sh(M_2 \times \bR_{>0})$. Suppose that $\dot{\ss}(\SF_1)$ and $\dot{\ss}(\SF_2)$ are $\bR_{>0}$-non-characteristic and $\psi_1(\dot{\ss}_\pi(\SF_1))$ and $\psi_2(\dot{\ss}_\pi(\SF_2))$ are pdff. Then there is an isomorphism of sheaves on $M_1 \times M_2$:
    $$\psi_{12}(\SF_1 \boxtimes_{\bR_{>0}} \SF_2) \simeq \psi_1 \SF_1 \boxtimes \psi_2\SF_2.$$
\end{proposition} 
\begin{proof}    
    We denote the relevant embeddings:
    \begin{gather*}
    M_1 \times M_2 \times 0 \xrightarrow{i} M_1 \times M_2 \times \bR \xleftarrow{j} M_1 \times M_2 \times \bR_{>0}, \\
    M_1 \times M_2 \times 0 \times \bR \xrightarrow{\ol{i}_1} M_1 \times M_2 \times \bR \times \bR\xleftarrow{\ol{j}_1} M_1 \times M_2 \times \bR_{>0} \times \bR, \\
    M_1 \times M_2 \times 0 \times 0 \xrightarrow{\ol{i}} M_1 \times M_2 \times \bR \times \bR \xleftarrow{\ol{j}} M_1 \times M_2 \times \bR_{>0} \times \bR_{>0}, \\
    M_1 \times 0 \xrightarrow{i_1} M_1 \times \bR \xleftarrow{j_1} M_1 \times \bR_{>0},\\
    \delta: \bR_{>0} \hookrightarrow \bR_{>0} \times \bR_{>0}, \; \delta: \bR \hookrightarrow \bR \times \bR.
    \end{gather*}

    We know that the left hand side equals to $i^*j_*\delta^*(\SF_1 \boxtimes \SF_2)$, and the right hand side equals to $\delta^*(i_1^*j_{1*}\SF_1 \boxtimes i_2^*j_{2*}\SF_2) = \ol{i}^*(j_{1*}\SF_1 \boxtimes j_{2*}\SF_2)$. Hence, it suffices to show that
    $$i^*j_*\delta^*(\SF_1 \boxtimes \SF_2) \simeq \ol{i}^*(j_{1*}\SF_1 \boxtimes j_{2*}\SF_2).$$
    On the left hand side, we use Lemma \ref{nearby commute with restrict} and deduce that
    $$i^*j_*\delta^*(\SF_1 \boxtimes \SF_2) \xrightarrow{\ref{nearby commute with restrict}} i^*\delta^*\ol{j}_*(\SF_1 \boxtimes \SF_2) \xrightarrow{\sim} \ol{i}^*\ol{j}_*(\SF_1 \boxtimes \SF_2).$$
    On the right hand side, we use Lemma \ref{nearby commute Nadler} and \ref{remove irrelevant factors} and deduce that
    \begin{align*}
    \ol{i}^*(j_{1*}\SF_1 \boxtimes j_{2*}\SF_2) &\xrightarrow{\,\sim\,} {i}_1^*j_{1*}\SF_1 \boxtimes {i}_2^*j_{2*}\SF_2 \xrightarrow{\ref{remove irrelevant factors}} \ol{i}_1^*\ol{j}_{1*}(\SF_1 \boxtimes i_2^*j_{2*}\SF_2) \\
    &\xrightarrow{\ref{remove irrelevant factors}} {i}_1^*{j}_{1*}\ol{i}_2^*\ol{j}_{2*}(\SF_1 \boxtimes \SF_2) \xrightarrow{\ref{nearby commute Nadler}} \ol{i}^*\ol{j}_*(\SF_1 \boxtimes \SF_2).
    \end{align*}
    This then completes the proof.
\end{proof}

    More generally, let $\mathrm{Flag}_k$ be the partial flags in the simplex $[k]$ with partial orders (which encodes all the different ways to take the $k$-fold box tensor product). Then we have:

\begin{corollary}\label{rem: box tensor commute higher}
    Consider $\SF_1 \in \Sh(M_1 \times \bR_{>0}), \dots, \SF_k \in \Sh(M_k \times \bR_{>0})$ such that $\dot\SS(\SF_j)$ are $\bR_{>0}$-non-characteristic and $\psi_j(\dot\SS_\pi(\SF_j))$ are pdff. Then for any $1 < l_1 < \dots < l_r < k$ there are natural isomorphisms
    $$\psi_{12\cdots k}(\SF_1 \boxtimes_{\bR_{>0}} \dots \boxtimes_{\bR_{>0}} \SF_k) \simeq \psi_{j_1\cdots j_{l_1}}(\SF_{j_1} \boxtimes_{\bR_{>0}} \dots \boxtimes_{\bR_{>0}} \SF_{j_{l_1}}) \boxtimes \cdots  \psi_{{j_{l_r+1}}\cdots j_k}(\SF_{j_{l_r+1}} \boxtimes_{\bR_{>0}} \dots \boxtimes_{\bR_{>0}} \SF_{j_k})$$
    that fit into a coherent diagram $\mathrm{Flag}_{k} \to \Fun^{ex}(\Sh(\prod_{j=1}^k (M_j \times \bR_{>0})), \Sh(\prod_{j=1}^k M_j))$ where partial flag given by $j_1, j_2, \dots, j_k$ and $1 < l_1 < \dots < l_r < k$ is sent to the right hand side of the equation.
\end{corollary}
\begin{proof}
    We have a diagram of functors with (non-proper) base changes induced by adjunctions as in Figure \ref{fig:nearby exterior tensor}, which defines a diagram of functors with (op)lax diagram of natural transformations by \cite[Corollary F]{HaugsengHebestreitLinskensNuiten}; see also \cite[Appendix A Section 12.3]{Gaitsgory-Rozenblyum}. Then by applying Corollary \ref{rem: nearby commute higher} iteratively, we can show that this is a commutative diagram of functors where natural transformations are invertible.
\end{proof}

%

\begin{figure}
    \centering
    \includegraphics[width=1\linewidth]{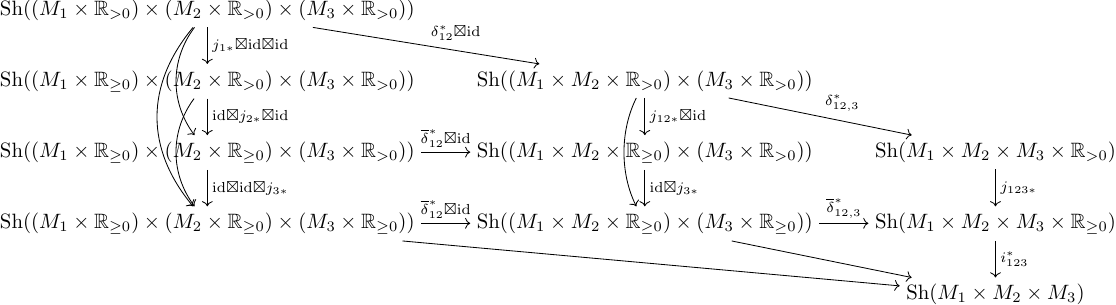}
    \caption{Part of the diagram of that appears in the commutativity of the exterior tensor products and nearby cycle functors, where the natural transformation in each square is induced by adjunctions of six functors.}
    \label{fig:nearby exterior tensor}
\end{figure}

\subsection{Global sections of nearby cycles and tensor products}

We recall the notion of gapped families of subsets. Write $T^*_\Delta(M\times M) \subset T^*(M \times M)$ for the conormal to the diagonal $\Delta \subset M \times M$.

\begin{definition}[{\cite[Definition 2.10]{Nadler-Shende}}]\label{def: gap}
    Let $(\Lambda_t, \Lambda_t')$ be a family of pairs of subsets inside the contact manifold $X$ with a given Reeb flow. Then $(\Lambda_t, \Lambda_t')$ is gapped if the lengths of all the Reeb chords between $\Lambda_t$ and $\Lambda_t'$ are bounded from below by some $\epsilon > 0$. $(\Lambda_t, \Lambda_t')$ is positively gapped if the lengths of all the Reeb chords from $\Lambda_t'$ to $\Lambda_t$ are bounded from below by some $\epsilon > 0$.
\end{definition}

    We recall the full faithfulness theorem on gapped nearby cycle functors: 
    
\begin{theorem}[{\cite[Theorem 5.1]{Nadler-Shende}}]\label{fullfaithful nearby}
    Let $\SF, \SG \in \Sh(M \times \bR_{>0})$. Suppose that $\ss(\SF)$ and $\ss(\SG)$ are $\bR_{>0}$-non-characteristic, $\psi(-\ss(\SF)) \times \psi(\ss(\SG))$ is pdff, and that
    the family $(\dot{\ss}_\pi(\SF), \dot{\ss}_\pi(\SG))$ is positively\footnote{In \cite{Nadler-Shende}, the fact that only positive gappedness is needed is not mentioned, but that is what the proof uses. See \cite[Lemma 2.19 \& 5.2]{Nadler-Shende}.} gapped, or equivalently that $(-\ss_\pi(\SF) \times_{\bR_{>0}} \ss_\pi(\SG), \dot{N}^*\Delta_{M} \times \bR_{>0})$ is positively gapped.\footnote{In \cite{Nadler-Shende}, this `or equivalently' does not appear in the statement, but does in the proof. See for example \cite[Lemma 2.19]{Nadler-Shende}.}  Then
    $$\Hom(\SF, \SG) \simeq \Hom(\psi \SF, \psi \SG).$$
\end{theorem}

    In this subsection we will show the analogous result for $\otimes$. 
    
    Similar to \cite[Proposition 2.25 \& Lemma 5.2]{Nadler-Shende}, we need a base change formula between push-forward and pull-backs. We explain a model case. 
    Consider the maps
    \[\begin{tikzcd}
        \bR_{>0} \times 0 \ar[d, "\ol{i}_1" left] \ar[r, "j_2"] & \bR_{\geq 0} \times 0 \ar[d, "i_1"] \\
        \bR_{>0} \times \bR_{\geq 0} \ar[r, "\ol{j}_2"] & \bR_{\geq 0} \times \bR_{\geq 0}
    \end{tikzcd}\]
    Using non-characteristic deformation Lemma \ref{lem: non-char}, we can easily show the following lemma (which plays the same role as \cite[Lemma 2.24]{Nadler-Shende}):
    
\begin{lemma}\label{lem: non-proper base change}
    For a sheaf $\SF \in \Sh(\bR_{>0} \times \bR_{\geq 0})$ such that
    \begin{enumerate}
        \item $\SS(\SF)$ is non-characteristic with respect to the projection $\pi_1: \bR_{>0} \times \bR_{\geq 0} \to \bR_{>0}$, and
        \item $\SS_{\pi_1}(\SF)$ is gapped from $\dot N^*(\bR_{>0} \times 0)$ are $\bR_{>0}$-families of subsets,
    \end{enumerate}
    the natural transformation defines an isomorphism
    $$i_{1}^*\ol{j}_{2*}\SF \xrightarrow{\sim} j_{2*} \ol{i}_1^*\SF.$$
\end{lemma}

    We denote the diagonal embeddings by
    $$\Delta_{\bR_{>0}}: M \times \bR_{>0} \hookrightarrow M \times M \times \bR_{>0}, \; \Delta: M \hookrightarrow M \times M.$$
    In our situation, the main technical lemma we need is the following of which we write a complete proof for expository reasons:

\begin{lemma}\label{gapped basechange}
    Let $\SF, \SG \in \Sh(M \times \bR_{>0})$ be with proper supports via $p: M \times \bR_{>0} \to \bR_{>0}$. Suppose $\ss(\SF) \times \ss(\SG)$ is $\bR_{>0}$-noncharacteristic, $\psi(\ss(\SF)) \times \psi(\ss(\SG))$ is pdff, and the family of pairs $(\ss_\pi(\SF) \times_{\bR_{>0}} \ss_\pi(\SG), \dot{N}^*\Delta_M \times \bR_{>0})$ is positively gapped, or equivalently that the family of pairs $(\ss_\pi(\SF), -\ss_\pi(\SG))$ is positively gapped. Then
    $$\Gamma(M, \psi \Delta_{\bR_{>0}}^*(\SF \boxtimes_{\bR_{>0}} \SG)) \simeq \Gamma(M, \Delta^* \psi (\SF \otimes_{\bR_{>0}} \SG)).$$
\end{lemma}
\begin{proof}
    We write the projection maps
    $$p: M \times \bR_t \to \bR_t, \; d: M \times M \times \bR_t \to \bR_d$$
    where $d$ is a submersion such that $d(\Delta_M \times \bR_t) = 0$. Under the assumption that $\SF$ and $\SG$ have proper supports, we may apply proper base change and the proof is reduced to
    $$(p \circ j)_* \Delta_{\bR_{>0}}^*(\SF \boxtimes_{\bR_{>0}} \SG) \simeq i_0^*(d \circ j)_*(\SF \boxtimes_{\bR_{>0}} \SG),$$
    where $i_0: 0 \times \bR_t \hookrightarrow \bR_d \times \bR_t$ is the embedding.

    Since $\psi(\ss(\SF) \times \ss(\SG))$ is pdff and $\ss(\SF) \times \ss(\SG)$ is $\bR_{>0}$-noncharacteristic, the stalks at $t = 0$ can be computed by the section on a small open neighbourhood by Lemma \ref{stalk}. On the left hand side, the stalk is
    $$\Gamma(\Delta_M \times (0, \epsilon], \SF \boxtimes_{\bR_{>0}} \SG).$$
    On the right hand side, considering tubular neighbourhoods $U_{\epsilon'}(\Delta_M)$ of $\Delta_M$ with smooth boundary, the stalk at $t = 0$ can be computed by
    $$\Gamma(\ol{U}_{\epsilon'}(\Delta_M) \times (0, \epsilon], \SF \boxtimes_{\bR_{>0}} \SG).$$
    Therefore, it suffices to show that one can apply non-characteristic deformation lemma to the family $\ol{U}_{\epsilon'}(\Delta_M) \times (0, \epsilon]$ for $\epsilon' \to 0$. 

    Let $R_s$ be the Reeb flow on $\dot{T}^*(M \times M)$. The image of $\dot{T}^*_\Delta(M \times M)$ under the Reeb flow is the outward conormal bundle of $U_{\epsilon'}(\Delta_M)$, i.e.~$R_{\epsilon'}(\dot{T}^*_\Delta(M \times M)) = \dot{N}_{out}^*U_{\epsilon'}(\Delta_M)$. It suffices to show that for any $\epsilon' > 0$ small enough,
    $$\ss(\SF) \times \ss(\SG) \cap R_{\epsilon'}(\dot{T}^*_\Delta(M \times M)) = \varnothing.$$
    This is because $\ss(\SF) \times \ss(\SG) \cap R_{\epsilon'}(\dot{T}^*_\Delta(M \times M))$ are short Reeb chords from $\dot{T}^*_\Delta(M \times M)$ to $\ss(\SF) \times \ss(\SG)$, or equivalently the short Reeb chords from $-\dot{\ss}(\SF)$ to $\dot{\ss}(\SG)$. Since the pair is known to be positively gapped, for $\epsilon' > 0$ small enough there are no Reeb chords between them. Thus, the isomorphism
    $$\Gamma(\Delta_M \times (0, \epsilon], \SF \boxtimes_{\bR_{>0}} \SG) \simeq \Gamma(\ol{U}_{\epsilon'}(\Delta_M) \times (0, \epsilon], \SF \boxtimes_{\bR_{>0}} \SG)$$
    follows from the non-characteristic deformation lemma.
\end{proof}

\begin{theorem}\label{nearby fully faithful}
    Let $\SF, \SG \in \Sh(M \times \bR_{>0})$ be sheaves with proper supports via $p: M \times \bR_{>0} \to \bR_{>0}$. Suppose $\ss(\SF)$ and $\ss(\SG)$ are $\bR_{>0}$-noncharacteristic, $\psi(\ss(\SF) \times \ss(\SG))$ is pdff, and the family of pairs $(\ss_\pi(\SF) \times \ss_\pi(\SG), \dot{T}^*_\Delta(M\times M) \times \bR_{>0})$ is positively gapped, or equivalently that the family of pairs $(\dot{\ss}_\pi(\SF), -\dot{\ss}_\pi(\SG))$ is positively gapped. Then
    $$\Gamma(M, \psi(\SF \otimes_{\bR_{>0}} \SG)) \simeq \Gamma(M, \psi\SF \otimes \psi\SG).$$
\end{theorem}
\begin{proof}
    The isomorphism follows from the following compositions
    \begin{align*}
    \Gamma(M, \psi(\SF \otimes_{\bR_{>0}} \SG)) &\xrightarrow{\,\,\sim\,\,} \Gamma(M, \psi\Delta_{\bR_{>0}}^*(\SF \boxtimes_{\bR_{>0}} \SG)) \\
    &\xrightarrow{\ref{gapped basechange}} \Gamma(M, \Delta^*\psi(\SF \boxtimes_{\bR_{>0}} \SG)) \\
    &\xrightarrow{\,\ref{nearby cycle product}\,} \Gamma(M, \Delta^*(\psi\SF \boxtimes \psi\SG)) \\
    &\xrightarrow{\,\,\sim\,\,} \Gamma(M, \psi\SF \otimes \psi\SG). \qedhere
    \end{align*}
\end{proof}

\begin{example}\label{ex: cotangent bundle hom}
    Let $M = N \times \bR_t$ and the diffeotopy $\varphi_s(x, t) = \varphi_s(x, st)$. Consider $\SF, \SG \in \Sh_{\tau>0}(N \times \bR_t \times \bR_{>0})$ be the pull-backs of some sheaves in $\Sh_{\tau>0}(N \times \bR_t)$ via the diffeotopy $\varphi: N \times \bR_t \times \bR_{>0} \to N \times \bR_t$. Then for sufficiently large $t \in \bR$, we know that $(\ss_\pi^\infty(\SF), \ss_\pi^\infty(T_t\SG))$ is positively gapped, as there are no Reeb chords from one to the other. Therefore, we have
    \begin{align*}
    \Hom(\psi (\SF), \psi(\SG)) &= \Hom(\psi(\SF), \psi (T_t\SG)) = \Hom(\SF, T_t\SG).
    \end{align*}
\end{example}


%
%
%
%
%
%

%% file: microsheafreview.tex

\section{Recollections on microsheaves}\label{sec:microsheaf recollections}

    We recall some basic results on microsheaves. In particular prove several folklore results, including presentability of the microsheaf categories, constructibility of the sheaf of microsheaf categories, and K\"unneth formulae for such categories under suitable Lagrangian hypotheses.

\subsection{Microsheaves}\label{ssec: microsheaf definition}

    Following \cite[Section 6.1]{KS},
    for any open subset $\underline\Omega \subset T^*M$, we set
    $$\msh^\text{pre}(\underline\Omega) = \Sh(M; \underline\Omega) = \Sh(M)/\Sh_{T^*M \setminus \underline\Omega}(M).$$
    For any $\underline\Omega' \subset \underline\Omega$, we have a restriction functor $\msh^\text{pre}(\underline\Omega) \to \msh^\text{pre}(\underline\Omega')$ induced by the inclusion $\Sh_{T^*M \setminus \underline\Omega}(M) \hookrightarrow \Sh_{T^*M \setminus \underline\Omega'}(M)$. 
    
    We view  $\msh^\text{pre}$ as a presheaf valued in the category of all  categories, and sheafify it as such to obtain $\msh$.  (See \cite[Remark 6.1]{Nadler-Shende} for a discussion of why we sheafify here as as opposed to in $\PrL$ or $\PrR$; the basic reason is that stalks do not work well for sheaves valued in the latter categories.) We refer to sections of $\msh$ as `microsheaves'. By computing the stalks of the internal Hom \cite[Theorem 6.1.2]{KS}, one can show that the internal Hom of $\msh$ is computed by the $\mu hom$ functor in Kashiwara--Schapira \cite[Defintion 4.1]{KS}.  

    The sheaf of categories  $\msh$ is invariant under the $\bR_+$-action by dilation in $T^*M$. One can restrict $\msh$ to the cosphere bundle $S^*M = (T^*M \setminus M)/\bR_+$. From the definitions there is an evident `microlocalization' functor:
    $$m_{S^*M}: \Sh(M) \xrightarrow{\sim} \msh(T^*M) \to \msh(S^*M).$$

    Let $\Lambda \subset S^*M$ be a subset, we write $\msh_\Lambda$ for the subsheaf of $\msh$ whose sections are objects that are microsupported on $\Lambda$.  Microlocalization restricts to the subcategories with prescribed microsupport: 
    $$m_\Lambda: \Sh_\Lambda(M) \to \msh_\Lambda(S^*M) \xrightarrow{\sim} \msh_\Lambda(\Lambda).$$

    As with objects in any quotient, to prove results about microsheaves it is necessary to argue in terms of representative sheaves. 
    The basic tool for producing sheaf representatives from given microsheaves is the microlocal cutoff of Kashiwara and Schapira, especially in its `refined' variant \cite[Proposition 6.1.4]{KS}. It is perhaps most powerful when the support of the microsheaf in $S^*M$ satisfies the finite position hypothesis as in Definition \ref{def: ptfp}, i.e.~when the projection to $M$ is finite-to-one.

    This condition ensures the following consequence via the refined microlocal cutoff:

\begin{proposition}[{\cite[Lemma 10.2.5]{Guillermou-survey}}, {\cite[Proposition 7.10]{Nadler-Shende}}]\label{prop: microstalk}
    Let $\Lambda \subset S^*M$ be a 
    locally closed subset in finite position at $x \in M$, and let $\pi: S^*M \to M$ be the projection. Then the natural morphism between sheaves of categories
    $$\Sh_\Lambda/\,\Loc \longrightarrow \pi_*\msh_\Lambda$$
    is an isomorphism at $x \in M$. 
\end{proposition}
\begin{remark}
    This fact is also (and earlier) used in the complex setting \cite[Section 5.1]{Waschkies}, where it however requires stronger results on the microlocal cutoff \cite[Theorem 3.2]{DAgnolo}. 
\end{remark}

    The following lemma asserts the invariance of microsheaves under contactomorphism and stabilization, and will be frequently used in later sections:

\begin{lemma}[{\cite[Theorem 7.2.1]{KS}, \cite[Lemma 6.3]{Nadler-Shende}}]    
    \label{lem:contact-transform-main}
    Let $\Lambda' \subset S^*M$ be a subset that is contactomorphic to $\Lambda \times 0_B \subset S^*N \times T^*B$, where $B$ is a ball and $\Lambda$ is contractible. Let $\pi: \Lambda' \to \Lambda$ be the projection. Then there is a (noncanonical) isomorphism $\pi^*\msh_\Lambda \cong \msh_{\Lambda'}$. 
\end{lemma}

\subsection{Doubling functor}\label{ssec: doubling}

    When $\Lambda$ is a smooth Legendrian, Guillermou upgraded Proposition \ref{prop: microstalk} to a construction of {\em global} representatives, by what is now called `doubling'  \cite[Part 11]{Guillermou-survey}. This was generalized to allow singular $\Lambda$ in \cite[Section 7]{Nadler-Shende}; here we recall a variant of this construction given in \cite[Section 4]{KuoLi-spherical}.

    Let $\Lambda \subset S^*M$ be a relatively compact subset with closure $\overline{\Lambda}$.  We write $\partial\Lambda = \overline{\Lambda} \setminus \Lambda$ for the boundary of $\Lambda$. Suppose $\Lambda$ is positively self displaceable as in Definition \ref{def: self displaceable}. Then we consider a contact Hamiltonian pushoff of $\Lambda$ which is positive on an open neighborhood of $\Lambda$ and zero on $\partial\Lambda$ such that $\Lambda_s$ is disjoint from $\Lambda$ for $s \neq 0$.
    
    We define the doubling movie to be $\Lambda^\prec = \Lambda_- \cup \Lambda_+ \subset S^*(M \times \bR)$ and the doubling at $s \in \bR$ to be the reduction $\overline{\Lambda}_{-s} \cup \overline{\Lambda}_s \subset S^*M$. $\Lambda_{\pm s}$ is the Hamiltonian pushoff of $\Lambda$ along the Reeb flow by time $\pm s^{3/2}$ and $\Lambda_\pm \subset S^*(M \times \bR)$ is the movie of the Hamiltonian pushoff by the Reeb flow for time $\pm s^{3/2}$. We let
    $$\underline\Lambda_{\cup,s} = (\overline{\Lambda}_{-s} \times \bR_+) \cup (\overline{\Lambda}_{s} \times \bR_+) \cup \bigcup_{-s \leq r \leq s} \pi(\Lambda_r) \subset T^*M.$$
    Respectively, we also let
    $$\underline\Lambda_{\cup}^\prec = (\ol{\Lambda}_{-} \times \bR_+) \cup (\ol{\Lambda}_{+} \times \bR_+) \cup \bigcup_{-s \leq r \leq s} (\pi(\ol{\Lambda}_r) \times s) \subset T^*(M \times \bR).$$
    Moreover, we will set $\underline\Lambda_{\cup,s}^\prec = \underline\Lambda_{\cup}^\prec \cap T^*(M \times \bR_{\leq s})$.

\begin{theorem}[{\cite[Theorem 4.47 \& Proposition 6.8]{KuoLi-spherical}}] \label{thm: relative-doubling} 
    Let $\Lambda \subset S^*M$ be a relatively compact sufficiently Legendrian subset\footnote{The reference assumes stronger constructibility hypotheses, but the proof goes through without change.}. Then for $s>0$ sufficiently small, the microlocalization $m_{\Lambda_{-s}}$ induces an equivalence: 
    $$m_{\Lambda_{-s}}: \Sh_{\underline\Lambda_{\cup,s}}(M) \xrightarrow{\sim} \msh_{\Lambda_{-s}}(\Lambda_{-s}).$$
    We denote the inverse as $w_{\Lambda_{-s}}$ and term it the doubling functor. Moreover, there is an equivalence
    $$\Sh_{\underline\Lambda_{\cup,s}^\prec}(M \times \bR_{\leq s}) \xrightarrow{\sim} \msh_{\Lambda_{-s}}(\Lambda_{-s}).$$
    We denote its inverse by $w_{\Lambda^\prec_s}$ and term it the parametric doubling functor. 
\end{theorem}

We will also sometimes use the following variant (essentially differing only by a change of coordinates):
    \begin{gather*}
    \underline\Lambda_{\cup,s}^+ = (\overline{\Lambda} \times \bR_+) \cup (\overline{\Lambda}_{s} \times \bR_+) \cup \bigcup_{0 \leq r \leq s} \pi(\Lambda_r) \subset T^*M, \\ \underline\Lambda_{\cup,s}^- = (\overline{\Lambda}_{-s} \times \bR_+) \cup (\overline{\Lambda} \times \bR_+) \cup \bigcup_{-s \leq r \leq 0} \pi(\Lambda_r) \subset T^*M.
    \end{gather*}
\begin{definition}\label{def: positive doubling}
    Let $\Lambda \subset S^*M$ be a relatively compact sufficiently Legendrian subset. Then for $s>0$ sufficiently small, the microlocalization $m_{\Lambda}$ induces an equivalence
    $$m_{\Lambda}: \Sh_{\underline\Lambda_{\cup,s}^\pm}(M) \xrightarrow{\sim} \msh_{\Lambda}(\Lambda).$$
    We denote the inverse as $w_{\Lambda}^\pm$ and also call it the (positive/negative) doubling functor.
\end{definition}

\begin{remark}\label{rem: doubling limit}
    When $\Lambda$ is only pertubable to finite position as in Definition \ref{def: ptfp} (potentially not self displaceable as in Definition \ref{def: self displaceable}), the first part of the theorem fails, but the second part still holds \cite[Lemma 7.27]{Nadler-Shende}:
    $$\Sh_{\underline\Lambda_{\cup,s}^\prec}(M \times \bR_{\leq s}) \xrightarrow{\sim} \msh_{\Lambda_{-s}}(\Lambda_{-s}).$$
    In fact, the above equivalence is enough for many of the applications. However, we do not go further into this direction and this will not be necessary for our main applications.
\end{remark}

    The original version of \cite[Theorem 7.28 \& Proposition 7.29]{Nadler-Shende} had the following additional hypothesis:   
\begin{definition}\label{def: contact collar}
    Let $\Lambda \subset S^*M$ be any subset. Then we say $\Lambda$ has a contact collar $(\partial\Lambda, U)$ for a subset $\partial \Lambda$ in some contact manifold $U$ if there is an open contact embedding $(\partial \Lambda \times (0,\epsilon), U \times T^*(0,\epsilon)) \hookrightarrow (\Lambda, S^*M)$. A contact flow is compatible with the contact collar if it is restricts to the pull-back of a contact flow on $U$.
\end{definition}
    Existence of a contact collar allows to produce a smooth double $\Lambda^\prec_{sm}$ given a smooth $\Lambda$. This smoothness is not necessary for the results in the present subsection, but will return later as a hypothesis. 
    
\subsection{Adjoints and doubling}

    Recall that the microsupport of a limit or colimit of sheaves is contained in the closure of the union of microsupports of all the terms (see \cite[Exercise V.7]{KS}, or its solution \cite[Proposition 3.4]{Guillermou-Viterbo}).  This has the following well-known immediate consequences:

\begin{lemma} \label{sheaves with prescribed microsupport closed under limits}
    Let $\underline\Lambda \subset T^*M$ be a closed conic subset. Then $\Sh_{\underline\Lambda}(M)$ is a presentable category, and the tautological inclusion
    $\iota_{\underline\Lambda *}: \Sh_{\underline\Lambda}(M) \hookrightarrow \Sh(M)$
    is limit and colimit preserving. 
\end{lemma}

\begin{corollary}\label{cor: sheaf-bicomplete}
    The inclusion $\iota_{\underline\Lambda *}: \Sh_{\underline\Lambda}(M) \hookrightarrow \Sh(M)$ has a left adjoint, which we will denote $\iota_{\underline\Lambda}^*$, and a right adjoint, which we will denote $\iota_{\underline\Lambda}^!$. 
\end{corollary}
\begin{proof}
    For a manifold $M$, $\Sh(M)$ is presentable by \cite[Corollary 7.1.4.4]{Lurie-HTT} and \cite[Corollary 2.24]{Volpe-six-operations}. Then Lemma \ref{sheaves with prescribed microsupport closed under limits} implies that $\Sh_{\underline\Lambda}(M)$ is also presentable, and the result follows from the adjoint functor theorem of presentable categories \cite[Corollary 5.5.2.9]{Lurie-HTT}.
\end{proof}

    In fact, the microlocalization functor admits left and right adjoints, and they factor through doubling:

\begin{theorem}[{\cite[Theorem 4.2 \& 4.47]{KuoLi-spherical}\footnote{The reference uses stronger constructibility hypothesis. However, the proof does not rely on knowing a priori that the category of microsheaves is presentable or the microlocalization admits adjoint functors, so establishes the present assertion.}, which generalizes \cite[Lemma 12.2.1]{Guillermou-survey}}]\label{thm: doubling adjoint} 
    Let $\Lambda' \subset S^*M$ be any compact subset and $\Lambda \subset \Lambda'$ be an open sufficiently Legendrian subset. Then the left (right) adjoint of 
    $$m_{\Lambda}: \Sh_{\Lambda'}(M) \to \msh_{\Lambda}(\Lambda)$$
    factors as a composition of doubling and the left (right) adjoint of the tautological inclusion
    $$\iota_{\Lambda'}^* \circ w_{\Lambda_{-s}}: \msh_{\Lambda}(\Lambda) \to \Sh_{\underline{\Lambda}{}_{\cup,s}}(M) \to \Sh_{\Lambda'}(M).$$
\end{theorem}


    We will want to apply the theorem to results established in \cite{Nadler-Shende}, which however used a different construction of doubling \cite[Section 7]{Nadler-Shende}. One can either observe that \cite{Nadler-Shende} goes through using the doubling from (wherever) or observe that the doublings are in fact equivalent in the following sense. We denote the relative doubling in \cite[Section 7.3]{Nadler-Shende} by $\Lambda^\prec_{sm}$ and the relative doubling in \cite[Section 4.6]{KuoLi-spherical} by $\Lambda^\prec$.

\begin{proposition} 
    Let $\Lambda \subset S^*M$ be sufficiently Legendrian with contact collar. Then there is a contact isotopy that sends $\Lambda^\prec_{sm}$ into an arbitrary small neighborhood of $\Lambda^\prec$ and vice versa, which induces a pair of nearby cycle functors that induces equivalences and fits into the diagram
    \[\begin{tikzcd}
    \Sh_{\underline\Lambda_{\cup,s,sm}}(M) \ar[r, "\sim"] \ar[d, "m_\Lambda" left] & \Sh_{\underline\Lambda_{\cup,s}}(M) \ar[d, "m_\Lambda"] \\
    \msh_\Lambda(\Lambda) \ar[r, "="] & \msh_\Lambda(\Lambda).
    \end{tikzcd}\]
\end{proposition}
\begin{proof}
    Fix a Reeb vector field $R$ that is compatible with the contact collar. Recall that away from the contact collar, both doublings are defined by the contact movie of $\Lambda $ under the Reeb flow. Within the contact collar $U \times T^*(0, \epsilon) \times T^*\bR$, the smooth doubling $\Lambda_{\pm s,sm}$ of $\partial \Lambda \times (0, \epsilon)$ is defined by $\partial \Lambda \times P_{\pm s} \subset U \times T^*(0, \epsilon) \times$ defined by capping of the doubling of $\partial \Lambda_{\pm s}$ by a smooth parabola:
    $$\partial \Lambda \times P_{\pm s} = \{(\varphi_R^{f_\pm(t)}(x), t, f'_\pm(t)) \mid x \in \partial \Lambda, t \in (0, \epsilon)\},$$
    where $f: \bR_{\geq 0} \to \bR$ is a function such that $f_\pm(t) = \pm t^{3/2}$ when $t$ is sufficiently close to $0$ and $f_\pm(t) = \pm s$ when $t$ is close to $\pm \epsilon$. On the other hand, the doubling in $\Lambda_{\pm s}$ of $\partial \Lambda \times (0, \epsilon)$ is defined by
    $$(\partial \Lambda \times (0, \epsilon))_{\pm s} = \{(\varphi_R^{\rho_\pm(t)}(x), t, \rho_\pm'(t)) \mid x \in \partial \Lambda, t \in (0, \epsilon)\},$$
    where $\rho_\pm: \bR_{\geq 0} \to \bR$ are smooth cut-off functions. Then we can define the contact isotopy by considering a family of smooth functions $f_{\pm,t} \to \rho_{\pm}$ that induces a contact isotopy $\phi$ in $U \times T^*(0, \epsilon)$. Then by Lemma \ref{lem: ss-nearby-cycle}, we have a nearby cycle functor
    $$\Sh_{\underline\Lambda_{\cup,s,sm}}(M) \xrightarrow{\sim} \Sh_{(\underline\Lambda_{\cup,s,sm})_\phi}(M \times \bR_{>0}) \to \Sh_{\underline\Lambda_{\cup,s}}(M).$$
    Since the contact isotopy is identity away from the contact collar, we can conclude the diagram in the statement commutes by Lemma \ref{lem:contact-transform-main}. Then it follows from Theorem \ref{thm: relative-doubling} that the nearby cycle induces an equivalence.
\end{proof}

    By Theorem \ref{thm: doubling adjoint}, for an open subset $\Omega \subset \Lambda$ of sufficiently Legendrian subsets, we can realize both the restriction functor and its adjoint functors using doubling:

\begin{corollary}\label{cor: restrict doubling}
    Let $\Lambda \subset S^*M$ be any relative compact sufficiently Legendrian subset and $\Omega \subset \Lambda$ an open subset. Then the restriction functor $r_{\Omega}^*: \msh_{\Lambda}(\Lambda) \to \msh_{\Lambda}(\Omega)$ is given by
    $$m_{\Omega} \circ \iota_{\underline\Omega_{\cup,s}^+}^! \circ w_{\Lambda}^+: \msh_\Lambda(\Lambda) \to \Sh_{\underline\Lambda_{\cup,s}^+}(M) \to \Sh_{\underline\Omega_{\cup,s}^+}(M) \to \msh_{\Lambda}(\Omega),$$
    and the left adjoint $r_{\Omega,!}: \msh_{\Lambda}(\Omega) \to \msh_{\Lambda}(\Lambda)$ to the restriction functor is given by the composition
    $$m_{\Lambda} \circ \iota_{\underline{\Lambda}_{\cup,s}^+}^* \circ w_{\Omega}^+: \msh_{\Lambda}(\Omega) \rightarrow \Sh_{\underline{\Omega}_{\cup,s}^+}(M) \rightarrow \Sh_{\underline\Lambda_{\cup,s}^+}(M) \to \msh_{\Lambda}(\Lambda).$$
\end{corollary}

\begin{corollary}\label{rem: double vs wrap} 
    Let $\Lambda \subset S^*M$ and $L \subset S^*M$ be a sufficiently Legendrian subsets such that $\Lambda \subset L$ is a closed subset. Then there is a commutative diagram of left adjoints of restrictions and microlocalizations
    \[\begin{tikzcd}
    \msh_{L}(L) \ar[r] \ar[d, "\iota_{\Lambda}^*" left] & \Sh_{\underline{L}_{\cup, \epsilon}^+}(M) \ar[d, "\iota_{\underline{\Lambda}{}^+_{\cup,\epsilon}}^*"] \\
    \msh_{\Lambda}(\Lambda) \ar[r] & \Sh_{\underline{\Lambda}{}_{\cup,\epsilon}^+}(M).
    \end{tikzcd}\]
\end{corollary}

    From Theorems \ref{thm: relative-doubling}, \ref{thm: doubling adjoint}, Lemma \ref{sheaves with prescribed microsupport closed under limits}, we can immediately conclude that the category of microsheaves supported on sufficiently Legendrian is also presentable, and the microlocalization functor is limit and colimit preserving.

\begin{corollary}\label{cor: presentable doubling}
    Let $\underline\Lambda \subset T^*M$ be a conic subset and $\Lambda \subset S^*M$ be any relative compact sufficiently Legendrian subset. Then the sheaf of categories of microsheaves $\msh_{\underline\Lambda}$ takes values in presentable categories and the restriction functors are limit and colimit preserving. In particular, the microlocalization functor
    $$m_\Lambda: \Sh_{\Lambda}(M) \longrightarrow \msh_{\Lambda}(\Lambda)$$
    is limit and colimit preserving and admits left and right adjoints.
\end{corollary}

\subsection{K\"unneth and dualizability}\label{ssec: microsheaf duality}

Recall that there is the Lurie tensor product on the category of presentable categories; see \cite[Section 5.3.6]{Lurie-HTT} or \cite[Section 4.8.1]{Lurie-HA}. 
The functor `sheaves' from locally compact Hausdorff spaces to presentable categories is monoidal (as shown for sheaves in spaces in \cite[Propositions 7.3.1.11 \& 7.3.3.9]{Lurie-HTT} and explicitly base changed to other coefficients in \cite[Proposition 2.30]{Volpe-six-operations}):
    $$\Sh(M_1) \otimes \Sh(M_2) \simeq \Sh(M_1 \times M_2).$$
    The same is known to hold when prescribing microsupports, per \cite[Theorem B]{Zhang-cutoff}: for any closed conic $\underline\Lambda_1 \subset T^*M_1$ and $\underline\Lambda_2 \subset T^*M_2$ and open conic 
    $\underline U_1 \subset T^*M_1$ and $\underline U_2 \subset T^*M_2$,
    \begin{equation} \label{bingyu kunneth}
    \Sh_{\underline\Lambda_1}(M_1; \underline U_1) \otimes \Sh_{\underline\Lambda_2}(M_2; \underline U_2) \simeq  \Sh_{\underline\Lambda_1 \times \underline\Lambda_2}(M_1 \times M_2; \underline U_1 \times \underline U_2) .
    \end{equation}
Such K\"unneth formulas imply that appropriately limit/colimit preserving functors between sheaf categories can be represented by integral kernels.  

In this section we develop similar results for microsheaf categories.  Note that we must restrict to microsheaf categories with sufficiently Legendrian support to ensure presentability of microsheaf categories (via Corollary \ref{cor: presentable doubling}).  

\begin{theorem}\label{thm:kunneth-microsheaf} 
    Let $\underline\Lambda_1 \subset T^*M_1$ and $\underline\Lambda_2 \subset T^*M_2$ be any closed conic subsets with $\Lambda_1, \Lambda_2$ compact stratified Legendrians. Then there is an equivalence of sheaves of categories
    $$\msh_{\underline\Lambda_1} \boxtimes \msh_{\underline\Lambda_2}  \simeq \msh_{\underline\Lambda_1 \times \underline\Lambda_2} .$$
    In particular, there is an equivalence of categories
    $$\msh_{\underline\Lambda_1}(\underline\Lambda_2) \otimes \msh_{\underline\Lambda_1}(\underline\Lambda_2) \simeq \msh_{\underline\Lambda_1 \times \underline\Sigma_2}(\underline\Lambda_1 \times \underline\Lambda_2) .$$
\end{theorem}
\begin{proof}
    It is immediate from Equation \eqref{bingyu kunneth} that there exists a natural functor of presheaves of categories:    $$\msh_{\underline\Lambda_1}^{pre} \boxtimes \msh_{\underline\Lambda_2}^{pre} \xrightarrow{\sim} \msh_{\underline\Lambda_1 \times \underline\Lambda_2}^{pre}.$$
    Since Proposition \ref{cor: microsheaf-bicomplete} shows that $\msh_{\underline\Lambda_1}$ and $\msh_{\underline\Lambda_2}$ take values in presentable categories, we can define the sheaf of categories $\msh_{\underline\Lambda_1} \boxtimes \msh_{\underline\Lambda_2}$ (resp.~the presheaf of categories $\msh_{\underline\Lambda_1}^{pre} \boxtimes \msh_{\underline\Lambda_2}^{pre}$). Using the universal property of the sheafification functor, we obtain a natural functor of sheaves of categories (where $(-)^{sh}$ is the sheafification functor)
    $$\msh_{\underline\Lambda_1} \boxtimes \msh_{\underline\Lambda_2} \to \big(\msh_{\underline\Lambda_1}^{pre} \boxtimes \msh_{\underline\Lambda_2}^{pre} \big)^{sh} \to \msh_{\underline\Lambda_1 \times \underline\Lambda_2}.$$
    Both morphisms are in fact isomorphisms, since this can be checked on stalks, which are determined by the corresponding presheaves. 
    
    Finally, when $\underline\Lambda_1$ and $\underline\Lambda_2$ are compact subsets, we apply the projection formula for the categories of sheaves valued in the bicomplete category of stable categories \cite[Propositions 2.30 \& 6.11]{Volpe-six-operations} and obtain that for the projection maps $\pi_{\underline\Lambda_1}: \underline\Lambda_1 \to pt$ and $\pi_{\underline\Lambda_2}: \underline\Lambda_2 \to pt$,
    $$ \pi_{\underline\Lambda_1 *}\msh_{\underline\Lambda_1} \otimes \pi_{\underline\Lambda_2*}\msh_{\underline\Lambda_2} = \pi_{\underline\Lambda_1 \times \underline\Lambda_2,*} \msh_{\underline\Lambda_1} \boxtimes \msh_{\underline\Lambda_2} = \pi_{\underline\Lambda_1 \times \underline\Lambda_2,*}\msh_{\underline\Lambda_1 \times \underline\Lambda_2}.$$
    This shows that $\msh_{\underline\Lambda_1}(\underline\Lambda_1) \otimes \msh_{\underline\Lambda_2}(\underline\Lambda_2) = \msh_{\underline\Lambda_1 \times \underline\Lambda_2}(\underline\Lambda_1 \times \underline\Lambda_2)$.
\end{proof}

    We recall some notions of products of contact manifolds and their submanifolds.

\begin{definition}\label{def: contact product}
    Let $X_1, X_2$ be contact manifolds. Then $X_1 \times \bR_{>0}, X_2 \times \bR_{>0}$ are homogeneous symplectic manifolds. The contact product is defined by
    $$X_1 \,\widehat\times\, X_2 = ((X_1 \times \bR_{>0}) \times (X_2 \times \bR_{>0}))/\bR_{>0} \cong X_1 \times X_2 \times \bR_{>0},$$
    where the contact form is $\alpha = r_1 \alpha_1 + r_2 \alpha_2$.
\end{definition}

\begin{example}
    The contact product $S^*M_1 \,\widehat\times\, S^*M_2$ is the open submanifold $\{[x_1, \xi_1, x_2, \xi_2] \mid \xi_1 \neq 0, \xi_2 \neq 0\} \subset S^*(M_1 \times M_2)$. Note that it is diffeomorphic to $S^*M_1 \times S^*M_2 \times \bR$.
\end{example}

    We will denote the points in the contact product $ X_1 \,\widehat\times\, X_2$ by their representatives in $(X_1 \times \bR_{>0}) \times (X_2 \times \bR_{>0})$. In particular, we will write $[x_1, r_1, x_2, r_2] \in X_1 \,\widehat\times\, X_2$ for the equivalence class of $(x_1, r_1, x_2, r_2) \in (X_1 \times \bR_{>0}) \times (X_2 \times \bR_{>0})$.
    
    We can define contact product of submanifolds inside contact manifolds.

\begin{definition}\label{def: contact product leg}
    Let $\Lambda_i \subset X_i$ be Legendrian submanifolds. Then the contact product is defined by
    $$\Lambda_1 \,\widehat\times\, \Lambda_2 = \{[x_1, r_1; x_2, r_2] \mid x_i \in \Lambda_i \} \subset X_1 \,\widehat\times\, X_2,$$
    whose conification is $(\Lambda_1 \times \bR_{>0}) \times (\Lambda_2 \times \bR_{>0}) \subset (X_1 \times \bR_{>0}) \times (X_2 \times \bR_{>0})$.
\end{definition}


\begin{example}
    When $\Lambda_1 \subset S^*M_1$ and $\Lambda_2 \subset S^*M_2$, 
    we can define the contact product $\Lambda_1 \,\widehat\times\, \Lambda_2 \subset S^*(M_1 \times M_2)$ as the subset whose conification in $\dot T^*(M_1 \times M_2)$ is the product of the cones
    $(\Lambda_1 \times \bR_{>0}) \times (\Lambda_2 \times \bR_{>0}) \subset \dot T^*M_1 \times \dot T^*M_2.$
    Note that there is a homeomorphism $\Lambda_1 \,\widehat\times\, \Lambda_2 \cong \Lambda_1 \times \Lambda_2 \times \bR$, and we can identify $\Lambda_1 \times \Lambda_2 \hookrightarrow \Lambda_1 \,\widehat\times\, \Lambda_2$ as the subset $\Lambda_1 \times \Lambda_2 \times 0$. 
\end{example}
    
    Thus, for contact products, the K\"unneth formula can be written as follows:


\begin{corollary}
    Let $\Lambda_1 \subset S^*M_1$ and $\Lambda_2 \subset S^*M_2$ be any closed stratified Legendrian subsets. Then there is an equivalence of sheaves of categories
    $$\msh_{\Lambda_1 \widehat\times \Lambda_2}|_{\Lambda_1 \times \Lambda_2} \simeq \msh_{\Lambda_1} \boxtimes \msh_{\Lambda_2}.$$
    In particular, there is an equivalence of categories
    $$\msh_{\Lambda_1 \widehat\times \Lambda_2}(\Lambda_1 \times \Lambda_2) \simeq \msh_{\Lambda_1}(\Lambda_1) \otimes \msh_{\Lambda_2}(\Lambda_2).$$
\end{corollary}

    For a locally compact Hausdorff space $M$, we know that the category of sheaves $\Sh(M)$ is compactly assembled \cite[Theorem 21.1.6.12 \& Proposition 21.1.7.1]{Lurie-SAG} and hence dualizable in $\PrLst$ \cite[Proposition D.7.3.1]{Lurie-SAG} (alternatively, this also follows from \cite[Proposition 2.30]{Volpe-six-operations} and \cite[Proposition 3.2]{Kuo-Shende-Zhang-Hochschild-Tamarkin}).

    The following results are stated in \cite{KuoLi-duality} under stronger constructibility hypotheses, but, after the constructibility and presentability results established above, their proofs carry through without change in the stratified Legendrian setting.  

\begin{theorem}[{\cite[Theorem 1.3]{KuoLi-duality}, \cite[Proposition 3.2]{Kuo-Shende-Zhang-Hochschild-Tamarkin}}]\label{thm: duality sheaf}
    Let $\underline\Lambda \subset T^*M$ be any conic subset. Then $\Sh_{\underline\Lambda}(M)$ is dual to $\Sh_{-\underline\Lambda}(M)$ with unit and counit
    $$\eta_\Lambda = \iota^*_{-\underline\Lambda \times \underline\Lambda} 1_\Delta, \; \epsilon_\Lambda(-, -) = \pi_!\Delta^*,$$
    where $\iota^*_{-\underline\Lambda \times \underline\Lambda}$ is the left adjoint of the tautological inclusion.
\end{theorem}

    For a contact manifold $X$, let $X^-$ be the contact manifold with the same contact distribution but opposite co-orientation. We define the diagonal submanifold $\Delta_X \subset X^- \,\widehat\times\, X$ by the quotient of the diagonal in the symplectization $\Delta_{X \times \bR_{>0}} \subset (X^- \times \bR_{>0}) \times (X \times \bR_{>0})$.

\begin{example}
    Let $X = S^*M$. Then the antipodal map on $S^*M$ provides a contactomorphism between $X^-$ and $X$. For $X^- \,\widehat\times\, X \cong S^*M \,\widehat\times\, S^*M \subset S^*(M \times M)$, the diagonal submanifold is just $S^*_\Delta(M \times M) \subset S^*(M \times M)$.
\end{example}

    From Theorems \ref{thm: duality sheaf} and \ref{thm: relative-doubling}, we deduce duality of microsheaves as in \cite{KuoLi-duality}, where we recall the notions from Definition \ref{def: positive doubling}:

\begin{corollary}
    [{\cite[Corollary 4.19]{KuoLi-duality}}]\label{thm:microsheaf-duality}
    Let $\Lambda \subset S^*M$ be a stratified Legendrian subset. Then $\msh_\Lambda(\Lambda)$ is dual to $\msh_{-\Lambda}(-\Lambda)$ in $\PrLst$, with unit and counit
    $$\eta_\Lambda = m_{-\Lambda \times \Lambda}(\iota^*_{-\underline\Lambda_{\cup,\epsilon}^+ \times \underline\Lambda_{\cup,\epsilon}^+} 1_\Delta), \; \epsilon_\Lambda(-, -) = \pi_!\Delta^*(w_{\Lambda}^+ - \boxtimes w_{-\Lambda}^+ -),$$
    where $\iota^*_{-\underline\Lambda_{\cup,\epsilon}^+ \times \underline\Lambda_{\cup,\epsilon}^+}$ is the left adjoint of the tautological inclusion and $w_{\Lambda}^+$ is the (positive) doubling.
\end{corollary}

    From the duality theorem, we can provide an definition of microlocal compositions for microsheaves on products of pdff subsets in finite positions.

\begin{definition}\label{def: alg composition}
    Let $\Lambda_1 \subset S^*M_1$, $\Lambda_2 \subset S^*M_2$ and $\Lambda_3 \subset S^*M_3$ be compact sufficiently Legendrian subsets with boundaries, with sufficiently Legendrian products. We define the algebraic microlocal composition as the functor
    \begin{gather*}
    \msh_{\Lambda_1 \widehat\times \Lambda_2}(\Lambda_1 \times \Lambda_2) \otimes \msh_{-\Lambda_2 \widehat\times \Lambda_3}(-\Lambda_2 \times \Lambda_3) \to \msh_{\Lambda_1 \widehat\times \Lambda_3}(\Lambda_1 \times \Lambda_3), \\ 
    \SF_{12} \boxtimes \SF_{23} \mapsto \SF_{23} \circ_a \SF_{12} = (\id_{\Lambda_1} \otimes\, \epsilon_{-\Lambda_2} \otimes \id_{\Lambda_3})(\SF_{12} \boxtimes \SF_{23}).
    \end{gather*}
\end{definition}

    From the duality and K\"unneth formula, it follows that integral kernels classify colimit preserving functors of microsheaf categories. Using the notation from Definition \ref{def: positive doubling}:

\begin{theorem}[{\cite[Theorem 1.4]{KuoLi-duality}}]\label{thm:microsheaf-fourier}
    Let $\Lambda_1 \subset S^*M_1, \Lambda_2 \subset S^*M_2$ be relatively compact sufficiently Legendrian subsets. Then there is an equivalence
    $$\msh_{-\Lambda_1 \widehat\times \Lambda_2}(-\Lambda_1 \times \Lambda_2) \simeq \mathrm{Fun}^\mathrm{L}(\msh_{\Lambda_1}(\Lambda_1), \msh_{\Lambda_2}(\Lambda_2)),$$
    which sends $\SK$ to the colimit preserving functor 
    $$\SK \circ_a (-) = (\epsilon_{-\Lambda_1} \otimes \id_{\Lambda_2})(- \boxtimes \SK) = m_{\Lambda_{2}}\pi_{2!}(w_{-\Lambda_{1} \times \Lambda_{2}}^+\SK \otimes \pi_1^*w_{\Lambda_{1}}^+-).$$
\end{theorem}

\begin{corollary}\label{rem: double stop removal}  
    For any conic sufficiently Legendrian subset $L$ such that $-\Lambda_1 \,\widehat\times\, \Lambda_2 \subset L$ is a closed subset, there is a commutative diagram
    \[\begin{tikzcd}[column sep=70pt]
    \msh_{L}(L) \ar[r, "\iota_{-\Lambda_1 \widehat\times \Lambda_2}^*"] \ar[d, "w_{L}^+" left] & \msh_{-\Lambda_1 \widehat\times \Lambda_2}(-\Lambda_1 \times \Lambda_2) \ar[d, "w_{-\Lambda_{1} \widehat\times \Lambda_{2}}^+"] \\
    \Sh_{\underline{L}_{\cup,s}}(M_1 \times M_2) \ar[r, "\iota_{-(\underline\Lambda_{1})_{\cup,\epsilon}^+ \times (\underline\Lambda_{2})_{\cup,\epsilon}^+}^*"] & \Sh_{-(\underline\Lambda_{1})_{\cup,\epsilon}^+ \times (\underline\Lambda_{2})^+_{\cup,\epsilon}}(M_1 \times M_2).
    \end{tikzcd}\]
    Hence for any $\SL \in \msh_L(L)$ we have an isomorphism between the following two colimit preserving functors:
    $$m_{\Lambda_2}(\iota_{-(\underline\Lambda_{1})_{\cup,\epsilon}^+ \times (\underline\Lambda_{2})_{\cup,\epsilon}^+}^* w_{L}^+\SL \circ w_{\Lambda_1}^+(-)) = m_{\Lambda_2}(w_{-\Lambda_1\widehat\times \Lambda_2}^+\iota_{-\Lambda_1 \widehat\times \Lambda_2}^*\SL \circ w_{\Lambda_1}^+(-)).$$
\end{corollary}
\begin{proof}
    This is a special case of Corollary \ref{rem: double vs wrap}.
\end{proof}

We have the limit preserving versions, whose proofs are the same as Theorem \ref{thm:microsheaf-duality}: 

\begin{theorem} \label{thm:microsheaf-duality-PrR} 
    Let $\Lambda \subset S^*M$ be a relatively compact sufficiently Legendrian. Then $\msh_\Lambda(\Lambda)$ is dual to $\msh_{-\Lambda}(-\Lambda)$ in $\PrRst$, with unit and counit
    $$\eta_\Lambda = m_{-\Lambda \times \Lambda}(\iota^!_{-\underline\Lambda_{\cup,\epsilon}^- \times \underline\Lambda_{\cup,\epsilon}^-} \omega_\Delta^{-1}), \; \epsilon_\Lambda(-, -) = \pi_*\Delta^!\sHom^\boxtimes(w_{\Lambda}^- -, w_{-\Lambda}^- -),$$
    where $\iota^!_{-\underline\Lambda_{\cup,\epsilon}^- \times \underline\Lambda_{\cup,\epsilon}^-}$ is the right adjoint of the tautological inclusion, and $w_{\Lambda}$ is the (negative) doubling.
\end{theorem}
\begin{proof}
    The result follows from Theorem \ref{thm:microsheaf-duality} by the equivalence $\PrRst = (\PrLst)^{op}$ sending left adjoints to right adjoints and the adjunction of the six functors.
\end{proof}

\begin{theorem}\label{thm:microsheaf-fourier-PrR}
    Let $\Lambda_1 \subset S^*M_1, \Lambda_2 \subset S^*M_2$ be relatively compact stratified Legendrian subsets. Then there is an equivalence
    $$\msh_{-\Lambda_1 \widehat\times \Lambda_2}(-\Lambda_1 \times \Lambda_2) \simeq \mathrm{Fun}^\mathrm{R}(\msh_{\Lambda_1}(\Lambda_1), \msh_{\Lambda_2}(\Lambda_2))^{op},$$
    which sends $\SK$ to the limit preserving functor $$\sHom^\circ_a(\SK, -) = m_{\Lambda_{2}}\pi_{2*}\Delta^!\sHom(w_{-\Lambda_{1} \times \Lambda_{2}}^-\SK, \pi_1^!w_{\Lambda_{1}}^- -).$$
\end{theorem}

    We briefly mention the associativity of colimit preserving functors in presentable categories. 
    When $\SD_1$ is a dualizable object in $\PrLst$ \cite[Section 2.1 \& 2.2]{Hoyois-Scherotzke-Sibilla}, we know that there is an equivalence of 1-categories \cite[Proposition 4.6.1.6 \& 4.8.1.17]{Lurie-HA} \cite[Chapter I Section 4.3.2]{Gaitsgory-Rozenblyum}
    $$\Fun^\mathrm{L}(\SD_1, \SD_2) \simeq \SD_1^\vee \otimes \SD_2.$$
    Such equivalences are compatible with compositions: 
    Let $\SD_1, \dots, \SD_k$ be dualizable objects in $\PrLst$. Then there are natural isomorphisms
    \begin{equation}\label{rem: associativity cat}
    \begin{tikzcd}
        \Fun^\mathrm{L}(\SD_1, \SD_2) \otimes \dots \otimes \Fun^\mathrm{L}(\SD_{k-1}, \SD_k) \ar[d, "(-) \circ \dots  \circ (-)" left] \ar[r, "\sim"] & (\SD_1^\vee \otimes \SD_2) \otimes \dots \otimes (\SD_{k-1}^\vee \otimes \SD_k) \ar[d, "\id_{\SD_1^\vee} \otimes \epsilon_{\SD_2} \otimes \dots \otimes \epsilon_{\SD_{k-1}} \otimes \id_{\SD_k}"] \\
        \Fun^\mathrm{L}(\SD_1, \SD_k) \ar[r, "\sim"] & \SD_1^\vee \otimes \SD_k
    \end{tikzcd}
    \end{equation}
    that induce an isomorphism of commutative diagrams $\Delta_1^{\times k-1} \to \PrLst$, which send the vertices to $\Fun^\mathrm{L}(\SD_1, \SD_{l_1}) \otimes \dots \otimes \Fun^\mathrm{L}(\SD_{l_r + 1}, \SD_k)$ and respectively $(\SD_1^\vee \otimes \SD_{l_1}) \otimes \dots \otimes (\SD_{l_r+1}^\vee \otimes \SD_k)$ and the edges to the composition functors satisfying the Segal condition.
    
    In fact, for the categories of sheaves (and microsheaves with sufficiently Lagrangian supports), since the counits and compositions are all realized by adjunctions and six-functors as in Theorems \ref{thm: duality sheaf} and \ref{thm:microsheaf-duality}, we can also conclude the above diagram by the coherence of proper base change encoded in the six-functor formalism \cite{Volpe-six-operations} plus the coherence of (op)lax natural transformations induced by adjunctions \cite[Corollary F]{HaugsengHebestreitLinskensNuiten}.
    

%
%
%
%
%

\subsection{Constructibility and compact generation}
    In fact, under an additional geometric hypothesis, we can show that $\msh_\Lambda$ is constructible and takes values in compactly generated categories.  This will never be necessary in the present article, and so we will never impose this hypothesis after this subsection. 

\begin{definition}
    We say that a closed subset $\Lambda \subset S^*M$ is a tamely stratified Legendrian if $\Lambda$ is sufficiently Legendrian and (after some contact isotopy) $\Lambda$ is contained in the conormal bundle of some locally finite $C^1$-Whitney stratification of $M$.
\end{definition}  

    For example, any subanalytic Legendrian $\Lambda \subset S^*M$ with respect to some analytic structure on $M$ is a tamely stratified Legendrian by \cite[Section 8.3]{KS}.
    
\begin{proposition}[{\cite[Theorem 4.13]{Ganatra-Pardon-Shende3}}]\label{prop:sheaf-compact-gen}
    Let $\Lambda \subset S^*M$ be a closed tamely stratified Legendrian. Then $\Sh_{\Lambda}(M)$ is compactly generated by the corepresentatives of the microstalks at smooth Legendrian points.
\end{proposition}

\begin{proposition}\label{cor: microsheaf-bicomplete}
    Let $\underline\Lambda \subset T^*M$ is a conic subset such that $\Lambda \subset S^*M$ is a tamely stratified Legendrian. Then $\msh_{\underline\Lambda}$ is compactly generated,
    and the restriction functors are limit and colimit preserving. 
\end{proposition}
\begin{proof}
    Let $B_\epsilon \subset M$ be a neighborhood of $x\in M$. By Proposition \ref{prop: microstalk}, we know that there is an isomorphism between sheaves of stable categories on $M$:
    $$m_\Lambda: \Sh_\Lambda/\Loc \xrightarrow{\sim} \pi_*\msh_\Lambda.$$
    By Proposition \ref{prop:sheaf-compact-gen}, the sections of the presheaves $\Sh_\Lambda(B_\epsilon)/\Loc(B_\epsilon)$ are compactly generated, and the restriction functors on presheaves $j_{\epsilon',\epsilon}^*: \Sh_\Lambda(B_\epsilon)/\Loc(B_\epsilon) \to \Sh_\Lambda(B_{\epsilon'})/\Loc(B_{\epsilon'})$ are limit and colimit preserving. Therefore, the sections and restriction functors after sheafification are still compactly generated and limit and colimit preserving, so the sections of sheaves $\msh_\Lambda(S^*B_\epsilon)$ are presentable compactly generated, and the restriction functors $\rho_{\epsilon',\epsilon}^*: \msh_\Lambda(S^*B_\epsilon) \to \msh_\Lambda(S^*B_{\epsilon'})$ are limit and colimit preserving functors.
    
    Finally, since $\Lambda$ is in finite position, we may assume that when $\epsilon > 0$ is sufficiently small, $\Lambda \cap S^*B_\epsilon = \bigsqcup_{i=1}^k \Lambda_i \cap S^*B_\epsilon$. Then we have
    $$\msh_\Lambda(S^*B_\epsilon) \simeq \oplus_{i=1}^k \msh_{\Lambda_i}(S^*B_\epsilon),$$
    and for any small open subset $U_{i,\epsilon} \subset S^*B_\epsilon$ so that $\Lambda \cap U_{i,\epsilon} = \Lambda_i \cap S^*B_\epsilon$, the sections $\msh_\Lambda(U_{i,\epsilon}) = \msh_{\Lambda_i}(S^*B_{\epsilon})$ take values in presentable compactly generated categories, and the restriction functors are limit and colimit preserving.
\end{proof}

    We now show that $\msh_\Lambda$ is a constructible sheaf of stable categories when $\Lambda$ is a tamely stratified Legendrian.

\begin{proposition}\label{prop:microsheaf-construct}
    Let $\Lambda \subset S^*M$ be closed tamely stratified Legendrian. Then at any $p \in T^*M$, there exists $\epsilon_0 > 0$ such that the natural restriction functor on sheaves of categories
    $$\msh_\Lambda(U_\epsilon(p)) \longrightarrow \msh_\Lambda(p)$$
    is an isomorphism for the open balls $U_\epsilon(p)$ of radius $\epsilon < \epsilon_0$ centered at $x$.
\end{proposition}
\begin{proof}
    By Proposition \ref{cor: microsheaf-bicomplete}, we know that for sufficiently small $\epsilon > 0$, there exists an equivalence of presheaves of categories
    $$m_\Lambda: (\Sh_{\Lambda}/\Loc)(B_{\epsilon}) \longrightarrow \msh_{\Lambda}(S^*B_\epsilon)$$
    For $\epsilon' < \epsilon$ sufficiently small, we have a commutative diagram
    \[\begin{tikzcd}
        (\Sh_\Lambda/\Loc)(B_\epsilon)  \ar[d, "\rotatebox{90}{$\sim$}"] \ar[r, "j_{\epsilon',\epsilon}^*"] & (\Sh_\Lambda/\Loc)(B_{\epsilon'})  \ar[d, "\rotatebox{90}{$\sim$}"] \\
        \msh_\Lambda(S^*B_\epsilon) \ar[r, "\rho_{\epsilon',\epsilon}^*"] & \msh_\Lambda(S^*B_\epsilon).
    \end{tikzcd}\]
    Hence, it suffices to show that $j_{\epsilon',\epsilon}^*: \Sh_\Lambda(B_\epsilon)/\Loc(B_\epsilon) \to \Sh_\Lambda(B_{\epsilon'})/\Loc(B_{\epsilon'})$ induces equivalences and then take sheafifications.

    Since $\Lambda$ is sufficiently Legendrian, consider a positive contact isotopy $\varphi_t$ on $S^*B_\epsilon$ such that $\varphi_t(\Lambda) \cap (\bigcup_{r\in[\epsilon',\epsilon]}S^*B_{r}) = \varnothing$. Consider any $\SF, \SG \in \Sh_\Lambda(B_\epsilon)$. Since $\dot\SS(\SF) \cap \dot\SS(\SG)_t = \Lambda \cap \Lambda_t = \varnothing$, we know by Formula \eqref{lem: ss-hom} that $\dot\SS(\sHom(\SF, \SG_t)) \subset -\dot\SS(\SF) \cup \dot\SS(\SG)_t = -\Lambda \cup \Lambda_t$ is also sufficiently Legendrian. By non-characteristic deformation Proposition \ref{prop: perturb-lim}, we know that
    $$\Hom(\SF, \SG_t) =\Hom(j_{\epsilon',\epsilon}^*\SF, j_{\epsilon',\epsilon}^*(\SG_t)).$$
    Therefore, by Proposition \ref{prop: perturb-lim}, we know that for any $\SF, \SG \in \Sh(B_\epsilon)$,
    $$\Hom(\SF, \SG) = \lim_{t\to 0^+}\Hom(\SF, \SG_t) = \lim_{t\to 0^+}\Hom(j_{\epsilon',\epsilon}^*\SF, j_{\epsilon',\epsilon}^*(\SG_t)) = \Hom(j_{\epsilon',\epsilon}^*\SF, j_{\epsilon',\epsilon}^*\SG).$$
    This implies the restriction functor is fully faithful. 
    
    On the other hand, for $\epsilon' < \epsilon$ sufficiently small, we may assume that the Whitney stratifications are contractible in $B$ and the inclusions of the strata from $B'$ to $B$ are homotopy equivalences by considering a tubular neighbourhood of the Whitney stratification \cite[Section 7 \& 8]{Mather}. We may also assume that $\Lambda \cap S^*B_\epsilon$ and $\Lambda \cap S^*B_\epsilon$ have the same set of smooth Legendrian components. Then for any $\SF' \in \Sh_\Lambda(B')$, it is constructible with respect to the Whitney stratification by \cite[Theorem 8.4.2]{KS}, and since the inclusions of the strata from $B'$ to $B$ are homotopy equivalences, $\SF'$ extends (canonically) to a constructible sheaf $\SF \in \Sh(B)$. Since $\Lambda \cap S^*B_\epsilon$ and $\Lambda \cap S^*B_\epsilon$ have the same set of smooth Legendrian components, we know by computing microstalks that
    $$\SF \in \Sh_\Lambda(B), \quad j_{\epsilon',\epsilon}^*\SF = \SF' \in \Sh_{\Lambda}(B').$$
    Therefore, the restriction functor $j_{\epsilon',\epsilon}^*$ is essentially surjective. This shows that when $\Lambda$ is a tamely stratified Legendrian, there are isomorphisms
    $$\pi_*\msh_\Lambda(S^*B_\epsilon) \xrightarrow{\sim} \pi_*\msh_\Lambda(S^*_xM).$$

    Finally, since $\Lambda$ is in finite position, we may assume that when $\epsilon > 0$ is sufficiently small, $\Lambda \cap S^*B_\epsilon = \bigsqcup_{i=1}^k \Lambda_i \cap S^*B_\epsilon$. Then we have
    $$\msh_\Lambda(S^*B_\epsilon) \simeq \oplus_{i=1}^k \msh_{\Lambda_i}(S^*B_\epsilon),$$
    and for any small open subset $U_{i,\epsilon} \subset S^*B_\epsilon$ so that $\Lambda \cap U_{i,\epsilon} = \Lambda_i \cap S^*B_\epsilon$ and $p_i = \Lambda_i \cap S^*_xB_\epsilon$, the restriction functor factors as
    $$\oplus_{i=1}^k \msh_{\Lambda_i}(U_{i,\epsilon}) \longrightarrow \oplus_{i=1}^k\msh_{\Lambda_i}(p_i).$$
    Hence each factor is an isomorphism.
\end{proof}

\begin{remark}
    The only place in the proof that the existence of a locally finite Whitney stratification is used is in the proof of the essential surjectivity of $j_{\epsilon',\epsilon}^*: \Sh_\Lambda(B_\epsilon)/\Loc(B_\epsilon) \to \Sh_\Lambda(B_{\epsilon'})/\Loc(B_{\epsilon'})$. In general, for any sufficiently Legendrian subsets, we can still prove $j_{\epsilon',\epsilon}^*$ is fully faithful, but we do not know how to prove essential surjectivity (though nor do we have examples of sufficient Legendrians where surjectivity fails).
\end{remark}

\begin{remark}
    If $\Lambda$ is subanalytic Lagrangian, the subset $(-\Lambda) \,\widehat+\, \Lambda$ is also subanalytic Lagrangian, and the standard microsupport estimate 
    $\SS(\sHom(\SF, \SG)) \subset -\SS(\SF) \,\widehat+\, \SS(G)$
    implies full faithfulness of the restriction functors, without appealing to Proposition \ref{prop: perturb-lim}. 
    
    However, we do not presently know whether  $\Lambda$ being stratified Legendrian implies that $(-\Lambda) \,\widehat+\, \Lambda$ is stratified Legendrian.  Thus we appeal to Proposition \ref{prop: perturb-lim}.
\end{remark}

\subsection{Microlocal nearby cycle functor}
    We recall the nearby cycle functor of microsheaves in \cite[Section 8]{Nadler-Shende}. 
    Consider a subset  $\Lambda \subset S^*(M \times \bR_{>0})$ that is $\bR_{>0}$-non-characteristic. We regard the projection $\Pi(\Lambda) \subset S^*M \times \bR_{>0}$ as an $\bR_{>0}$-family of subsets. Suppose that the $\bR_{>0}$-family of subsets is uniformly self displaceable. Then by Theorem \ref{thm: relative-doubling}, we can define the relative doubling functor. Following the notation in Definition \ref{def: positive doubling}, we can consider the following nearby cycle functor
    $$w_\Lambda^+: \msh_\Lambda(\Lambda) \rightarrow \Sh_{\underline\Lambda_{\cup,\epsilon}^+}(M \times \bR_{>0}).$$
    The microsupport estimate of Lemma \ref{lem: ss-nearby-cycle} constrains the image of the following nearby cycle functor:
    $$\psi = i^*j_*: \Sh_{\underline\Lambda_{\cup,\epsilon}^+}(M \times \bR_{>0}) \to \Sh_{\underline{\psi(\Lambda)}{}_{\cup,\epsilon}^+}(M).$$
    This allows us to define gapped nearby cycles of microsheaves:

\begin{definition}\label{def: gap micro nearby}
    Let $\Lambda \subset S^*(M \times \bR_{>0})$ be an $\bR_{>0}$-non-characteristic subset such that $\Pi(\Lambda) \subset S^*M \times \bR_{>0}$ defines a family of relative compact subsets that is uniformly self displaceable and perturbable to finite position. Then the microlocal nearby cycle functor $\psi: \msh_{\Lambda}(\Lambda) \to \msh_{\psi(\Lambda)}(\psi(\Lambda))$ is defined by
    $$\psi(\SF) = m_{\psi(\Lambda)}(\psi(w_{\Lambda}^+(\SF))).$$
\end{definition}
\begin{remark}
    The condition used in \cite[Section 8]{Nadler-Shende} to define the microlocal nearby cycle functor is that $\Pi(\Lambda) \subset S^*M \times \bR_{>0}$ is self gapped with respect to some contact flow compatible with the contact collar of $\Pi(\Lambda)$ in Definition \ref{def: contact collar}. Our condition above is weaker than that, and therefore, this condition does not always ensure that the nearby cycle functor is fully faithful as in \cite[Theorem 8.3]{Nadler-Shende}.
\end{remark}

    Here is a family of subsets which is not gapped with itself, where the gapped nearby cycle functor cannot be defined:

\begin{example}\label{ex: gapped specialization undefined} 
    Let $U = \{(x, t) \mid |x| \leq t\} \subset \bR \times \bR_{>0}$ and consider the conic Lagrangian subset $\Lambda = \dot{N}^*(0 \times \bR_{>0}) \cup \dot N^*_{out}U$, where $N^*_{out}U$ here means the outward conormal bundle of the smooth boundary $\partial U$. Then $\Lambda \subset S^*(\bR \times \bR_{>0})$ is $\bR_{>0}$-non-characteristic and $\Pi(\Lambda) \subset S^*\bR \times \bR_{>0}$ has (families of) self Reeb chords of length $t \to 0$. 
\end{example}

    Here is a family of subsets $\Pi(\Lambda)$ which is uniformly self displaceable but the doubling $\Pi(\Lambda) \cup \Pi(\Lambda_\epsilon)$ is not gapped with itself:

\begin{example}\label{ex: gapped collar}
    Let $U = \{(x, t) \mid |x| \leq t\} \subset \bR \times \bR_{>0}$ and the conic subset $\Lambda = N^*_{out}(0 \times \bR_{>0}) \cup N^*_{out}U \subset T^*(\bR \times \bR_{>0})$. For a smooth non-negative function $f: \bR \to \bR_{\geq 0}$, consider the embedding
    $$T^*(\bR \times \bR_{>0}) \hookrightarrow S^*(\bR^2 \times \bR_{>0}), \quad (x, t, \xi, \tau) \mapsto (x + f'(\xi), f(\xi), t, \xi, 1, \tau).$$
    Then for the image of $\Lambda$ under the embedding, it defines a family $\Pi(\Lambda)$ that is uniformly self gapped with itself. One can define the doubling by considering the embedding for a smooth function $f$ and $f_\epsilon$ such that $f_\epsilon - f = \epsilon \rho$ where $\rho$ is a smooth cut-off function supported on a compact interval $[-r, r]$. Then the doubling
    $\Pi(\Lambda) \cup \Pi(\Lambda_\epsilon)$ is not gapped with itself: For the pair of points $(t, \xi), (0, \xi) \in \Lambda_t$, consider their images under the embeddings
    $$(t + f'(\xi), f(\xi), \xi, 1) \in \Lambda_t, \quad (f'_\epsilon(\xi), f_\epsilon(\xi), \xi, 1) \in \Lambda_{t,\epsilon}.$$
    When $t + f'(\xi) = f'_\epsilon(\xi)$, there is a Reeb chord between the pair of points of length $f(\xi) - f_\epsilon(\xi)$. As $t \to 0$, the solution $\xi \to r$ and the length of the Reeb chord converges to $0$. See Figure \ref{fig:gapped non collar}.
\end{example}

\begin{figure}
    \centering
    \includegraphics[width=0.75\linewidth]{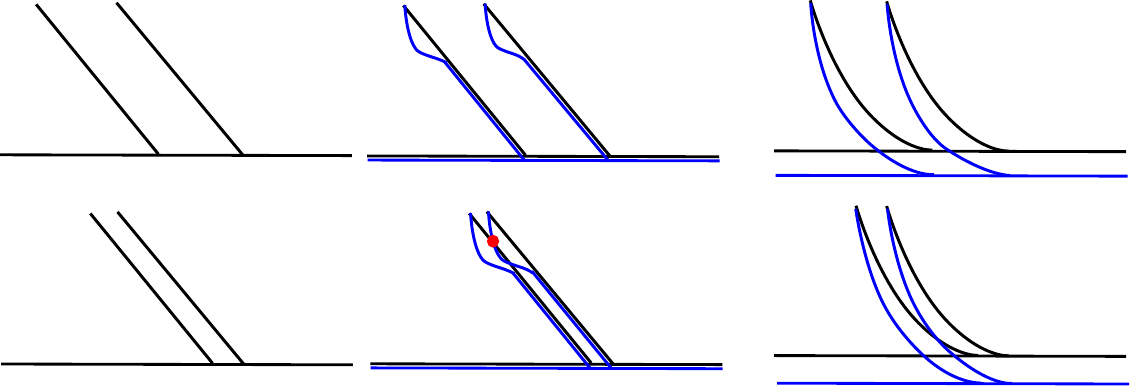}
    \caption{The self gapped family of Legendrians $\Lambda_t \subset S^*_{\eta>0}\bR^2$ that has non-gapped doubling in Example \ref{ex: gapped collar}. The figure on the left and middle are the Lagrangian projections of the Legendrians $\Lambda_t$ and doublings $(\Lambda_t)_{\pm\epsilon}$ where the Reeb chords are the double points. The figure on the right is the front projection of the doubling $(\Lambda_t)_{\pm\epsilon}$.}
    \label{fig:gapped non collar}
\end{figure}

    From the above example, one can see that uniform self displaceability of $\Pi(\Lambda)$ does not ensure self gappedness of $\Pi(\Lambda) \cup \Pi(\Lambda_\epsilon)$. The reason is the cut-off of the Reeb flow for relative doubling near the `boundary' of $\Pi(\Lambda)$: under the cut-off of the Reeb flow, failure of gappedness for the `boundary' of $\Pi(\Lambda)$ results in failure of gappedness for the doubling. To ensure that the doubling is self gapped, we therefore use the contact collar condition in Definition \ref{def: contact collar} introduced in \cite[Section 7.3]{Nadler-Shende}.

\begin{definition}\label{def: gapped collar}
    Let $\Lambda \subset S^*M \times \bR_{>0}$ be a family of subsets with a uniform family of contact collars. We say that $\Lambda$ is gapped with itself with gapped collar if $\Lambda$ is gapped with itself with respect to a positive contact flow compatible with the contact collar.
\end{definition}

\begin{lemma}\label{lem: gapped collar}
    Let $\Lambda \subset S^*M \times \bR_{>0}$ be a family of subsets with a uniform family of contact collars. If $\Lambda$ is gapped with itself with gapped collar, then the relative doubling $\Lambda_{-\epsilon} \cup \Lambda_\epsilon$ is gapped with itself.
\end{lemma}
\begin{proof}
    Consider the positive contact isotopy that is compatible with the collar $\partial\Lambda \subset U \times \bR_{>0}$ of $\Lambda \subset S^*M \times \bR_{>0}$, we can define the relative doubling by the compatible contact isotopy on $\Lambda \setminus \partial \Lambda \times (0, \epsilon)$ and the contact isotopy of a cut-off function on the contact collar $\partial \Lambda \times T^*(0,\epsilon)$. Then the doubling is gapped with itself by direct computation on $\Lambda \setminus \partial \Lambda \times (0, \epsilon)$ and on the contact collar $\partial \Lambda \times T^*(0,\epsilon)$.
\end{proof}

\begin{theorem}[{\cite[Theorem 8.3]{Nadler-Shende}}] \label{gapped nearby cycle} 
    Let $\Lambda \subset S^*(M \times \bR_{>0})$ be a non-characteristic family of relatively compact sufficiently Legendrian subsets such that $\Pi(\Lambda) \subset S^*M \times \bR_{>0}$ is gapped with itself with gapped collar. Then the gapped nearby cycle functor $\psi$ is fully faithful:
    $$\msh_{\Lambda}(\Lambda) \hookrightarrow \msh_{\psi(\Lambda)}(\psi(\Lambda)).$$
\end{theorem}

    Here we also establish the `$\otimes$' version of Theorem \ref{gapped nearby cycle}: 

\begin{theorem}\label{thm: gapped specialization}
    Let $\Lambda \subset S^*(M \times \bR_{>0})$ be $\bR_{>0}$-non-characteristic such that $\Pi(\Lambda) \subset S^*M \times \bR_{>0}$ is a family of relatively compact sufficiently Legendrian subsets that is gapped with itself with gapped collar. Then there exists a commutative diagram of nearby cycle functors 
    \[\begin{tikzcd}[column sep=80pt]
    \msh_{-\Lambda}(-\Lambda) \otimes \msh_{\Lambda}(\Lambda) \ar[r,"\Gamma_c(w_{-\Lambda}^+ - \otimes w_{\Lambda}^+ - )"] \ar[d,"\psi" left] & \cC \ar[d, "{\rotatebox{90}{$\sim$}}"] \\
    \msh_{-\psi(\Lambda)}(-\psi(\Lambda)) \otimes \msh_{\psi(\Lambda)}(\psi(\Lambda)) \ar[r,"\Gamma_c(w_{-\psi(\Lambda)}^+ - \otimes w_{\psi(\Lambda)}^+ - )"] & \cC.
    \end{tikzcd}\]
\end{theorem}
\begin{proof}
    The proof is parallel to that of Theorem \ref{gapped nearby cycle} or \cite[Theorem 8.3]{Nadler-Shende}. 
    Assume that $\Pi(\Lambda)$ is positively $\epsilon$-gapped with itself with gapped collar. Then by Lemma \ref{lem: gapped collar}, for $0 < s < \epsilon$, the doubling $\Lambda \cup \Lambda_s$ is also positively gapped with itself. By Theorem \ref{nearby fully faithful}, we know that for $\SF \in \msh_{-\Lambda}(-\Lambda)$ and $\SG \in \msh_{\Lambda}(\Lambda)$,
    $$\Gamma_c(M, \psi(w_{-\Lambda}^+\SF) \otimes \psi(w_{\Lambda}^+\SG)) = \Gamma_c(M, \psi(w_{-\Lambda}^+\SF \otimes w_{\Lambda}^+\SG)).$$
    We know that $\Gamma_c(M, \psi(w_{-\Lambda}^+\SF \otimes w_{\Lambda}^+\SG)) = \Gamma_c(M, w_{-\Lambda}^+\SF \otimes w_{\Lambda}^+\SG)$. This shows that the functors lift to microsheaves and hence completes the proof.
\end{proof}

    We recall the following lemma, showing a commutative square of gapped nearby cycle functors for a $J$-parametric family of gapped sufficiently Legendrian subsets where $J$ is a contractible compact manifold. This will be later used to show the compatibility between the gapped nearby cycle functors under contact isotopies.

\begin{proposition}\label{prop: nearby commute micro} 
    Let $\Lambda \subset S^*(M \times J \times \bR_{>0})$ be an $J \times \bR_{>0}$-non-characteristic subset, such that $\Pi(\Lambda) \subset S^*M \times J \times \bR_{>0}$ defines a family of relatively compact sufficiently Legendrian subsets that is uniformly self displaceable. Let $\Lambda_s = \Pi(\Lambda) \cap S^*M \times s \times \bR_{>0}$. Then there exists a commutative diagram of microlocal nearby cycle functors
    \[\begin{tikzcd}
    \msh_{\Lambda}(\Lambda) \ar[d, "i_s^*"] \ar[r, "\psi"] & \msh_{\psi(\Lambda)}(\psi(\Lambda)) \ar[d, "i_s^*"] \\
    \msh_{\Lambda_s}(\Lambda_s) \ar[r, "\psi_s"] & \msh_{\psi_s(\Lambda_s)}(\psi_s(\Lambda_s)).
    \end{tikzcd}\]
\end{proposition}
\begin{proof}
    Suppose $\Lambda \subset S^*M \times J \times \bR_{>0}$ is $\epsilon$-gapped with itself. By Definition \ref{def: gap micro nearby}, it suffices to show the commutativity of the following diagram
    \[\begin{tikzcd}
    \Sh_{\underline\Lambda_{\cup,\epsilon}}(M \times J \times \bR_{>0}) \ar[d, "i_s^*"] \ar[r, "\psi"] & \Sh_{\psi(\underline\Lambda_{\cup,\epsilon})}(M \times J) \ar[d, "i_s^*"] \\
    \Sh_{(\underline\Lambda_s)_{\cup,\epsilon}}(M \times s \times \bR_{>0}) \ar[r, "\psi_s"] & \Sh_{\psi_s((\underline\Lambda_s)_{\cup,\epsilon})}(M \times s).
    \end{tikzcd}\]
    Then the result follows from Proposition \ref{nearby commute with restrict}.
\end{proof}

\subsection{Microsheaf categories for Weinstein manifolds}\label{ssec: weinstein}
    We now recall how microsheaf categories can be defined for contact manifolds $X$ admitting appropriate topological `Maslov data' \cite{Shende-h-principle,Nadler-Shende}.  The construction is then adapted to Liouville manifolds $W$ by taking the contactization $X = W \times \mathbb{R}$, and ultimately to the skeleta of Weinstein manifolds by requiring microsupport in $\mathfrak{c}_W \times 0 \subset W \times \mathbb{R}$. 

    The first step is to take a high codimension contact embedding $X \hookrightarrow S^{*}\bR^N$.  Gromov's $h$-principle for contact embeddings \cite[Section 3.4.3]{Gromov-PDR}, see also \cite[Theorem 20.3.1]{CieliebakEliashbergMishachev}, ensures the existence of such embeddings, and also their uniqueness up to a space of choices which becomes arbitrarily connected as $N \to \infty$.  
    
    Given a Lagrangian sub-bundle $\sigma$ of the symplectic normal bundle of the embedding, we write $X_\sigma$ for a thickening of $X$ along $\sigma$, and define 
    $$\msh_{X, \sigma} := \msh_{X_\sigma}|_X.$$ 
    The basic stabilization and isotopy invariance of microsheaf categories (Lemma \ref{lem:contact-transform-main}) implies that this definition is independent of all choices except the stable equivalence class of $\sigma$.  If we begin with a subset $X \subset S^*M$, and choose $\sigma$ by restricting the fiber polarization on the cosphere bundle, we recover the previous notion of $\msh_X$ \cite[Corollary 4.13]{CKNS}.
    
    Note that said $\sigma$ may be regarded as a section of the stable Lagrangian Grassmannian bundle $LGr(X)\to X$.  In fact, it is shown in \cite[Section 11]{Nadler-Shende} $\msh_{X, \sigma}$ is the pullback of along $\sigma$ of a sheaf of categories $\msh_{LGr(X)}$ over $LGr(X)$, defined using the universal polarization on the contact manifold given by the relative cotangent bundle of $\mathfrak f: LGr(X) \to X$ \cite[Section 11.1]{Nadler-Shende}.  It is moreover shown that $\msh_{LGr(X)}$ structure descends to a certain $BPic(\mathcal{C})$ bundle over $X$ \cite[Theorem 11.14]{Nadler-Shende}, and correspondingly that the essential data to define $\msh_{X}$ is not in fact a stable polarization $\sigma$ -- which is a null-homotopy of the natural map $X \to B(U/O)$ classifying the Lagrangian Grassmannian bundle -- but rather a null-homotopy of the further composition $X \to B(U/O) \to B^2 Pic(\mathcal{C})$.  We refer to such a null-homotopy as {\em Maslov data}. Given Maslov data $\mu$ for $X$, we write $\msh_{X, \mu}$ for the corresponding sheaf of categories on $X$. In particular, for a Weinstein manifold $W$ with Lagrangian skeleton $\mathfrak{c}_W$, we take $X = W \times \bR$ and write 
    $$\msh_{\mathfrak{c}_{W,\mu}} = \msh_{W \times \bR,\mu}|_{\mathfrak{c}_W}.$$
    We will omit the notation for $\mu$ whenever we have fixed some Maslov data, or the assertion in question holds whenever the same (or compatible) Maslov data are chosen for all microsheaf categories appearing.

\begin{definition}\label{def: universal legendrian}
    Let $X$ be a contact manifold and $L \subset X$ be a relatively compact subset. Then $L$ is called universally sufficiently Legendrian if for any contact embedding $X \hookrightarrow S^*\bR^N$, the Legendrian thickening along the symplectic normal bundle $L_\sigma \subset S^*\bR^N$ is sufficiently Legendrian.
\end{definition}


\begin{lemma}\label{lem: universal legendrian}
    Let $W$ be a Weinstein manifold with Morse--Smale gradient like function for the Liouville vector field and $\Lambda \subset \partial W$ be a Whitney stratified Legendrian. Then the relative core $\mathfrak{c}_{W,\Lambda} = \mathfrak{c}_W \cup \Lambda \times \bR_{>0}$ is universally sufficient Legendrian.
\end{lemma}
\begin{proof}
    When the Weinstein manifold is equipped with a Morse--Smale gradient like function for the Liouville vector field, we know that the relative core $\mathfrak{c}_{W,\Lambda}$ is a Whitney stratified Legendrian by \cite[Theorem 4.33]{Nicolaescu}. For any contact embedding $W \times \bR \hookrightarrow S^*\bR^N$, the Legendrian thickening is a Whitney stratified Legendrian and contained in an exact symplectic hypersurface, and therefore it is sufficiently Legendrian by Lemmas \ref{lem: sufficient legendrian whitney}.
\end{proof}

\begin{lemma}\label{lem: universal legendrian analytic}
    Let $W$ be a Weinstein manifold with an analytic gradient like function for the Liouville vector field and $\Lambda \subset \partial W$ be a subanalytic Legendrian. Then the relative core $\mathfrak{c}_{W,\Lambda} = \mathfrak{c}_W \cup \Lambda \times \bR_{>0}$ is universally sufficient Legendrian.
\end{lemma}
\begin{proof}
    For any contact embedding $W \times \bR \hookrightarrow S^*\bR^N$, we may assume that it is analytic up to a small contact perturbation. Then the Legendrian thickening of $\Lambda$ is also a subanalytic Legendrian, and therefore it is sufficiently Legendrian by Lemmas \ref{lem: pdff analytic}, \ref{lem: ptfp analytic} and \ref{lem: self displace analytic}.
\end{proof}

\begin{definition}\label{def: eventually conic}
    Let $W$ be a Liouville manifold and $L \subset W \times \bR$ be any subset. Denote by $\partial W$ the (ideal) contact boundary and $\partial W \times \bR_{>0}$ the cylindrical end determined by the Liouville vector field. We say that $L$ is an eventually conical subset if there exists a compact subset $\Lambda \subset \partial W$, $r \in \bR_{>0}$ and $t \in \bR$ such that $L \cap (\partial W \times \bR_{>r} \times \bR) \subset \Lambda \times \bR_{>0} \times t$. Under such condition, we say that $\Lambda \subset \partial W$ is the conical end of $L$.
\end{definition}

\begin{definition}\label{def: quantization immerse}
    Let $W$ be a Weinstein manifold with Maslov data and $L \subset W \times \bR$ be an eventually conical universally sufficient Legendrian subset. By the category of conic sheaf quantizations of $L$, we mean the category of microsheaves  
    $$\msh_{\mathfrak{c}_{W}^L}(\mathfrak{c}_W^L) = \msh_{\mathfrak{c}_{W \times T^*\bR,L}}(\mathfrak{c}_{W \times T^*\bR,L})_0$$
    where $\mathfrak{c}_W^L = \mathfrak{c}_{W \times T^*\bR,L}$ and the subscript $0$ means the subcategory of microsheaves with $0$ stalks in a neighborhood of $-\infty \in [-\infty, +\infty] = \ol{\bR}$.
\end{definition}

    We have natural restriction functors 
    $$\msh_{L}(L) \xleftarrow{i_L^*} \msh_{\mathfrak{c}_{W}^L}(\mathfrak{c}_W^L) \xrightarrow{i_\infty^*} \msh_{\mathfrak{c}_{W,\partial L}}(\mathfrak{c}_{W,\partial L}).$$

The fundamental microsheaf quantization result \cite[Cor. 9.9]{Nadler-Shende} -- ultimately a consequence of gapped specialization for sheaves -- amounts to the assertion that when $L \to W$ is an embedding, left arrow is an equivalence, and the right arrow is fully faithful; we give a more explicit presentation of this assertion in Theorem \ref{thm: nearby = restriction at infty} below.  

For now, we recall the following fundamental result invariance result, also a consequence of the results on gapped specialization: 
\begin{theorem}[{\cite[Theorem 9.14]{Nadler-Shende}}]\label{thm: weinstein invariance}
    Let $\lambda, \lambda'$ be Weinstein Liouville forms on $W$ that are connected by a Weinstein cobordism $W \times T^*I$. Then the restriction functors induce equivalences:
    $$\msh_{\mathfrak{c}_{W,\lambda}}(\mathfrak{c}_{W,\lambda}) \xleftarrow{\sim} \msh_{\mathfrak{c}_{W\times T^*I}}(\mathfrak{c}_{W \times T^*I}) \xrightarrow{\sim} \msh_{\mathfrak{c}_{W,\lambda'}}(\mathfrak{c}_{W,\lambda'}).$$
\end{theorem}

    By Theorem \ref{thm: weinstein invariance}, 
    we have a commutative diagram for a compactly supported contact isotopy, with $L_t \subset W \times \bR$ the induced Legendrian isotopy

\begin{corollary}\label{cor: Hamilton invariance quantization}
    Let $W$ be a Weinstein manifold with Maslov data and $L \subset W \times \bR$ a universally sufficient Legendrian. For a compactly supported contact isotopy, with $L_t \subset W \times \bR$ the induced Legendrian isotopy and $L_H \subset W \times T^*I \times \bR$ the induced contact movie, we have a commutative diagram where vertical arrows define equivalences
    \begin{equation}\label{eq: Hamilton invariance quantization main}
    \begin{tikzcd}
        \msh_{L_{0}}(L_0) & \msh_{\mathfrak{c}_{W,\sigma}^{L_0}}(\mathfrak{c}_W^{L_0}) \ar[l, "i_{L_0}^*" above] \ar[r, "i_\infty^*"] & \msh_{\mathfrak{c}_{W,\partial L_0}}(\mathfrak{c}_{W,\partial L_0}) \\
        \msh_{L_H}(L_H) \ar[u, "\rotatebox{90}{$\sim$}"] \ar[d, "\rotatebox{90}{$\sim$}" left] & \msh_{\mathfrak{c}_{W\times T^*I}^{L_H}}(\mathfrak{c}_{W\times T^*I}^{L_H}) \ar[l, "i_{L_H}^*" above] \ar[r, "i_\infty^*"] \ar[d, "\rotatebox{90}{$\sim$}"] \ar[u, "\rotatebox{90}{$\sim$}" right] & \msh_{\mathfrak{c}_{W \times T^*I,\partial L_H}}(\mathfrak{c}_{W \times T^*I,\partial L_H}) \ar[d, "\rotatebox{90}{$\sim$}"] \ar[u, "\rotatebox{90}{$\sim$}" right] \\
        \msh_{L_{1}}(L_1) & \msh_{\mathfrak{c}_{W}^{L_1}}(\mathfrak{c}_W^{L_1}) \ar[l, "i_{L_1}^*" above] \ar[r, "i_\infty^*"] & \msh_{\mathfrak{c}_{W,\partial L_1}}(\mathfrak{c}_{W,\partial L_1}).
    \end{tikzcd}
    \end{equation}
\end{corollary}
\begin{proof}
    Consider the Weinstein manifolds with stops $(W \times T^*\bR, L_0)$ and $(W \times T^*\bR, L_1)$ related by the Weinstein cobordism $(W \times T^*\bR \times T^*I, L_H)$. Then we have the commutative diagram by restriction functors. The fact that the vertical arrows by restriction functors define equivalences now follows from Theorem \ref{thm: weinstein invariance}.
\end{proof}

\subsection{Maslov data and products}
    Let $X_1$ and $X_2$ be contact manifolds equipped with Maslov data $\mu_1$ and $\mu_2$ -- which are null homotopies of the maps $X_i \to B^2Pic(\SC)$. We define the product Maslov data $\mu_{12}$ as the null homotopy of the product map composed with the composition map on $B^2Pic(\SC)$:
    $$X_1 \,\widehat\times\, X_2 \to B^2Pic(\SC) \times B^2Pic(\SC) \to B^2Pic(\SC).$$
    Note that the Maslov data define sections of the $BPic(\SC)$-torsors $BPic(\SC)(X_i) \to X_i$, which define a section on the Whitney product
    $$\pi_1^*BPic(\SC)(X_1) \times \pi_2^*BPic(\SC)(X_2) \to X_1 \,\widehat\times\, X_2.$$
    By the fiberwise composition, this determines a section over $BPic(\SC)(X_1 \,\widehat\times\, X_2) \to X_1 \,\widehat\times\, X_2$. 
    
    We translate the results on duality and K\"unneth to this setting: 

\begin{theorem}\label{thm: weinstein kunneth duality maslov}
    Let $W_1, W_2$ be Weinstein manifolds with Maslov data $\mu_1, \mu_2$. Then there is an equivalence
    $$\msh_{\mathfrak{c}_{W_1 \times W_2,\mu_{12}}}(\mathfrak{c}_{W_1 \times W_2}) \simeq \msh_{\mathfrak{c}_{W_1,\mu_1}}(\mathfrak{c}_{W_1}) \otimes \msh_{\mathfrak{c}_{W_2,\mu_2}}(\mathfrak{c}_{W_2}).$$
    Moreover, $\msh_{\mathfrak{c}_{W_1,\mu_1}}(\mathfrak{c}_{W_1})$ is dualizable in $\PrLst$ with dual $\msh_{\mathfrak{c}_{W_1^-,\mu_1^-}}(\mathfrak{c}_{W_1^-})$, and
    $$\msh_{\mathfrak{c}_{W_1^- \times W_2, \mu_{12}^-}}(\mathfrak{c}_{W_1^- \times W_2}) \simeq \Fun^\mathrm{L}(\msh_{\mathfrak{c}_{W_1,\mu_1}}(\mathfrak{c}_{W_1}), \msh_{\mathfrak{c}_{W_2,\mu_2}}(\mathfrak{c}_{W_2})).$$
\end{theorem}
\begin{proof}
    Denote the fiber product structure on the product Lagrangian Grassmannian (given by taking direct sums of Lagrangian subspaces in each component) by
    $$m: \pi_1^*LGr(X_1) \times \pi_2^*LGr(X_2) \to LGr(X_1 \,\widehat\times\, X_2).$$
    Consider the universal microsheaves on the Lagrangian Grassmannian $LGr(X_i)$ determined by the relative cotangent bundles $\mathfrak f_i: LGr(X_i) \to X_i$ and $\mathfrak f_{12}: LGr(X_1 \,\widehat\times\, X_2) \to X_1 \,\widehat\times\, X_2$ following \cite[Section 11.1]{Nadler-Shende}. We have
    $m^*T^*\mathfrak f_{12} = T^*\mathfrak f_1 \,\widehat\times\, T^*\mathfrak f_2.$
    Therefore, when defining the universal microsheaves over the Lagrangian Grassmannian, taking the contact embedding of the product relative cotangent bundle $T^*\mathfrak f_{12}$ is equivalent to the contact product of the relative cotangent bundles $T^*\mathfrak f_1 \,\widehat\times\, T^*\mathfrak f_2$.
    
    We know that microsheaves on $\pi_1^*LGr(X_1) \times \pi_2^*LGr(X_2)$ descends to a sheaf of categories on
    $\pi_1^*BPic(\SC) (X_1) \times \pi_2^*BPic(\SC)(X_2)$, and microsheaves $LGr(X_1 \,\widehat\times\, X_2)$ descends to a sheaf of categories on $BPic(\SC)(X_1 \widehat\times X_2)$ by \cite[Theorem 11.14]{Nadler-Shende}.
    Then, since the natural map $U/O \to BPic(\SC)$ is a monoidal functor, we have a commutative diagram on the fiberwise products
    \begin{equation}\label{eq: LGr to BPic}
    \begin{tikzcd}
    \pi_1^*LGr(X_1) \times \pi_2^*LGr(X_2) \ar[d] \ar[r,"m"] & LGr(X_1 \,\widehat\times\, X_2) \ar[d] \\
    \pi_1^*BPic(\SC)(X_1) \times \pi_2^*BPic(\SC)(X_2) \ar[r, "m"] & BPic(\SC)(X_1 \,\widehat\times\, X_2)
    \end{tikzcd}
    \end{equation}
    Therefore, we have is a natural isomorphism between sheaves of categories as both sides are obtained by descents from the same sheaf of categories
    $$\mu_{12}^*\msh_{BPic(\SC)(X_1 \,\widehat\times\, X_2)} = (\mu_1 \times \mu_2)^*\msh_{\pi_1^*BPic(\SC)(X_1) \widehat\times \pi_2^*BPic(\SC)(X_2)}.$$
    When $W_i$ are Weinstein domains and $X_i = W_i \times \bR$, we know that the Lagrangian skeleta of $T^*\mathfrak f_i$ are the restrictions of the Lagrangian Grassmannian $LGr(W_i)|_{\mathfrak{c}_{W_i}}$. Hence the K\"unneth formula follows from Theorem \ref{thm:kunneth-microsheaf}. 
    
    Finally, note that the fiberwise inverse map $a: LGr(X) \to LGr(X)$ changes the co-orientation of the relative cotangent bundle $\mathfrak f: LGr(X) \to X$. Therefore, microsheaves with respect to the inverse Maslov data is simply realized by microsheaves on the image of the antipodal map (in the cosphere bundle). Hence duality and the classification of colimit preserving functors follow from Theorems
    \ref{thm:microsheaf-duality} and \ref{thm:microsheaf-fourier}.
\end{proof}

    By taking the right adjoints as in Theorem \ref{thm:microsheaf-duality-PrR}, we can also show duality and classification of limit preserving functors:

\begin{theorem}
    Let $W_1, W_2$ be Weinstein manifolds with Maslov data $\sigma_1, \sigma_2$. $\msh_{\mathfrak{c}_{W_1,\sigma_1}}(\mathfrak{c}_{W_1})$ is dualizable with dual $\msh_{\mathfrak{c}_{W_1^-,\sigma_1^-}}(\mathfrak{c}_{W_1^-})$ in $\PrRst$, and
    $$\msh_{\mathfrak{c}_{W_1^- \times W_2, \sigma_{12}^-}}(\mathfrak{c}_{W_1^- \times W_2}) \simeq \Fun^\mathrm{R}(\msh_{\mathfrak{c}_{W_1,\sigma_1}}(\mathfrak{c}_{W_1}), \msh_{\mathfrak{c}_{W_2,\sigma_2}}(\mathfrak{c}_{W_2}))^{op}.$$
\end{theorem}

%% file: compatibilities.tex

\section{Compatibilities}

    We explain the relationship of various functors involved in our constructions: the relationship between doubling, K\"unneth, and nearby cycle functors, and the relationship between sheaf quantizations and nearby cycle functors.

\subsection{Doubling, K\"unneth, and nearby cycle}
    We explain the relationship between the K\"unneth formula of microsheaves and doubling functors. We have the following theorem which generalizes \cite[Theorem 1.4 \& Proposition 3.21]{KuoLi-duality}.

    We use the notation for doubles from Definition \ref{def: positive doubling} in the following statement and proof.

\begin{theorem}\label{thm: kunneth-doubling}
    Let $\Lambda_1 \subset S^*M_1, \Lambda_2 \subset S^*M_2$ be compact sufficiently Legendrian subsets. Then for sufficiently small $0 < \epsilon' \ll \epsilon$, there is a commutative diagram of equivalences
    \[\begin{tikzcd}[column sep=50pt]
    \Sh_{(\underline{\Lambda_1 \widehat\times \Lambda_2})_{\cup, \epsilon'}^+}(M_1 \times M_2) \ar[d, "\rotatebox{90}{$\sim$}"] \ar[r, "m_{\Lambda_1 \times \Lambda_2}"] & \msh_{\Lambda_1 \widehat\times \Lambda_2}(\Lambda_1 \times \Lambda_2) \\ 
    \Sh_{(\underline{\Lambda_1})_{\cup,\epsilon}^+ \times (\underline\Lambda_{2})_{\cup,\epsilon}^+}(M_1 \times M_2) \ar[r, "m_{\Lambda_1} \otimes m_{\Lambda_2}"] & \msh_{\Lambda_1}(\Lambda_1) \otimes \msh_{\Lambda_2}(\Lambda_2) \ar[u, "\rotatebox{90}{$\sim$}"] ,
    \end{tikzcd}\]
    where the left vertical equivalence is the composition:     $$\iota_{(\underline\Lambda_{1})^+_{\cup,\epsilon} \times (\underline\Lambda_{2})^+_{\cup,\epsilon}}^*: \Sh_{(\underline{\Lambda_1 \widehat\times \Lambda_2})^+_{\cup, \epsilon'}}(M_1 \times M_2) \hookrightarrow \Sh(M_1 \times M_2) \lr \Sh_{(\underline\Lambda_{1})^+_{\cup,\epsilon} \times (\underline\Lambda_{2})^+_{\cup,\epsilon}}(M_1 \times M_2).$$
    Here, the second map is the left 
    adjoint to the inclusion going in the opposite direction. 
\end{theorem}

\begin{proof}
    For any $\SF_1 \boxtimes \SF_2 \in \Sh_{(\underline\Lambda_1)_{\cup,s}^+ \times (\underline\Lambda_2)_{\cup,s}^+}(M_1 \times M_2)$, by Theorem \ref{thm:kunneth-microsheaf}, we have isomorphisms
    $$m_{\Lambda_{1} \times \Lambda_{2}}(\SF_1 \boxtimes \SF_2) = m_{\Lambda_{1}}(\SF_1) \boxtimes m_{\Lambda_{2}}(\SF_2).$$
    Then, consider the left adjoints of $m_{\Lambda_{1}}$ and $m_{\Lambda_{2}}$ given by the doubling functors $w_{\Lambda_{1}}^+$ and $w_{\Lambda_{2}}^+$ by Theorem \ref{thm: doubling adjoint}. Since the left adjoint is unique, it follows that 
    $$m_{\Lambda_{1} \times \Lambda_{2}}^l(\SF_1 \boxtimes \SF_2) = w_{\Lambda_{1}}^+(\SF_1) \boxtimes w_{\Lambda_{2}}^+(\SF_2).$$
    On the other hand, by Theorem \ref{thm: doubling adjoint}, we also know that $m_{\Lambda_{1} \times \Lambda_{2}}^l$ can be written as the composition of doubling and left adjoint of tautological inclusions 
    $$m_{\Lambda_{1} \times \Lambda_{2}}^l(\SF_1 \boxtimes \SF_2) = \iota_{(\underline\Lambda_1)^+_{\cup,\epsilon} \times (\underline\Lambda_2)^+_{\cup,\epsilon}}^* \circ w_{(\underline{\Lambda_1 \widehat\times \Lambda_2})^+_{\cup,\epsilon}}(\SF_1 \boxtimes \SF_2).$$
    This implies that there is a commutative diagram 
    \[\begin{tikzcd}[column sep=60pt]
    \msh_{\Lambda_1 \widehat\times \Lambda_2}(\Lambda_1 \times \Lambda_2) \ar[r, "w_{\Lambda_{1} \widehat\times \Lambda_{2}}^+"] &  \Sh_{(\underline{\Lambda_1 \widehat\times \Lambda_2})^+_{\cup, \epsilon'}}(M_1 \times M_2) \\ 
    \msh_{\Lambda_1}(\Lambda_1) \otimes \msh_{\Lambda_2}(\Lambda_2) \ar[r, "w_{\Lambda_{1}}^+ \otimes w_{\Lambda_{2}}^+"] \ar[u, "\rotatebox{90}{$\sim$}"] & \Sh_{(\underline{\Lambda_1})^+_{\cup,\epsilon} \times (\underline\Lambda_{2})^+_{\cup,\epsilon}}(M_1 \times M_2) \ar[u]
    \end{tikzcd}\]
    where the right vertical functor is the composition
    $$\iota_{(\underline\Lambda_{1})^+_{\cup,\epsilon} \times (\underline\Lambda_{2})^+_{\cup,\epsilon}}^*: \Sh_{(\underline{\Lambda_1 \widehat\times \Lambda_2})^+_{\cup, \epsilon'}}(M_1 \times M_2) \hookrightarrow \Sh(M_1 \times M_2) \lr \Sh_{(\underline\Lambda_{1})^+_{\cup,\epsilon} \times (\underline\Lambda_{2})^+_{\cup,\epsilon}}(M_1 \times M_2).$$
    Since Theorem \ref{thm: relative-doubling} shows that $w_{\Lambda_{1}}^+ \otimes w_{\Lambda_{2}}^+$ and $w_{\Lambda_{1} \times \Lambda_{2}}^+$ are inverse equivalences to $m_{\Lambda_{1}} \otimes m_{\Lambda_{2}}$ and $m_{\Lambda_{1} \times \Lambda_{2}}$, the above diagram also commutes when we replace the horizontal doubling functors by the microlocalization functors. Since Theorem \ref{thm:kunneth-microsheaf} shows that the left vertical functor is an equivalence, it follows that the right vertical functor is also an equivalence.
\end{proof}

Recall that the left and right adjoints to the tautological inclusion $\iota_{\Lambda *}: \Sh_{\Lambda}(M) \hookrightarrow \Sh(M)$ 
are given `by wrapping' \cite[Theorem 1.2]{Kuo-wrapped-sheaves}. We recall the precise formulation:  
    
\begin{definition}[{\cite[Definition 3.14]{Kuo-wrapped-sheaves}}]\label{def: wrapping}
    Let $\Lambda \subset S^*M$ be a closed subset. The category $W(\Lambda)$ consists of non-negative compactly supported contact isotopies on $S^*M \setminus \Lambda$ as objects, whose $j$-morphisms are $\Delta^j$-families of positive contact isotopies.
\end{definition}

\begin{theorem}[{\cite[Theorem 1.2]{Kuo-wrapped-sheaves}}]\label{rem:wrapping} 
    Let $\Lambda \subset S^*M$ be a closed subset. There is a functor $W(\Lambda) \to \Fun^\mathrm{L}(\Sh(M), \Sh(M)), \varphi \mapsto \SK_\varphi$, and for the tautological inclusion $\iota_{\Lambda *}: \Sh_{\Lambda}(M) \hookrightarrow \Sh(M)$, its adjoint functors are
    $$\iota_{\Lambda}^*\SF = \operatorname{colim}_{\varphi \in W(S^*M \setminus \Lambda)}\SK_{\varphi} \circ \SF, \quad \iota_{\Lambda}^!\SF = \operatorname{lim}_{\varphi \in W(S^*M \setminus \Lambda)}\SK_{\varphi} \circ \SF.$$
\end{theorem}

There is the following criterion for when a cofinal sequence in the colimit computing $i_\Lambda^*$ can be obtained from a given positive contact isotopy:

\begin{proposition}[{\cite[Theorem 5.15]{Kuo-wrapped-sheaves}}]\label{prop:wrapping nearby}
    Let $\Lambda, \Lambda' \subset S^*M$ be compact pdff subsets such that $\Lambda \cap \Lambda' = \varnothing$. Suppose there is a positive contact isotopy $\varphi_t$ in $W(\Lambda)$ such that for any open neighborhood $\Omega$ of $\Lambda$, there is $T \in \bR$ such that $\varphi_t(\Lambda') \subset \Omega$ for $t \geq T$. Then 
    for any $\SF \in \Sh_{\Lambda'}(M)$, 
    $$\iota_\Lambda^*\SF = \operatorname{colim}_{t\to \infty}\SK_{\varphi}^t \circ \SF = \psi(\SK_\varphi \circ \SF),$$
    and the functor $\iota_{\Lambda}^*: \Sh_{\Lambda'}(M) \to \Sh_\Lambda(M)$ is fully faithful.
\end{proposition}

Finally, we show that the left adjoint of the tautological inclusion in Theorem \ref{thm: kunneth-doubling} can be realized as a nearby cycle functor.

\begin{theorem}\label{thm: doubling-small-piece}
    Let $\Lambda_1 \subset S^*M_1$, $\Lambda_2 \subset S^*M_2$ be sufficiently Legendrian subsets with boundaries. There is a contact isotopy that defines a family of sufficiently Legendrian subsets $({\Lambda_{1} \,\widehat\times\,\Lambda_2})_{\cup,s} \subset S^*(M_1 \times M_2) \times \bR_{>0}$ with the nearby cycle functor
    $$\psi_{(\underline\Lambda_1)_{\pm \epsilon} \times (\underline\Lambda_2)_{\pm \epsilon}}: \Sh_{(\underline{\Lambda_{1} \widehat\times \Lambda_{2}})_{\cup,\epsilon'}}(M_1 \times M_2) \xrightarrow{\sim} \Sh_{(\underline\Lambda_1)_{\cup,\epsilon} \times (\underline\Lambda_2)_{\cup,\epsilon}}(M_1 \times M_2).$$
    Moreover, it is equivalent to the left adjoint of the tautological inclusion
    $$\iota_{(\underline\Lambda_1)_{\pm \epsilon} \times (\underline\Lambda_2)_{\pm \epsilon}}^*: \Sh_{(\underline{\Lambda_{1} \widehat\times \Lambda_{2} })_{\cup,\epsilon'}}(M_1 \times M_2) \xrightarrow{\sim} \Sh_{(\underline\Lambda_1)_{\cup,\epsilon} \times (\underline\Lambda_2)_{\cup,\epsilon}}(M_1 \times M_2).$$
\end{theorem}
\begin{proof}
    Let $\Lambda \subset S^*M$ be a sufficiently Legendrian subset. Then the $U$-shape filling $\Lambda \times \cup_{\epsilon,\epsilon'}$ of the double copy $\Lambda_{\epsilon} \cup \Lambda_{\epsilon'}$ is defined as follows. Let $f: (\epsilon, \epsilon') \to \bR_{>0}$ be a smooth function such that $f(s) \to +\infty$ when $s \to \epsilon$ or $\epsilon'$. 
    \begin{equation*}
    \Lambda \times \cup_{\epsilon,\epsilon'} = \{(x, r\xi) \mid (x, \xi) \in \Lambda_s \subset S^*M, r = f(s) \in \bR_{>0}\} \subset T^*M.
    \end{equation*}
    We can choose the smooth function $f$ such that $\Lambda \times \cup_{\pm\epsilon}$ is contained in an arbitrary small neighbourhood of $\underline\Lambda_{\cup,\epsilon}$.

    Consider the sufficiently Legendrian subset $(\underline\Lambda_1)_{\cup,\epsilon} \times (\underline\Lambda_2)_{\cup,\epsilon}$. We construct a contact isotopy that induces a family of relative doubling subsets $(\Lambda_1 \,\widehat\times\, \Lambda_2)_{\pm s}$ that are contained in arbitrary small neighbourhoods of $(\underline\Lambda_1)_{\cup,\epsilon} \times (\underline\Lambda_2)_{\cup,\epsilon}$. Consider the decomposition
    \begin{align*}
    (\underline\Lambda_1)_{\cup,\epsilon} \times (\underline\Lambda_2)_{\cup,\epsilon} = &\, \big((\underline\Lambda_1)_{\cup,\epsilon} \times (\Lambda_{2,\pm\epsilon} \times \bR_{>0})\big) \cup \big((\Lambda_{1,\pm\epsilon} \times \bR_{>0}) \times (\underline\Lambda_2)_{\cup,\epsilon} \big) \\
    &\,\cup \Big(\bigcup\nolimits_{-\epsilon < s < \epsilon}\pi_1(\Lambda_{1,s}) \times \bigcup\nolimits_{-\epsilon < s < \epsilon}\pi_2(\Lambda_{2,s}) \Big).
    \end{align*}
    First, consider $(\underline\Lambda_{1})_{\cup,\epsilon} \times \Lambda_{2,\pm \epsilon}$ with a small neighbourhood $D^*_r M_1 \times U_r(\Lambda_{2,\pm \epsilon})$ where $D_r^*M_1$ is the disk bundle of $M$ and $U_r(\Lambda_{2,\pm \epsilon})$ is a neighbourhood of $\Lambda_{2,\pm \epsilon}$. Construct the standard U-shape filling $\Lambda_1 \times \cup_{\pm \epsilon}$ of $\Lambda_{1,\pm \epsilon}$ inside the disk bundle $D_r^*M_1$. We know that it is contained in a small neighbourhood of $(\underline\Lambda_{1})_{\cup,\epsilon}$. For $s \ll \epsilon$ we now have a pdff subset
    $$(\Lambda_1 \times \cup_{\pm \epsilon - s}) \times \Lambda_{2,\pm \epsilon-s} \subset D_r^*M \times U_r(\Lambda_{2,\pm \epsilon}).$$
    Next, consider $\Lambda_{1,\pm \epsilon} \times \underline\Lambda_{2,\cup,\epsilon}$. We can similarly construct a pdff subset for $s \ll \epsilon$
    $$\Lambda_{1,\pm \epsilon-s} \times (\Lambda_2 \times \cup_{\pm \epsilon-s}) \subset U_r(\Lambda_{1,\pm \epsilon}) \times D^*_r M_2.$$
    Both of their boundaries are equal to $\Lambda_{1,\pm\epsilon} \times \Lambda_{2,\pm\epsilon}$ and hence can be glued together which defines a doubling of $\Lambda_{1,-\epsilon} \,\widehat\times\, \Lambda_{2,-\epsilon}$.

    We claim that in $S^*M_1 \,\widehat\times\, S^*M_2$, the doubling is contact isotopic to some relative doubling $(\Lambda_1 \,\widehat\times\, \Lambda_2)_{\pm s}$ in the complement of $(\underline\Lambda_{1})_{\cup,\epsilon} \times (\underline\Lambda_{2})_{\cup,\epsilon}$ for $\epsilon' \ll \epsilon$. First, note that the two branches $\Lambda_{1,\epsilon} \times (\Lambda_2 \times \cup_{\pm\epsilon})$ and $\Lambda_{1,-\epsilon} \times (\Lambda_2 \times \cup_{\pm\epsilon})$ are connected through a contact isotopy $\Lambda_{1,t} \times (\Lambda_2 \times \cup_{\pm \epsilon})$; the filling $\Lambda_1 \times \cup_{\pm \epsilon}$ of $\Lambda_{1,\pm \epsilon}$ and the filling $\Lambda \times \cup_{-\epsilon \pm s}$ of $\Lambda_{\pm\epsilon}$ are connected through a contact isotopy. Therefore, we have the first contact isotopy 
    \begin{align*}
    &(\Lambda_{1,\pm \epsilon} \times (\Lambda_2 \times \cup_{\pm \epsilon})) \cup ((\Lambda_1 \times \cup_{\pm\epsilon}) \times \Lambda_{2,\pm \epsilon}) \\
    &\xrightarrow{\sim} (\Lambda_{1,-\epsilon\pm s} \times (\Lambda_2 \times \cup_{\pm \epsilon})) \cup ((\Lambda_1 \times \cup_{-\epsilon\pm s}) \times \Lambda_{2,\pm \epsilon}).
    \end{align*}
    Second, consider the contact isotopy connecting $\Lambda_{1,-\epsilon \pm s} \times (\Lambda_2 \times \cup_{\pm \epsilon}) \cong \Lambda_{1,-\epsilon \pm s} \times (\Lambda_2 \times \cup_{-\epsilon \pm s})$ and $(\Lambda_{1} \times \Lambda_{2} \times \bR)_{\pm \epsilon'}$, and the isotopy connecting $(\Lambda_1 \times \cup_{-\epsilon\pm s}) \times \Lambda_{2,\pm\epsilon}$ and $(\Lambda_1 \times \cup_{-\epsilon \pm s}) \times \Lambda_{2,-\epsilon \pm s}$.
    This implies that we have a contact isotopy
    $$(\Lambda_{1,-\epsilon\pm s} \times (\Lambda_2 \times \cup_{-\epsilon\pm s})) \cup ((\Lambda_1 \times \cup_{\pm \epsilon}) \times \Lambda_{2,\pm \epsilon} ) \xrightarrow{\sim} (\Lambda_{1} \,\widehat\times\, \Lambda_{2})_{\pm s}.$$
    Thus, by Proposition \ref{prop:wrapping nearby}, we have a nearby cycle functor induced by the contact isotopy $(\Lambda_1 \,\widehat\times\, \Lambda_2)_{\pm s}$, where we write $(\underline\Lambda_i)_{\pm\epsilon} = 0_{M_i} \cup (\Lambda_i)_{\pm\epsilon} \times \bR_{>0}$:
    $$\psi_{(\underline\Lambda_1)_{\pm\epsilon} \times (\underline\Lambda_2)_{\pm\epsilon}} : \Sh_{(\underline{\Lambda_1 \widehat\times \Lambda_2})_{\cup,\epsilon'}}(M_1 \times M_2) \longrightarrow \Sh_{(\underline\Lambda_1)_{\pm\epsilon} \times (\underline\Lambda_2)_{\pm\epsilon}}(M_1 \times M_2).$$
    Moreover, Proposition \ref{prop:wrapping nearby} implies that the nearby cycle functor realizes $\iota_{(\underline\Lambda_1)_{\pm\epsilon} \times (\underline\Lambda_2)_{\pm\epsilon}}^*$ whose image is in $\Sh_{(\underline\Lambda_1)_{\pm\epsilon} \times (\underline\Lambda_2)_{\pm\epsilon}}(M_1 \times M_2)$.
    
    Finally, we will show that the essential image of the above nearby cycle functor lands in $\Sh_{(\underline\Lambda_1)_{\cup,\epsilon} \times (\underline\Lambda_2)_{\cup,\epsilon}}(M_1 \times M_2)$.  We have already established this microsupport estimate away from the zero section; what remains is to estimate the support. We can show that under the contact isotopy and the projection $\pi_{12}: T^*(M_1 \times M_2) \to M_1 \times M_2$, we have
    $$\pi_{12}\big((\Lambda_{1} \,\widehat\times\, \Lambda_{2})_{\pm s} \big) \subset U_r\Big(\bigcup\nolimits_{-\epsilon < s < \epsilon}\pi_1(\Lambda_{1,s}) \times \bigcup\nolimits_{-\epsilon < s < \epsilon}\pi_2(\Lambda_{2,s}) \Big).$$
    Indeed, the contact isotopy $\Lambda_{1,\pm \epsilon} \times (\Lambda_2 \times \cup_{\pm \epsilon}) \xrightarrow{\sim} \Lambda_{1,-\epsilon\pm s} \times (\Lambda_2 \times \cup_{\pm \epsilon})$ is contained in a neighbourhood of $\bigcup_{-\epsilon<s<\epsilon}\Lambda_{1,s} \times (\Lambda_2 \times \cup_{\pm \epsilon})$, and the contact isotopy $((\Lambda_1 \times \cup_{\pm\epsilon}) \times \Lambda_{2,\pm \epsilon}) \xrightarrow{\sim} ((\Lambda_1 \times \cup_{-\epsilon\pm s}) \times \Lambda_{2,\pm \epsilon})$ is contained in a neighbourhood of $\bigcup_{-\epsilon<s<\epsilon}\Lambda_{1,s} \times \Lambda_{2,\pm\epsilon}$. Thus, the projection of the first contact isotopy is contained in a neighbourhood of
    $$\bigcup\nolimits_{-\epsilon<s<\epsilon}\pi_1(\Lambda_{1,s}) \times \pi_2(\Lambda_2 \times \cup_{\pm \epsilon}) \subset \bigcup\nolimits_{-\epsilon < s < \epsilon}\pi_1(\Lambda_{1,s}) \times \bigcup\nolimits_{-\epsilon < s < \epsilon}\pi_2(\Lambda_{2,s}).$$
    Similarly, the projection of the second contact isotopy is contained in a neighbourhood of
    $$\pi_1(\Lambda_{1} \times \cup_{-\epsilon \pm s}) \times \bigcup\nolimits_{-\epsilon<s<\epsilon}\pi_2(\Lambda_{2,s}) \subset \bigcup\nolimits_{-s<t<s}\pi_1(\Lambda_{1,-\epsilon+t}) \times \bigcup\nolimits_{-\epsilon < s < \epsilon}\pi_2(\Lambda_{2,s}).$$
    Therefore, we know that the nearby cycle of sheaves microsupported in the family $(\Lambda_{1} \,\widehat\times\, \Lambda_{2})_{\cup,s} \subset S^*M_1 \,\widehat\times\, S^*M_2$ defined the following functor by Lemma \ref{lem: ss-nearby-cycle}:
    $$\psi_{(\underline\Lambda_1)_{\cup, \epsilon} \times (\underline\Lambda_2)_{\cup, \epsilon}}: \Sh_{(\underline{\Lambda_{1} \widehat\times \Lambda_{2}})_{\cup,s}}(M_1 \times M_2) \hookrightarrow \Sh_{(\underline\Lambda_1)_{\cup,\epsilon} \times (\underline\Lambda_2)_{\cup,\epsilon}}(M_1 \times M_2).$$
    Then we know that the nearby cycle functor that realizes the functor $\iota_{(\underline\Lambda_1)_{\cup, \epsilon} \times (\underline\Lambda_2)_{\cup, \epsilon}}^*$ has image in $\Sh_{(\underline\Lambda_1)_{\cup,\epsilon} \times (\underline\Lambda_2)_{\cup,\epsilon}}(M_1 \times M_2)$, which finishes the proof.
\end{proof}

\begin{figure}
    \centering
    \includegraphics[width=0.8\linewidth]{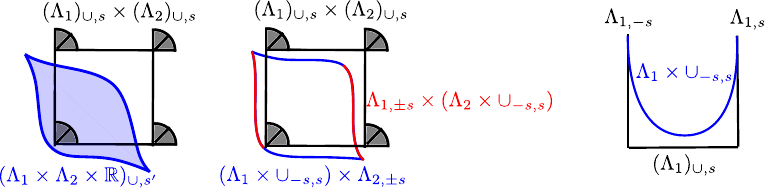}
    \caption{Let the base be $M_1 = M_2 = \bR$ and the Legendrians be $\Lambda_1 = \Lambda_2 = \{(0, 1)\} \subset S^*\bR$. The figure on the left shows the front projection of $(\Lambda_1)_{\cup,s} \times (\Lambda_2)_{\cup,s}$, $(\Lambda_1 \,\widehat\times\, \Lambda_2)_{\cup,s'}$ on the base $\bR^2$. The figure on the right shows the $U$-shape Lagrangian filling $\Lambda_1 \times \cup_{-s,s}$ and $\Lambda_2 \times \cup_{-s,s}$ into $T^*\bR$.}
    \label{fig:enter-label}
\end{figure}

\subsection{Sheaf quantization as nearby cycle} \label{sec:quantization}
    We will now give some variations on the basic microsheaf quantization result \cite[Corollary 9.9]{Nadler-Shende}, in particular giving many presentations of microsheaf quantization categories, and separating out some assertions which remain valid for immersed, rather than embedded, Lagrangians.  
    We introduce the following condition:

\begin{definition}
    A subset of $W \times \bR_t$ is $\epsilon$-thin if it is contained in $W \times (-\epsilon, \epsilon) \subset W$. An exact Lagrangian immersion $\ol{L} \looparrowright W$ is called $\epsilon$-thin if it admits an $\epsilon$-thin Legendrian lift.
\end{definition}

    Note that being $\epsilon$-thin is not a severe restriction. 
    The Liouville vector field $Z$ on $W$ defined by $\iota(Z)d\lambda = \lambda$ lifts to a contact vector field $\partial/\partial t + Z$. Abuse notations and denote the contact flow by $\varphi_Z$ as well. Then we have the following:

\begin{lemma}\label{lem: thinness}
    Let $W$ be a Liouville manifold and $L \subset W \times \bR$ be any eventually conical  subset in the sense of Definition \ref{def: eventually conic}. Let $\varphi_Z^{z}$ be the contact lift of Liouville flow on $W \times \bR$. Then there always exists $z_0 \in \bR$ such that when $z > z_0$, $\varphi_Z^{-z}(L) \subset W \times \bR$ is an $\epsilon$-thin subset. 
\end{lemma}

\begin{theorem}\label{thm: quantize nearby = microlocalize}
    Let $W$ be a Weinstein manifold with Maslov data and $L \subset W \times \bR$ be an $\epsilon$-thin eventually conical Legendrian. Then there is a diagram induced by nearby cycle functors and microlocalizations
    \[\begin{tikzcd}
    \Sh_{\underline L^+_{\cup,\epsilon}}(\bR^N) \ar[d, "\psi_{W}^L" left] \ar[r, "m_{L}"] & \msh_{L}(L) \\
    \Sh_{(\underline{\mathfrak{c}}_W^L)^+_{\cup,\epsilon}}(\bR^N) \ar[r, "m_{{\mathfrak{c}}_{W}^L}"] & \msh_{\mathfrak{c}_W^L}(\mathfrak{c}_W^L) \ar[u, "i_{L}^*" right].
    \end{tikzcd}\]
    Moreover, the composition $m_{{\mathfrak{c}}_{W}^L} \circ \psi_{\mathfrak{c}_{W}^L}$ is an equivalence. In particular, sheaf quantizations of the Legendrian $L$ are equivalent to objects in the category $\Sh_{\underline L_{\cup,\epsilon}}(\bR^N)$.
\end{theorem}
\begin{proof}
    Consider the doubling of the Lagrangian subset $\mathfrak{c}_W^L$ with respect to the thickening $\sigma$ inside $S^*\bR^N$. By Theorem \ref{thm: relative-doubling}, we know that there is an equivalence
    $$w_{{\mathfrak{c}}_{W}^L}^+ : \msh_{{\mathfrak{c}}_{W}^L}({\mathfrak{c}}_{W}^L) \simeq \Sh_{(\underline{\mathfrak{c}}_{W}^L)^+_{\cup,\epsilon}}(\bR^N) : m_{{\mathfrak{c}}_{W}^L}.$$
    We now construct the nearby cycle functor
    $$\psi_W^L: \Sh_{\underline L^+_{\cup,\epsilon}}(\bR^N) \to \Sh_{(\underline{\mathfrak{c}}_W^L)^+_{\cup,\epsilon}}(\bR^N) .$$
    and show commutativity of the diagram. 
    
    For the Legendrian submanifold $L \subset W \times \bR$, a smooth function $f: \bR_{>0} \to \bR_{>0}$ such that $f(0) \to +\infty$ and $f(+\infty) \to 0$ with $F(r) = \int_0^r f(s)ds$, consider the $U$-shape exact Lagrangian 
    \begin{equation}\label{eq: u-shape filling}
    L \times \cup_{0,+\infty} = \{(\varphi_Z^s(x), t + f^{-1}(s), s) \mid (x, t) \in L, s \in \bR_{>0} \} \subset W \times T^*\bR.
    \end{equation}
    Consider the Legendrian doubling of the $U$-shape Lagrangian filling. There exists a contact isotopy induced by the Liouville flow $Z = Z_\lambda + \tau \partial/\partial \tau$ on $W \times T^*\bR$ that sends $L \times \cup_{0,+\infty}$ to an arbitrary small neighbourhood of $\mathfrak{c}_W^L = (\mathfrak{c}_{W} \times 0_\bR) \cup (L \times \bR_{>0})$. Therefore, the contact movie of the Legendrian doubling $(\underline L \times \cup_{0,\infty})_{Z,\cup,\epsilon}$ satisfies the condition $\psi(\underline L \times \cup_{0,\infty})^+_{Z,\cup,\epsilon} \subset (\underline{\mathfrak{c}}_W^L)^+_{\cup,\epsilon}$. Note that the Legendrian lift of $(L \times \cup_{0,\infty})_Z \subset W \times T^*\bR \times (-\epsilon, \epsilon)$ is gapped with respect to the standard Reeb flow since $L \times \cup_{0,\infty} \subset W \times T^*\bR$ is embedded Lagrangian. Then by Theorem \ref{thm: gapped specialization} that there is a fully faithful embedding
    $$\psi_{W}^L: \Sh_{(\underline{L}\times \cup_{0,\infty})^+_{\cup,\epsilon}}(\bR^N) \xrightarrow{\sim} \Sh_{(\underline{L}\times \cup_{0,\infty})_{Z,\cup,\epsilon}^+}(\bR^N \times \bR_{>0}) \hookrightarrow \Sh_{(\underline{\mathfrak{c}}_{W}^L)_{\cup,\epsilon}^+}(\bR^N).$$
    Since the contact isotopy induced by the Liouville flow preserves the conical end of the Legendrian $L \times \cup_{0,+\infty} \subset W \times T^*\bR$, we have a commutative diagram
    \[\begin{tikzcd}
    \Sh_{(\underline{L}\times \cup_{0,\infty})^+_{\cup,\epsilon}}(\bR^N) \ar[r, "\sim"] \ar[d, "m_{L}"] & \Sh_{(\underline{L}\times \cup_{0,\infty})^+_{Z,\cup,\epsilon}}(\bR^N \times \bR_{>0}) \ar[r] \ar[d, "m_{L \times \bR_{>0}}"] & \Sh_{(\underline{\mathfrak{c}}_{W}^L)^+_{\cup,\epsilon}}(\bR^N) \ar[d, "m_{L}"] \\
    \msh_{L}(L) \ar[r,"\sim"] & \msh_{L \times \bR_{>0}}(L \times \bR_{>0}) \ar[r,"\sim"] & \msh_{L}(L)
    \end{tikzcd}\]
    where the bottom horizontal functors are isomorphisms by  Lemma \ref{lem:contact-transform-main}. This shows that the diagram in the proposition commutes.
    
    Finally, we show that the nearby cycle functor is not only fully faithful but an equivalence. Due to the contact invariance of the categories of sheaves and microsheaves, the nearby cycle functor does not depend on the contact embedding we choose. Without loss of generality, fixing an exact embedding $\varphi : W \hookrightarrow S^*\bR^{N-1}$ and a non-negative smooth function $f: \bR \to \bR_{\geq 0}$ increasing on $\bR_{\geq 0}$, we consider an exact embedding 
    $$W \times T^*\bR \hookrightarrow S^*(\bR^N \times \bR), (x, t, \tau) \mapsto (\varphi(x); t+f'(\tau), \tau; f(\tau)-2\epsilon, 1).$$
    Then we know $\mathfrak{c}_{W \times T^*\bR} \cap S^*(\bR^{N-1} \times \bR_{\geq 0}) = \varnothing$ and $\mathfrak{c}_{W \times T^*\bR, L} \cap T^*(\bR^{N-1} \times \bR_{\geq 0}) = L \times \bR_{\geq 0}$ and the doubling satisfies $(\underline{\mathfrak{c}}_{W \times T^*\bR})^+_{\cup,\epsilon} \cap T^*(\bR^{N-1} \times \bR_{\geq 0}) = \varnothing$ and 
    $$(\underline{\mathfrak{c}}_{W \times T^*\bR, L})^+_{\cup,\epsilon} \cap T^*(\bR^{N-1} \times \bR_{\geq 0}) = (\underline L \times \cup_{0,\infty})^+_{\cup,\epsilon} \cap T^*(\bR^{N-1} \times \bR_{\geq 0}) \cong \underline L_{\cup,\epsilon}^{+,\prec}.$$
    This is because for large enough $r > 0$, the intersection $(L \times \cup_{0,+\infty}) \cap (W \times T^*_{>r}\bR)$ is contact isotopic to $L \times \bR_{>r}$. Thus the doubling $(L \times \cup_{0,+\infty})^+_{\pm\epsilon} \cap S^*(\bR^{N-1} \times \bR_{>0})$ is contact isotopic to the union of $(L \times \bR_{>r})^+_{\pm\epsilon}$ where the contact push-off is supported away from the conical ends.
    Consider the restriction of the doubling to $\bR^{N-1} \times \bR_{>0}$. The above condition implies
    $$i_+^*: \Sh_{(\underline{\mathfrak{c}}_{W}^L)^+_{\cup,\epsilon}}(\bR^N) \hookrightarrow \Sh_{\underline L_{\cup,\epsilon}^{+,\prec}}(\bR^{N-1} \times \bR_{>0}).$$
    Moreover, there is a commutative diagram of restriction functors
    \[\begin{tikzcd}
    \Sh_{(\underline L\times \cup_{0,\infty})^+_{\cup,\epsilon}}(\bR^N) \ar[d, "i_+^*"] \ar[r, "\psi_{W}^L"] & \Sh_{(\underline{\mathfrak{c}}_{W}^L)^+_{\cup,\epsilon}}(\bR^N) \ar[d, "i_+^*"] \\ \Sh_{L_{\cup,\epsilon}^{+,\prec}}(\bR^{N-1} \times \bR_{>0}) \ar[r, "="] & \Sh_{\underline L_{\cup,\epsilon}^{+,\prec}}(\bR^{N-1} \times \bR_{>0}).
    \end{tikzcd}\]
    Since $(\underline L\times \cup_{0,\infty})^+_{\cup,\epsilon}$ can be realized as the doubling $\underline L_{\cup,\epsilon}^{+,\prec} \cup_{\underline L_{\cup,\epsilon}^+} \underline L_{\cup,\epsilon}^{+,\prec}$, the restriction functor 
    $i_+^*: \Sh_{(\underline L\times \cup_{0,\infty})^+_{\cup,\epsilon}}(\bR^N) \rightarrow \Sh_{\underline L_{\cup,\epsilon}^{+,\prec}}(\bR^{N-1} \times \bR_{>0})$
    is an equivalence. We know that $i_+^*$ is the inverse equivalence of the nearby cycle functor.
\end{proof}

\begin{theorem}\label{thm: nearby = restriction at infty}
    Let $W$ be a Weinstein manifold with Maslov data and ${L} \subset W \times \bR$ be an $\epsilon$-thin eventually conical sufficiently Legendrian. Then there is a commutative diagram
    \begin{equation}\label{eq: sheaf quantization}
    \begin{tikzcd}[row sep=10pt]
    & \Sh_{\underline L_{\cup,\epsilon}}(\bR^N) \ar[dl, "m_L" above] \ar[dd, "\psi_W^L"] \ar[dr, "\psi_{W,\partial L}"] & \\
    \msh_{L}(L) & & \msh_{\mathfrak{c}_{W,\partial L}}(\mathfrak{c}_{W,\partial L}) \\
    & \msh_{\mathfrak{c}_{W}^L}(\mathfrak{c}_W^L) \ar[ul, "i_L^*"] \ar[ur, "i_\infty^*" below] &
    \end{tikzcd}
    \end{equation}
    When $L \subset W$ is an exact Lagrangian embedding, $i_L^*$ and $m_L$ are equivalences and $i_\infty^*$ and $\psi_{\mathfrak{c}_{W,\partial L}}$ are fully faithful, and going from left to right in the diagram  agrees with the functor of \cite[Corollary 9.9]{Nadler-Shende}: 
    $$\psi_{W,\partial L}: \msh_{L}(L) \hookrightarrow \msh_{\mathfrak{c}_{W,\partial L}}(\mathfrak{c}_{W,\partial L}).$$
\end{theorem}
\begin{proof}
    The commutativity of the left triangle in Formula \eqref{eq: sheaf quantization} is the content of Theorem \ref{thm: quantize nearby = microlocalize}. Hence, it suffices to show the commutativity of the right square in Formula \eqref{eq: sheaf quantization}. We consider the Hamiltonian flow on $W \times T^*\bR$ induced by the lift of the Liouville flow $Z = Z_\lambda + t\partial/\partial t$ which sends $(\mathfrak{c}_{W,\partial L} \times \bR) \cup (L \times \bR_{>0})$ to $(\mathfrak{c}_{W,\partial L} \times \bR) \cup (\mathfrak{c}_{W,\partial L} \times \bR_{>0})$. We write it as $\mathfrak{c}_{W,\partial L} \times \sqcup_{0,\infty}$ where $\sqcup_{0,\infty} = \bR_{\geq 0} \cup (0 \times \bR_{>0}) \subset T^*\bR$. Since the contact flow fixes (a neighborhood of) $\mathfrak{c}_{W,\partial L} \times \infty$, by Proposition \ref{prop: nearby commute micro}, we know that $i_\infty^*$ factors as
    \begin{equation}\label{eq: factor nearby cycle}
    \msh_{\mathfrak{c}_{W}^L}(\mathfrak{c}_W^L) \xrightarrow{\psi_Z} \msh_{\mathfrak{c}_{W,\partial L} \times \sqcup_{0,\infty}}(\mathfrak{c}_{W,\partial L} \times \sqcup_{0,\infty}) \xrightarrow{i_\infty^*} \msh_{\mathfrak{c}_{W,\partial L}}(\mathfrak{c}_{W,\partial L}).
    \end{equation}
    On the other hand, we can consider a contact isotopy from $\sqcup_{0,\infty}$ to an arbitary small neighborhood of $\cup_{0,\infty}$ in $T^*\bR$. Consider the nearby cycle functor $\psi_{Z'}$ induced by the contact flow $Z'$. Then as it fixes (a neighborhood of) $\mathfrak{c}_{W,\partial L} \times \infty$, by Proposition \ref{prop: nearby commute micro}, we know that the second restriction functor $i_\infty^*$ in Formula \eqref{eq: factor nearby cycle} further factors as
    \begin{equation*}
    \msh_{\mathfrak{c}_{W,\partial L} \times \sqcup_{0,\infty}}(\mathfrak{c}_{W,\partial L} \times \sqcup_{0,\infty}) \xrightarrow{\psi_{Z'}} \msh_{\mathfrak{c}_{W,\partial L} \times \cup_{0,\infty}}(\mathfrak{c}_{W,\partial L} \times \cup_{0,\infty}) \xrightarrow{\sim} \msh_{\mathfrak{c}_{W,\partial L}}(\mathfrak{c}_{W,\partial L}).
    \end{equation*}
    Since the family $\underline{\mathfrak{c}}_{W,\partial L}\times\sqcup_{0,\infty}$ are Lagrangian skeleta of $W \times T^*\bR$ with respect to a family of Liouville forms, we can consider the family of Liouville flows that send $L \times \cup_{0,\infty}$ to neighbourhoods of the skeleton. Then, by Proposition \ref{nearby commute Nadler}, we can get the natural commutative diagram between nearby cycle functors and their restrictions
    \[\begin{tikzcd}
    \Sh_{(\underline L \times \cup_{0,\infty})_{\cup,\epsilon}}(\bR^N) \ar[d, "\psi_{W}" left] \ar[r, "\psi_{W,\partial L}"] & \Sh_{(\underline{\mathfrak{c}}_{W,\partial L} \times \cup_{0,\infty})_{\cup,\epsilon}}(\bR^N) \\
    \Sh_{(\underline{\mathfrak{c}}_{W}^L)_{\cup,\epsilon}}(\bR^N) \ar[r, "\psi_Z"] & \Sh_{(\underline{\mathfrak{c}}_{W,\partial L}\times\sqcup_{0,\infty})_{\cup,\epsilon}}(\bR^N) \ar[u, "\psi_{Z'}" right].
    \end{tikzcd}\]
    This shows the commutativity of the right triangle about nearby cycles and restriction at infinity in Formula \eqref{eq: sheaf quantization}, and it is clear that the top vertical functor is exactly the nearby cycle functor in \cite[Corollary 9.9]{Nadler-Shende}. 
    
    Finally, when $\ol{L} \hookrightarrow W$ is an exact Lagrangian embedding, it has no Reeb chords of length less than $\epsilon$, and by Theorem \ref{thm: relative-doubling} we know that the microlocalization defines an equivalence
    $$m_{L}: \Sh_{\underline L_{\cup,\epsilon}^+}(\bR^N) \simeq \msh_{L}(L): w_{L}^+.$$
    Then by the left commutative triangle of Formula \eqref{eq: sheaf quantization}, the restriction $i_{L}^*$ defines an equivalence and the inverse can be realized by the nearby cycle functor $\psi_W^L$. By the right commutative triangle of Formula \eqref{eq: sheaf quantization}, the restriction $i_\infty^* \circ \psi_W^L$ can also be realized by the functor $\psi_{W,\partial L}$ and this completes the proof. 
\end{proof}

\begin{figure}
    \centering
    \includegraphics[width=1.0\linewidth]{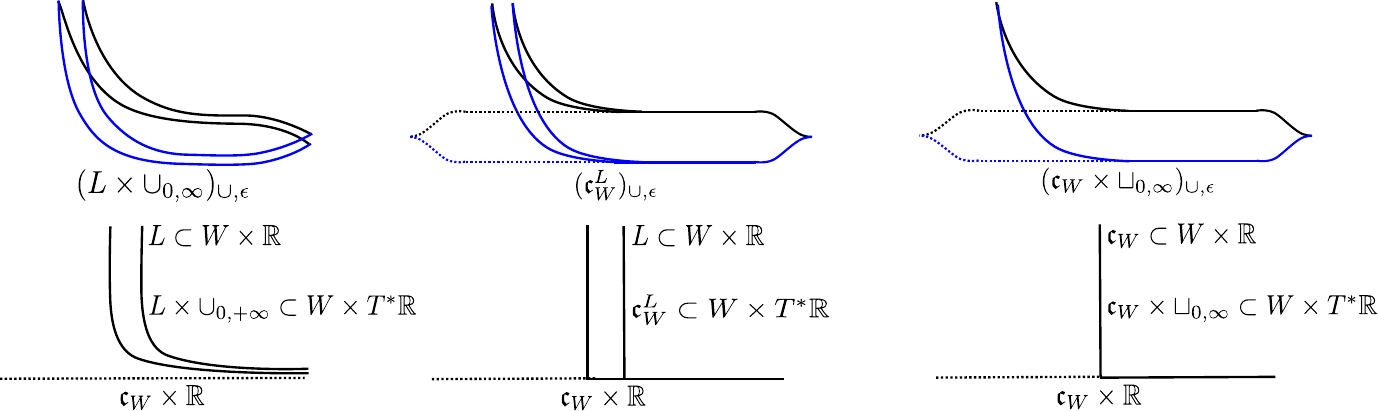}
    \caption{Let $W = pt$ and $L \subset W \times \bR$ be a Legendrian consisting of two points whose Reeb chord has length less than $\epsilon$. The figure on the bottom shows the Lagrangians $L \times \cup_{0,+\infty}$, $\mathfrak{c}_W^L = (\mathfrak{c}_{W} \times \bR) \cup (L \times \bR_{>0})$ and $\mathfrak{c}_W \times \sqcup_{0,+\infty} \subset W \times T^*\bR$, and the figure on the top shows the front projection of the doublings $(\underline{L} \times \cup_{0,+\infty})_{\cup,\epsilon}$, $(\underline{\mathfrak{c}}{}_{W}^L)_{\cup,\epsilon}$, and $(\underline{\mathfrak{c}}{}_{W} \times \sqcup_{0,+\infty})_{\cup,\epsilon} \subset S^*\bR^2$ under the contact embedding onto the base $\bR^2$.}
    \label{fig:double_quantization}
\end{figure}

    When $W = T^*M$ and $L \subset T^*M$ is an exact eventually conical Lagrangian, the sheaf quantizations of $L$ define sheaves in $M \times \bR$. Such notions have been studied extensively in previous works \cite{Guillermou-survey,Jin-Treumann}, and we now discuss the compatibility between the sheaf quantization in \cite{Nadler-Shende} and the sheaf quantizations therein.

\begin{example}\label{ex: guillermou quantize} \label{ex: quantizate cotangent bundle}
    Consider $W = T^*M$, a compact exact Lagrangian $L\subset T^*M$, and the diagram
    $$\msh_{L}(L) \xleftarrow{m_L} \msh_{\mathfrak{c}_{T^*M}^L}(\mathfrak{c}_{T^*M}^L) = \Sh_L(M \times \bR)_0 \xrightarrow{i_\infty^*} \Loc(M).$$
    Guillermou showed the left arrow is an equivalence and the right arrow is fully faithful \cite[Theorem 12.4.3 \& 12.4.4]{Guillermou-survey}. 
    Theorem \ref{thm: nearby = restriction at infty} recovers this result, and moreover asserts that $i_\infty^* \circ m_L^{-1}$  is realized by a nearby cycle functor, and hence matches the corresponding construction in  \cite{Nadler-Shende}.

    More generally, when we allow $L$ to be eventually conical, in the corresponding  diagram
    $$\msh_{L}(L) \xleftarrow{m_L}  \msh_{\mathfrak{c}_{T^*M}^L}(\mathfrak{c}_{T^*M}^L) = \Sh_{L \cup (\partial L \times \bR)}(M \times \bR)_0 \xrightarrow{i_\infty^*} \Sh_{\partial L}(M),$$
    Theorem \ref{thm: nearby = restriction at infty} continues to assert that 
    the left arrow is an equivalence and the right arrow is fully faithful, and that $i_\infty^* \circ m_L^{-1}$ is realized by a nearby cycle functor, and hence matches the corresponding construction in \cite{Nadler-Shende}.
\end{example}

\begin{remark}
    The case of eventually conical $L \subset T^*M$ was previously studied by Jin and Treumann \cite[Theorem 3.17 \& Section 3.18]{Jin-Treumann}.  
    
    Their construction was as follows. Consider perturbation in $S^*M \times (r, +\infty)$ by the Hamiltonian flow with $H_\epsilon(x, \xi) = \epsilon/|\xi|$ for $|\xi| \geq r+1$ and $H(x, \xi) = 0$ for $|\xi| \leq r$. The primitive of such perturbation $L_\epsilon$ of the Lagrangian $L$ is bounded from below as in \cite[Section 3.6 Corollary (i)]{Jin-Treumann}. They showed that in diagram
    $$\msh_{L_\epsilon}(L_\epsilon) \xleftarrow{m_L} \Sh_{L_\epsilon}(M \times \bR)_0 \xrightarrow{i_\infty^*} \Loc(M),$$
    the left arrow is an equivalence and the right arrow is fully faithful. For such Lagrangian perturbations $L_\epsilon$, Theorem \ref{thm: nearby = restriction at infty} still applies and it asserts that $i_\infty^* \circ m_L^{-1}$ is also realized by a nearby cycle functor.
    
    A precise comparison of the quantization in Example \ref{ex: quantizate cotangent bundle} and that of Jin and Treumann is as follows. Observe $$\bigcap\nolimits_{\epsilon>0}\ol{\bigcup\nolimits_{0<t<\epsilon}L_t} \subset L \cup (\partial L \times \bR).$$
    Therefore, by Lemma \ref{lem: ss-nearby-cycle}, the nearby cycle induced by the Legendrian isotopy by Theorem \ref{thm: GKS sheaf} sends $\Sh_{L_\epsilon}(M \times \bR)_0$ to $\Sh_{L \cup (\partial L \times \bR)}(M \times \bR)_0$, one can show a commutative diagram
    \[\begin{tikzcd}
    \msh_{L_\epsilon}(L_\epsilon) \ar[d, "\psi_L"] & \Sh_{L_\epsilon}(M \times \bR)_0 \ar[l, "m_{L_\epsilon}" above] \ar[r, "i_\infty^*"] \ar[d, "\psi_{L \cup (\partial L \times \bR)}"] & \Sh_{\partial L}(M) \ar[d, "\rotatebox{90}{=}"] \\
    \msh_{L}(L) & \Sh_{L \cup (\partial L \times \bR)}(M \times \bR)_0 \ar[l, "m_L" above] \ar[r, "i_\infty^*"] & \Sh_{\partial L}(M).
    \end{tikzcd}\]
    Here $\psi_L: \msh_{L_\epsilon}(L_\epsilon) \to \msh_{L}(L)$ is an equivalence because $L_\epsilon \cap T^*_{<r}M \times \bR = L \cap T^*_{<r}M \times \bR$ and it is a deformation retract of $L_\epsilon$ and $L$ respectively.
\end{remark}

We now explain how to compute homomorphisms of sheaf quantizations of exact Lagrangian immersions via doubling and why this is compatible with the sheaf quantization by the nearby cycle functor.

\begin{definition}\label{def: cofinal position} 
    Let $W$ be a Liouville manifold and $L \subset W \times \bR$ and $K \subset W \times \bR$ be eventually conical Legendrian immersions. Then we say that $L \ll K$ if there are no Reeb chords from $K$ to $L$.
\end{definition}

\begin{lemma}
    Let $W$ be a Liouville manifold and $L \subset W \times \bR$ and $K \subset W \times \bR$ be eventually conical Legendrians. We have $L \ll K_t$ for $t \gg 0$, where $K_t$ is the time $t$ Reeb pushoff of $K$. 
\end{lemma}
\begin{proof}
    Since $L$ and $K$ are eventually conical, we know there exists $t > 0$ such that $L \subset W \times (-t/2, t/2)$ and $K \subset W \times (-t, t)$.
\end{proof}

\begin{theorem}\label{thm: lagrangian immersion tensor}
    Let $W$ be a Weinstein manifold with Maslov data, and $L, K \subset W \times \bR$ be any $\epsilon$-thin eventually concial sufficiently Legendrian subsets with conical ends contained in a sufficiently Legendrian $\Lambda \subset \partial W$ such that $L \ll K$. Then we have a commutative diagram
    \[\begin{tikzcd}[column sep=40pt]
    \Sh_{\underline{L}^-_{\cup,\epsilon}}(\bR^N) \otimes \Sh_{\underline{K}_{\cup,\epsilon}}(\bR^N) \ar[r, "\Gamma_c(-\otimes-)"] \ar[d, "\psi_{W^-}^{L^-} \otimes \psi_{W}^K" left] & \SC \ar[d, "\rotatebox{90}{=}"] \\
    \msh_{\mathfrak{c}_{W^-}^{L^-}}(\mathfrak{c}_{W^-}^{L^-}) \otimes \msh_{\mathfrak{c}_{W}^{K}}(\mathfrak{c}_{W}^{K}) \ar[d, "i_\infty^* \otimes i_\infty^*" left] \ar[r, "\Gamma_c(-\otimes-)"] & \SC \ar[d, "\rotatebox{90}{=}"] \\
    \msh_{\mathfrak{c}_{W^-,\Lambda}}(\mathfrak{c}_{W^-,\Lambda}) \otimes \msh_{\mathfrak{c}_{W,\Lambda}}(\mathfrak{c}_{W,\Lambda}) \ar[r, "\Gamma_c(-\otimes-)"] & \SC.
    \end{tikzcd}\]
\end{theorem}
\begin{proof}
    Using Theorem \ref{thm: nearby = restriction at infty}, we can replace both vertical functors by the nearby cycle functors. Since the Lagrangian $\mathfrak{c}_{W^-,\Lambda}, \mathfrak{c}_{W,\Lambda}$ and $\mathfrak{c}_{W^-}^{L^-}, \mathfrak{c}_{W}^K$ are contained in exact symplectic hypersurfaces, they are automatically gapped with respect to some Reeb flow. Thus the commutativity of the upper square follows from Theorem \ref{thm: gapped specialization}. Since $L$ and $K$ are $\epsilon$-thin and  $L \ll K$, they are also gapped with respect to some Reeb flow. Thus the commutativity of the lower square follows from Theorem \ref{nearby fully faithful}.
\end{proof}

\begin{theorem}\label{thm: lagrangian immersion hom}
    Let $W$ be a Weinstein manifold with Maslov data, and $L, K \subset W \times \bR$ be any $\epsilon$-thin eventually conical sufficiently Legendrian subsets with conical ends contained in the sufficiently Legendrian $\Lambda \subset \partial W$ such that $L \ll K$. Then we have a commutative diagram
    \[\begin{tikzcd}[column sep=40pt]
    \Sh_{\underline{L}_{\cup,\epsilon}}(\bR^N) \otimes \Sh_{\underline{K}_{\cup,\epsilon}}(\bR^N) \ar[r, "{\Hom(-,-)}"] \ar[d, "\psi_{W}^{L} \otimes \psi_{W}^K" left] & \SC \ar[d, "\rotatebox{90}{=}"]\\
    \msh_{\mathfrak{c}_{W}^{L}}(\mathfrak{c}_{W}^{L}) \otimes \msh_{\mathfrak{c}_{W}^{K}}(\mathfrak{c}_{W}^{K}) \ar[d, "i_\infty^* \otimes i_\infty^*" left] \ar[r, "{\Hom(-,-)}"] & \SC \ar[d, "\rotatebox{90}{=}"] \\
    \msh_{\mathfrak{c}_{W,\Lambda}}(\mathfrak{c}_{W,\Lambda}) \otimes \msh_{\mathfrak{c}_{W, \Lambda}}(\mathfrak{c}_{W,\Lambda}) \ar[r, "{\Hom(-,-)}"] & \SC.
    \end{tikzcd}\]
\end{theorem}
\begin{proof}
    Similar to the previous theorem, we can replace both vertical functors by the nearby cycle functors. Since the Lagrangian $\mathfrak{c}_{W,\Lambda}$ and $\mathfrak{c}_{W}^{L}, \mathfrak{c}_{W}^K$ are contained in exact symplectic hypersurfaces, they are automatically gapped with respect to some Reeb flow. Thus the commutativity of the upper square follows from Theorem \ref{gapped nearby cycle}. Since $L$ and $K$ are $\epsilon$-thin and  $L \ll K$, they are also gapped with respect to some Reeb flow. Thus the commutativity of the lower square follows from Theorem \ref{fullfaithful nearby}.
\end{proof}

%% file: gappedsheafcomposition.tex

\section{Gapped composition of sheaves}\label{sec: gapped composition}

    Here we give a geometric criterion in terms of microsupports (`gappedness') for when compositions of sheaf kernels commute with nearby cycle functors.    
    These are relative versions of the gapped specialization theorem of \cite{Nadler-Shende}, recalled as Theorems  \ref{fullfaithful nearby} and  \ref{nearby fully faithful} above.

\subsection{Microsupport estimates for composition}

    Let $M_1, M_2, M_3$ be manifolds and write $\pi_{ij}: M_1 \times M_2 \times M_3 \to M_i \times M_j$ and $\Delta_2: M_1 \times M_2 \times M_3 \hookrightarrow M_1 \times M_2 \times M_2 \times M_3$ be the diagonal embedding. Let $\SF_{12} \in \Sh(M_1 \times M_2)$ and $\SF_{23} \in \Sh(M_2 \times M_3)$. We define the composition of sheaves to be 
    $$\SF_{23} \circ \SF_{12} = \pi_{13!}\Delta_2^*(\SF_{12} \boxtimes \SF_{23}).$$

\begin{definition}
    For $\Lambda_{12} \subset T^*(M_1 \times M_2)$ and $\Lambda_{23} \subset T^*(M_2 \times M_3)$, we define the composition to be the following subset in $T^*(M_1 \times M_3)$, where $\pi_{13}: T^*(M_1 \times M_2 \times M_2 \times M_3) \to T^*(M_1 \times M_3)$ is the projection:
    \begin{align}
    \Lambda_{23} \circ \Lambda_{12} := \pi_{13}((\Lambda_{12} \times \Lambda_{23}) \cap T^*M_1 \times T^*_\Delta(M_2 \times M_2) \times T^*M_3).
    \end{align}
\end{definition}

\begin{lemma}\label{lem: ss-composition}
    Let $\SF_{12} \in \Sh(M_1 \times M_2)$ and $\SF_{23} \in \Sh(M_2 \times M_3)$.
    Assume that $\pi_{13}$ is proper on $\pi_{12}^{-1} \supp(\SF_{12}) \cap \pi_{23}^{-1} \supp(\SF_{23})$. Then, we have $\SS(F_{23} \circ F_{12}) \subset \SS(F_{23}) \circ \SS(F_{12})$. 
\end{lemma}
\begin{proof}
    By Formulae \eqref{lem: ss-pull back} and \eqref{lem: ss-hom}, we know that $\SS(\Delta_2^*(\SF_{12} \boxtimes \SF_{23}))$ is contained in the set of points $(x_1, x_2, x_3, \xi_1, \xi_2, \xi_3) \in T^*(M_1 \times M_2 \times M_3)$ such that there exist 
    $(x_{1,n}, x_{2,n}, \xi_{1,n}, \xi_{2,n}) \in \SS(\SF_{12})$, $(y_{2,n}, y_{3,n}, \eta_{2,n}, \eta_{3,n}) \in \SS(\SF_{23})$ with
\begin{gather*}
    x_{i,n} \to x_i,\quad y_{i,n} \to x_i, \quad
    (\xi_{1,n}, \xi_{2,n} + \eta_{2,n}, \eta_{3,n}) \to (\xi_1, \xi_2, \xi_3), \\
    |(x_1, x_2, x_2, x_3) - (x_{1,n}, x_{2,n}, y_{2,n}, y_{3,n})| |(\xi_{1,n}, \xi_{2,n}, \eta_{2,n}, \eta_{3,n})| \to 0.
\end{gather*}
    By Formula \eqref{lem: ss-push forward}, since $\pi_{13}$ is proper on $\pi_{23}^{-1} \supp(\SF_{23}) \cap \pi_{12}^{-1} \supp(\SF_{12})$, $\SS(\SF_{23} \circ \SF_{12})$ is contained in the set of points $(x_1, x_3, \xi_1, \xi_3) \in T^*(M_1 \times M_3)$ such that there exists $x_2 \in M_2$ with $(x_1, x_2, x_3, \xi_1, 0, \xi_3) \in \SS(\Delta_2^*(\SF_{12} \boxtimes \SF_{23}))$. This means that there exist $(x_{1,n}, x_{2,n}, \xi_{1,n}, \xi_{2,n}) \in \SS(\SF_{12})$, $(y_{2,n}, y_{3,n}, \eta_{2,n}, \eta_{3,n}) \in \SS(\SF_{23})$ with
\begin{gather*}
    x_{i,n} \to x_i,\quad y_{i,n} \to x_i, \quad
    (\xi_{1,n}, \xi_{2,n} + \eta_{2,n}, \eta_{3,n}) \to (\xi_1, 0, \xi_3), \\
    |(x_1, x_2, x_2, x_3) - (x_{1,n}, x_{2,n}, y_{2,n}, y_{3,n})| |(\xi_{1,n}, \xi_{2,n}, \eta_{2,n}, \eta_{3,n})| \to 0.
\end{gather*}
This implies that $\xi_{2,n} - \eta_{2,n} \to 0$ and $|x_{2,n} - y_{2,n}||\xi_{2,n}| \to 0$. We may assume that $\xi_{2,n} \to \xi_2$ and thus $\eta_{2,n} \to -\xi_2$. Then $(x_{1,n}, x_{2,n}, \xi_{1,n}, \xi_{2,n}) \to (x_1, x_2, \xi_1, \xi_2)$ and $(y_{2,n}, y_{3,n}, \eta_{2,n}, \eta_{3,n}) \to (x_2, x_3, -\xi_2, \xi_3)$. Since microsupports are closed subsets, this implies that there exists $(x_2, \xi_2) \in T^*M_2$ such that 
$$(x_1, x_2, \xi_1, \xi_2) \in \SS(\SF_{12}), \quad (x_2, x_3, -\xi_2, \xi_3) \in \SS(\SF_{23}).$$
Therefore, $(x_1, x_3, \xi_1, \xi_3) \in \SS(\SF_{23} \circ \SF_{12})$ if and only if the above condition holds.
\end{proof}

Similarly, we can also show the following:

\begin{lemma}\label{lem: ss-hom-composition}
    Let $\SF_{12} \in \Sh(M_1 \times M_2)$ and $\SF_{23} \in \Sh(M_2 \times M_3)$.
    Assume that $\pi_{13}$ is proper on $\pi_{12}^{-1} \supp(\SF_{12}) \cap \pi_{23}^{-1} \supp(\SF_{23})$. Then, we have $\SS(\sHom^\circ(\SF_{12}, \SF_{23})) \subset \SS(\SF_{23}) \circ (-\SS(\SF_{12}))$. 
\end{lemma}

\subsection{Gapped \texorpdfstring{$\otimes$}{Tensor} composition}
    We will show the commutativity of compositions of integral kernels and nearby cycles under gappedness assumptions.

    We recall the colimit preserving composition of integral kernels $\SF_{12}\in \Sh(M_1 \times M_2)$ and $\SF_{23} \in \Sh(M_2 \times M_3)$ \cite[Section 3.6]{KS}:\footnote{Our convention of composition differs from \cite[Section 3.6]{KS}. What we denote by $\SF_{23} \circ \SF_{12}$ here is instead denoted by $\SF_{12} \circ \SF_{23}$ there.} 
    $$\SF_{23} \circ \SF_{12} = \pi_{13!}(\pi_{12}^*\SF_{12} \otimes \pi_{23}^*\SF_{23}).$$
    
\begin{definition}\label{def: gap composition}
   Consider manifolds $M_1, M_2, M_3$. We say that the conic subsets $\Lambda_{12} \subset T^*(M_1 \times M_2 \times \bR_{>0})$ and $\Lambda_{23} \subset T^*(M_2 \times M_3 \times \bR_{>0})$ have {\em (positively) gapped composition} if 
   \begin{enumerate}
       \item $\Lambda_{12} \subset T^*(M_1 \times M_2 \times \bR_{>0})$, $\Lambda_{23} \subset T^*(M_2 \times M_3 \times \bR_{>0})$ are $\bR_{>0}$-non-characteristic,
       \item the composition of the pair $\psi_{23}(\Lambda_{23}) \circ \psi_{12}(\Lambda_{12}) \subset \dot{T}^*(M_1 \times M_3)$ is pdff, and
       \item the pair
       $$(\Lambda_{12} \times_{\bR_{>0}} \Lambda_{23}, 0_{M_1} \times \dot{T}^*_\Delta{M_2^2} \times 0_{M_3} \times \bR_{>0}) $$
       is (positively) gapped for a Reeb flow on $\dot{T}^*(M_1 \times M_2 \times M_2 \times M_3)$.
   \end{enumerate}
   For sheaves $\SF_{12} \in \Sh(M_1 \times M_2 \times \bR_{>0})$ and $\SF_{23} \in \Sh(M_2 \times M_3 \times \bR_{>0})$, we say that they have {\em (positively) gapped composition} if $\SS(\SF_{12})$ and $\SS(\SF_{23})$ have positively gapped composition.
\end{definition}

    Similar to Theorem \ref{nearby fully faithful}, the main technical ingredient is a base change formula between push-foward and pull-backs that is a family version of Lemma \ref{gapped basechange}.

    Here is the model case.
    Consider the maps
    \[\begin{tikzcd}
        \bR_{>0} \times 0 \times M \ar[d, "{i}_1" left] \ar[r, "j_2"] & \bR_{\geq 0} \times 0 \times M \ar[d, "\ol{i}_1"] \\
        \bR_{>0} \times \bR_{\geq 0} \times M \ar[r, "\ol{j}_2"] & \bR_{\geq 0} \times \bR_{\geq 0} \times M
    \end{tikzcd}\]
    Using non-characteristic deformation (Lemma \ref{lem: non-char}), we can show the following lemma (which plays the same role as Lemma \ref{lem: non-proper base change}):
    
\begin{lemma}\label{lem: non-proper base change compose}
    For a sheaf $\SF \in \Sh(\bR_{>0} \times \bR_{\geq 0} \times M)$ such that 
    \begin{enumerate}
        \item $\SS(\SF)$ is non-characteristic with respect to $\pi_1: \bR_{>0} \times \bR_{\geq 0} \times M \to \bR_{>0} \times 0 \times M$, 
        \item $\Pi_M(\ol{\SS_{\pi_1}(\SF)}) \cap T^*M$ is pdff for $\Pi_M: T^*(\bR_{\geq 0} \times M) \times \bR_{\geq 0} \to T^*M \times \bR_{\geq 0}^2$, and
        \item $\SS_{\pi_1}(\SF)$ is gapped from $\dot N^*_{out}(0 \times M) \times \bR_{>0}$ as an $\bR_{>0}$-family of subsets,
    \end{enumerate}
    the natural transformation defines an isomorphism
    $$\ol{i}_{1}^*\ol{j}_{2*}\SF \xrightarrow{\sim} j_{2*} {i}_1^*\SF.$$
\end{lemma}
\begin{proof}
    Since $\SS(\SF)$ is non-characteristic with respect to the projection $\pi_1: \bR_{>0} \times \bR_{\geq 0} \times M \to \bR_{>0} \times 0 \times M$ and the projection of $\Pi_M(\ol{\SS_{\pi_1}(\SF)}) \cap T^*M$ is a pdff subset, by Lemma \ref{stalk}, on the left hand side, the stalks of the sheaf at $0 \times x$ is
    $$\Gamma((0, \epsilon) \times 0 \times U_{\delta}(x), \SF).$$
    On the right hand side, the stalks of the sheaf at $0 \times x$ is
    $$\Gamma((0, \epsilon) \times [0, \epsilon') \times U_\delta(x), \SF).$$
    The result will follow if we can establish the necessary microsupport estimate to apply the noncharacteristic deformation lemma to the family $(0, \epsilon) \times [0, \epsilon') \times U_\delta(x)$ for $\epsilon' \to 0$.

    Since $\SF$ is non-characteristic with respect to the projection $\pi_1: \bR_{>0} \times \bR_{\geq 0} \times M \to \bR_{>0} \times 0 \times M$, there is an injection $\Pi_1: \SS(\SF) \hookrightarrow \SS_{\pi_1}(\SF)$. Thus, it suffices to estimate the intersection of the relative singular supports $\SS_{\pi_1}(\SF)$ with $\dot N^*_{out}([0, \epsilon') \times U_\delta(x)) \times (0, \epsilon)$.

    First, consider outward conormal vectors of $[0, \epsilon')$. Since $\SS_{\pi_1}(\SF)$ is gapped from $\dot N^*_{out}(\bR_{>0} \times 0 \times M)$, it follows that for $\epsilon' > 0$ sufficiently small
    $$\SS_{\pi_1}(\SF) \cap (\dot N^*_{out}[0, \epsilon') \times 0_M) \times (0, \epsilon) = \varnothing.$$

    Then, we consider the intersection of the microsupport with the outward conormal vectors of $U_\delta(x)$ and $[0, \epsilon)$ at the corner:
    $$\SS_{\pi_1}(\SF) \cap (\dot N^*_{out}[0, \epsilon') + \dot N^*_{out}U_\delta(x)) \times (0, \epsilon).$$
    Since $\SS_{\pi_1}(\SF) \cap (\dot N^*_{out}[0, \epsilon') \times 0_M) \times (0, \epsilon) = \varnothing$, we know that such intersection admits a further injective projection via
    $\Pi_M: \dot T^*(\bR_{\geq 0} \times M) \times [0, \epsilon) \to \dot T^*M \times \bR_{\geq 0} \times [0, \epsilon).$
    When $\epsilon$ and $\epsilon' \to 0$, by compactness, we know that the points in the intersection converges to points inside
    \begin{align*}
        \Pi_M: \ol{\SS_{\pi_1}(\SF)} \cap (\dot N^*_{out}0 + \dot N^*_{out}U_\delta(x)) \hookrightarrow \Pi_M(\ol{\SS_{\pi_1}(\SF)}) \cap \dot N^*_{out}U_\delta(x) \subset \dot T^*M.
    \end{align*}
    Since the projection $\Pi_M(\ol{\SS_{\pi_1}(\SF)}) \cap \dot T^*M$ is pdff, we know thhat when $\delta > 0$ is sufficiently small, the above intersection must be empty. Therefore, when $\delta > 0$ is sufficiently small, we can find $\epsilon' > 0$ sufficiently small such that the original intersection is also empty. This therefore completes the proof.
\end{proof}

    The following consequence of the above lemma is the key step for the commutativity of nearby cycle functors and compositions.

\begin{lemma}\label{gapped composition basechange}
    Let $\SF_{12} \in \Sh(M_1 \times M_2 \times \bR_{>0})$ and $\SF_{23} \in \Sh(M_2 \times M_3 \times \bR_{>0})$ have positively 
    gapped composition for some Reeb flow. Then
    $$\pi_{13 !}(\psi_{123}\, \Delta_{M_2/\bR_{>0}}^*(\SF_{12} \boxtimes_{\bR_{>0}} \SF_{23})) \simeq \pi_{13!}(\Delta_{M_2}^*(\psi_{1223}(\SF_{12} \boxtimes_{\bR_{>0}} \SF_{23})).$$
\end{lemma}
\begin{proof}
    We need to show that
    $$\pi_{13 !}(i^*j_* \Delta_{M_2/\bR_{>0}}^*(\SF_{12} \boxtimes_{\bR_{>0}} \SF_{23})) \simeq \pi_{13!}(i^*\Delta_{M_2}^*j_*(\SF_{12} \boxtimes_{\bR_{>0}} \SF_{23})).$$
    We consider $d: M_2 \times M_2 \to \bR_d$ to be the distance function to the diagonal $\Delta_{M_2}$ and write 
    \begin{gather*}
    \pi_{13,d}: M_1 \times M_2 \times M_2 \times M_3 \times \bR_t \to M_1 \times M_3 \times \bR_d \times \bR_t,\\ 
    (x_1, x_2, x_2', x_3, t) \mapsto (x_1, x_3, d(x_2, x_2'), t).
    \end{gather*}
    Under the assumption that $\SF_{12}$ and $\SF_{23}$ have proper supports, we may apply proper base change formula, and then it suffices to show that
    $$i^* j_* i_0^*\pi_{13,d!}(\SF_{12} \boxtimes_{\bR_{>0}} \SF_{23}) \simeq i^* \ol{i}_0^*\ol{j}_*\pi_{13,d!}(\SF_{12} \boxtimes_{\bR_{>0}} \SF_{23}),$$
    where $i_0$, $\ol{i}_0$, $j_*$ and $\ol{j}_*$ are the natural embeddings defined by the following diagram:
    \[\begin{tikzcd}
    \bR_{t>0} \times 0 \times M_1 \times M_3 \ar[r, "j"] \ar[d, "i_0"] & \bR_t \times 0 \times M_1 \times M_3 \ar[d, "\ol{i}_0"] \\
    \bR_{t>0} \times \bR_d \times M_1 \times M_3 \ar[r, "\ol{j}"] & \bR_t \times \bR_d \times M_1 \times M_3.
    \end{tikzcd}\]
    We will consider sheaves in the above diagram and show the Conditions (1)--(3) in Lemma \ref{lem: non-proper base change compose} using Conditions (1)--(3) in Definition \ref{def: gap composition} for gapped compositions as follows: 
    \begin{enumerate}
        \item Given the assumption that $\SS(\SF_{12})$ and $\SS(\SF_{23})$ are $\bR_{t>0}$-non-characteristic, we know that $\SS(\pi_{13,d!}(\SF_{12} \boxtimes_{\bR_{>0}} \SF_{23}))$ is also $\bR_{t>0}$-non-characteristic. 
        \item Given the assumption that $\psi(\SS(\SF_{23})) \circ \psi(\SS(\SF_{12}))$ is a pdff subset, we know that the projection $\Pi_{13}(\ol{\SS(\pi_{13,d!}(\SF_{12} \boxtimes_{\bR_{>0}} \SF_{23})}) \cap \dot T^*(M_1 \times M_3)$ is also a pdff subset. 
        \item Given the assumption that $\SS_\pi(\SF_{12}) \times_{\bR_{t>0}} \SS_\pi(\SF_{23})$ and $0_{M_1} \times \dot N^*\Delta_{M_2} \times 0_{M_3} \times \bR_{>0}$ are gapped, $\SS_\pi(\pi_{13,d!}(\SF_{12} \boxtimes_{\bR_{>0}} \SF_{23}))$ and $\dot N^*_{out}(0 \times M_1 \times M_3) \times \bR_{>0}$ are also gapped.
    \end{enumerate}
    Applying Lemma \ref{lem: non-proper base change compose} completes the proof.
    \end{proof}

\begin{remark}
    As Lemma \ref{gapped composition basechange}
     is central to our argument, we wish to further clarify what its proof is doing by repeating the proof of Lemma \ref{lem: non-proper base change compose} (as opposed to just invoking its assertion) in this specific setting. Since $\ss(\SF_{12})$ and $\ss(\SF_{23})$ are $\bR_{>0}$-non-characteristic, and $\psi(\ss(\SF_{23})) \circ \psi(\ss(\SF_{12}))$ is pdff, the stalks at $x_1 \times x_3 \times 0$ can be computed by the sections on small open neighbourhoods by Lemma \ref{stalk}. On the left hand side, the stalk is
    $$\Gamma(U_\delta(x_1) \times \Delta_{M_2} \times U_\delta(x_3) \times (0, \epsilon), \SF_{12} \boxtimes_{\bR_{>0}} \SF_{23}).$$
    On the right hand side, considering tubular neighbourhoods $U_{\epsilon'}(\Delta_{M_2})$ of $\Delta_{M_2}$ with smooth boundary, the stalk is
    $$\Gamma(U_\delta(x_1) \times U_{\epsilon'}(\Delta_{M_2}) \times U_\delta(x_3) \times (0, \epsilon), \SF_{12} \boxtimes_{\bR_{>0}} \SF_{23}).$$
    Hence the result will follow if we can establish the necessary microsupport estimate to apply the noncharacteristic deformation lemma to the family $U_\delta(x_1) \times U_{\epsilon'}(\Delta_{M_2}) \times U_\delta(x_3) \times (0, \epsilon)$ for $\epsilon' \to 0$. This is exactly ensured by the singular support estimates of Lemma \ref{lem: non-proper base change compose}, which we can unpack again as follows.
    
    Since $\ss(\SF_{12})$ and $\ss(\SF_{23})$ are $\bR_{>0}$-non-characteristic, it suffices to estimate the intersection of the relative microsupport $\ss_\pi(\SF_{12}) \times \ss_\pi(\SF_{23})$ with $\dot{N}^*_{out}(U_\delta(x_1) \times U_{\epsilon'}(\Delta_{M_2}) \times U_\delta(x_3)) \times (0, \epsilon)$.
    
    Consider the intersection of the microsupport with the conormal bundle of $U_{\epsilon'}(\Delta_{M_2})$ (see the red boundary of the slices in Figure \ref{fig:gapped composition base change}). Since the pair $(\ss_\pi(\SF_{12}) \times \ss_\pi(\SF_{23}), 0_{M_1} \times \dot{N}^*\Delta_{M_2} \times 0_{M_3} \times (0, \epsilon))$ is positively gapped, we know that for any $\epsilon' > 0$ small enough,
    $$\ss_\pi(\SF_{12}) \times \ss_\pi(\SF_{23}) \cap R_\epsilon(0_{M_1} \times \dot{N}^*\Delta_{M_2} \times 0_{M_3}) \times (0, \epsilon) = \varnothing.$$

    Then consider the intersection of the microsupport with the sum of the outward conormal bundle of $U_\delta(x_1) \times U_\delta(x_3)$ and of $U_{\epsilon'}(\Delta_{M_2})$ at the corner (see the blue boundary and dots of the slices in Figure \ref{fig:gapped composition base change}):
    $$\ss_\pi(\SF_{12}) \times \ss_\pi(\SF_{23}) \cap (\dot{N}_{out}^*(U_\delta(x_1) \times U_\delta(x_3)) + \dot{N}_{out}^*U_{\epsilon'}(\Delta_{M_2})) \times (0, \epsilon).$$
    When $\epsilon$ and $\epsilon' \to 0$, by compactness, we know that the intersection converges to intersection points of the form
    \begin{align*}
    \Pi_{13}: &\, \psi(\ss_\pi(\SF_{12})) \times \psi(\ss_\pi(\SF_{23})) \cap (\dot{N}^*(U_\delta(x_1) \times U_\delta(x_3)) \times \dot{N}^*\Delta_{M_2}) \\
    &\, \hookrightarrow \Pi_{13}(\psi(\ss_\pi(\SF_{12})) \times \psi(\ss_\pi(\SF_{23})) \cap \dot{N}^*(U_\delta(x_1) \times U_\delta(x_3)) \times \dot{N}^*\Delta_{M_2}) \\
    &\, \hookrightarrow \psi(\ss(\SF_{23})) \circ \psi(\ss(\SF_{12})) \cap \dot{N}^*(U_\delta(x_1) \times U_\delta(x_3)) \subset \dot{T}^*(M_1 \times M_3).
    \end{align*}
    Since $\psi(\ss(\SF_{23})) \circ \psi(\ss(\SF_{12}))$ is pdff, we know that when $\delta > 0$ is sufficiently small, the intersection must be empty. Therefore, when $\delta > 0$ sufficiently small, we can find $\epsilon' > 0$ sufficiently small, so that the original intersection is also empty:
    $$\ss_\pi(\SF_{12}) \times \ss_\pi(\SF_{23}) \cap (\dot{N}_{out}^*(U_\delta(x_1) \times U_\delta(x_3)) + \dot{N}_{out}^*U_{\epsilon'}(\Delta_{M_2})) \times (0, \epsilon) = \varnothing.$$
    This finishes the discussion on the necessary singular support estimations, and our discussion of the proof of Lemma \ref{gapped composition basechange}.
\end{remark}

\begin{figure}
    \centering
    \includegraphics[width=0.75\linewidth]{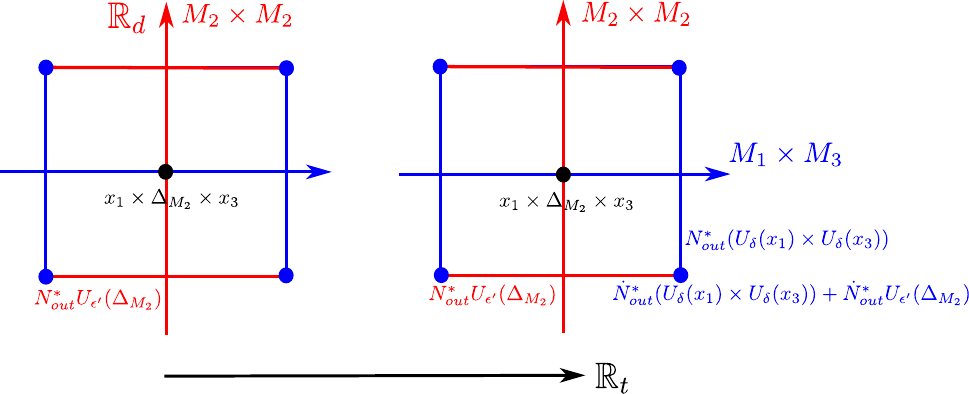}
    \caption{The geometric setting in Lemma \ref{gapped composition basechange}. The black horizontal direction is the $\bR_t$-direction, the red vertical direction on each slice is the $\bR_d$ direction, and the blue horiontal direction on each slice is the $M_1 \times M_3$ direction. The square on each slice is the open neighbourhood $U_\delta(x_1) \times U_\delta(x_3) \times U_{\epsilon'}(\Delta_{M_2})$.}
    \label{fig:gapped composition base change}
\end{figure}

\begin{theorem}\label{thm:gapped composition}
    Let $\SF_{12} \in \Sh(M_1 \times M_2 \times \bR_{>0})$ and $\SF_{23} \in \Sh(M_2 \times M_3 \times \bR_{>0})$ be sheaves with positively gapped composition for some Reeb flow, and suppose $\psi(\ss_\pi(\SF_{12})) \times \psi(\ss_\pi(\SF_{23}))$ is pdff. Then
    $$\psi_{13} (\SF_{23} \circ \SF_{12}) \simeq \psi_{23} \SF_{23} \circ \psi_{12} \SF_{12}.$$
\end{theorem}
\begin{proof}
    The isomorphism follows from the following compositions:
    \begin{align*}
     \psi_{13} (\SF_{23} \circ \SF_{12}) & \xrightarrow{\,\,\sim\,\,}  \pi_{13!}  (\psi_{123}  \Delta_{M_2/\bR_{>0}}^* (\SF_{12} \boxtimes_{\bR_{>0}} \SF_{23})) \\
    & \xrightarrow{ \ref{gapped composition basechange} } \pi_{13 !} \Delta_{M_2}^* (\psi_{1223} ( \SF_{12} \boxtimes_{\bR_{>0}} \SF_{23})) \\
    & \xrightarrow{ \,\ref{nearby cycle product} \,} \pi_{13!}\Delta_{M_2}^* (\psi_{12}\SF_{12} \boxtimes \psi_{23}\SF_{23}) \\
    & \xrightarrow{ \,\,\sim\,\,} \psi_{23} \SF_{23} \circ \psi_{12} \SF_{12}. \qedhere
    \end{align*}
\end{proof}

    More generally, consider the sheaf composition functor as a functor $\Delta_1^{\times k-1} \to \Cat_{st}$ (which is used to encode all possible orders of compositions). We have:

\begin{definition}\label{def: gap composition pairwise}
   Consider manifolds $M_1, M_2, \dots, M_k$. We say that the conic subsets $\Lambda_{12} \subset T^*(M_1 \times M_2) \times \bR_{>0}, \dots, \Lambda_{k-1,k} \subset T^*(M_{k-1} \times M_k) \times \bR_{>0}$ (with compositions $\Lambda_{ik} = \Lambda_{jk} \circ_{\bR_{>0}} \Lambda_{ij} \subset T^*(M_i \times M_k) \times \bR_{>0}$) have {\em (positively) gapped composition} if any pair
   $$(\Lambda_{ij} \times_{\bR_{>0}} \Lambda_{jk}, 0_{M_i} \times \dot{T}^*_\Delta{M_j^2} \times 0_{M_k} \times \bR_{>0}) $$
   is (positively) gapped for a Reeb flow on $\dot{T}^*(M_i \times M_j \times M_j \times M_k)$, and the composition of this pair
   $$\psi_{jk}\Lambda_{jk} \circ \psi_{ij}\Lambda_{ij} \subset \dot T^*(M_i \times M_k)$$
   is a pdff subset. For sheaves $\SF_{12} \in \Sh(M_1 \times M_2 \times \bR_{>0}), \dots, \SF_{k-1,k} \in \Sh(M_{k-1} \times M_k \times \bR_{>0})$, we say that they have {\em (positively) gapped composition} if the singular supports are $\bR_{>0}$-non-characteristic, and $\SS_\pi(\SF_{12}), \dots, \SS_\pi(\SF_{k-1,k})$ have positively gapped composition.
\end{definition}

\begin{corollary}\label{rem: gapped composition higher}
    Let $\SF_{12} \in \Sh(M_1 \times M_2 \times \bR_{>0}), \dots, \SF_{k-1,k} \in \Sh(M_{k-1} \times M_k \times \bR_{>0})$ be sheaves with positively gapped composition for some Reeb flow, and suppose $\psi(\SS_\pi(\SF_{12})) \times \dots \times \psi(\SS_\pi(\SF_{k-1,k}))$ is pdff. Then for any $1 < l_1 < \dots < l_r < k$, we have natural isomorphisms 
    $$\psi_{12\cdots k}(\SF_{k-1,k} \circ \dots \circ \SF_{12}) \simeq \psi_{l_r\cdots k}(\SF_{k-1,k} \circ \dots \circ \SF_{l_r,l_r+1}) \circ \dots \circ \psi_{12\cdots l_1}(\SF_{l_1-1,l_1} \circ \dots \circ \SF_{12})$$
    that fit into a natural transformation, from the diagram  $\Delta_1^{\times k-1} \to \Cat_{st}$ satisfying the Segal condition that sends the vertices to $\Sh(M_1 \times M_{l_1} \times \bR_{>0}) \otimes \dots \otimes \Sh(M_{l_r} \times M_k \times \bR_{>0})$ and the edges to the composition functors, to the diagram  $\Delta_1^{\times k-1} \to \Cat_{st}$ satisfying the Segal condition that sends the vertices to $\Sh(M_1 \times M_{l_1}) \otimes \dots \otimes \Sh(M_{l_r} \times M_k)$ and the edges to the composition functors, such that all the 2-morphisms are equivalences.
\end{corollary}
\begin{proof}
    We have a diagram of (non-proper or smooth) base changes induced by adjunctions as in Figure \ref{fig:nearby exterior tensor} for iterated exterior tensor products and Figure \ref{fig:gap sheaf composition diagram} for the gapped composition base change formula, which forms a diagram of (op)lax natural transformations by \cite[Corollary F]{HaugsengHebestreitLinskensNuiten}. Then, by iteratively applying Lemma \ref{gapped composition basechange}, we can show that the natural transformations are equivalences and hence we have a commutative diagram. 
\end{proof}

\begin{figure}
    \centering
    \includegraphics[width=1\linewidth]{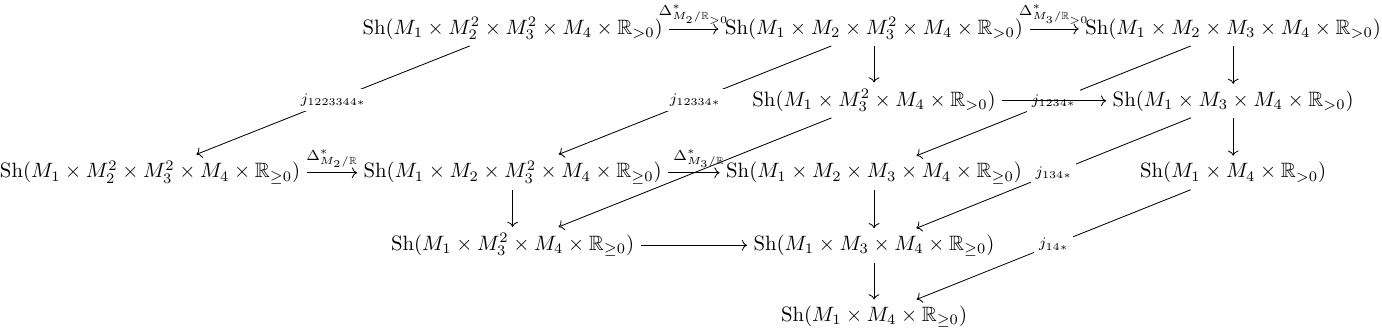}
    \caption{Part of the diagram of functors that appear in the iterated gapped composition of sheaves, where the natural transformation in each square is induced by the adjunctions of six-functors.}
    \label{fig:gap sheaf composition diagram}
\end{figure}

\begin{remark}\label{rem: gap composition higher alternative}
    Consider the set of divisions of the $k$-simplex $\mathrm{Div}_k$ with the partial order by refinements (which is used to encode all possible orders of compositions). We note that the natural transformation between the two coherent diagrams $\Delta_1^{\times k-1} \to \Cat_{st}$ is also the same as the coherent diagram $\mathrm{Div}_k \to \Fun^{ex}(\bigotimes_{j=1}^{k-1}\Sh(M_j \times M_{j+1} \times \bR_{>0}), \Sh(M_1 \times M_k))$ where a division $1 < l_1 < \dots < l_r < k$ is sent to the functor on the right hand side of the equation.
\end{remark}

\subsection{Gapped Hom composition}
We recall the limit-preserving composition of integral kernels  $\SF_{12}\in \Sh(M_1 \times M_2)$ and $\SF_{23} \in \Sh(M_2 \times M_3)$ \cite[Section 3.6]{KS}: 
    $$\sHom^\circ(\SF_{12}, \SF_{23}) = \pi_{13*}\sHom(\pi_{12}^*\SF_{12}, \pi_{23}^!\SF_{23}).$$
    
\begin{definition}\label{def: gap hom composition}
   Consider manifolds $M_1, M_2, M_3$. We say that conic subsets $\Lambda_{12} \subset T^*(M_1 \times M_2 \times \bR_{>0})$ and $\Lambda_{23} \subset T^*(M_2 \times M_3 \times \bR_{>0})$ have {\em (positively) gapped hom composition} if 
   \begin{enumerate}
       \item $\Lambda_{12} \subset T^*(M_1 \times M_2 \times \bR_{>0})$, $\Lambda_{23} \subset T^*(M_2 \times M_3 \times \bR_{>0})$ are $\bR_{>0}$-non-characteristic,
       \item the composition of the pair $\psi_{23}(\Lambda_{23}) \circ \psi_{12}(-\Lambda_{12}) \subset \dot{T}^*(M_1 \times M_3)$ is pdff, and
       \item the pair
       $$(-\Lambda_{12} \times_{\bR_{>0}} \Lambda_{23}, 0_{M_1} \times \dot{T}^*_\Delta{M_2^2} \times 0_{M_3} \times \bR_{>0}) $$
       is (positively) gapped for a Reeb flow on $\dot{T}^*(M_1 \times M_2 \times M_2 \times M_3)$.
   \end{enumerate}
\end{definition}

    Similar to the $\otimes$ version, the main technical lemma we need is a family version of \cite[Lemma 2.24 \& Proposition 2.25]{Nadler-Shende}.

    Here is the model case. 
    Consider the maps
    \[\begin{tikzcd}
        0 \times 0 \times M \ar[d, "{y}_1" left] \ar[r, "y_2"] & \bR_{\geq 0} \times 0 \times M \ar[d, "{i}_1"] \\
        0 \times \bR_{\geq 0} \times M \ar[r, "{i}_2"] & \bR_{\geq 0} \times \bR_{\geq 0} \times M
    \end{tikzcd}\]
    Using non-characteristic deformation Lemma \ref{lem: non-char}, we can show the following lemma (this should also be related to the hyperbolic localizations \cite{Braden}):
    
\begin{lemma}\label{lem: non-proper base change hom}
    For a sheaf $\SF \in \Sh(\bR_{\geq 0}^2 \times M)$ such that 
    \begin{enumerate}
        \item $\SS(\SF)$ is non-characteristic with respect to $\pi_1: \bR^2_{\geq 0} \times M \to \bR_{\geq 0} \times 0 \times M$, 
        \item $\Pi_M(\SS(\SF)) \cap T^*M$ is pdff for $\Pi_M: T^*(\bR_{\geq 0} \times M) \times \bR_{\geq 0} \to T^*M \times \bR_{\geq 0}^2$, and
        \item $\SS_{\pi_1}(\SF)$ is gapped from $\dot N^*_{out}(0 \times M) \times \bR_{>0}$ as an $\bR_{\geq 0}$-family of subsets,
    \end{enumerate}
    the natural transformation defines an isomorphism
    $$y_{1}^*{i}_{2}^!\SF \xrightarrow{\sim} y_{2}^! {i}_1^*\SF.$$
\end{lemma}

    Instead of proving the above lemma, we will prove the following one which is the precise situation we are in:

\begin{lemma}\label{gapped hom composition basechange}
    Let $\SF_{12} \in \Sh(M_1 \times M_2 \times \bR_{>0})$ and $\SF_{23} \in \Sh(M_2 \times M_3 \times \bR_{>0})$ have gapped composition for some Reeb flow. Then
    $$\pi_{13 *}(\psi_{123}\, \Delta_{M_2/\bR_{>0}}^! \sHom^{\boxtimes/\bR_{>0}}(\SF_{12}, \SF_{23})) \simeq \pi_{13*}\Delta_{M_2}^!(\psi_{1223}\sHom^{\boxtimes/\bR_{>0}}(\SF_{12}, \SF_{23})).$$
\end{lemma}
\begin{proof} 
    Using the same notation as in Lemma \ref{gapped composition basechange}, it suffices to show that
    $$i^*(\pi_{13} \circ j)_*\Delta_{M_2/\bR_{>0}}^!\sHom^{\boxtimes/\bR_{>0}}(\SF_{12}, \SF_{23}) \simeq i_0^!(d_{13} \circ j)_*\sHom^{\boxtimes/\bR_{>0}}(\SF_{12}, \SF_{23}).$$
    On the left hand side, one can show that the stalk is
    $$\Hom(1_{U_\delta(x_1) \times \Delta_{M_2} \times U_\delta(x_3) \times (0,\epsilon)}, \sHom^{\boxtimes/\bR_{>0}}(\SF_{12}, \SF_{23})).$$
    On the right hand side, the stalk is
    $$\Hom(1_{U_\delta(x_1) \times \overline{U}_{\epsilon'}(\Delta_{M_2}) \times U_\delta(x_3) \times (0,\epsilon)}, \sHom^{\boxtimes/\bR_{>0}}(\SF_{12}, \SF_{23})).$$
    The result will then follow if we can apply noncharacteristic deformation (Lemma \ref{lem: non-char}) to the family $U_\delta(x_1) \times \overline{U}_{\epsilon'}(\Delta_{M_2}) \times U_\delta(x_3) \times (0,\epsilon)$ as $\epsilon' \to 0$.  To do so we must show that for some sufficiently small $\epsilon > 0, \delta> 0$, as $\epsilon' \to 0$, the outward conormals $N^*_{out}(U_\delta(x_1) \times \overline{U}_{\epsilon'}(\Delta_{M_2}) \times U_\delta(x_3) \times (0,\epsilon))$ are disjoint from $-\SS(\SF_{12}) \times \SS(\SF_{23})$. The rest of the argument is singular support estimations and is the same as Lemma \ref{gapped composition basechange}.
\end{proof}

\begin{theorem}\label{thm:gapped hom composition}
    Let $\SF_{12} \in \Sh(M_1 \times M_2 \times \bR_{>0})$ and $\SF_{23} \in \Sh(M_2 \times M_3 \times \bR_{>0})$ be sheaves with poritive gapped hom composition for some Reeb flow, and suppose $\psi(\dot{\ss}_\pi(\SF_{12})) \times \psi(\dot{\ss}_\pi(\SF_{23}))$ is pdff. Then
    $$\psi_{13} \sHom^\circ(\SF_{12}, \SF_{23}) \simeq \sHom^\circ(\psi_{12} \SF_{12}, \psi_{23} \SF_{23}).$$
\end{theorem}
\begin{proof}
    The isomorphism follows from the following compositions:
    \begin{align*}
     \psi_{13} \sHom^\circ(\SF_{12}, \SF_{23}) & \xrightarrow{\,\,\sim\,\,}  \pi_{13*}  (\psi_{123}  \Delta_{M_2/\bR_{>0}}^! \sHom^{\boxtimes/\bR_{>0}}(\SF_{12}, \SF_{23})) \\
    & \xrightarrow{ \ref{gapped hom composition basechange} } \pi_{13 *} \Delta_{M_2}^! (\psi_{1223} \sHom^{\boxtimes/\bR_{>0}}( \SF_{12}, \SF_{23})) \\
    & \xrightarrow{ \,\ref{nearby cycle hom product} \,} \pi_{13*}\Delta_{M_2}^! \sHom^{\boxtimes}(\psi_{12}\SF_{12}, \psi_{23}\SF_{23}) \\
    & \xrightarrow{ \,\,\sim\,\,} \sHom^\circ(\psi_{12} \SF_{12}, \psi_{23} \SF_{23}). \qedhere
    \end{align*}
\end{proof}

\subsection{Examples of gapped compositions}
    One class of examples that we are interested in is the case when the manifold is $N \times \bR_t$ and the microsupport of the sheaf $\SS(\SF)$ is a Legendrian contained in $S^*_{\tau\geq 0}(N \times \bR_t)$. We will give some examples of compositions of sheaves with microsupports on such Legendrians and illustrate when the gappedness condition holds. Later in Sections \ref{sec:quantization} and \ref{ssec:main theorems}, we will explain how this is related to our result of exact Lagrangian correspondences in cotangent bundles.

\begin{example}\label{ex: cotangent bundle composition}
    Let $M_i = N_i \times \bR_t$ for $1 \leq i \leq 3$ and consider the diffeotopy $\varphi(x_i, t; s) = \varphi_s(x_i, t) = (x_i, st)$ for $s\in \bR_{>0}$. Consider $\SF_{ij} \in \Sh_{\tau>0}(N_i \times N_j \times \bR_t \times \bR_{>0})$ be the pull-back of some sheaves in $\Sh_{\tau>0}(N_i \times N_j \times \bR_t)$ via the diffeotopy $\varphi: N_i \times N_j \times \bR_t \times \bR_{>0} \to N_i \times N_j \times \bR_t$. Write $s_{ij}: N_i \times \bR \times N_j \times \bR \to N_i \times N_j \times \bR, (x_i, t_i, x_j, t_j) \mapsto (x_i, x_j, t_i - t_j)$. Then consider the sheaves
    $$\ol{\SF}_{ij} = s_{ij}^*\SF_{ij}.$$
    We can show that $(\ss_\pi(\ol{\SF}_{12}) \times \ss_\pi(\ol{\SF}_{23}), 0_{N_1 \times \bR_{t_1}} \times \Delta_{N \times \bR_{t_2}} \times 0_{N_3 \times \bR_{t_3}})$ is always gapped. 
    Thus, we always have
    $$\psi_{13}(\ol{\SF}_{12} \circ \ol{\SF}_{23}) \simeq \psi_{12}\ol{\SF}_{12} \circ \psi_{23}\ol{\SF}_{23}.$$
    This is a basic example where the composition of sheaves singularly supported on the Legendrians $L_{ij} \subset S^*(N_i \times N_j \times \bR)$ is compatible with nearby cycle functors. Later we will see that it is a special case of Theorem \ref{thm: main composition}.
\end{example}

\begin{example}\label{ex: cotangent bundle hom recover}
    Consider $M_1 = M_3 = \bR$ and $M_2 = N \times \bR_t$, and the diffeotopy $\varphi(x, t; s) = \varphi_s(x, t) = (x, st)$ for $s\in \bR_{>0}$. Consider $\SF, \SG \in \Sh_{\tau>0}(N \times \bR_t \times \bR_{>0})$ to be the pull-back of some sheaves in $\Sh_{\tau>0}(N \times \bR_t)$ via the diffeotopy $\varphi$. Write $s_{12}: M_1 \times M_2 \to M_2, (t_1, x, t_2) \mapsto (x, t_1 - t_2)$. Then consider the sheaves
    $$\ol{\SF} = s_{23}^*\SF, \; \ol{\SG} = s_{12}^*\SG.$$
    By Example \ref{ex: cotangent bundle composition}, we know that $(\ss_\pi(\ol{\SF}) \times \ss_\pi(\ol{\SG}), 0_{\bR_{t_1}} \times \Delta_{N \times \bR_{t_2}} \times 0_{\bR_{t_3}})$ is always gapped. Thus, we always have
    $$\psi_{13}(\ol{\SF} \circ_{N \times \bR} \ol{\SG}) \simeq \psi_{23}\ol{\SF} \circ_{N \times \bR} \psi_{12}\ol{\SG}.$$
    The resulting sheaf is the pull-back of a sheaf on $\bR$ via the map $s_{13}: \bR^2 \to \bR$. This shows that in the setting of Example \ref{ex: cotangent bundle composition}, when $N_1 = N_3 = pt$, the composition of sheaves in $N \times \bR$ with microsupports in Legendrians $L, K \subset S^*(N \times \bR)$ yields a sheaf on $\bR$ or a $\bR$-filtered object in $\SC$. Later, we will see that this is a special case of Theorem \ref{cor: recover tensor}.
    
    We can also further take composition $\ol{\SF} \circ_{N \times \bR^2} \ol{\SG}$ over $N \times \bR \times \bR$, by pulling back $\ol{\SF} \circ_{N \times \bR} \ol{\SG}$ to the diagonal and then pushing forward to a point. Reeb chords between $\ss(\ol{\SF} \circ_{N \times \bR} \ol{\SG})$ and $\dot{T}^*_{\Delta}\bR^2$ correspond bijectively to Reeb chords between $\ss_\pi({\SF}) \times \ss_\pi({\SG})$ and $\dot{T}^*_\Delta (N \times \bR)^2$. When there are no Reeb chords from $-\ss_\pi(\SF)$ to $\ss_\pi(\SG)$, by Proposition \ref{gapped composition basechange} we have
    $$\psi_1\Delta_\bR^*(\ol{\SF} \circ_{N \times \bR} \ol{\SG})) \simeq \Delta_\bR^*\psi_{13}(\ol{\SF} \circ_{N \times \bR} \ol{\SG})).$$
    can show that when there are no Reeb chords from $-\ss_\pi(\SF)$ to $\ss_\pi(\SG)$,
    \begin{align*}
    \ol{\SF} \circ_{N \times \bR^2} \ol{\SG} &\simeq \pi_{\bR!}\psi_1\Delta_\bR^*(\ol{\SF} \circ_{N \times \bR} \ol{\SG}))) \simeq \pi_{\bR!}\Delta_\bR^*\psi_{13}(\ol{\SF} \circ_{N \times \bR} \ol{\SG}) \\
    &\simeq \pi_{\bR!}\Delta_\bR^*(\psi_{23}\ol{\SF} \circ_{N \times \bR} \psi_{12}\ol{\SG}) \simeq \psi_{23}\ol{\SF} \circ_{N \times \bR^2} \psi_{12}\ol{\SG}.
    \end{align*}
    This shows that in the setting of Example \ref{ex: cotangent bundle composition}, when $N_1 = N_3 = pt$, from the composition of sheaves with microsupports in Legendrians $L, K \subset S^*(N \times \bR)$ which yields a sheaf on $\bR$, we can further take composition and nearby cycle to obtain an actual object in $\SC$. Later, we will see that this is a special case of Theorem \ref{prop: recover tensor full faithful}.
\end{example}

\begin{example}
\label{rem: necessity not restrict to fiber 1}
    One could have tried to prove Theorems \ref{thm:gapped composition} and \ref{thm:gapped hom composition} by restricting to each fiber $i_{x_1,x_3}: x_1 \times M_2^2 \times x_3 \hookrightarrow M_1 \times M_2^2 \times M_3$ using Proposition \ref{nearby commute with restrict}, and then applying Theorems \ref{fullfaithful nearby} and \ref{nearby fully faithful} to each fiber.  This fails: the gapped condition is not preserved by passing to the fiber. 
    
    For example, let $M_1 = M_2 = M_3 = \bR$ and $M_2 = pt$. 
      Let us write  $i_{x_1}: x_1 \times M_2 \hookrightarrow M_1 \times M_2$ and $i_{x_3}: M_2 \times x_3 \hookrightarrow M_2 \times M_3$.  
    Consider 
    $$\SF_{12} = 1_{x_1 = x_2} \in \Sh(\bR^2 \times \bR_{>0}), \quad \SF_{23} = 1_{x_2=x_3+s} \in \Sh(\bR^2 \times \bR_{>0}).$$
    Observe that
    $$i_0^*\SF_{12} = 1_{x_2 = 0} \in \Sh(\bR \times \bR_{>0}), \quad i_0^*\SF_{23} = 1_{x_2 = s} \in \Sh(\bR \times \bR_{>0}).$$
    Now $(\ss_\pi(i_0^*\SF_{12}) \times \ss_\pi(i_0^*\SF_{23}), \dot{T}^*_\Delta\bR^2 )$ are not gapped, 
    and hence Theorem \ref{fullfaithful nearby} does not apply. 
    Nevertheless, $(\ss_\pi(\SF_{12}) \times \ss_\pi(\SF_{23}), 0_\bR \times \dot{T}^*_\Delta\bR^2 \times 0_\bR)$ are gapped -- there are no Reeb chords between the pair for any $s > 0$ -- hence Theorem \ref{thm:gapped composition} applies.  By said theorem, or a direct computation, one has
    $$\psi\SF_{23} \circ \psi\SF_{12} = 1_{x_2=x_3} \circ 1_{x_1=x_2} = 1_{x_1=x_3} = \psi(\SF_{23} \circ \SF_{12}).$$
\end{example}

%
%
%
%
%

%% file: gappedmicrosheafcomposition.tex

\section{Gapped composition of microsheaves}\label{sec: gapped composition microsheaf}

    We will define the composition of microsheaves with prescribed sufficiently Legendrian supports, using doubling and the gapped composition theorem from the previous Section \ref{sec: gapped composition}. 

    Moreover, we discuss the relationship between gapped compositions of microsheaves and integral transforms of microsheaves in Section \ref{ssec: microsheaf duality}, and the relation between gapped compositions and stop removal functors in Section \ref{ssec: doubling}.

\subsection{Contact composition}

    We recall the definition of contact compositions of subsets. In the whole section, we assume that all the contact manifolds are co-oriented.
    
    First, we describe contact reductions along coisotropic submanifolds. We recall that for a contact manifold $(X, \xi)$, a submanifold $C \subset X$ is said to be {\it coisotropic} if $TC \cap \xi|_C \subset \xi|_C$ is a coisotropic distribution, and {\it regular coisotropic} if, in addition, $TC \cap \xi|_C$ has constant rank.

\begin{definition}
    Let $X$ be a contact manifold and $C \subset X$ a regular coisotropic submanifold of dimension at least $n + 1$. Then $X \times \bR_{>0}$ is a homogeneous symplectic manifold and $C \times \bR_{>0}$ is a homogeneous coisotropic submanifold. Let $\mathcal{F}_{C \times \bR_{>0}}$ be the characteristic foliation on $C \times \bR_{>0}$. The contact reduction of $X$ with respect to $Y$ is defined by
    $$X_{red} \times \bR_{>0} = (C \times \bR_{>0}) / \cF_{C \times \bR_{>0}},$$
    where there is a natural projection map $\pi_{red}: C \rightarrow X_{red}$.
\end{definition}

\begin{example}
    For a contact manifold $X_i$, let $X^-_i$ be the contact manifold with the same distribution but opposite co-orientation. Consider $X_1^- \,\widehat\times\, X_2 \,\widehat\times\, X_2^- \,\widehat\times\, X_3$ and the regular coisotropic
    $$X_1^- \,\widehat\times\, \Delta_{X_2} \,\widehat\times\, X_3 = \{[x_1, r_1; x_2, r_2; x_2, r_2; x_3, r_3] \mid (x_i, r_i) \in X_i \times \bR_{>0} \}.$$
    Then the corresponding contact reduction is the contact product $X_1 \,\widehat\times\, X_3$. When $X_i = S^*M_i = \dot{T}^*M_i / \bR_{>0}$, the regular coisotropic submanifold is the $\bR_{>0}$-quotient of
    $$\dot{T}^*X_1 \times \dot{T}^*_\Delta({M_2} \times M_2) \times \dot{T}^*X_3 \subset \dot{T}^*(M_1 \times M_2 \times M_2 \times M_3),$$
    and the contact reduction is the product $S^*M_1 \,\widehat\times\, S^*M_3 \subset S^*(M_1 \times M_3)$.
\end{example}

    We are interested in contact compositions of Legendrian subsets and contact reductions of Legendrian subsets.

\begin{definition}
    Let $X$ be a contact manifold, $C$ be a regular coisotropic submanifold and $\Lambda \subset X$ be a subset. Then for the contact reduction $\pi_{red}: C \to X_{red}$, the reduction of $\Lambda$ with respect to $C$ is
    $$\Lambda_{red} = \pi_{red}(\Lambda \cap C).$$
\end{definition}

\begin{definition}\label{def: contact composition}
    Let $\Lambda_{12} \subset X_1^- \,\widehat\times\, X_2, \Lambda_{23} \subset X_2^- \,\widehat\times\, X_3$ be two subsets, their composition is the reduction of $\Lambda_{12} \,\widehat\times\, \Lambda_{23}$ with respect to the regular coisotropic $X_1^- \,\widehat\times\, \Delta_{M_2} \,\widehat\times\, X_3$, where $\pi_{13}: X_1^- \,\widehat\times\, X_2 \,\widehat\times\, X_2^- \,\widehat\times\, X_3 \to X_1^- \,\widehat\times\, X_3$ is the projection map that defines the reduction:
    \begin{align}
    \Lambda_{23} \circ \Lambda_{12} := \pi_{13}(\Lambda_{12} \,\widehat\times\, \Lambda_{23} \cap X_1^- \,\widehat\times\, \Delta_{M_2} \,\widehat\times\, X_3).
    \end{align}
\end{definition}

    We relate compositions of conic subsets in cotangent bundles to compositions of subsets in the contact products of cosphere bundles.

    Let $M_1, M_2$ and $M_3$ be smooth manifolds. Consider subsets $\Lambda_{12} \subset S^*M_1 \,\widehat\times\, S^*M_2$ and $\Lambda_{23} \subset S^*M_2 \,\widehat\times\, S^*M_3$, and conic subsets
    $$\underline\Lambda_{12} = 0_{12} \cup \Lambda_{12} \times \bR_{>0} \subset T^*(M_1 \times M_2), \; \underline\Lambda_{23} = 0_{23} \cup \Lambda_{23} \times \bR_{>0} \subset T^*(M_2 \times M_3),$$
    where $0_{12} \subset 0_{M_1 \times M_2}$ and $0_{23} \subset 0_{M_2 \times M_3}$ are some prescribed subsets.

\begin{lemma}\label{ss of composition for doubling 1}
    Let $\Lambda_{12} \subset S^*M_1 \,\widehat\times\, S^*M_2$ and $\Lambda_{23} \subset S^*M_2 \,\widehat\times\, S^*M_3$ be any subsets. Then
    \begin{align*}
    (\underline\Lambda_{12} \times \underline\Lambda_{23}) \cap (T^*{M_1} \times \dot{T}^*_\Delta{M_2^2} \times T^*{M_3}) = (\Lambda_{12} \,\widehat\times\, \Lambda_{23} \cap S^*M_1 \,\widehat\times\, S^*_\Delta M_2^2 \,\widehat\times\, S^*M_3) \times \bR_{>0}.
    \end{align*}
\end{lemma}
\begin{proof}
    Since $\underline\Lambda_{12} \subset 0_{12} \cup (\dot{T}^*M_1 \times \dot{T}^*{M_2})$ and $\underline\Lambda_{23} \subset 0_{23} \cup (\dot{T}^*M_2 \times \dot{T}^*{M_3})$, we have the identity
    $$(\underline\Lambda_{12} \times \underline\Lambda_{23}) \cap (T^*{M_1} \times \dot{T}^*_\Delta{M_2^2} \times T^*{M_3}) = (\underline\Lambda_{12} \times \underline\Lambda_{23}) \cap (\dot{T}^*{M_1} \times \dot{T}^*_\Delta{M_2^2} \times \dot{T}^*{M_3}).$$
    Using Definition \ref{def: contact product leg} of coisotropics, we can write 
    $$\dot{T}^*{M_1} \times \dot{T}^*_\Delta{M_2^2} \times \dot{T}^*{M_3} = (S^*M_1 \,\widehat\times\, S^{*}_\Delta{M_2^2} \,\widehat\times\, S^*M_3) \times \bR_{>0}.$$
    Then, using Definition \ref{def: contact product leg} of contact products of Legendrians, we can also write
    \begin{align*}
    (\underline\Lambda_{12} \times \underline\Lambda_{23}) \cap (\dot{T}^*{M_1} \times \dot{T}^*_\Delta{M_2^2} \times \dot{T}^*{M_3}) = (\Lambda_{12} \,\widehat\times\, \Lambda_{23} \cap S^*M_1 \,\widehat\times\, S^*_\Delta M_2^2 \,\widehat\times\, S^*M_3) \times \bR_{>0}.
    \end{align*}
    This completes the proof.
\end{proof}

We also try to understand the relations between the compositions of the conic Lagrangian subsets and compositions of the corresponding Legendrian subsets. 

\begin{lemma}\label{ss of composition for doubling 2}
    Let $\underline\Lambda_{12} \subset T^*M_1 \times T^*M_2$, $\underline\Lambda_{23} \subset T^*M_1 \times T^*M_2$ be conic subsets and $\Lambda_{12} \subset S^*M_1 \,\widehat\times\, S^*M_2$, $\Lambda_{23} \subset S^*M_2 \,\widehat\times\, S^*M_3$ be the corresponding subsets at infinity. Then
    \begin{align*}
    &(\underline\Lambda_{12} \times \underline\Lambda_{23}) \cap (T^*{M_1} \times {T}^*_\Delta{M_2^2} \times T^*{M_3}) \\
    &= ((\Lambda_{12} \,\widehat\times\, \Lambda_{23} \cap S^*M_1 \,\widehat\times\, S^*_\Delta M_2^2 \,\widehat\times\, S^*M_3) \times \bR_{>0}) \cup (0_{12} \times 0_{23}).
    \end{align*}
    In particular, $\underline\Lambda_{12} \circ \underline\Lambda_{23} = (\Lambda_{12} \circ \Lambda_{23} \times \bR_{>0}) \cup (0_{12} \circ 0_{23})$.
\end{lemma}
\begin{proof}
    Since $\underline\Lambda_{12} \subset 0_{12} \cup (\dot{T}^*M_1 \times \dot{T}^*{M_2})$ and $\underline\Lambda_{23} \subset 0_{23} \cup (\dot{T}^*M_2 \times \dot{T}^*{M_3})$, we have the identity
    \begin{align*}
        &(\underline\Lambda_{12} \times \underline\Lambda_{23}) \cap (T^*M_1 \times T^*_\Delta M_2^2 \times T^*M_3) \cap (\dot{T}^*{(M_1 \times M_2^2 \times M_3)}) \\
        &=(\underline\Lambda_{12} \times \underline\Lambda_{23}) \cap (T^*M_1 \times T^*_\Delta M_2^2 \times T^*M_3) \cap (\dot{T}^*{(M_1 \times M_2)} \times \dot{T}^*(M_2 \times M_3)) \\
        &= (\underline\Lambda_{12} \times \underline\Lambda_{23}) \cap (\dot{T}^*M_1 \times \dot{T}^*_\Delta M_2^2 \times \dot{T}^*M_3).
    \end{align*}
    Then we complete the proof using Lemma \ref{ss of composition for doubling 1}. The statement about Lagrangian composition follows by  taking the projection to $\dot T^*M_1 \times \dot T^*M_3$ and the quotient by $\bR_{>0}$. 
\end{proof}

Finally, we will need to show that products of conic subsets are positively gapped with respect to the conormal of the diagonal $0_{M_1} \times \dot{S}^*_\Delta{M_2^2} \times 0_{M_3}$.

\begin{lemma}\label{ss gapped for doubling}
    Let $\Lambda_{12} \subset S^*M_1 \,\widehat\times\, S^*M_2$ and $\Lambda_{23} \subset S^*M_2 \,\widehat\times\, S^*M_3$ be any subsets. Let $R_t$ be the Reeb flow on $S^*(M_1 \times M_2 \times M_2 \times M_3)$. Then for any $t$ sufficiently small,
    $$(\underline\Lambda_{12} \times \underline\Lambda_{23}) \cap R_t(0_{M_1} \times \dot{T}^*_\Delta{M_2^2} \times 0_{M_3}) = \varnothing.$$
\end{lemma}
\begin{proof}
    Since $\underline\Lambda_{12} \subset 0_{12} \cup (\dot{T}^*M_1 \times \dot{T}^*{M_2})$ and $\underline\Lambda_{23} \subset 0_{23} \cup (\dot{T}^*M_2 \times \dot{T}^*{M_3})$, it follows that the intersection is always empty as 
    $$(\underline\Lambda_{12} \times \underline\Lambda_{23}) \cap R_t(0_{M_1} \times \dot{T}^*_\Delta{M_2^2} \times 0_{M_3}) = \varnothing.$$
    This then completes the proof.
\end{proof}

\subsection{Definition and associativity}
    We define microlocal gapped composition in this section, upon the following additional assumption on the pair of sufficiently Legendrian subsets:

\begin{definition}\label{def: composable}
    Let $\Lambda_{12} \subset S^*(M_1 \times M_2)$ and $\Lambda_{23} \subset S^*(M_2 \times M_3)$ be relatively compact and sufficiently Legendrian. We say that $\Lambda_{12}$ and $\Lambda_{23}$ are {\em composable} if their contact composition $\Lambda_{13}:=\Lambda_{23} \circ \Lambda_{12}$ is also sufficiently Legendrian, and, moreover, in terms of the positive isotopy that displaces $\Lambda_{12}$ and $\Lambda_{23}$, we have that $\Lambda_{23,\epsilon'} \circ \Lambda_{12,\epsilon}$ is disjoint from $\Lambda_{23} \circ \Lambda_{12}$ for sufficiently small $\epsilon \geq 0, \epsilon' \geq 0$ such that $\epsilon + \epsilon' > 0$.
\end{definition}

\begin{lemma}\label{lem: composable openness}
    Let $\Lambda_{12} \subset S^*(M_1 \times M_2)$ and $\Lambda_{23} \subset S^*(M_2 \times M_3)$ be relatively compact sufficiently Legendrian subsets that are composable in the sense of Definition \ref{def: composable}. Then for the relative doublings $(\underline\Lambda_{23})_{\cup,\epsilon}^+$ and $(\underline\Lambda_{12})_{\cup,\epsilon}^+$,
    $$\Lambda_{23} \circ \Lambda_{12} \subset (\underline\Lambda_{23})_{\cup,\epsilon}^+ \circ (\underline\Lambda_{12})_{\cup,\epsilon}^+$$
    as an open subset. 
\end{lemma}
\begin{proof}
    We know that for the relative doublings of $\Lambda_{12}$ and $\Lambda_{23}$, their composition is $$(\underline\Lambda_{23})_{\cup,\epsilon}^+ \circ (\underline\Lambda_{12})_{\cup,\epsilon}^+ \setminus 0_{M_1 \times M_3} = (\Lambda_{23} \circ \Lambda_{12}) \cup (\Lambda_{23,\epsilon} \circ \Lambda_{12}) \cup (\Lambda_{23} \circ \Lambda_{12,\epsilon}) \cup (\Lambda_{23,\epsilon} \circ \Lambda_{12,\epsilon}).$$
    Since $\Lambda_{12}$ and $\Lambda_{23}$ are composable and relatively compact subsets, we know that $\Lambda_{23} \circ \Lambda_{12}$ is a relative open subset in $(\underline\Lambda_{23})_{\cup,\epsilon}^+ \circ (\underline\Lambda_{12})_{\cup,\epsilon}^+$.
\end{proof}

    The condition that $\Lambda_{23,\epsilon'} \circ \Lambda_{12,\epsilon}$ is disjoint from $\Lambda_{23} \circ \Lambda_{12}$ can easily fail: 

\begin{example}
    Let $M_1 = M_3 = pt$ and $M_2 = \bR^2$. Then for $\Lambda, \Lambda' \subset S^*\bR^2$, the composition $\Lambda \circ \Lambda'$ is given by the intersection points $\Lambda \cap \Lambda'$. Let $f: \bR \to \bR$ be a smooth function such that $f^{-1}(0) = 0$ and the critical values have an accumulation point at $0$. Consider 
    $$\Lambda = \{(x, 0; 0, 1) \mid x \in \bR\}, \quad \Lambda' = \{(x, f(x); f'(x), 1) \mid x \in \bR\}.$$
    Then $\Lambda \circ \Lambda'$ is a single point. However, for any sufficiently small $\epsilon$ and $\epsilon' > 0$, the Reeb push-off $\Lambda_\epsilon \circ \Lambda'_{\epsilon'}$ is non-empty and hence cannot be disjoint from $\Lambda \circ \Lambda'$.
\end{example}

    Let $\Lambda_{12} \subset S^*(M_1 \times M_2)$ and $\Lambda_{23} \subset S^*(M_2 \times M_3)$ be relatively compact sufficiently Legendrian subsets. Consider the doubling functor in Definition \ref{def: positive doubling}
    $$w_{\Lambda_{ij}}^+: \msh_{\Lambda_{12}}(\Lambda_{12}) \hookrightarrow \Sh_{(\underline\Lambda_{12})_{\cup,\epsilon}^+}(M_i \times M_j).$$
    Then for $\SS(w_{\Lambda_{12}}^+\SF_{12}) \subset (\underline\Lambda_{12})_{\cup,\epsilon}^+$ and $\SS(w_{\Lambda_{23}}^+\SF_{23}) \subset (\underline\Lambda_{23})_{\cup,\epsilon}^+$, by Lemma \ref{lem: ss-composition}, we know that
    $$\SS(w_{\Lambda_{23}}^+\SF_{23} \circ w_{\Lambda_{12}}^+\SF_{12}) \subset (\underline\Lambda_{23})_{\cup,\epsilon}^+ \circ (\underline\Lambda_{12})_{\cup,\epsilon}^+.$$
    Then, Lemma \ref{lem: composable openness} allows us to define compositions of microsheaves supported on composable pairs of sufficient Legendrians as follows:

\begin{definition-proposition}\label{def: gap micro composition}
    Let $\Lambda_{12} \subset S^*(M_1 \times M_2)$ and $\Lambda_{23} \subset S^*(M_2 \times M_3)$ be a composable pair of relatively compact sufficiently Legendrian subsets as in Definition \ref{def: composable} and let $\Lambda_{13} := \Lambda_{23} \circ \Lambda_{12} \subset S^*(M_1 \times M_3)$. 
    Then for any $\SF_{12} \in \msh_{\Lambda_{12}}(\Lambda_{12}), \SF_{23} \in \msh_{\Lambda_{23}}(\Lambda_{23})$, and for sufficiently small doubling parameter $\epsilon > 0$ such that $\Lambda_{12}, \Lambda_{23}$ and $\Lambda_{13}$ have no Reeb chords of length less than $\epsilon$, the microsheaves
    $$m_{\Lambda_{13}}(w_{\Lambda_{23,\epsilon}}^+\SF_{23} \circ w_{\Lambda_{12,\epsilon}}^+\SF_{12}) \in \msh_{\Lambda_{13}}(\Lambda_{13})$$
    are canonically isomorphic. We define it as $\SF_{23} \circ_g \SF_{12}$. 
\end{definition-proposition}
\begin{proof}
    Let $\epsilon > 0$ be smaller than the shortest Reeb chords on $\Lambda_{12}, \Lambda_{23}$ and $\Lambda_{23}$. Then for any $\SF_{ij} \in \msh_{\Lambda_{ij}}(\Lambda_{ij})$, following notations in Definition \ref{def: positive doubling}, we have $\SS(w_{\Lambda_{ij}}^{\prec,+}\SF_{ij}) \subset (\underline\Lambda_{ij})_{\cup,\epsilon}^{\prec,+} \subset T^*(M_i \times M_j \times \bR_{\leq \epsilon})$. Then by Lemma \ref{lem: ss-composition}, we have
    $$\SS(w_{\Lambda_{23}}^{\prec,+}\SF_{23} \circ w_{\Lambda_{12}}^{\prec,+}\SF_{12}) \subset (\underline\Lambda_{23})_{\cup,\epsilon}^{\prec,+} \circ_{\bR_{\leq\epsilon}} (\underline\Lambda_{12})_{\cup,\epsilon}^{\prec,+}.$$
    Since the pair is composable in the sense of Definition \ref{def: composable}, $\Lambda_{13} \times \bR_{\leq \epsilon} =(\Lambda_{23} \circ \Lambda_{12}) \times \bR_{\leq\epsilon} \subset (\underline\Lambda_{23})_{\cup,\epsilon}^{\prec,+} \circ_{\bR_{\leq\epsilon}} (\underline\Lambda_{12})_{\cup,\epsilon}^{\prec,+}$ is an open subset that is sufficiently Legendrian, we know by the contact transformation Lemma \ref{lem:contact-transform-main} that the microlocalization along $\Lambda_{13} \times \bR_{\leq \epsilon}$ is constant on $[0, \epsilon]$. This completes the proof.
\end{proof}

\begin{definition-proposition} \label{def: gap micro hom composition} 
    Let $\Lambda_{12} \subset S^*(M_1 \times M_2)$ and $\Lambda_{23} \subset S^*(M_2 \times M_3)$ be a composable pair of relative compact sufficiently Legendrian subsets as in Definition \ref{def: composable} and let $\Lambda_{\bar13} := \Lambda_{23} \circ (-\Lambda_{12}) \subset S^*(M_1 \times M_3)$. Then for any $\SF_{12} \in \msh_{\Lambda_{12}}(\Lambda_{12}), \SF_{23} \in \msh_{\Lambda_{23}}(\Lambda_{23})$, and any sufficiently small $\epsilon > 0$ such that $\Lambda_{12}, \Lambda_{23}$ and $\Lambda_{\bar 13}$ have no Reeb chords of length less than $2\epsilon$, the microsheaves
    $$m_{\Lambda_{13}}\sHom^\circ(w_{\Lambda_{12}}^+\SF_{12}, w_{\Lambda_{23}}^+\SF_{23}) \in \msh_{\Lambda_{\bar13}}(\Lambda_{\bar13})$$
    are canonically isomorphic. We define it as $\sHom^\circ_g(\SF_{12}, \SF_{23})$. 
\end{definition-proposition}

\begin{remark}
    One important feature in the definition is that $\epsilon > 0$ not only depends on the self Reeb chords of $\Lambda_{12}, \Lambda_{23}$, but the self Reeb chords of their compositions $\Lambda_{13}$ as well.
\end{remark}

\begin{definition}\label{def: composable pairwise}
    Let $\Lambda_{12} \subset S^*(M_1 \times M_2), \Lambda_{23} \subset S^*(M_2 \times M_3), \dots , \Lambda_{k-1,k} \subset S^*(M_{k-1} \times M_k)$ be sufficiently Legendrian subsets. We say that $\Lambda_{12}, \Lambda_{23}, \dots, \Lambda_{k-1,k}$ are \emph{composable} if for any $1\leq i < l < j \leq k$, the iterated contact compositions defined inductively as $\Lambda_{ij} := \Lambda_{lj} \circ \Lambda_{il} \subset S^*(M_i \times M_j)$ are sufficiently Legendrian, and moreover, in terms of the positive isotopies that displace $\Lambda_{il}$ and $\Lambda_{lj}$, we have that $\Lambda_{lj,\epsilon} \circ \Lambda_{il,\epsilon'}$ is disjoint from $\Lambda_{lj} \circ \Lambda_{il}$ for sufficiently small $\epsilon \geq 0, \epsilon' \geq 0$ such that $\epsilon+\epsilon' > 0$.
\end{definition}

\begin{proposition}\label{prop: gapped composition associative}
    Let $\Lambda_{12} \subset S^*(M_1 \times M_2), \Lambda_{23} \subset S^*(M_2 \times M_3)$ and $\Lambda_{34} \subset S^*(M_3 \times M_4)$ be  composable relative compact sufficiently Legendrian subsets  
    in the sense of in Definition \ref{def: composable pairwise}. Then for any $\SF_{12} \in \msh_{\Lambda_{12}}(\Lambda_{12}), \SF_{23} \in \msh_{\Lambda_{23}}(\Lambda_{23})$ and $\SF_{34} \in \msh_{\Lambda_{34}}(\Lambda_{34})$, there is a canonical isomorphism
    $$(\SF_{34} \circ_g \SF_{23}) \circ_g \SF_{12} \simeq \SF_{34} \circ_g (\SF_{23} \circ_g \SF_{12}).$$
\end{proposition}
\begin{proof}
    First, by Theorem \ref{thm: doubling adjoint}, we know that there is an adjunction 
    $$m_{\Lambda_{13}}: \Sh_{(\underline\Lambda_{23})^+_{\cup,\epsilon} \times (\underline\Lambda_{12})^+_{\cup,\epsilon}}(M_1 \times M_3) \rightleftharpoons \msh_{\Lambda_{13}}(\Lambda_{13}): w_{\Lambda_{13}}^+.$$
    Hence, for any $\SF_{12} \in \msh_{\Lambda_{12}}(\Lambda_{12}), \SF_{23} \in \msh_{\Lambda_{23}}(\Lambda_{23})$, the counit of the adjunction induces
    $$w_{\Lambda_{13}}^+m_{\Lambda_{13}}(w_{\Lambda_{12}}^+\SF_{12} \circ w_{\Lambda_{23}}^+\SF_{23}) \to w_{\Lambda_{23}}^+\SF_{23} \circ w_{\Lambda_{12}}^+\SF_{12}.$$
    Moreover, since $m_{\Lambda_{13}}w_{\Lambda_{13}}^+m_{\Lambda_{13}} \to m_{\Lambda_{13}}$ is an isomorphism by Theorem \ref{thm: relative-doubling}, the microsupport of the cofiber of the counit satisfies is disjoint from $\Lambda_{13}$.
    Now, we show that the natural transformation given by the counit induces a canonical isomorphism
    \begin{gather*}
    m_{\Lambda_{14}} (w_{\Lambda_{24}}^+m_{\Lambda_{24}}(w_{\Lambda_{34}}^+\SF_{34} \circ w_{\Lambda_{23}}^+\SF_{23}) \circ w_{\Lambda_{12}}^+ \SF_{12})
    \xrightarrow{\sim} m_{\Lambda_{14}}(w_{\Lambda_{34}}^+\SF_{34} \circ w_{\Lambda_{23}}^+\SF_{23} \circ w_{\Lambda_{12}}^+ \SF_{12}).
    \end{gather*}
    Consider the cofiber of the counit of the adjunction. 
    The microsupport estimate of Lemma \ref{lem: ss-composition} shows that after composition with $w_{\Lambda_{34}}^+ \SF_{34}$, the microsupport of
    the cofiber
    \begin{gather*}
    w_{\Lambda_{34}}^+ \SF_{34} \circ w_{\Lambda_{13}}^+m_{\Lambda_{13}}(w_{\Lambda_{23}}^+\SF_{23} \circ w_{\Lambda_{12}}^+\SF_{12})
    \to w_{\Lambda_{34}}^+ \SF_{34} \circ  w_{\Lambda_{23}}^+\SF_{23} \circ w_{\Lambda_{12}}^+\SF_{12}
    \end{gather*}
    is disjoint from $\Lambda_{14}$. Thus, the microlocalization of the natural transformation along $\Lambda_{14}$ is an isomorphism. For the same reason, we also know that the natural transformation
    \begin{gather*}
    m_{\Lambda_{14}} (w_{\Lambda_{34}}^+\SF_{34} \circ w_{\Lambda_{13}}^+ m_{\Lambda_{13}}(w_{\Lambda_{23}}^+ \SF_{23} \circ w_{\Lambda_{12}}^+ \SF_{12})) \\
    \xrightarrow{\sim} m_{\Lambda_{34} \circ \Lambda_{23} \circ \Lambda_{12}}(w_{\Lambda_{34}}^+ \SF_{34} \circ w_{\Lambda_{23}}^+\SF_{23} \circ w_{\Lambda_{12}}^+ \SF_{12}).
    \end{gather*}
    induces a canonical isomorphism. This shows that $(\SF_{34} \circ_g \SF_{23}) \circ_g \SF_{12} \simeq \SF_{34} \circ_g (\SF_{23} \circ_g \SF_{12}).$
\end{proof}

\begin{corollary}\label{rem: associativity}
    In general, let $\Lambda_{12} \subset S^*(M_1 \times M_2),  \dots \Lambda_{k-1,k} \subset S^*(M_{k-1} \times M_k)$ be relative compact stratified Legendrian subsets and let $\Lambda_{ij}:= \Lambda_{lj} \circ \Lambda_{il} \subset S^*(M_i \times M_j)$ such that they are all composable as in Definition \ref{def: composable}. Then for any $\SF_{12} \in \msh_{\Lambda_{12}}(\Lambda_{12}), \dots, \SF_{k-1,k} \in \msh_{\Lambda_{k-1,k}}(\Lambda_{k-1,k})$, we have isomorphisms for any $1 < l_1 < \dots < l_r < k$
    $$\SF_{k-1,k} \circ_g \dots \circ_g \SF_{12} \simeq (\SF_{k-1,k} \circ_g \dots \circ_g \SF_{l_r,l_r+1}) \circ_g \dots \circ_g (\SF_{l_1-1,l_1} \circ_g \dots \circ_g \SF_{12})$$
    that fit into a coherent diagram $\Delta_1^{\times k-1} \to \Cat_{st}$ satisfying the Segal condition, which sends the vertices to $\msh_{\Lambda_{1,l_1}}(\Lambda_{1,l_1}) \otimes \dots \otimes \msh_{\Lambda_{l_r,k}}(\Lambda_{l_r,k})$ and the edges to the composition functors.
\end{corollary}
\begin{proof}
    For iterated microlocal compositions, we have a diagram of composition functors that are related by the natural transformations induced by the counits of the adjunction $m_{\Lambda_{ij}} \dashv w_{\Lambda_{ij}}^+$ as in Figure \ref{fig:microlocal_composition}, which form a diagram of (op)lax natural transformations by \cite[Corollary F]{HaugsengHebestreitLinskensNuiten}.
    Then, by iteratively using the isomorphism $m_{\Lambda_{ij}} \circ w_{\Lambda_{ij}}^+ \circ m_{\Lambda_{ij}} \simeq m_{\Lambda_{ij}}$, we can show that the natural transformations induced by adjunctions are equivalences and hence we have a commutative diagram.
\end{proof}

\begin{figure}
    \centering
    \includegraphics[width=1\linewidth]{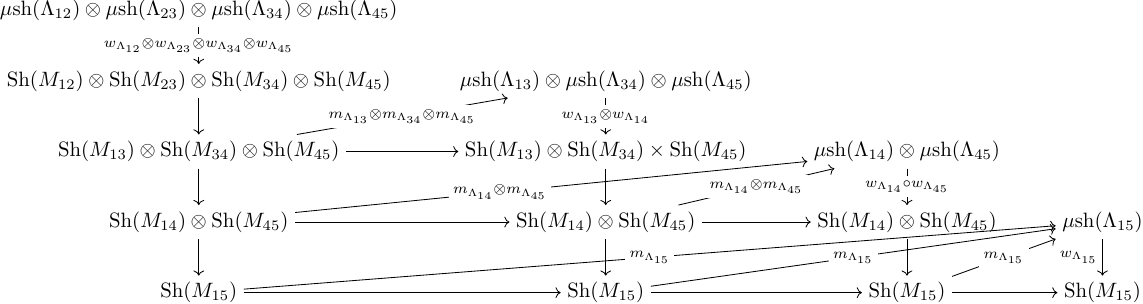}
    \caption{The diagram of functors that appears in the microlocal gapped composition of microsheaves, where we write $\msh(\Lambda_{ij})$ for $\msh_{\Lambda_{ij}}(\Lambda_{ij})$ and write $\Sh(M_{ij})$ for $\Sh(M_i \times M_j)$. For each square the natural transformation is the identity, and for each triangle the natural transformation is induced by the counit of the adjunction $m_{\Lambda_{ij}} \dashv w_{\Lambda_{ij}}^+$.}
    \label{fig:microlocal_composition}
\end{figure}

\begin{remark}
    It is not generally true that
    \begin{equation} 
    \label{associativity problem} w_{\Lambda_{13}}^+(\SF_{23} \circ_g \SF_{12}) \stackrel{?}{=} w_{\Lambda_{23}}^+\SF_{23} \circ_g w_{\Lambda_{12}}^+\SF_{12},\end{equation} 
    as we will see presently.  Our proof of Proposition \ref{prop: gapped composition associative} proceeded by showing that both sides of Formula \eqref{associativity problem} nevertheless agree after microlocalization along $\Lambda_{13}$.

    We now give the counterexample. 
    Let $M_1 = M_2 = M_3 = \bR$ and $\Lambda_{12} = \Lambda_{23} = \{(x, x, -\xi, \xi) \mid x \in \bR, \xi \in \bR_{>0}\}$. Then $\Lambda_{23}\circ \Lambda_{12} = \{(x, x, -\xi, \xi) \mid x \in \bR, \xi \in \bR_{>0}\}$. Consider the standard Reeb flow defined by the geodesic flow. We have $\Lambda_{12,\epsilon} = \Lambda_{23,\epsilon} = \{(x-\epsilon, x+\epsilon, -\xi, \xi) \mid x \in \bR, \xi \in \bR_{>0}\}$, and 
    $\Lambda_{23,\epsilon} \circ \Lambda_{12,\epsilon'} = \Lambda_{13,\epsilon+\epsilon'}.$
    Consider the rank 1 objects $\SF_{12} = 1_{\Lambda_{12}} \in \msh_{\Lambda_{12}}(\Lambda_{12})$ and $\SF_{23} = 1_{\Lambda_{23}} \in \msh_{\Lambda_{23}}(\Lambda_{23})$. Then
    $$w_{\Lambda_{12}}^+\SF_{12} = w_{\Lambda_{23}}^+\SF_{23} = 1_{\{(x, y) \mid -\epsilon < x - y \leq 0\}}.$$
    By direct computation, we know that
    $$w_{\Lambda_{23}}^+\SF_{23} \circ w_{\Lambda_{12}}^+\SF_{12} = 1_{\{(x, y) \mid -4\epsilon < x - y \leq -2\epsilon\}}[1] \oplus 1_{\{(x, y) \mid -2\epsilon < x - y \leq 0\}}.$$
    Therefore, by microlocalizing along $\Lambda_{13}$, we know that 
    $$\SF_{23} \circ_g \SF_{12} = m_{\Lambda_{13}}(w_{\Lambda_{23}}^+\SF_{23} \circ w_{\Lambda_{12}}^+\SF_{12}) = 1_{\Lambda_{23} \circ \Lambda_{12}}.$$
    However, $w_{\Lambda_{13}}^+(\SF_{23} \circ_g \SF_{12}) \neq w_{\Lambda_{23}}^+\SF_{23} \circ w_{\Lambda_{12}}^+\SF_{12}$. Instead, the former is only a direct summand of the latter. In fact, this is a general phenomenon when defining the composition using doublings.
\end{remark}

\begin{remark}\label{rem: composition KS}
    Finally, we explain the relation between our gapped composition and the theory of microlocal compositions in the work of Kashiwara and Schapira \cite[Section 7.3]{KS}. 
    Those authors defined the microlocal composition pointwise on the stalks of microsheaves when the microsupports are `microlocal composable'. More precisely, when $\Lambda_{12} \subset \dot T^*(M_1 \times M_2)$ and $\Lambda_{23} \subset \dot T^*(M_2 \times M_3)$ satisfy the condition that
    \begin{equation}\label{eq: microlocal composable}
    \Lambda_{12} \times \Lambda_{23} \cap (p_1 \times T^*_\Delta M_2^2 \times p_3) = (p_1, p_2, p_2, p_3),
    \end{equation}
    they are called `microlocally composable' at $(p_1, p_2, p_2, p_3)$. Under the above condition, they defined a microlocal composition functor for $\Lambda_{13} = \Lambda_{23} \circ \Lambda_{12}$ in \cite[Proposition 7.3.1]{KS}:
    $$\circ_\mu : \msh_{\Lambda_{12}}(p_1,p_2) \times \msh_{\Lambda_{23}}(p_2,p_3) \to \msh_{\Lambda_{13}}(p_1, p_3),$$
    by the pro-object in the category $\Fun^{ex}(\msh_{\Lambda_{13}}(p_1, p_3), \SC)$ for $\SF'_{ij} \in \msh^{pre}(S^*(M_i \times M_j))$ such that $m_{(p_i, p_j)}\SF'_{ij} = \SF_{ij}$:
    $$\SF_{23} \circ_\mu \SF_{12} = m_{(p_1,p_3)}(\lim\nolimits_{\SF'_{23} \to \SF_{23},\SF'_{12} \to \SF_{12}} \SF'_{23} \circ \SF'_{12}),$$
    However, the microlocally composable condition in Formula \eqref{eq: microlocal composable} is extremely restrictive (and fails for many examples of interest to us).

    Nevertheless, when the supports $\Lambda_{12}$ and $\Lambda_{23}$ are composable sufficiently Legendrian subsets in the sense of Definition \ref{def: composable} which are moreover microlocally composable in the sense of Kashiwara and Schapira, our gapped composition agrees with the microlocal composition in \cite[Proposition 7.3.1]{KS}. In fact, by Theorems \ref{thm: relative-doubling} and \ref{thm: doubling adjoint}, for any $\SF'_{ij} \in \msh^{pre}_{\Lambda_{ij}}(S^*(M_i \times M_j))$ such that $m_{(p_i, p_j)}\SF'_{ij} = \SF_{ij}$, there is a natural morphism $w_{\Lambda_{ij,-s}} \circ m_{\Lambda_{ij,-s}}\SF_{ij}' \to \SF_{ij}'$ for $s > 0$ sufficiently small. This implies that 
    $$\SF_{12} \circ_\mu \SF_{23} = m_{(\Lambda_{12}\circ\Lambda_{23})_{-s}}(\lim\nolimits_{s \to 0} w_{\Lambda_{12,-s}}\SF_{12} \circ w_{\Lambda_{23,-s}}\SF_{23}).$$
    When $\Lambda_{12}$ and $\Lambda_{23}$ are only perturbable to finite positions, then the same argument still works if we follow Remark \ref{rem: composition limit} as define the composition via the pro-object as a limit of compositions of doublings as $s \to 0$.
\end{remark}

\begin{remark}\label{rem: composition limit}
   If we had only assumed that the $\Lambda_{ij}$ were perturbable to finite position (and perhaps not pdff or self displaceable), we could have still used the doubling in Remark \ref{rem: doubling limit} and defined the microlocal composition by taking
    $$\SF_{23} \circ_g \SF_{12} = m_{\Lambda_{13} \times 0}(w_{\Lambda_{23}}^{\prec,+}\SF_{23} \circ_{\bR} w_{\Lambda_{12}}^{\prec,+}\SF_{12}) = m_{\Lambda_{13}}(\lim\nolimits_{\epsilon \to 0} w_{\Lambda_{23}}^{+}\SF_{23} \circ_{\bR} w_{\Lambda_{12}}^{+}\SF_{12}).$$
    We expect that many of the arguments would still work.  
\end{remark}

\subsection{Microlocal gapped composition commutes with microlocal nearby cycle}
    We will consider families of sufficiently Legendrian subsets that satisfy Definitions \ref{def: gap}, \ref{def: gap composition} and \ref{def: composable} and show commutativity of gapped composition and gapped nearby cycles. For a subset $\Lambda \subset S^*(M \times \bR_{>0})$, we use the notations $\Pi(\Lambda)$ and $\pi(\Lambda)$ defined in Formula \eqref{eq: nearby subset}.

\begin{definition}\label{def: gap microlocal composition}
    Let $\Lambda_{12} \subset S^*(M_1 \times M_2 \times \bR_{>0})$ and $\Lambda_{23} \subset S^*(M_2 \times M_3 \times \bR_{>0})$ be $\bR_{>0}$-non-characteristic subsets that are relatively compact and sufficiently Legendrian. Let $\Lambda_{13} = \Lambda_{23} \circ_{\bR_{>0}} \Lambda_{12} \subset S^*(M_1 \times M_3)$. We say that $\Lambda_{12}$ and $\Lambda_{23}$ have \emph{(positively) gapped microlocal composition} if 
    \begin{enumerate}
        \item $\Pi_{12}(\Lambda_{12})$, $\Pi_{23}(\Lambda_{23})$, $\Pi_{13}(\Lambda_{13})$ are self gapped as in Definition \ref{def: gap},
        \item $\Pi_{12}(\Lambda_{12})$ and $\Pi_{23}(\Lambda_{23})$ are (positively) composable as in Definition \ref{def: composable}, 
        \item $\psi_{12}(\Lambda_{12})$ and $\psi_{23}(\Lambda_{23})$ are (positively) composable as in Definition \ref{def: composable}, and
        \item $\Pi_{12}(\Lambda_{12})$ and $\Pi_{23}(\Lambda_{23})$ have (positively) gapped composition as in Definition \ref{def: gap composition}.
    \end{enumerate}
\end{definition}

\begin{theorem}\label{thm:gapped composition microsheaf} 
    Let $\Lambda_{12} \subset S^*(M_1 \times M_2 \times \bR_{>0})$ and $\Lambda_{23} \subset S^*(M_2 \times M_3 \times \bR_{>0})$ be relatively compact sufficiently Legendrian subsets that have positively gapped microlocal composition. Then the following diagram commutes:
    \begin{equation}\label{eq: gap composition microsheaf}
    \begin{tikzcd}
        \msh_{\Lambda_{12}}(\Lambda_{12}) \times \msh_{\Lambda_{23}}(\Lambda_{23}) \ar[r, " \circ_g"] \ar[d, "\psi_{12} \times \psi_{23}" left] & \msh_{\Lambda_{13}}(\Lambda_{13}) \ar[d, "\psi_{13}"] \\
        \msh_{\psi_{12}(\Lambda_{12})}(\psi_{12}(\Lambda_{12})) \times \msh_{\psi_{23}(\Lambda_{23})}(\psi_{23}\Lambda_{23}) \ar[r, " \circ_g"]  & \msh_{\psi_{13}(\Lambda_{13})}(\psi_{23}(\Lambda_{13})).
    \end{tikzcd}
    \end{equation}
\end{theorem}
\begin{proof}
    Fix $\epsilon > 0$ sufficiently small. Since $\Lambda_{12}$ and $\Lambda_{23}$ are $\epsilon$-gapped  in the sense of Definition \ref{def: gap}, the gapped nearby cycle functors in Formula \eqref{eq: gap composition microsheaf} are well defined. Since $\Lambda_{12}, \Lambda_{23}$ and $\psi_{12}(\Lambda_{12}), \psi_{23}(\Lambda_{23})$ are both composable in the sense of Definition \ref{def: composable}, the composition functors in Formula \eqref{eq: gap composition microsheaf} are well-defined. Since the pair $(\Lambda_{12}, \Lambda_{23})$ has $\epsilon$-gapped composition in the sense of Definition \ref{def: gap composition},   the pair of doublings $((\underline\Lambda_{12})^+_{\cup,\epsilon}, (\underline\Lambda_{23})^+_{\cup,\epsilon})$ also has $\epsilon$-gapped composition. Then by Theorem \ref{thm:gapped composition}, we know that for $\SF_{12} \in \msh_{\Lambda_{12}}(\Lambda_{12})$ and $\SF_{23} \in \msh_{\Lambda_{23}}(\Lambda_{23})$,
    $$\psi_{23}(w_{\Lambda_{23}}^+\SF_{23}) \circ \psi_{12}(w_{\Lambda_{12}}^+\SF_{12}) = \psi_{13}(w_{\Lambda_{23}}^+\SF_{23} \circ w_{\Lambda_{12}}^+\SF_{12}).$$
     Since $\psi_{ij}\SF_{ij} = m_{\psi(\Lambda_{ij})}(\psi_{ij}(w_{\Lambda_{ij}}^+\SF_{ij}))$ by Definition \ref{def: gap micro nearby}, microlocalizing along $\psi_{13}(\Lambda_{13})$, by Definition \ref{def: gap micro composition}, we have
    $$\psi_{23}(\SF_{23}) \circ_g \psi_{12}(\SF_{12}) = \psi_{13}(\SF_{23} \circ_g \SF_{12}).$$
    This completes the proof of the theorem.
\end{proof}

    More generally, similar to Corollary \ref{rem: gapped composition higher}, consider the gapped microsheaf composition functor as a functor $\Delta_1^{\times k-1} \to \Cat_{st}$ (which is used to encode all possible orders of compositions). We have:

\begin{corollary}\label{rem: gapped micro composition higher}
    Let $\Lambda_{12} \subset S^*(M_1 \times M_2 \times \bR_{>0}), \dots, \Lambda_{l-1,l} \subset S^*(M_{k-1} \times M_k \times \bR_{>0})$ be relatively compact sufficiently Legendrians such that all compositions $\Lambda_{ij} = \Lambda_{lj} \circ_{\bR_{>0}} \Lambda_{il} \subset S^*(M_i \times M_j \times \bR_{>0})$ have (positively) gapped microlocally composition. Then there are natural isomorphisms for any $1 < l_1 < \dots < l_r < k$ 
    $$\psi_{12\dots k}(\SF_{k-1,k} \circ_g \dots \circ_g \SF_{12}) \simeq \psi_{l_r\dots k}(\SF_{l_r,l_r+1} \circ_g \dots \circ_g \SF_{k,k+1}) \circ_g \dots \circ_g \psi_{1\dots l_1}(\SF_{l_1-1,l_1} \circ_g \dots \circ_g \SF_{12})$$
    that fit into a natural transformation, from the coherent diagram $\Delta_1^{\times k-1} \to \Cat_{st}$ satisfying the Segal condition that sends the vertices to $\msh(\Lambda_{1,l_1}) \otimes \dots \otimes \msh(\Lambda_{l_r,k})$ and the edges to the composition functors, to the coherent diagram $\Delta_1^{\times k-1} \to \Cat_{st}$ satisfying the Segal condition, which sends the vertices to $\msh(\psi_{1,l_1}(\Lambda_{1,l_1})) \otimes \dots \otimes \msh(\psi_{l_r,k}(\Lambda_{l_r,k}))$ and the edges to the composition functors, such that all the 2-morphisms are equivalences.
\end{corollary}
\begin{proof}
    For iterated microlocal compositions, we have a diagram of composition functors and nearby cycle functors that are related by the natural transformations induced by the (non-proper or smooth) base changes and the counits of the adjunction $m_{\Lambda_{ij}} \dashv w_{\Lambda_{ij}}^+$, which form a diagram of (op)lax natural transformations by \cite[Corollary F]{HaugsengHebestreitLinskensNuiten}. Then, by iteratively using the isomorphisms above, we can show that the natural transformations induced by adjunctions are equivalences and hence we have a commutative diagram. 
\end{proof}

\begin{remark}\label{rem: gap micro composition higher alternative}
    Consider the set of divisions of the $k$-simplex $\mathrm{Div}_k$ with the partial order by refinements (which is used to encode all possible orders of compositions). We note that the natural transformation between the two coherent diagrams $\Delta_1^{\times k-1} \to \Cat_{st}$ is also the same as the coherent diagram $\mathrm{Div}_k \to \Fun^{ex}(\bigotimes_{j=1}^{k-1}\msh_{\Lambda_{j,j+1}}(\Lambda_{j,j+1}), \msh_{\Lambda_{1,k}}(\Lambda_{1,k}))$ where a division $1 < l_1 < \dots < l_r < k$ is sent to the functor on the right hand side of the equation.
\end{remark}

\begin{remark}
    Compared to Theorem \ref{thm:gapped composition} in the sheaf setting, the above Theorem \ref{thm:gapped composition microsheaf} has the additional hypothesis that $\Lambda_{12}$ and $\Lambda_{23}$ need to be gapped with themselves, necessary to use the gapped specialization of microsheaves (Definition \ref{def: gap micro nearby}). 

    Let us give an example where this additional hypothesis does not hold. 
    Let $\Delta = \{(x, y) \mid x = y\} \subset \bR^2$ be the diagonal and $U(\Delta \times \bR_{>0}) = \{(x, y, t) \mid |x - y| < t\} \subset \bR^2$ be the neighbourhood of the diagonal of radius $t$. Consider the conic Lagrangian subset $\Lambda_{12} = \dot N^*(\Delta \times \bR_{>0})$ and $\Lambda_{23} = \dot{N}^*(\Delta \times \bR_{>0}) \cup \dot N^*_{out}U(\Delta \times \bR_{>0})$ where here $N^*_{out}U$ means the outward conormal bundle of the smooth boundary $\partial U$. Then the pair $(\Lambda_{12}, \Lambda_{23})$ has gapped composition in the sense of Definition \ref{def: gap composition}.  For any $\SF_{12} \in \Sh_{\Lambda_{12}}(\bR^2 \times \bR_{>0})$ and $\SF_{23} \in \Sh_{\Lambda_{23}}(\bR^2 \times \bR_{>0})$, by Theorem \ref{thm:gapped composition}, we have
    $$\psi_{13}(\SF_{23} \circ \SF_{12}) = \psi_{23}\SF_{23} \circ \psi_{12}\SF_{12}.$$
    However, as explained in Example \ref{ex: gapped specialization undefined}, $\Pi(\Lambda_{23})$ is not gapped with itself because the family has self Reeb chords of length $t \to 0$.
\end{remark}

    We also have compatibility with restriction:

\begin{theorem}\label{thm: gapped composition restrict}
    Let $\Lambda_{12} \subset S^*(M_1 \times M_2)$ and $\Lambda_{23} \subset S^*(M_2 \times M_3)$ be relatively compact sufficiently Legendrian subsets that are composable. Let $\Lambda_{13} = \Lambda_{23} \circ \Lambda_{12}$. Suppose given open subsets $\Sigma_{ij} \subset \Lambda_{ij}$ such that 
    $$\Sigma_{12} \,\widehat\times\,\Sigma_{23} = \pi_{13}^{-1}(\Sigma_{13}) \cap \Lambda_{12} \,\widehat\times\, \Lambda_{23}.$$
    Then the following diagram commutes:
    \[\begin{tikzcd}
    \msh_{\Lambda_{12}}(\Lambda_{12}) \otimes \msh_{\Lambda_{23}}(\Lambda_{23}) \ar[r, " \circ_g"] \ar[d, "i_{12}^* \otimes i_{23}^*" left] & \msh_{\Lambda_{13}}(\Lambda_{13}) \ar[d, "i_{13}^*"] \\
    \msh_{\Sigma_{12}}(\Sigma_{12}) \otimes \msh_{\Sigma_{23}}(\Sigma_{23}) \ar[r, " \circ_g"] & \msh_{\Sigma_{13}}(\Sigma_{13}).
    \end{tikzcd}\]
\end{theorem}
\begin{proof}
    First, we use Corollary \ref{cor: restrict doubling} to realize the microlocal restriction functors via the doubling functor in Theorems \ref{thm: relative-doubling} and \ref{thm: doubling adjoint}:
    $$m_{\Sigma_{ij}} \circ \iota^!_{(\underline\Sigma_{ij})_{\cup,\epsilon}^+} \circ w_{\Lambda_{ij}^+}: \msh_{\Lambda_{ij}}(\Lambda_{ij}) \to \Sh_{(\underline\Lambda_{ij})_{\cup,\epsilon}^+}(M) \to \Sh_{(\underline\Sigma_{ij})_{\cup,\epsilon}^+}(M) \to \msh_{\Sigma_{ij}}(\Sigma_{ij}).$$
    By Theorem \ref{rem:wrapping} and Proposition \ref{prop:wrapping nearby}, the right adjoint functor $$\iota^!_{(\underline\Sigma_{ij})_{\cup,\epsilon}^+}: \Sh_{(\underline\Lambda_{ij})_{\cup,\epsilon}^+}(M) \to \Sh_{(\underline\Sigma_{ij})_{\cup,\epsilon}^+}(M)$$
    can be realized by the nearby cycle functor $\psi_{ij}$ of the negative contact push-off that sends $(\Lambda_{ij})_\epsilon$ back to $(\Lambda_{ij} \setminus \Sigma_{ij}) \cup (\Sigma_{ij})_\epsilon$. Then for any $\SF_{12} \in \msh_{\Lambda_{12}}(\Lambda_{12})$ and $\SF_{23} \in \msh_{\Lambda_{12}}(\Lambda_{12})$, their microlocal gapped composition is
    $$i_{23}^*\SF_{23} \circ_g i_{12}^*\SF_{12} = m_{\Sigma_{23} \circ \Sigma_{12}}(\psi_{23}(w_{\Lambda_{23}^+}\SF_{23}) \circ \psi_{12}( w_{\Lambda_{12}^+}\SF_{12})).$$
    Since $\Sigma_{12} \,\widehat\times\, \Sigma_{23} = \pi_{13}^{-1}(\Sigma_{13}) \cap \Lambda_{12}\,\widehat\times\,\Lambda_{23}$, we know the composition of $\Lambda_{ij} \setminus \Sigma_{ij}$ is disjoint from $\Sigma_{13}$. Since $\Lambda_{12}, \Lambda_{23}$ are composable, we know that for the negative contact push-off from $(\Lambda_{ij})_\epsilon$ back to $(\Lambda_{ij} \setminus \Sigma_{ij}) \cup (\Sigma_{ij})_\epsilon$, their composition is also disjoint from $\Sigma_{13}$. Therefore, by Theorem \ref{thm: gapped composition family}, we can conclude that 
    $$i_{23}^*\SF_{23} \circ_g i_{12}^*\SF_{12} = m_{\Sigma_{13}}(\psi_{13}(w_{\Lambda_{23}^+}\SF_{23} \circ w_{\Lambda_{12}^+}\SF_{12})) = i_{13}^*(\SF_{23} \circ_g \SF_{12}).$$
    This completes the proof.
\end{proof}

\begin{remark}
    When $\Lambda_{12}, \Lambda_{23}$ and $\Sigma_{12}, \Sigma_{23}$ are both conic subsets in cotangent bundles containing part of the zero sections with $\pi(\Lambda_{ij}) \subset \Lambda_{ij}$ and $\pi(\Sigma_{ij}) \subset \Sigma_{ij}$, the above formula can be deduced from proper base change.
\end{remark}

Similarly, using Theorem \ref{thm:gapped hom composition}, we can prove the following result:

\begin{theorem}\label{thm:gapped hom composition microsheaf}  
    Let $\Lambda_{12} \subset S^*(M_1 \times M_2 \times \bR_{>0})$ and $\Lambda_{23} \subset S^*(M_2 \times M_3 \times \bR_{>0})$ define families of relatively compact sufficiently Legendrians subsets with gapped microlocal hom composition. Then there exists a commutative diagram between nearby cycle functors and compositions functors
    \[\begin{tikzcd}
        \msh_{\Lambda_{12}}(\Lambda_{12}) \otimes \msh_{\Lambda_{23}}(\Lambda_{23}) \ar[r, " \sHom^\circ_g"] \ar[d, "\psi_{12} \otimes \psi_{23}" left] & \msh_{\Lambda_{\bar 13}}(\Lambda_{\bar 13}) \ar[d, "\psi_{13}"] \\
        \msh_{\psi_{12}(\Lambda_{12})}(\psi_{12}(\Lambda_{12})) \otimes \msh_{\psi_{23}(\Lambda_{23})}(\psi_{23}(\Lambda_{23})) \ar[r, "\sHom^\circ_g"]  & \msh_{\psi_{13}(\Lambda_{\bar 13})}(\psi_{13}(\Lambda_{\bar 13})).
    \end{tikzcd}\]
\end{theorem}

    

%
%
%
%
%

\subsection{Gapped composition computes the integral transform}

    Let $\Lambda_1 \subset S^*M_1$, $\Lambda_2 \subset S^*M_2$ and $S^*M_3 \subset S^*M_3$ be sufficiently Legendrian and let $\Lambda_{\bar 12} = -\Lambda_1 \,\widehat\times\, \Lambda_2, \Lambda_{\bar 23} = -\Lambda_2 \,\widehat\times\, \Lambda_3$ and $\Lambda_{\bar13} = -\Lambda_1 \,\widehat\times\, \Lambda_3$. They have gapped composition and are composable:

\begin{lemma}\label{lem: product composable}
    Let $\Lambda_1 \subset S^*M_1$, $\Lambda_2 \subset S^*M_2$ and $S^*M_3 \subset S^*M_3$ be sufficiently Legendrian subsets. Then $\Lambda_{\bar 12} = -\Lambda_1 \,\widehat\times\, \Lambda_2, \Lambda_{\bar 23} = -\Lambda_2 \,\widehat\times\, \Lambda_3$ are composable in the sense of Definition \ref{def: composable} and have gapped composition in the sense of Definition \ref{def: gap composition}, and $\Lambda_{\bar 12}, \Lambda_{\bar 23}, \Lambda_{\bar 13}$ are sufficiently Legendrian.
\end{lemma}
\begin{proof}
    First, we show that $\Lambda_{\bar 12} = -\Lambda_1 \,\widehat\times\, \Lambda_2, \Lambda_{\bar 23} = -\Lambda_2 \,\widehat\times\, \Lambda_3$ and $\Lambda_{\bar 13} = -\Lambda_1 \,\widehat\times\, \Lambda_3$ are sufficiently Legendrian. Consider the contact flows $R_i$ on $S^*M_i$ or equivalently on $\dot T^*M_i$. Then we can define the product contact flow by
    $$R_{ij}^t(x_i, \xi_i, r_i; x_j, \xi_j, r_j) = (R_i^{tr_i}(x_i, \xi_i, r_i), R_j^{tr_j}(x_j, \xi_j, r_j)), \quad (x_i, \xi_i) \in S^*M_i, \; r_i^2 + r_j^2 = 1.$$
    For the contact flows $R_i$ that displace $\Lambda_i$ from cotangent fibers (resp.~from itself), the product flow $R_{ij}$ displaces $\Lambda_{ij}$ from cotangent fibers (resp.~from itself). Thus the product $\Lambda_{ij}$ are sufficiently Legendrian. They are composable because
    $\Lambda_{\bar 13} = -\Lambda_1 \,\widehat\times\, \Lambda_3 = (-\Lambda_2 \,\widehat\times\, \Lambda_3 ) \circ (-\Lambda_1 \,\widehat\times\, \Lambda_2 ) = \Lambda_{\bar 23} \circ \Lambda_{\bar 12}$, and thus if we consider the product Reeb flow, 
    $$\Lambda_{\bar23,\epsilon'} \circ \Lambda_{\bar12,\epsilon} = \Lambda_{\bar13,\epsilon+\epsilon'}.$$
    Since $\Lambda_{\bar 13}$ is self displaceable, we know that $\Lambda_{\bar13,\epsilon+\epsilon'}$ is disjoint from $\Lambda_{\bar 13}$ and thus the pair is composable in the sense of Definition \ref{def: composable}. Finally, by Lemma \ref{ss gapped for doubling}, the pair has gapped composition in the sense of Definition \ref{def: gap composition}.
\end{proof}
    
    One can define the algebraic composition of microsheaves with sufficient Legendrian supports
    using duality and K\"unneth formula in Definition \ref{def: alg composition}, or define the gapped composition using doubling in Definition \ref{def: gap micro composition}.
    Our main theorem in this section is that, in the above special case, the gapped composition agrees with the algebraic composition:

\begin{theorem}\label{thm: gap composition kunneth}
    Let $\Lambda_1 \subset S^*M_1$, $\Lambda_2 \subset S^*M_2$ and $\Lambda_3 \subset S^*M_3$ be relatively compact sufficiently Legendrian subsets and let $\Lambda_{\bar 12} = -\Lambda_1 \,\widehat\times\, \Lambda_2, \Lambda_{\bar 23} = -\Lambda_2 \,\widehat\times\, \Lambda_3$ and $\Lambda_{\bar 13} = -\Lambda_1 \,\widehat\times\, \Lambda_3$. Then there is a commutative diagram
    \[\begin{tikzcd}
    \msh_{\Lambda_{\bar 12}}(\Lambda_{\bar 12}) \otimes \msh_{\Lambda_{\bar 23}}(\Lambda_{\bar 23}) \ar[r, "\circ_g"] \ar[d, "\rotatebox{90}{$=$}" left] & \msh_{\Lambda_{13}}(\Lambda_{13}) \ar[d, "\rotatebox{90}{$=$}"] \\
    \msh_{\Lambda_{\bar 12}}(\Lambda_{\bar 12}) \otimes \msh_{\Lambda_{\bar 23}}(\Lambda_{\bar 23}) \ar[r, "\circ_a"] & \msh_{\Lambda_{\bar 13}}(\Lambda_{\bar 13}).
    \end{tikzcd}\]
    between the composition in Definition \ref{def: alg composition} and the gapped composition in Definition \ref{def: gap micro composition}.
\end{theorem}
\begin{proof}
    First, by Theorem \ref{thm:microsheaf-duality}, we know that there is a commutative diagram of algebraic compositions of sheaves and microsheaves in Definition \ref{def: alg composition}:
    \[\begin{tikzcd}
    \msh_{\Lambda_{\bar 12}}(\Lambda_{\bar 12}) \otimes \msh_{\Lambda_{\bar 23}}(\Lambda_{\bar 23}) \ar[r, "\circ_a"] \ar[d, "(w_{-\Lambda_1}^+ \otimes w_{\Lambda_2}^+) \otimes (w_{-\Lambda_2}^+ \otimes w_{\Lambda_3}^+)" left] & \msh_{\Lambda_{\bar 13}}(\Lambda_{\bar 13}) \\
    \Sh_{(-\underline\Lambda_{1})^+_{\cup,\epsilon} \times (\underline\Lambda_{2})^+_{\cup,\epsilon}}(M_1 \times M_2) \otimes \Sh_{(-\underline\Lambda_{2})^+_{\cup,\epsilon} \times (\underline\Lambda_{3})^+_{\cup,\epsilon}}(M_2 \times M_3) \ar[r] & \Sh_{(-\underline\Lambda_{1})^+_{\cup,\epsilon} \times (\underline\Lambda_{3})^+_{\cup,\epsilon}}(M_1 \times M_3)  \ar[u, "m_{-\Lambda_1} \otimes m_{\Lambda_3}" right].
    \end{tikzcd}\]
    Then, per Definition \ref{def: gap micro composition}, we know that the gapped composition is defined via the following commutative diagram
    \[\begin{tikzcd}
    \msh_{\Lambda_{\bar 12}}(\Lambda_{\bar 12}) \otimes \msh_{\Lambda_{\bar 23}}(\Lambda_{\bar 23}) \ar[r, "\circ_g"] \ar[d, "w_{\Lambda_{\bar 12}}^+ \otimes w_{\Lambda_{\bar 23}}^+" left] & \msh_{\Lambda_{\bar 13}}(\Lambda_{\bar 13}) \\
    \Sh_{(\underline\Lambda_{\bar 12})^+_{\cup,\epsilon'}}(M_1 \times M_2) \otimes \Sh_{(\underline\Lambda_{\bar 23})^+_{\cup,\epsilon'}}(M_2 \times M_3) \ar[r] & \Sh_{(\underline\Lambda_{\bar 23})^+_{\cup,\epsilon'} \circ (\underline\Lambda_{\bar 12})^+_{\cup,\epsilon'}}(M_1 \times M_3) \ar[u, "m_{\Lambda_{\bar 13}}"].
    \end{tikzcd}\]
    Therefore, by Theorem \ref{thm: kunneth-doubling}, it suffices to show that the following diagram commutes:
    \[\begin{tikzcd}[column sep=10pt]
    \Sh_{(\underline\Lambda_{\bar 12})^+_{\cup,\epsilon'}}(M_1 \times M_2) \otimes \Sh_{(\underline\Lambda_{\bar 23})^+_{\cup,\epsilon'}}(M_2 \times M_3) \ar[r] \ar[d, "(w_{-\Lambda_1}^+ \otimes w_{\Lambda_2}^+ \otimes w_{-\Lambda_2}^+ \otimes w_{\Lambda_3}^+) \circ (m_{\Lambda_{\bar 12}} \otimes m_{\Lambda_{\bar 23}})" left] & \Sh_{(\underline\Lambda_{\bar 23})^+_{\cup,\epsilon'} \circ (\underline\Lambda_{\bar 12})^+_{\cup,\epsilon'}}(M_1 \times M_3) \ar[d, "(w_{-\Lambda_1}^+ \otimes w_{\Lambda_3}^+) \circ m_{\Lambda_{13^-}}"]  \\
    \Sh_{(-\underline\Lambda_{1})^+_{\cup,\epsilon} \times (\underline\Lambda_{2})^+_{\cup,\epsilon}}(M_1 \times M_2) \otimes \Sh_{(-\underline\Lambda_{2})^+_{\cup,\epsilon} \times (\underline\Lambda_{3})^+_{\cup,\epsilon}}(M_2 \times M_3) \ar[r] & \Sh_{(-\underline\Lambda_{1})^+_{\cup,\epsilon} \times (\underline\Lambda_{3})^+_{\cup,\epsilon}}(M_1 \times M_3).
    \end{tikzcd}\]
    By Theorem \ref{thm: doubling-small-piece}, we know that the vertical functors are realized by nearby cycle functors
    $$\psi_{ij} : \Sh_{(\underline\Lambda_{\bar ij})^+_{\cup,s}}(M_i \times M_j) \xrightarrow{\sim} \Sh_{(-\underline\Lambda_{i})^+_{\cup,\epsilon} \times (\underline\Lambda_{j})^+_{\cup,\epsilon}}(M_i \times M_j).$$
    Moreover, from the proof of Theorem \ref{thm: doubling-small-piece}, we know that we can find an contact isotopy of sufficiently Legendrian subsets $(\underline\Lambda_{\bar ij})^+_{\cup,s}$ into an arbitrary small neighbourhood of $(-\underline\Lambda_{i})^+_{\cup,\epsilon} \times (\underline\Lambda_{j})^+_{\cup,\epsilon}$. This means that
    $\psi_{ij}((\underline\Lambda_{\bar ij})^+_{\cup,s}) \subset (-\underline\Lambda_{i})^+_{\cup,\epsilon} \times (\underline\Lambda_{j})^+_{\cup,\epsilon}.$
    Then, by Theorem \ref{thm:gapped composition}, we can show that for any $\SF_{12} \in \msh_{\Lambda_{\bar 12}}(\Lambda_{\bar 12}), \SF_{23} \in \msh_{\Lambda_{\bar 23}}(\Lambda_{\bar 23})$, we have
    $$\psi_{23}(w_{\Lambda_{\bar 23}}^+\SF_{23}) \circ \psi_{12}(w_{\Lambda_{\bar 12}}^+\SF_{12}) = \psi_{13}(w_{\Lambda_{\bar 23}}^+\SF_{23} \circ w_{\Lambda_{\bar 12}}^+\SF_{12}).$$
    Since $\Lambda_{\bar 23} \circ \Lambda_{\bar 12} \subset ((\underline\Lambda_{\bar 23})^+_{\cup,s} \circ (\underline\Lambda_{\bar 12})^+_{\cup,s})$ is an open subset for any $s > 0$ by Lemma \ref{lem: product composable}, we can microlocalize along $\Lambda_{\bar 13} = \Lambda_{\bar 23} \circ \Lambda_{\bar 12}$ on the right hand side show by Lemma \ref{lem:contact-transform-main} that
    $$m_{\Lambda_{\bar 13}}(w_{\Lambda_{\bar 23}}^+\SF_{23} \circ w_{\Lambda_{\bar 12}}^+\SF_{12}) = m_{\Lambda_{\bar 13}} \circ \psi_{13}(w_{\Lambda_{\bar 23}}^+\SF_{23} \circ w_{\Lambda_{\bar 12}}^+\SF_{12}).$$
    This then shows the commutativity of the diagram, which completes the proof.
\end{proof}

    Combining Corollary \ref{rem: gapped micro composition higher} with the coherence of associativity of compositions of colimit preserving functors in Formula \eqref{rem: associativity cat}, we can also conclude that:

\begin{corollary}\label{rem: gap composition kunneth higher}
    Let $\Lambda_1 \subset S^*M_1, \dots, \Lambda_k \subset S^*M_k$ be relatively compact sufficiently Legendrian subsets and let $\Lambda_{ij}^- = -\Lambda_i \,\widehat\times\, \Lambda_j$. Then there are natural equivalences between the composition in Definition \ref{def: alg composition} and the gapped composition in Definition \ref{def: gap micro composition}
    \[\begin{tikzcd}
    \msh_{\Lambda_{\bar 12}}(\Lambda_{\bar 12}) \otimes \dots \otimes \msh_{\Lambda_{\ol{k-1},k}}(\Lambda_{\ol{k-1},k}) \ar[r, "\circ_g"] \ar[d, "\rotatebox{90}{$=$}" left] & \msh_{\Lambda_{\bar 1k}}(\Lambda_{\bar 1k}) \ar[d, "\rotatebox{90}{$=$}"] \\
    \msh_{\Lambda_{\bar 12}}(\Lambda_{\bar 12}) \otimes \dots \otimes \msh_{\Lambda_{\ol{k-1},k}}(\Lambda_{\ol{k-1},k}) \ar[r, "\circ_a"] & \msh_{\Lambda_{\bar 1k}}(\Lambda_{\bar 1k}).
    \end{tikzcd}\]
    that fit into an equivalence of coherent diagrams $\Delta_1^{\times k-1} \to \Cat_{st}$ where the vertices are sent to $\msh_{\Lambda_{\bar 1l_1}}(\Lambda_{\bar 1l_1}) \otimes \dots \otimes \msh_{\Lambda_{\bar l_r+1,k}}(\Lambda_{\bar l_r+1,k})$ and the arrows are sent to the composition functors.
\end{corollary}

\subsection{Gapped composition and stop removal}\label{ssec: stop removal}

In applications, our desired kernel originates as some (only eventually conic) Lagrangian, whose conification is still not supported in 
$\Lambda_1 \,\widehat\times\, \Lambda_2$, but
rather in some larger $L_{12}$ containing $\Lambda_1 \,\widehat\times\, \Lambda_2$.  For instance, if $-\Lambda_1 = \Lambda_2$ is the core of some Weinstein manifold, then the diagonal of the Weinstein manifold will not be contained in $\Lambda_1 \times \Lambda_2$. 

Thus, we need to understand the relation between gapped compositions and the localization functor
$$\iota_{L_{12}}^*: \msh_{L_{12}}(L_{12}) \to \msh_{\Lambda_1 \,\widehat\times\, \Lambda_2}(\Lambda_1 \,\widehat\times\, \Lambda_2).$$
where we assume that $\Lambda_1 \,\widehat\times\, \Lambda_2 \subset L_{12}$ is a closed subset.

\begin{lemma}\label{lem: composition wrap}
    Let $\underline\Lambda_i \subset T^*M_i$ be closed conic subsets, $\underline\Lambda_{\bar ij} = -\underline\Lambda_i \times \underline\Lambda_j$, and $\underline{L}_{ij} \subset T^*(M_i \times M_j)$ be closed subsets containing $\underline{\Lambda}_{\bar ij}$ such that $\underline{L}_{\bar 13} = \underline{L}_{\bar 23} \circ \underline{L}_{\bar 12}$ and
    $\underline{L}_{ij} \circ \underline\Lambda_i \subset \underline\Lambda_j.$
    Then we have a commutative diagram
    \[\begin{tikzcd}
    \Sh_{\underline{L}_{\bar 12}}(M_1 \times M_2) \otimes \Sh_{\underline{L}_{\bar 23}}(M_2 \times M_3) \ar[r, "\circ"] \ar[d, "\iota_{\underline\Lambda_{\bar 12}}^* \otimes \iota_{\underline\Lambda_{\bar 23}}^*" left] & \Sh_{\underline{L}_{\bar 13}}(M_1 \times M_3) \ar[d, "\iota_{\underline\Lambda_{\bar 13}}^*"] \\
    \Sh_{-\underline\Lambda_1 \times \underline\Lambda_2}(M_1 \times M_2) \otimes \Sh_{-\underline\Lambda_2 \times \underline\Lambda_3}(M_2 \times M_3) \ar[r, "\circ"] & \Sh_{-\underline\Lambda_1 \times \underline\Lambda_3}(M_1 \times M_3).
    \end{tikzcd}\]
\end{lemma}
\begin{proof}
    This was proven in \cite[Lemma 4.6]{KuoLi-duality} for $M_3 = pt$; the proof in the present case follows the same argument.  We give it for completeness. 

    Per Theorem \ref{thm:microsheaf-fourier}, colimit preserving functors are classified by integral kernels.  Thus it suffices to check that for any $\SF_{12} \in \Sh_{\underline{L}_{12}}(M_1 \times M_2)$, $\SF_{23} \in \Sh_{\underline{L}_{23}}(M_2 \times M_3)$ and $\SF_1 \in \Sh_{\underline\Lambda_1}(M_1)$, we have
    $$\iota_{\underline\Lambda_{\bar 23}}^*\SF_{23} \circ \iota_{\underline\Lambda_{\bar 12}}^*\SF_{12} \circ \SF_1 \simeq \iota_{\underline\Lambda_{\bar 13}}^*(\SF_{23} \circ \SF_{12}) \circ \SF_1.$$
    Consider $\SG_3 \in \Sh_{\underline\Lambda_3}(M_3)$. Using adjunctions and in particular the fact that $\iota_{\underline\Lambda_{\bar ij}}^*\SF_{ij} \circ \SF_i = \iota_{\underline\Lambda_j}^*(\SF_{ij} \circ \SF_i)$ when $\SF_i \in \Sh_{\underline\Lambda_i}(M_i)$, we have
    \begin{align*}
    \Hom(\iota_{\underline\Lambda_{\bar 23}}^*\SF_{23} \circ \iota_{\underline\Lambda_{\bar 12}}^*\SF_{12} \circ \SF_1, \SG_3) &= \Hom(\iota_{\underline\Lambda_{3}}^*(\SF_{23} \circ (\iota_{\underline\Lambda_{\bar 12}}^*\SF_{12} \circ \SF_1)), \SG_3) \\
    &\simeq \Hom(\SF_{23} \circ (\iota_{\underline\Lambda_{\bar 12}}^*\SF_{12} \circ \SF_1), \SG_3) \\
    & = \Hom(\iota_{\underline\Lambda_{\bar 12}}^*\SF_{12} \circ \SF_1, \sHom^\circ(\SF_{23}, \SG_3)) \\
    &= \Hom(\iota_{\underline\Lambda_2}^*(\SF_{12} \circ \SF_1), \sHom^\circ(\SF_{23}, \SG_3)) \\
    &\simeq \Hom(\SF_{12} \circ \SF_1, \sHom^\circ(\SF_{23}, \SG_3)) \\
    &= \Hom(\SF_1, \sHom^\circ(\SF_{23} \circ \SF_{12}, \SG_3)).
    \end{align*}
    Here, we use the fact that $\underline{L}_{\bar ij} \circ \underline\Lambda_i \subset \underline\Lambda_j$ in the second and fifth isomorphism. On the other hand, we also have
    \begin{align*}
    \Hom(\iota_{\underline\Lambda_{\bar 13}}^*(\SF_{23} \circ \SF_{12}) \circ \SF_1, \SG_3) &= \Hom(\iota_{\underline\Lambda_3}^*((\SF_{23} \circ \SF_{12}) \circ \SF_1), \SG_3) \\
    &\simeq \Hom((\SF_{23} \circ \SF_{12}) \circ \SF_1, \SG_3) \\
    &= \Hom(\SF_1, \sHom^\circ(\SF_{23} \circ \SF_{12}, \SG_3)).
    \end{align*}
    Here, we use the fact that $\underline{L}_{\bar ij} \circ \underline\Lambda_i \subset \underline\Lambda_j$ in the second isomorphism. This then completes the proof.
\end{proof}

\begin{theorem}\label{thm: gap composition wrap kunneth}
    Let $\Lambda_1 \subset S^*M_1$, $\Lambda_2 \subset S^*M_2$ and $\Lambda_3 \subset S^*M_3$ be relatively compact sufficiently Legendrian subsets and let $\Lambda_{\bar 12} = -\Lambda_1 \,\widehat\times\, \Lambda_2, \Lambda_{\bar 23} = -\Lambda_2 \,\widehat\times\, \Lambda_3$ and $\Lambda_{\bar 13} = -\Lambda_1 \,\widehat\times\, \Lambda_3$. Let ${L}_{\bar ij} \subset S^*(M_i \times M_j)$ be composable relatively compact sufficiently Legendrian subsets containing $\Lambda_{\bar ij}$ such that ${L}_{\bar 13} = {L}_{\bar 23} \circ {L}_{\bar 12}$ and
    ${L}_{\bar ij} \circ \Lambda_i \subset \Lambda_j$. Then there is an equivalence between the algebraic composition in Definition \ref{def: alg composition} and the gapped composition in Definition \ref{def: gap micro composition}
    \[\begin{tikzcd}
    \msh_{L_{\bar 12}}(L_{\bar 12}) \otimes \msh_{L_{\bar 23}}(L_{\bar 23}) \ar[d, "\iota_{\Lambda_{\bar 12}}^* \otimes \iota_{\Lambda_{\bar 23}}^*" left]\ar[r,"\circ_g"] & \msh_{L_{\bar 13}}(L_{\bar 13}) \ar[d, "\iota_{\Lambda_{\bar 13}}^*"] \\
    \msh_{\Lambda_{\bar 12}}(\Lambda_{\bar 12}) \otimes \msh_{\Lambda_{\bar 23}}(\Lambda_{\bar 23}) \ar[r,"\circ_a"] & \msh_{\Lambda_{\bar 13}}(\Lambda_{\bar 13}).
    \end{tikzcd}\]
\end{theorem}
\begin{proof}
    First, by Lemma \ref{lem: composition wrap}, we have a commutative diagram
    \[\begin{tikzcd}
    \Sh_{(\underline L_{\bar 12})_{\cup,\epsilon}}(M_1 \times M_2) \otimes \Sh_{(\underline L_{\bar 23})_{\cup,\epsilon}}(M_2 \times M_3) \ar[d, "\iota_{(\underline\Lambda_{\bar 12})_{\cup,\epsilon}}^* \otimes \iota_{(\underline\Lambda_{\bar 23})_{\cup,\epsilon}}^*" left]\ar[r,"\circ"] & \Sh_{(\underline L_{\bar 13})_{\cup,\epsilon}}(M_1 \times M_3) \ar[d, "\iota_{(\underline\Lambda_{\bar 13})_{\cup,\epsilon}}^*"] \\
    \Sh_{(\underline\Lambda_{\bar 12})_{\cup,\epsilon}}(M_1 \times M_2) \otimes \Sh_{(\underline\Lambda_{\bar 23})_{\cup,\epsilon}}(M_2 \times M_3) \ar[r,"\circ"] & \Sh_{(\underline \Lambda_{\bar 13})_{\cup,\epsilon}}(M_1 \times M_3).
    \end{tikzcd}\]
    Then, by Theorem \ref{thm: relative-doubling} and Corollary \ref{rem: double vs wrap} (see also Example \ref{rem: double stop removal}), we can conclude that there is a commutative diagram between the geometric compositions
    \[\begin{tikzcd}
    \msh_{L_{\bar 12}}(L_{\bar 12}) \otimes \msh_{L_{\bar 23}}(L_{\bar 23}) \ar[d, "\iota_{\Lambda_{\bar 12}}^* \otimes \iota_{\Lambda_{\bar 23}}^*" left]\ar[r,"\circ_g"] & \msh_{L_{\bar 13}}(L_{\bar 13}) \ar[d, "\iota_{\Lambda_{\bar 13}}^*"] \\
    \msh_{\Lambda_{\bar 12}}(\Lambda_{\bar 12}) \otimes \msh_{\Lambda_{\bar 23}}(\Lambda_{\bar 23}) \ar[r,"\circ_g"] & \msh_{\Lambda_{\bar 13}}(\Lambda_{\bar 13}).
    \end{tikzcd}\]
    Then, Theorem \ref{thm: gap composition kunneth} implies the commutative diagram between gapped compositions and algebraic compositions, so we can finish the proof.
\end{proof}

\begin{remark} 
    The assertions of Lemma \ref{lem: composition wrap} and  Theorem \ref{thm: gap composition wrap kunneth} are structurally similar, and indeed, when the subsets involved in the Lemma satisfy the additional hypotheses of the Theorem, the assertion of the Lemma becomes a special case of the Theorem.  

    Nevertheless the Theorem is rather subtler -- it requires  Legendrian hypotheses, and appeals to the results on gapped composition.  The basic source for the additional difficulty is that in the microsheaf setting, we only have a K\"unneth formula with prescribed microsupports.  As a consequence, the microsheaves correspondences -- which are conic, but not biconic -- are not determined by the functors which they induce.  
\end{remark}

\begin{corollary}\label{thm: gapped hom composition family} 
    Let $\Lambda_1 \subset S^*M_1$, $\Lambda_2 \subset S^*M_2$ and $S^*M_3 \subset S^*M_3$ be relatively compact sufficiently Legendrian subsets. Let ${L}_{\bar ij} \subset S^*(M_i \times M_j \times \bR_{>0})$ be relatively compact sufficiently Legendrian subsets containing $\Lambda_{\bar ij} \times 0_{\bR_{>0}}$ such that 
    $\psi_{ij}({L}_{\bar ij}) \circ \Lambda_i \subset \Lambda_j$. Write $L_{\bar 13} = L_{\bar 23} \circ_{\bR_{>0}} L_{\bar 12}$.
    
    When $(L_{\bar 12}, L_{\bar 23})$ have gapped microlocal compositions in the sense of Definition \ref{def: gap microlocal composition}, then there is an equivalence
    \[\begin{tikzcd}
    \msh_{L_{\bar 12}}(L_{\bar 12}) \otimes \msh_{L_{\bar 23}}(L_{\bar23}) \ar[d, "(\psi_{12} \otimes \psi_{23}) \circ (\iota_{\Lambda_{\bar 12}}^* \otimes \iota_{\Lambda_{\bar 23}}^*)" left]\ar[r,"\circ_g"] & \msh_{L_{\bar 13}}(L_{\bar 13}) \ar[d, "\psi_{13} \circ \iota_{\Lambda_{\bar 13}}^*"] \\
    \msh_{\Lambda_{\bar 12}}(\Lambda_{\bar 12}) \otimes \msh_{\Lambda_{\bar 23}}(\Lambda_{\bar 23}) \ar[r,"\circ_a"] & \msh_{\Lambda_{\bar 13}}(\Lambda_{\bar 13}).
    \end{tikzcd}\]
\end{corollary}
\begin{proof}
    When $(L_{\bar12}, L_{\bar23})$ have gapped microlocal compositions in the sense of Definition \ref{def: gap micro composition}, by Theorems \ref{thm:gapped composition} and \ref{thm: relative-doubling}, we know that there is a commutative diagram
    \[\begin{tikzcd}
    \msh_{L_{\bar 12}}(L_{\bar 12}) \otimes \msh_{L_{\bar 23}}(L_{\bar 23}) \ar[d, "\psi_{12} \otimes \psi_{23}" left]\ar[r,"\circ_g"] & \msh_{L_{\bar 13}}(L_{\bar 13}) \ar[d, "\psi_{13}"] \\
    \msh_{\psi_{12}(L_{\bar 12})}(\psi_{12}(L_{\bar 12})) \otimes \msh_{\psi_{23}(L_{\bar 23})}(\psi_{23}(L_{\bar 23})) \ar[r,"\circ_g"] & \msh_{\psi_{13}(L_{\bar 13})}(\psi_{13}(L_{\bar 13})).
    \end{tikzcd}\]
    Then since $\psi_{ij}({L}_{\bar ij}) \circ \Lambda_i \subset \Lambda_j$, the result follows from Theorem \ref{thm: gap composition wrap kunneth}. 
\end{proof}

    Finally, one can replace the composition by the hom composition and prove similar results:

\begin{theorem}\label{thm: gap hom composition wrap kunneth}
    Let $\Lambda_1 \subset S^*M_1$, $\Lambda_2 \subset S^*M_2$ and $S^*M_3 \subset S^*M_3$ be relative compact sufficiently Legendrians and let $\Lambda_{12} = \Lambda_1 \,\widehat\times\, \Lambda_2, \Lambda_{23} = \Lambda_2 \,\widehat\times\, \Lambda_3$ and $\Lambda_{\bar 13} = -\Lambda_1 \,\widehat\times\, \Lambda_3$. Let ${L}_{ij} \subset S^*(M_i \times M_j)$ be composable pairs of sufficiently Legendrian subsets containing $\Lambda_{ij}$ as closed subsets such that ${L}_{\bar 13} = {L}_{23} \circ (-{L}_{12})$ and
    ${L}_{ij} \circ (-\Lambda_i) \subset \Lambda_j$. Then there is an equivalence between the algebraic hom composition in Definition \ref{def: alg composition} and the gapped hom composition in Definition \ref{def: gap micro composition}
    \[\begin{tikzcd}
    \msh_{L_{12}}(L_{12}) \otimes \msh_{L_{23}}(L_{23}) \ar[d, "\iota_{\Lambda_{12}}^* \otimes \iota_{\Lambda_{23}}^*" left]\ar[r,"\sHom^\circ_g"] & \msh_{L_{\bar 13}}(L_{\bar 13}) \ar[d, "\iota_{\Lambda_{\bar 13}}^*"] \\
    \msh_{\Lambda_{12}}(\Lambda_{12}) \otimes \msh_{\Lambda_{23}}(\Lambda_{23}) \ar[r,"\sHom^\circ_a"] & \msh_{\Lambda_{\bar 13}}(\Lambda_{\bar 13}).
    \end{tikzcd}\]
\end{theorem}

\begin{corollary}\label{thm: gapped composition family}
    Let $\Lambda_1 \subset S^*M_1$, $\Lambda_2 \subset S^*M_2$ and $\Lambda_3 \subset S^*M_3$ be relative compact sufficiently Legendrians. Let ${L}_{ij} \subset S^*(M_i \times M_j \times \bR_{>0})$ be a composable pair of relatively compact sufficiently Legendrians containing $\Lambda_{ij} \times 0_{\bR_{>0}}$ such that 
    $\psi_{ij}({L}_{ij}) \circ (-\Lambda_i) \subset \Lambda_j$. Write ${L}_{\bar 13} = {L}_{23} \circ_{\bR_{>0}} (-{L}_{12})$.
    
    When $(L_{12}, L_{23})$ have gapped hom compositions in the sense of Definition \ref{def: gap microlocal composition}, then there is an equivalence
    \[\begin{tikzcd}
    \msh_{L_{12}}(L_{12}) \otimes \msh_{L_{23}}(L_{23}) \ar[d, "(\psi_{12} \otimes \psi_{23}) \circ (\iota_{\Lambda_{12}}^* \otimes \iota_{\Lambda_{23}}^*)" left]\ar[r,"\sHom^\circ_g"] & \msh_{L_{13}^-}(L_{13}^-) \ar[d, "\psi_{13} \circ \iota_{\Lambda_{13}}^*"] \\
    \msh_{\Lambda_{12}}(\Lambda_{12}) \otimes \msh_{\Lambda_{23}}(\Lambda_{23}) \ar[r,"\sHom^\circ_a"] & \msh_{\Lambda_{\bar 13}}(\Lambda_{\bar 13}).
    \end{tikzcd}\]
\end{corollary}

%% file: contactproduct.tex

%
%
%
%

\section{Elementary contact geometry of products}\label{appen: contact}

    In this section, we discuss contact products, reductions and correspondences, and their relations with products, reductions and correspondences of exact symplectic manifolds. We also explain a simple property of contact flows in the product that lifts Liouville flows in the product of Liouville manifolds. We will write $W_1, W_2, \dots$ for exact symplectic manifolds and $X_1, X_2, \dots$ for contact manifolds.

\subsection{Relation of symplectic and contact compositions}
    Consider exact symplectic manifolds $W_i$ (and $W_i^-$ be the same manifolds with opposite symplectic forms) and exact subsets $\ol{L}_{12} \looparrowright W_1^- \times W_2, \ol{L}_{23} \looparrowright W^-_2 \times W_3$, where being exact subsets means they come with prescribed liftings to subsets in the contactization
    $$L_{12} \subset W_1^- \times W_2 \times \bR, \quad L_{23} \subset W_2^- \times W_3 \times \bR.$$

\begin{definition}\label{def: contact composition final}
    Let $W_1, W_2$ and $W_3$ be exact symplectic manifolds and $L_{12} \subset W_1 \times W_2 \times \bR, L_{23} \subset W_2 \times W_3 \times \bR$ be any subsets in the contactization. Then we define their contact composition to be
    \begin{equation}\label{eq: contact composition final}
    L_{23} \circ L_{12} := \{(x_1, x_3, f_{23}(x_1, x_2) - f_{12}(x_2, x_3)) \mid (x_1, x_2) \in L_{12}, (x_2, x_3) \in L_{23}\}.
    \end{equation}
\end{definition}

    However, it is not immediately clear how to interpret the above ad-hoc definition by contact compositions as in Definition \ref{def: contact composition} since $W_1^- \times W_2 \times \bR$ and $W_2^- \times W_3 \times \bR$ do not appear to be products of contact manifolds as in Definition \ref{def: contact product}.
    We would like to understand their compositions via compositions of the products of contact manifolds $W_1 \times \bR$, $W_2 \times \bR$ and $W_3 \times \bR$. This requires us to add one dimension to the Legendrian $L_{12} \subset W_1^- \times W_2 \times \bR$ to get a Legendrian in $(W_1^- \times \bR) \,\widehat\times\, (W_2 \times \bR)$.

    We define Reeb extension of exact subsets in exact symplectic manifolds. In particular, the Reeb extension of exact Lagrangians in $W_1^- \times W_2$ or Legendrians in $W_1^- \times W_2 \times \bR$ will be Legendrians in the contact product $(W_1^- \times \bR) \,\widehat\times\, (W_2 \times \bR)$.

\begin{definition}\label{def: reeb extension}
    Let $L \looparrowright X_1 \times X_2$ be an exact subset with primitive $f_L: L \rightarrow \bR$. Let $R_i^s$ the Reeb flow on $X_i$. Then the Reeb extension of $L$ is defined by
    $$L_{ex} = \{[R_1^{-s}(x_1), r; R_2^{s-f_L(x_1, x_2)}(x_2), r] \mid (x_1, x_2) \in L\} \subset X_1 \,\widehat\times\, X_2.$$
\end{definition}

\begin{remark}
    In general, the Reeb extension of an exact locally closed subset is not necessarily a locally closed subset (due to the complicated dynamics of the Reeb flow). However, for $\epsilon > 0$ sufficiently small, if $L \subset X_1 \,\widehat\times\, X_2$ is locally closed, then the Reeb extension for $(-\epsilon, \epsilon)$
    $$L_{ex} = \{[R_1^{-s}(x_1), r; R_2^{s-f_L(x_1, x_2)}(x_2), r] \mid (x_1, x_2) \in L , s \in (-\epsilon, \epsilon) \} \subset X_1 \,\widehat\times\, X_2$$
    is also a locally closed subset.
\end{remark}

\begin{lemma}
    Let $X_1, X_2$ be contact manifold. The Reeb extension $L_{ex} \subset X_1 \,\widehat\times\, X_2$ of an exact submanifold $L \subset X_1 \,\widehat\times\, X_2$ is an isotropic submanifold.
\end{lemma}
\begin{proof}
    We know that $(r_1\alpha_1 + r_2\alpha_2) |_{L_{ex}} = r \alpha_1 + r \alpha_2 - r \alpha_2(R_2) df_L = 0$ because $\alpha_2(R_2) \equiv 1$ and thus $rdf_L = (r_1\alpha_1 + r_2\alpha_2)|_L$ (since $R_2$ is the Reeb vector field for the contact form $\alpha_2$).
\end{proof}

    The following example is the main class of examples we are interested in, namely Reeb extensions of exact subsets in the product of exact symplectic manifolds, as suggested in the beginning of the subsection.

\begin{example}\label{example reeb extension}
    Let $\underline L \looparrowright T^*N_1 \times T^*N_2$ be an exact subsets with primitive $f_L$, which lift to a subset
    $$L \subset T^*N_1 \times T^*N_2 \times \bR.$$
    Consider the embedding $L \subset (T^*N_1 \times 0) \,\widehat\times\, (T^*N_2 \times 0) \subset J^1N_1 \,\widehat\times\, J^1N_2$. Then under the standard Reeb vector fields $R_i = \partial/\partial {t_i}$, the Reeb extension of $L$ is 
    $$L_{ex} = \{[x_1, s, r; x_2, f_L(x_1, x_2) - s, r] \mid (x_1, x_2) \in L, s\in \bR, r \in \bR_{>0}\}.$$

    More generally, let $W_1, W_2$ be exact symplectic manifolds and $\underline L \subset W_1 \times W_2$ be an exact Lagrangian submanifold with primitive $f_L$, which lift to a subset
    $$L \subset W_1 \times W_2 \times \bR.$$
    Consider the embedding $L \subset (W_1 \times 0) \,\widehat\times\, (W_2 \times 0) \subset (W_1 \times \bR_{t_1}) \,\widehat\times\, (W_2 \times \bR_{t_2})$. Then under the standard Reeb vector fields $R_i = \partial/\partial {t_i}$, the Reeb extension of $L$ is 
    $$L_{ex} = \{[x_1, s, r; x_2, f_L(x_1, x_2) - s, r] \mid (x_1, x_2) \in L, s \in \bR, r \in \bR_{>0}\}.$$
\end{example}

    Let $W_1, W_2, W_3$ be exact symplectic manifolds. Let $L_{12} \subset W_1^- \times W_2$, $L_{23} \subset W_2^- \times W_3$ be exact subsets, or equivalently, subsets in the contactizations 
    $$L_{12} \subset W_1^- \times W_2 \times \bR, \quad L_{23} \subset W_2^- \times W_3 \times \bR.$$
    Now let $X_i = W_i \times \bR$, $X_i^- = W_i^- \times \bR$, and $L_{12,ex} \subset X_1^- \,\widehat\times\, X_2$, $L_{23,ex} \subset X_2^- \,\widehat\times\, X_3$ be their Reeb extensions as in the above example. We will show the composition $L_{13} \subset W_1^- \times W_3 \times \bR$ has Reeb extension $L_{13,ex}$ that equals the contact composition:
    $L_{13,ex} = L_{23,ex} \circ L_{12,ex} \hookrightarrow X_1^- \,\widehat\times\, X_3.$
    as in Definition \ref{def: contact product leg}.

    First, we compare the intersections of the Lagrangian product with the coisotropic $W_1^- \times \Delta_{W_2} \times W_3 \subset W_1^- \times W_2 \times W_2^- \times W_3$ and the intersections of the product of Reeb extensions with the coisotropic $X_1^- \,\widehat\times\, \Delta_{X_2} \,\widehat\times\, X_3 \subset X_1^- \,\widehat\times\, X_2 \,\widehat\times\, X_2^- \,\widehat\times\, X_3$ defined by Definition \ref{def: contact product}.

    It is important to notice that, once we take the Reeb extension of exact Lagrangians along the extra contact factors, the intersection between the product Lagrangian and the coisotropic in $W_1^- \times W_2 \times W_2^- \times W_3$ no longer correspond to Reeb chords of the Legendrians and the coisotropic in $X_1^- \,\widehat\times\, X_2 \,\widehat\times\, X_2^- \,\widehat\times\, X_3$. All the intersection points in the symplectic manifold will be visible through the extension by the Reeb vector field.

    As a warm-up exercise for the general case, we recall the following lemma:
    
\begin{lemma}\label{reeb chord Lagrangian intersection}
    Let $L, K \subset W \times \bR$ be Legendrians and $\ol{L}, \ol{K} \looparrowright W$ be the Lagrangian projections.
    Then there is a bijection between Lagrangian intersections of $\ol{L}$ and $\ol{K}$ and the Reeb chords of $L$ and $ K$. Moreover, action of the Lagrangian intersection corresponds to the length of the Reeb chords.
\end{lemma}
\begin{proof}
    Let $f_L$ and $f_K$ be the primitives of $L$ and $K$. Reeb chords between $L$ and $K$ are between $(x, f_L(x)) \in L$ and $(x, f_K(x)) \in K$, so $x \in \ol{L} \cap \ol{K}$ and the length is $f_K(x) - f_L(x)$.
\end{proof}

\begin{lemma}\label{reeb extension intersection}
    Let $W_1, W_2, W_3$ be exact symplectic manifolds, $X_1, X_2, X_3$ be the contactizations and ${L}_{12} \subset W_1^- \times W_2 \times \bR, {L}_{23} \subset W_2^- \times W_3 \times \bR$ be any subsets. There is a bijection between the intersections of the pair
    $$\big(\ol{L}_{12} \times \ol{L}_{23}, \, W_1^- \times \Delta_{W_2} \times W_3 \big)$$
    and the intersections of the following pair quotient by a free $\bR$-action
    $$\big(L_{12,ex} \,\widehat\times \,L_{23,ex},\, X_1^- \,\widehat\times\, \Delta_{X_2}\, \widehat\times\, X_3 \big).$$
\end{lemma}
\begin{proof}
    Let the primitives of the exact subsets $L_{12}$ and $L_{23}$ be $f_{12}$ and $f_{23}$. The contact product of the Reeb extensions of Legendrians is given by
    \begin{align*}
    L_{12,ex} \,\widehat\times\, L_{23,ex} = \{ &\, [x_1, x_2, s_{12}, f_{12}(x_1, x_2)-s_{12}, r_{12}, r_{12}; x_2', x_3, s_{23}, f_{23}(x_2', x_3)-s_{23}, r_{23}, r_{23}] \\
    & \mid (x_1, x_2) \in \underline L_{12}, (x_2', x_3) \in \underline L_{23}, s_{12}, s_{23} \in \bR, r_{12}, r_{23} \in \bR_{>0}\},
    \end{align*}
    and the coisotropic submanifold is given by 
    $$X_1 \,\widehat\times\, \Delta_{X_2} \,\widehat\times\, X_3 = \{[x_1, x_2, s_1, s_2, r_1, r_2; x_2, x_3, s_2, s_3, r_2, r_3] \mid (x_i, s_i) \in X_i, r_i \in \bR_{>0}\}.$$
    Therefore, the intersection points between $L_{12,ex} \,\widehat\times\, L_{23,ex}$ and $ X_1 \,\widehat\times\, \Delta_{X_2} \,\widehat\times\, X_3$ satisfy
    \begin{gather*}
    x_2 = x_2',\; r_1 = r_2 = r_3 = r_{12} = r_{23}, \; s_1 = s_{12}, \; s_2 = s_{23}, \\
    f_{12}(x_1, x_2) - s_{12} = s_2, \; f_{23}(x_2', x_3) - s_{23} = s_3.
    \end{gather*}
    This is determined by points $(x_1, x_2) \in L_{12}$ and $(x_2, x_3) \in L_{23}$, and $s_1, s_2, s_3$ are determined via the following relations, with $s_1$ being a free variable:
    $$s_2 = f_{12}(x_1, x_2) - s_1, \; s_3 = f_{23}(x_2, x_3) - s_2.$$
    This shows the bijection.
\end{proof}

    Then we can understand the relation between symplectic reductions and contact reductions, and in particular Lagrangian and Legendrian compositions.

\begin{lemma}\label{reeb extension composition}
    Let $W_1, W_2, W_3$ and $X_1, X_2, X_3$ be as above. Then the Legendrian composition of the Reeb extension $L_{23,ex} \circ L_{12,ex} \subset X_1^- \,\widehat\times\, X_3$ is the Reeb extension of the Lagrangian composition ${L}_{23} \circ {L}_{12} \subset W_1^- \times W_3 \times \bR$.
\end{lemma}
\begin{proof}
    Using Proposition \ref{reeb extension intersection} and Definition \ref{def: contact composition}, the Legendrian correspondence $L_{23,ex} \circ L_{12,ex}$ is given by
    \begin{align*}
    L_{23,ex} \circ L_{12,ex} = \{ & \,[x_1, s_1, r; x_3, s_3, r] \mid (x_1, x_2) \in L_{12}, (x_2, x_3) \in L_{23}, \\
    &s_3 - s_1 = f_{23}(x_2, x_3) - f_{12}(x_1, x_2), r \in \bR_{>0}\}.
    \end{align*}
    This is the Reeb extension of the Lagrangian composition ${L}_{23} \circ {L}_{12} \subset W_1^- \times W_3 \times \bR$ by Definitions \ref{def: contact composition final} and \ref{def: reeb extension}.
\end{proof}

\begin{example}\label{ex: geometry composition pt}
    Let $W_1 = W_3 = pt$ and $W_2 = T^*N_2$. Consider Legendrians $L_{12} \subset T^*N_2 \times \bR$ and $L_{23} \subset T^*N_2 \times \bR$ with primitives $f_{12}$ and $f_{23}$. Then
    $$J^1pt\, \,\widehat\times\, J^1N_2 \cong J^1(N_2 \times \bR),$$
    and the Reeb extensions of the Lagrangian $L_{12}$ is given by
    $$L_{12,ex} = \{[s, r; x_2, \xi_2, f_{12}(x_2, \xi_2) - s, -r] \mid (x_2, \xi_2, f_{12}(x_2, \xi_2)) \in L_{12}, s \in \bR, r \in \bR_{>0}\}.$$
    Therefore, the intersection between $L_{12,ex} \widehat{\times} L_{23,ex}$ and $J^1pt\,  \,\widehat\times\, \Delta_{J^1N_2} \,\widehat\times\, J^1pt$ are given by pairs
    \begin{gather*}
    [s, -r; x_2, \xi_2, f(x_2, \xi_2) - s, r; x_2', -\xi_2', f'(x_2', \xi_2') - s', -r'; s', r'] \\
    = [s_1, -r_1; x_2, \xi_2, s_2, r_2; x_2, -\xi_2, s_2, -r_2; s_3, r_3]
    \end{gather*}
    such that the following equations hold:
    \begin{gather*}
    s_1 = s, \; s_2 = f(x_2, \xi_2) - s = f'(x_2', \xi_2') - s', \; s_3 = s', \\
    r_1 = r_2 = r_3 = r = r', \; x_2 = x_2', \; \xi_2 = \xi_2'.
    \end{gather*}
    We get that $x_2, \xi_2, x_2', \xi_2'$ and $s - s'$ are determined by the following set of equations:
    $$x_2 = x_2', \; \xi_2 = \xi_2', \; f(x_2, \xi_2) = f'(x_2', \xi_2') + (s - s'),$$
    $s$ is a free variable, and $s_1, s_2, s_3$, $r_1, r_2, r_3$, $r$ and $r'$ are all uniquely determined by the above variables. This is an example for Proposition \ref{reeb extension intersection}.
    
    Then, by definition, the Legendrian composition $L_{23,ex} \circ L_{12,ex}$ is given by
    $$L_{23,ex} \circ L_{12,ex} = \{[s, -r; s + f'(x_2, \xi_2) - f(x_2, \xi_2), r] \mid (x_2, \xi_2) \in \ol{L}_{12} \cap \ol{L}_{23}, r \in \bR_{>0}\}.$$
    It is the Legendrian extension of the composition $L_{23} \circ L_{12} \subset pt\, \times pt$, whose primitive is determined by the actions of  the Lagrangian intersection $\ol{L}_{23} \cap \ol{L}_{12}$. This is an example for Corollary \ref{reeb extension composition}.
    
    Consider the Reeb extension $L_{23,ex} \circ L_{12,ex} \subset J^1pt \,\widehat\times\, J^1pt$. The Reeb chords between the pair of Legendrians 
    $$L_{23,ex} \circ L_{12,ex}, \Delta_{J^1pt}$$
    correspond to Lagrangian intersections between the pair $(\ol{L}_{23} \circ \ol{L}_{12}, \Delta_{pt})$, and the lengths of Reeb chords correspond to the actions of $\ol{L}_{12} \cap \ol{L}_{23}$ as in Lemma \ref{reeb chord Lagrangian intersection}.
\end{example}

    Finally, one can show by direct computation that taking Reeb extensions commutes with Reeb push-offs of the Legendrians. In particular, this shows that Reeb extensions are composable in the sense of Definition \ref{def: composable}:

\begin{lemma}\label{reeb extension composition pushoff}
    Let $W_1, W_2, W_3$ and $X_1, X_2, X_3$ be as above. Let $L_{12,\epsilon}$ and $L_{23,\epsilon'}$ be Reeb push-offs of $L_{12} \subset W_1^- \times W_2 \times \bR$ and $L_{23} \subset W_2^- \times W_3 \times \bR$. Then the Legendrian composition of the Reeb extension $L_{23,\epsilon',ex} \circ L_{12,\epsilon,ex} \subset X_1^- \,\widehat\times\, X_3$ is the Reeb extension of the Lagrangian composition $(L_{23} \circ L_{12})_{\epsilon+\epsilon'} = {L}_{23,\epsilon'} \circ {L}_{12,\epsilon} \subset W_1^- \times W_3 \times \bR$.
\end{lemma}
\begin{proof}
    Let the Reeb push-off be $L_{ij,\epsilon} = \{(x_i, x_j, f_{ij} + \epsilon) \mid (x_i, x_j, f_{ij}) \in L_{ij}\}$. Then with respect to the Reeb extension, we can write
    $L_{ij,\epsilon,ex} = \{[x_i, x_j, s_{ij}, f_{ij} + \epsilon - s_{ij}, r_{ij}, -r_{ij}] \mid (x_i, x_j, f_{ij}) \in L_{ij}, s_{ij} \in \bR, r_{ij} \in \bR_{>0}\}.$
    Hence the result follows directly from the formula in Lemma \ref{reeb extension composition}.
\end{proof}

\subsection{A lemma on limits of mixed conic flows}\label{appen: liouville}


    
    Let $W_1, W_2$ be Liouville manifolds. The Liouville flows on the product commutes:
    $$\varphi_{Z_1}^{t_1} \circ \varphi_{Z_2}^{t_2} = \varphi_{Z_2}^{t_2} \circ \varphi_{Z_1}^{t_1}, \quad t_1, t_2 \in \bR.$$
    Let $L \hookrightarrow W_1 \times W_2 \times \bR$ be a Legendrian subset. We will give a sufficient condition for limits under the mixed Liouville flows to be Legendrian subsets in the corresponding Liouville manifolds:

\begin{lemma}\label{lem: liouville flow pdff} 
    Let $W_1, W_2$ be Liouville manifolds and $X_1 = W_1 \times \bR, X_2 = W_2 \times \bR$ be their contactizations. Let ${L} \hookrightarrow W_1 \times W_2 \times \bR$ be a subset such that $\ol{L} = \ol{L}_1 \times \ol{L}_2$ outside a compact set, where $\ol{L}_i \subset W_i$ are eventually conical subsets. Then
    $$\lim_{t_1, t_2 \to \infty} \varphi_{Z_1}^{-t_1} \circ \varphi_{Z_2}^{-t_2}(\ol{L}) \subset \mathfrak{c}_{W_1, \partial L_1} \times \mathfrak{c}_{W_2, \partial L_2}.$$
    When $W_i$ are Weinstein and $L_i$ are stratified Legendrian, then the limit is stratified Lagrangian.
\end{lemma}
\begin{proof}
    Consider a compact subset $K \subset W_1 \times W_2$. Then, within the compact subset $K$, the statement is automatically true because by the definition of the core we have
    $$\lim_{t_1, t_2 \to \infty} \varphi_{Z_1}^{-t_1} \circ \varphi_{Z_2}^{-t_2} (L \cap K) \subset \mathfrak{c}_{W_1} \times \mathfrak{c}_{W_2}.$$
    Outside the compact subset $K \subset W_1 \times W_2$, one may note that
    $$\varphi_{Z_1}^{-t_1} \varphi_{Z_2}^{-t_2} (L \setminus K) = \varphi_{Z_1}^{-t_1}(L_1) \times \varphi_{Z_2}^{-t_2} (L_2) \setminus \varphi_{Z_1}^{-t_1} \varphi_{Z_2}^{-t_2} (K).$$
    Let $\partial{L}_{1}, \partial{L}_{2}$ denote the (conic) boundaries of $L_1$ and $L_2$. When $t_1, t_2 \gg 0$, we know that $\varphi_{Z_1}^{-t_1}(L_1)$ is contained in small neighbourhoods of $\mathfrak{c}_{W_1, \partial{L}_{1}}$, and $\varphi_{Z_2}^{-t_2}(L_2)$ is contained in small neighbourhoods of $\mathfrak{c}_{W_2,\partial{L}_{2}}$. Therefore, their contact product
    $$\lim_{t_1, t_2 \to \infty} \varphi_{Z_1}^{-t_1}(L_1) \times \varphi_{Z_2}^{-t_2} (L_2) \subset \mathfrak{c}_{W_1,\partial{L}_{1}} \times \mathfrak{c}_{W_2,\partial{L}_{2}}.$$
    This then completes the proof. When $W_i$ are Weinstein and $L_i$ are stratified Lagrangian, $\mathfrak{c}_{W_1,\partial{L}_{1}} \times \mathfrak{c}_{W_2,\partial{L}_{2}}$ is stratified Lagrangian.
\end{proof}

    Let $X_1, X_2$ be the contactizations of the Liouville manifolds $W_1, W_2$. On the contact product $X_1 \,\widehat\times\, X_2$, we can define the Hamiltonian vector field defined by ${Z}_1$ and ${Z}_2$. by the contact Hamiltonian
    $$H_1(x_1, t_1, r_1, x_2, t_2, r_2) = r_1t_1, \quad H_2(x_1, t_1, r_1, x_2, t_2, r_2) = r_2t_2.$$
    The Hamiltonians are Poisson commutative. Hence the contact lifts of the Liouville flows commute:
    $$\varphi_{Z_1}^{t_1} \circ \varphi_{Z_2}^{t_2} = \varphi_{Z_2}^{t_2} \circ \varphi_{Z_1}^{t_1}, \quad t_1, t_2 \in \bR.$$
    The above lemma also gives a sufficient condition for limits under the mixed Liouville flows of contact products to be Legendrian subsets:
    
\begin{lemma}\label{lem: liouville flow pdff contact}
    Let $W_1, W_2$ be Liouville manifolds and $X_1 = W_1 \times \bR, X_2 = W_2 \times \bR$ be their contactizations. Let $L \hookrightarrow X_1 \times X_2$ be a subset such that $L = {L}_1 \,\widehat\times\, {L}_2$ outside a compact set, where ${L}_i \subset X_i$ are eventually conical subsets. Then
    $$\lim_{t_1, t_2 \to \infty} \varphi_{Z_1}^{-t_1} \circ \varphi_{Z_2}^{-t_2}({L}) \subset \mathfrak{c}_{W_1, \partial L_1} \,\widehat\times\, \mathfrak{c}_{W_2, \partial L_2}.$$
    When $W_i$ are Weinstein and $L_i$ are stratified Legendrian, then the limit is stratified Legendrian.
\end{lemma}

    Note that when the exact Lagrangians or Legendrians are not of product type at infinity, the limit may not be Lagrangian or Legendrian any more:

\begin{example}\label{ex: limit liouville flow not pdff}
    Consider the diagonal $\Delta_{T^*\bR} \subset T^*\bR \times T^*\bR$ where the Liouville vector fields are $Z_i = \xi_i\partial/\partial \xi_i$. By considering different paths of the form $t_1 = \gamma t_2$ for all $\gamma \in \bR$, one can easily compute that
    $$\lim_{t_1, t_2 \to\infty}\varphi_{Z_1}^{-t_1}\varphi_{Z_2}^{-t_2}(\Delta_{T^*\bR}) = T^*\bR \times T^*\bR.$$
    In general, this is not true for Lagrangians $L$ that is tangent to the product Liouville vector field on the product ${Z}_{12}$. We will see in the next section that this will result in failure of the condition in gapped compositions Definition \ref{def: gap composition} that the limit should be pdff.
\end{example}

%
%
%
%
%

%% file: lagrangiancomposition.tex

\section{Composition of Lagrangian correspondences}\label{sec:main theorems}

    In this section, we prove our main theorems showing that sheaf quantization intertwines the algebraic composition of microsheaves and the geometric gapped compositions of microsheaves.

\subsection{Gapped compositions of Legendrian correspondences}\label{ssec:main theorems}
    First, 
    we consider Legendrian subsets in contact manifolds and define their microsheaf composition, following Definition \ref{def: gap micro composition}:

\begin{definition}\label{def: gap composition weinstein} 
    Let $X_1, X_2$ and $X_3$ be contact manifolds with Maslov data. Let $\Lambda_{12} \subset X_1^- \,\widehat\times\, X_2$ and $\Lambda_{23} \subset X_2^- \,\widehat\times\, X_3$ be relatively compact sufficiently Legendrian subsets that are composable. Then for $\SF_{12} \in \msh_{\Lambda_{12}}(\Lambda_{12})$ and $\SF_{23} \in \msh_{\Lambda_{23}}(\Lambda_{23})$, we pick some contact embedding $X_i \hookrightarrow S^*\bR^N$ for $1 \leq i \leq 3$ as in Section \ref{ssec: weinstein} and define the gapped composition by the gapped microlocal composition in Definition \ref{def: gap microlocal composition}
    $$\SF_{23} \circ_g \SF_{12} = m_{\Lambda_{13}}(w_{\Lambda_{23}}^+\SF_{23} \circ w_{\Lambda_{12}}^+\SF_{12}),$$
    where the doubling functor $w_{\Lambda_{12}}^+$ and $w_{\Lambda_{23}}^+$ are defined as the relative doubling in the sense of Theorem \ref{thm: relative-doubling} for the Legendrian thickening $\sigma_{12}$ and $\sigma_{23}$ of $\Lambda_{12}$ and $\Lambda_{23}$ under the contact embeddings in the sense of Section \ref{ssec: weinstein}.
\end{definition}

    Let $W_1, W_2$ and $W_3$ be Weinstein manifolds with Maslov data and $X_1, X_2, X_3$ be the contactizations. Let $L_{12} \subset W_1^- \times W_2 \times \bR$ and $L_{23} \subset W_2^- \times W_3 \times \bR$ be relatively compact universally sufficiently Legendrian subsets. Then we define the composition of microsheaves by compositions of their Reeb extensions:
    \begin{equation}\label{eq: composition weinstein}
    \begin{tikzcd}
    \msh_{L_{12}}(L_{12}) \otimes \msh_{L_{23}}(L_{23}) \ar[r, "\circ"] \ar[d, "\rotatebox{90}{$\sim$}" left] & \msh_{L_{13}}(L_{13}) \ar[d, "\rotatebox{90}{$\sim$}"] \\
    \msh_{L_{12,ex}}(L_{12,ex}) \otimes \msh_{L_{23,ex}}(L_{23,ex}) \ar[r, "\circ"] & \msh_{L_{13,ex}}(L_{13,ex}).
    \end{tikzcd}
    \end{equation}
    Here the vertical isomorphisms follow from the compatibility of the compositions in Lemma \ref{reeb extension composition} and the invariance of microsheaves in Lemma \ref{lem:contact-transform-main}.

    The main class of examples we will work with is universally sufficiently Legendrian subsets in the contactizations of Weinstein manifolds. 

\begin{lemma}\label{lem: weinstein composable}
    Let $W_1, W_2, W_3$ be Liouville manifolds with Maslov data. Let $L_{12} \subset W_1^- \times W_2 \times \bR$, $L_{23} \subset W_2^- \times W_3 \times \bR$ be relatively compact universally sufficiently Legendrian subsets with universally sufficiently Legendrian compositions. Then the Reeb extensions of the Legendrian thickenings $\mathfrak{c}_{W_1^- \times W_2,ex}^{L_{12}}, \mathfrak{c}_{W_2^- \times W_3,ex}^{L_{23}}$ are composable in the sense of Definitions \ref{def: composable} and
    $$\mathfrak{c}_{W_2^- \times W_3,ex}^{L_{23}} \circ \mathfrak{c}_{W_1^- \times W_2,ex}^{L_{12}} \subset \mathfrak{c}_{W_1^- \times W_3,ex}^{L_{13}}.$$
\end{lemma}
\begin{proof}
    Consider the composition by Reeb extensions in Definition \ref{def: reeb extension}. For their composability, we know $\mathfrak{c}_{W_1^-\times W_2,ex}^{L_{12}}$ and $\mathfrak{c}_{W_2^-\times W_3,ex}^{L_{23}}$ are composable in the sense of Definition \ref{def: composable} by Lemma \ref{reeb extension composition pushoff}. For their composition, we have
    $$\mathfrak{c}_{W_2^-\times W_3,ex}^{L_{23}} \circ \mathfrak{c}_{W_1^-\times W_2,ex}^{L_{12}} \subset \mathfrak{c}_{W_1^-\times W_3,ex}^{L_{13}}.$$
    by Lemma \ref{reeb extension composition pushoff}. This completes the proof.
\end{proof}
    
    The associativity of gapped compositions follows immediately from Proposition \ref{prop: gapped composition associative} and Corollary \ref{rem: associativity}:

\begin{proposition}\label{prop: gapped composition associative weinstein}
    Let $W_1, W_2, W_3$ and $W_4$ be Weinstein manifolds with Maslov data. Let $L_{12} \subset W_1^- \times W_2 \times \bR$, $L_{23} \subset W_2^- \times W_3 \times \bR$ and $L_{34} \subset W_3^- \times W_4 \times \bR$ be relatively compact universally sufficiently Legendrian subsets with universally sufficiently Legendrian compositions. Then for $\SF_{12} \in \msh_{L_{12}}(L_{12}), \SF_{23} \in \msh_{L_{23}}(L_{23})$ and $\SF_{34} \in \msh_{L_{34}}(L_{34})$, there is a canonical isomorphism
    $$(\SF_{34} \circ \SF_{23}) \circ \SF_{12} \simeq \SF_{34} \circ (\SF_{23} \circ \SF_{12}).$$
\end{proposition}
\begin{proof}
    It follows from Lemma \ref{lem: weinstein composable} that the Reeb extensions of $L_{12}, L_{23}$ and $L_{34}$ are composable in the sense of Definition \ref{def: composable pairwise}. Then the result follows from Proposition \ref{prop: gapped composition associative}.
\end{proof}

\begin{corollary}
    Let $W_1, \dots, W_k$ be Weinstein manifolds with Maslov data. Let $L_{i,i+1} \hookrightarrow W_i^- \times W_{i+1} \times \bR$ be relatively compact universally sufficiently Legendrian subsets with arbitrary universally sufficiently Legendrian compositions. Then for $\SF_{12} \in \msh_{L_{12}}(L_{12}), \dots, \SF_{k-1,k} \in \msh_{L_{k-1,k}}(L_{k-1,k})$, there are canonical isomorphisms for any $1 < l_1 < \dots < l_r < k$
    $$\SF_{k-1,k} \circ \dots \circ \SF_{12} \simeq (\SF_{k-1,k} \circ \dots \circ \SF_{l_r,l_r+1}) \circ \dots \circ (\SF_{l_1-1,l_1} \circ \dots \circ \SF_{12})$$
    that fit into a commutative diagram $\Delta_1^{\times k} \to \Cat_{st}$ which sends the vertices to $\msh_{L_{1,l_1}}(L_{1,l_1}) \otimes \dots \otimes \msh_{L_{l_r,k}}(L_{l_r,k})$ and the arrows to the composition on the right hand side of the equation.
\end{corollary}
\begin{proof}
    It follows from Lemma \ref{lem: weinstein composable} that the Reeb extensions of $L_{12}, L_{23}$ and $L_{34}$ are always composable. Then the corollary follows from Corollary \ref{rem: associativity cat}.
\end{proof}

    Consider Weinstein manifolds $W_1, W_2, W_3$ with Maslov data, and universally sufficiently Legendrians $L_{12} \subset W_1^- \times W_2 \times \bR$, $L_{23} \subset W_2^- \times W_3 \times \bR$ and $L_{13} = L_{12} \circ L_{23} \subset W_1^- \times W_3 \times \bR$ that are eventually conical. Then as described in Theorem \ref{thm: nearby = restriction at infty}, sheaf quantizations on $L_{ij}$ induce the following diagram:
    \begin{equation}
    \msh_{L_{ij}}(L_{ij}) \xleftarrow{i_{L_{ij}}^*} \msh_{\mathfrak{c}_{W_i^-\times W_j}^{L_{ij}}}(\mathfrak{c}_{W_i^-\times W_j}^{L_{ij}}) \xrightarrow{i_\infty^*} \msh_{\mathfrak{c}_{W_i^-\times W_j,\partial L_{ij}}}(\mathfrak{c}_{W_i^-\times W_j,\partial L_{ij}}).
    \end{equation}
    Considering the Reeb extensions, we can define their contact compositions as in Formula \eqref{eq: composition weinstein}. Our main theorm shows that the above diagram is compatible with compositions:

\begin{theorem}\label{thm: main composition}
    Let $W_1, W_2, W_3$ be Weinstein manifolds with Maslov data. Let $L_{12} \subset W_1^- \times W_2 \times \bR$ and $L_{23} \subset W_2^- \times W_3 \times \bR$ be eventually conic and universally sufficiently Legendrian and $L_{13} = L_{23} \circ L_{12} \subset W_1^- \times W_3 \times \bR$ also be universally sufficiently Legendrian. Then there is a commutative diagram
    \[\begin{tikzcd}
    \msh_{L_{12}}(L_{12}) \otimes \msh_{L_{23}}(L_{23}) \ar[r, "\circ_g"] & \msh_{L_{13}}(L_{13}) \\
    \msh_{\mathfrak{c}_{W_1^- \times W_2}^{L_{12}}}(\mathfrak{c}_{W_1^- \times W_2}^{L_{12}}) \otimes \msh_{\mathfrak{c}_{W_2^- \times W_3}^{L_{23}}}(\mathfrak{c}_{W_2^- \times W_3}^{L_{23}}) \ar[d, "i_\infty^* \otimes i_\infty^*" left] \ar[u, "i_{L_{12}}^* \otimes i_{L_{23}}^*"] \ar[r, "\circ_g"] & \msh_{\mathfrak{c}_{W_1^- \times W_3}^{L_{13}}}(\mathfrak{c}_{W_1^- \times W_3}^{L_{13}}) \ar[d, "i_\infty^*"] \ar[u, "i_{L_{13}}^*" right] \\
    \msh_{\mathfrak{c}_{W_1^-\times W_2,\partial L_{12}}}(\mathfrak{c}_{W_1^- \times W_2}) \otimes \msh_{\mathfrak{c}_{W_2^-\times W_3,\partial L_{23}}}(\mathfrak{c}_{W_2^- \times W_3}) \ar[r, "\circ_g"] & \msh_{\mathfrak{c}_{W_1^-\times W_3,\partial L_{13}}}(\mathfrak{c}_{W_1^- \times W_3}).
    \end{tikzcd}\]
\end{theorem}
\begin{proof}
    First, we show the commutativity of the upper square. This follows from the commutativity of gapped compositions and restrictions in the target. Consider the composition by Reeb extensions in Definition \ref{def: reeb extension}. By Lemma \ref{lem: weinstein composable}, we know $\mathfrak{c}_{W_1^-\times W_2,ex}^{L_{12}}$ and $\mathfrak{c}_{W_2^-\times W_3,ex}^{L_{23}}$ are composable in the sense of Definition \ref{def: composable}, and
    $$\mathfrak{c}_{W_2^-\times W_3,ex}^{L_{23}} \circ \mathfrak{c}_{W_1^-\times W_2,ex}^{L_{12}} \subset \mathfrak{c}_{W_1^-\times W_3,ex}^{L_{13}}.$$
    Moreover, we know that in the above composition, the preimage of the cone of $L_{13,ex}$ is equal to the cone of $L_{12,ex} \,\widehat\times\, L_{23,ex}$. Therefore, by Theorem \ref{thm: gapped composition restrict}, we have the commutative diagram
    \[\begin{tikzcd}
    \msh_{L_{12}}(L_{12}) \otimes \msh_{L_{23}}(L_{23}) \ar[r, "\circ_g"] & \msh_{L_{13}}(L_{13}) \\
    \msh_{\mathfrak{c}_{W_1^- \times W_2}^{L_{12}}}(\mathfrak{c}_{W_1^- \times W_2}^{L_{12}}) \otimes \msh_{\mathfrak{c}_{W_2^- \times W_3}^{L_{23}}}(\mathfrak{c}_{W_2^- \times W_3}^{L_{23}}) \ar[u, "i_{L_{12}}^* \otimes i_{L_{23}}^*"] \ar[r, "\circ_g"] & \msh_{\mathfrak{c}_{W_1^- \times W_3}^{L_{13}}}(\mathfrak{c}_{W_1^- \times W_3}^{L_{13}}) \ar[u, "i_{L_{13}}^*" right].
    \end{tikzcd}\]
    
    Then, we show the commutativity of the lower square. Note that the same argument does not apply as the preimage of $\mathfrak{c}_{W_1^- \times W_3,ex}$ is not contained in the product $\mathfrak{c}_{W_2^- \times W_3,ex} \,\widehat\times\, \mathfrak{c}_{W_1^- \times W_2,ex}$. We therefore construct a nearby cycle functor by a contact flow that realizes the restriction functor. Consider the contact lift of the Liouville flows by $Z_{ij} = Z_{-\lambda_i \times \lambda_j} + t\partial/\partial t$ inside $W_i^- \times W_j \times T^*\bR$. By Lemmas \ref{ss gapped for doubling} and \ref{lem: liouville flow pdff contact}, we know that the Legendrian thickening of the contact movies always have gapped composition in the sense of Definition \ref{def: gap composition}. Then by Lemma \ref{lem: weinstein composable}, we know that the contact movies $(\mathfrak{c}_{W_1^-\times W_2}^{L_{12}})_{Z_{12}}$ and $(\mathfrak{c}_{W_2^-\times W_3}^{L_{23}})_{Z_{23}}$ have gapped microlocal composition in Definition \ref{def: gap micro composition} and 
    $$\psi_{12}(\mathfrak{c}_{W_1^-\times W_2}^{L_{12}})_{Z_{12}} \subset \mathfrak{c}_{W_1^-\times W_2,\partial L_{12}} \times \sqcup_{0,\infty}, \quad \psi_{23}(\mathfrak{c}_{W_2^-\times W_3}^{L_{23}})_{Z_{23}} \subset \mathfrak{c}_{W_2^-\times W_3,\partial L_{23}} \times \sqcup_{0,\infty},$$
    where $\sqcup_{0,\infty} = 0_{\bR_{\geq 0}} \times (0 \times \bR_{>0}) \subset T^*\bR$ as in the proof of Proposition \ref{thm: quantize nearby = microlocalize} and Theorem \ref{thm: nearby = restriction at infty}.
    Theorem \ref{thm: nearby = restriction at infty} implies that the gapped nearby cycle functors are induced by restrictions at infinity. Therefore, 
    we can apply Theorem \ref{thm:gapped composition microsheaf} to conclude that the lower square of the diagram commutes:
    \[\begin{tikzcd}
    \msh_{\mathfrak{c}_{W_1^- \times W_2}^{L_{12}}}(\mathfrak{c}_{W_1^- \times W_2}^{L_{12}}) \otimes \msh_{\mathfrak{c}_{W_2^- \times W_3}^{L_{23}}}(\mathfrak{c}_{W_2^- \times W_3}^{L_{23}}) \ar[d, "i_\infty^* \otimes i_\infty^*" left] \ar[r, "\circ_g"] & \msh_{\mathfrak{c}_{W_1^- \times W_3}^{L_{13}}}(\mathfrak{c}_{W_1^- \times W_3}^{L_{13}}) \ar[d, "i_\infty^*"] \\
    \msh_{\mathfrak{c}_{W_1^-\times W_2,\partial L_{12}}}(\mathfrak{c}_{W_1^- \times W_2}) \otimes \msh_{\mathfrak{c}_{W_2^-\times W_3,\partial L_{23}}}(\mathfrak{c}_{W_2^- \times W_3}) \ar[r, "\circ_g"] & \msh_{\mathfrak{c}_{W_1^-\times W_3,\partial L_{13}}}(\mathfrak{c}_{W_1^- \times W_3}).
    \end{tikzcd}\]
    This completes the proof.
\end{proof}
\begin{remark}
    Moreover, when $L_{12}, L_{23}$ and $L_{23}$ are $\epsilon$-thin, the upper commutative square in the main theorem also factors as the following commutative diagram by Theorem \ref{thm: nearby = restriction at infty}, where the bottom nearby cycle functors define equivalences
    \[\begin{tikzcd}
    \msh_{L_{12}}(L_{12}) \otimes \msh_{L_{23}}(L_{23}) \ar[r, "\circ_g"] & \msh_{L_{13}}(L_{13}) \\
    \Sh_{(\underline L_{12} \times \cup_{0,\infty})_{\cup,\epsilon}^+}(\bR^N) \otimes \Sh_{(\underline L_{23}\times \cup_{0,\infty})_{\cup,\epsilon}^+}(\bR^N) \ar[d, "\psi_{12} \otimes \psi_{23}" left] \ar[u, "m_{L_{12} \circ L_{23}}"] \ar[r, dashed] & \Sh_{(\underline L_{13}\times \cup_{0,\infty})_{\cup,\epsilon}^+}(\bR^N) \ar[d, "\psi_{13}"] \ar[u, "m_{L_{13}}" right] \\
    \msh_{\mathfrak{c}_{W_1^- \times W_2}^{L_{12}}}(\mathfrak{c}_{W_1^- \times W_2}^{L_{12}}) \otimes \msh_{\mathfrak{c}_{W_2^- \times W_3}^{L_{23}}}(\mathfrak{c}_{W_2^- \times W_3}^{L_{23}}) \ar[r, "\circ_g"] & \msh_{\mathfrak{c}_{W_1^- \times W_3}^{L_{13}}}(\mathfrak{c}_{W_1^- \times W_3}^{L_{13}}).
    \end{tikzcd}\]
\end{remark}

\begin{corollary}\label{thm: main composition higher} 
    Let $W_1, W_2, \dots, W_k$ be Weinstein manifolds with Maslov data. Let $L_{12} \subset W_1^- \times W_2 \times \bR, \dots, L_{k-1,k} \subset W_{k-1}^- \times W_k \times \bR$ and $L_{ij} = L_{kj} \circ L_{ik} \subset W_i^- \times W_j \times \bR$ be eventually conic universally sufficient Legendrians. Then there is commutative diagram of natural transformations
    \[\begin{tikzcd}[column sep=15pt]
    \msh_{L_{12}}(L_{12}) \otimes \dots \otimes \msh_{L_{k-1,k}}(L_{k-1,k}) \ar[r, "\circ_g"] & \msh_{L_{1k}}(L_{1k}) \\
    \msh_{\mathfrak{c}_{W_1^- \times W_2}^{L_{12}}}(\mathfrak{c}_{W_1^- \times W_2}^{L_{12}}) \otimes \dots \otimes \msh_{\mathfrak{c}_{W_{k-1}^- \times W_k}^{L_{k-1,k}}}(\mathfrak{c}_{W_{k-1}^- \times W_k}^{L_{k-1,k}}) \ar[d, "i_\infty^* \otimes \dots \otimes i_\infty^*" left] \ar[u, "i_{L_{12}}^* \otimes \dots \otimes i_{L_{k-1,k}}^*"] \ar[r, "\circ_g"] & \msh_{\mathfrak{c}_{W_1^- \times W_k}^{L_{1k}}}(\mathfrak{c}_{W_1^- \times W_k}^{L_{1k}}) \ar[d, "i_\infty^*"] \ar[u, "i_{L_{1k}}^*" right] \\
    \msh_{\mathfrak{c}_{W_1^-\times W_2,\partial L_{12}}}(\mathfrak{c}_{W_1^- \times W_2}) \otimes \dots \otimes \msh_{\mathfrak{c}_{W_{k-1}^-\times W_k,\partial L_{k-1,k}}}(\mathfrak{c}_{W_{k-1}^- \times W_k}) \ar[r, "\circ_g"] & \msh_{\mathfrak{c}_{W_1^-\times W_k,\partial L_{1k}}}(\mathfrak{c}_{W_1^- \times W_k}),
    \end{tikzcd}\]
    where each horizontal arrow is interpreted as a coherent diagram $\Delta_1^{\times k-1} \to \Cat_{st}$ satisfying the Segal condition and each square defines 2-morphisms that are equivalences.
\end{corollary}
\begin{proof}
    This follows from Corollaries \ref{rem: associativity} and \ref{rem: gapped micro composition higher}.  
\end{proof}

\begin{example} The result of Theorem \ref{thm: main composition} is already new for cotangent bundles; let us unpack it.  
Take $W_1 = T^*M_1$, $W_2 = T^*M_2$ and $W_3 = T^*M_3$, and exact Lagrangians $L_{12} \subset T^*(M_1 \times M_2)$, $L_{23} \subset T^*(M_2 \times M_3)$ and $L_{13} = L_{12} \circ L_{23}$. Following Example \ref{ex: guillermou quantize}, we can consider sheaf quantizations in $\Sh_{L_{ij}}(M_i \times M_j \times \bR)$ where we use $L_{ij} \subset T^*(M_i \times M_j) \times \bR$ for the Legendrian lift of the exact Lagrangians. Then Theorem \ref{thm: main composition} says that the following diagram commutes:
    \begin{equation}\label{eq: composition quantization cotangent}
    \begin{tikzcd}
    \msh_{L_{12}}(L_{12}) \otimes \msh_{L_{23}}(L_{23}) \ar[r, "\circ_g"] & \msh_{L_{13}}(L_{13}) \\    \Sh_{{L}_{12}}(M_1 \times M_2 \times \bR) \otimes \Sh_{{L}_{23}}(M_2 \times M_3 \times \bR) \ar[r, "\circ_g"] \ar[d, "i_\infty^* \otimes i_\infty^*" left] \ar[u, "m_{L_{12}} \otimes m_{L_{23}}"] & \Sh_{\tilde{L}_{13}}(M_1 \times M_3 \times \bR^2) \ar[d, "i_\infty^*"] \ar[u, "m_{L_{13}}" right] \\
    \Sh_{\partial L_{12}}(M_1 \times M_2) \otimes \Sh_{\partial L_{23}}(M_2 \times M_3) \ar[r, "\circ"] & \Sh_{\partial L_{13}}(M_1 \times M_3).
    \end{tikzcd}
    \end{equation}

    We now unpack the definitions of the microsheaf composition functors and rewrite the above diagram in a purely sheaf theoretic way. Consider the contact product $(T^*M_i \times \bR) \,\widehat\times\, (T^*M_j \times \bR)$, the Reeb extension ${L}_{ij,ex} \subset T^*(M_i \times M_j) \times T^*\bR \times \bR$ is defined by the product $L_{ij} \times \bR$. Let $s_{ij}: M_i \times M_j \times \bR^2 \to M_i \times M_j \times \bR, (x_i, x_j, t_i, t_j) \mapsto (x_i, x_j, t_i - t_j)$. Then there is an equivalence $s_{ij}^*: \Sh_{{L}_{12}}(M_1 \times M_2 \times \bR) \xrightarrow{\sim} \Sh_{{L}_{12,ex}}(M_1 \times M_2 \times \bR^2)$. We claim that the gapped (microlocal) composition functor can be realized as the composition
    \begin{gather*}
    \circ_g: \Sh_{{L}_{12,ex}}(M_1 \times M_2 \times \bR^2) \otimes \Sh_{{L}_{23,ex}}(M_2 \times M_3 \times \bR^2) \to \Sh_{{L}_{13,ex}}(M_1 \times M_3 \times \bR^2) \\
    \SF_{12} \circ_g \SF_{23} = s_{13*}(s_{12}^*\SF_{12} \circ s_{23}^*\SF_{23}).
    \end{gather*}
    Indeed, to show that the above composition agrees with the composition defined using doubling, we can consider the canonical contact embedding
    $$T^*(M_i \times \bR) \times \bR \hookrightarrow S^*(M_i \times \bR^2), \quad (x, t, \xi, \tau) \mapsto (x, t - f'(\tau), f(\tau), \xi, \tau, 1).$$
    for some smooth function $f: \bR \to \bR_{\geq 0}$ that is increasing on $\bR_{>0}$ and decreasing on $\bR_{<0}$. Then we have the embedding of contact products in Definition \ref{def: contact product} 
    $$(T^*(M_i \times \bR) \times \bR) \,\widehat\times\, (T^*(M_j \times \bR) \times \bR) \hookrightarrow S^*(M_i \times \bR^2) \,\widehat\times\, S^*(M_j \times \bR^2).$$
    Since the essential image of the doubling functor has zero stalk on $M_i \times M_j \times \{(t_i, t_j) \mid t_i - t_j > 0\}$, the microlocalization along $M_i \times M_j \times \Delta$ can be identified with the restriction:
    $$m_{0_{M_i \times M_j \times \bR}} = i_{M_i \times M_j \times \Delta}^*: \Sh_{(\underline{L}_{ij} \times N^*_+\Delta_\bR)_{\cup,\epsilon}^+}(M_i \times M_j \times \bR^2) \to \Sh_{L_{ij}}(M_i \times M_j \times \bR).$$
    Then we can compute by proper base change that the composition commutes with microlocalization, i.e. $\SF_{12} \circ_g \SF_{23} = s_{13*}(s_{12}^*\SF_{12} \circ s_{23}^*\SF_{23}).$
    This completes the reformulation of the middle arrow in Formula \eqref{eq: composition quantization cotangent} in terms of sheaf theoretic operations.
    
    Note that the lower square is not simply a  proper base change:  we are taking restriction functors not for the first and last component but also in the middle components.  We can prove it without explicit mention of microsheaves (i.e. appealing to our sheaf results which are precursors to  Theorem \ref{thm: main composition} rather than to the theorem itself) as follows.  The commutativity of the lower square can be translated into the commutativity of compositions and nearby cycle functors as indicated in Example \ref{ex: guillermou quantize}.  This commutativity then reduces to Example \ref{ex: cotangent bundle composition} as a special case of the gapped composition Theorem \ref{thm:gapped composition}.
\end{example}

    When the Lagrangian correspondences satisfy an additional assumption, we can show the compatibility between compositions and stop removals, which shows that for conic microsheaves, algebraic compositions intertwine with geometric gapped compositions:

\begin{theorem}\label{thm: main composition stop removal}
    Let $W_1, W_2, W_3$ be Weinstein manifolds with Maslov data. Let $L_{12} \subset W_1^- \times W_2 \times \bR$ and $L_{23} \subset W_2^- \times W_3 \times \bR$ be eventually conic universally sufficient Legendrians and $L_{13} = L_{23} \circ L_{12} \subset W_1^- \times W_3 \times \bR$. When
    \begin{equation}\label{stop removal hypothesis main}
    \mathfrak{c}_{W_i^- \times W_j,\partial L_{ij}} \circ \mathfrak{c}_{W_i} \subset \mathfrak{c}_{W_j}
    \end{equation}
    for any $i \leq j$, we have a further commutative diagram
    \[\begin{tikzcd}
    \msh_{\mathfrak{c}_{W_1^-\times W_2,\partial L_{12}}}(\mathfrak{c}_{W_1^- \times W_2}) \otimes \msh_{\mathfrak{c}_{W_2^-\times W_3,\partial L_{23}}}(\mathfrak{c}_{W_2^- \times W_3}) \ar[r, "\circ_g"] \ar[d, "\iota_{\partial L_{12}}^* \otimes \iota_{\partial L_{23}}^*" left] & \msh_{\mathfrak{c}_{W_1^-\times W_3,\partial L_{13}}}(\mathfrak{c}_{W_1^- \times W_3}) \ar[d, "\iota_{\partial L_{13}}^*"] \\
    \msh_{\mathfrak{c}_{W_1^-\times W_2}}(\mathfrak{c}_{W_1^- \times W_2}) \otimes \msh_{\mathfrak{c}_{W_2^-\times W_3}}(\mathfrak{c}_{W_2^- \times W_3}) \ar[r, "\circ_a"] & \msh_{\mathfrak{c}_{W_1^-\times W_3}}(\mathfrak{c}_{W_1^- \times W_3}).
    \end{tikzcd}\]
\end{theorem}
\begin{proof}
    Since $\mathfrak{c}_{W_i^- \times W_j,\partial L_{ij}} \circ \mathfrak{c}_{W_i} \subset \mathfrak{c}_{W_j}$, we obtain a further commutative diagram by directly applying Theorem \ref{thm: gap composition wrap kunneth}. Finally, for iterated compositions, we can show that they fit into a commutative diagram by Corollary \ref{rem: gap composition kunneth higher}. This completes the proof.
\end{proof}

\begin{remark}\label{rem: gao proper condition}
    The interaction of wrapping and Floer-theoretic composition is studied in Gao's work \cite{Gao1,Gao2}.  
    The condition in Formula \eqref{stop removal hypothesis main} is related to Gao's assumption that the Lagrangian correspondence projects to the target properly in the case of Liouville manifolds with no stops \cite{Gao1} in the following sense: Any Lagrangian correspondence that projects to the domain properly satisfies our assumption.
    When $W_i$ and $W_j$ are Weinstein manifolds, we know that $\mathfrak{c}_{W_i^- \times W_j, \partial L_{ij}} \circ \mathfrak{c}_{W_i} \subset \mathfrak{c}_{W_j}$ if and only if the projection $L_{ij} \to W_i$ is proper: First, we know that the projection of $\mathfrak{c}_{W_i^- \times W_j, \partial L_{ij}} \circ \mathfrak{c}_{W_i}$ contains the projection of $\partial L_{ij} \cap (\mathfrak{c}_{W_i} \times \partial W_j)$; for the projection to be contained in $\mathfrak{c}_{W_j}$, it is therefore necessary and sufficient that $\partial L_{ij} \cap (\mathfrak{c}_{W_i} \times \partial W_j) = \varnothing$. Then, this implies that $\partial L_{ij} \times \bR_{>0}$ is transverse to the fiber of the projection $W_i \times W_j \to W_i$. This means the projection $L_{ij} \to W_i$ is proper, which is exactly the condition imposed by Gao \cite{Gao1} except that the domain and the target are switched (which is because the functors Gao and we consider are adjoint to each other).

    Finally, note that Gao's properness assumption is used by him in two different places: First, it ensures that compositions of Lagrangians commute with large positive Hamiltonian perturbations near the boundary (i.e. positive wrappings); second, it also ensures that the colimit preserving functor induced by the Lagrangian in fact preserves compact objects (equivalently, the right adjoint functor preserves limits and colimits). Here we have used the hypothesis only for the first purpose; and have not discussed here when our functors preserve compact objects. 
\end{remark}

\begin{corollary}\label{thm: main composition stop removal higher}
    Let $W_1, W_2, \dots, W_k$ be Weinstein manifolds with Maslov data. Let $L_{12} \subset W_1^- \times W_2 \times \bR, \dots, L_{k-1,k} \subset W_{k-1}^- \times W_k \times \bR$ and $L_{ij} = L_{kj} \circ L_{ik} \looparrowright W_1^- \times W_3 \times \bR$ be eventually conic universally sufficient Legendrians. When
    \begin{equation}
    \mathfrak{c}_{W_i^- \times W_j,\partial L_{ij}} \circ \mathfrak{c}_{W_i} \subset \mathfrak{c}_{W_j}
    \end{equation}
    for any $i \leq j$, we have a further commutative diagram
    \[\begin{tikzcd}[column sep=15pt]
    \msh_{\mathfrak{c}_{W_1^-\times W_2,\partial L_{12}}}(\mathfrak{c}_{W_1^- \times W_2}) \otimes \dots \otimes \msh_{\mathfrak{c}_{W_{k-1}^-\times W_k,\partial L_{k-1,k}}}(\mathfrak{c}_{W_{k-1}^- \times W_k}) \ar[r, "\circ_g"] \ar[d, "\iota_{\partial L_{12}}^* \otimes \dots \otimes \iota_{\partial L_{k-1,k}}^*" left] & \msh_{\mathfrak{c}_{W_1^-\times W_k,\partial L_{1k}}}(\mathfrak{c}_{W_1^- \times W_k}) \ar[d, "\iota_{\partial L_{1k}}^*"] \\
    \msh_{\mathfrak{c}_{W_1^-\times W_2}}(\mathfrak{c}_{W_1^- \times W_2}) \otimes \dots \otimes  \msh_{\mathfrak{c}_{W_{k-1}^-\times W_k}}(\mathfrak{c}_{W_{k-1}^- \times W_k}) \ar[r, "\circ_a"] & \msh_{\mathfrak{c}_{W_1^-\times W_k}}(\mathfrak{c}_{W_1^- \times W_k}),
    \end{tikzcd}\]
    where each horizontal arrow is interpreted as a coherent diagram $\Delta_1^{\times k-1} \to \Cat_{st}$ satisfying the Segal condition and each square defines 2-morphisms that are equivalences.
\end{corollary}

\begin{proof}
    This follows from Corollaries \ref{rem: gapped micro composition higher} and \ref{rem: gap composition kunneth higher}.
\end{proof}

    Similarly, using Theorem \ref{thm:gapped hom composition}, one can deduce the following theorem on hom compositions of Lagrangian correspondences.

\begin{theorem}\label{thm: main composition hom}
    Let $W_1, W_2, W_3$ be Weinstein manifolds with Maslov data. Let $L_{12} \subset W_1 \times W_2 \times \bR$ and $L_{23} \subset W_2 \times W_3 \times \bR$ be any eventually conical universally sufficiently Legendrians and $L_{\bar 13} = L_{23} \circ (-L_{12}) \subset W_1^- \times W_3 \times \bR$. Then there is a commutative diagram
    \[\begin{tikzcd}
    \msh_{L_{12}}(L_{12}) \otimes \msh_{L_{23}}(L_{23}) \ar[r, "\sHom^\circ_g"] & \msh_{L_{\bar 13}}(L_{\bar 13}) \\    \msh_{\mathfrak{c}_{W_1 \times W_2}^{L_{12}}}(\mathfrak{c}_{W_1 \times W_2}^{L_{12}}) \otimes \msh_{\mathfrak{c}_{W_2 \times W_3}^{L_{23}}}(\mathfrak{c}_{W_2 \times W_3}^{L_{23}}) \ar[d, "i_\infty^* \otimes i_\infty^*" left] \ar[r, "\sHom^\circ_g"] \ar[u, "i_{L_{12}}^* \otimes i_{L_{23}}^*"] & \msh_{\mathfrak{c}_{W_1^- \times W_3}^{L_{\bar 13}}}(\mathfrak{c}_{W_1 \times W_3}^{L_{\bar 13}}) \ar[d, "i_\infty^*"] \ar[u, "i_{L_{\bar 13}}^*" right] \\
    \msh_{\mathfrak{c}_{W_1\times W_2,\partial L_{12}}}(\mathfrak{c}_{W_{12}}) \otimes \msh_{\mathfrak{c}_{W_2\times W_3,\partial L_{23}}}(\mathfrak{c}_{W_2 \times W_3}) \ar[r, "\sHom^\circ_g"] & \msh_{\mathfrak{c}_{W_1^-\times W_3,\partial L_{\bar 13}}}(\mathfrak{c}_{W_1^- \times W_3}) \\
    \end{tikzcd}\]
    Moreover, when $\mathfrak{c}_{W_i \times W_j,\partial L_{ij}} \circ \mathfrak{c}_{W_i} \subset \mathfrak{c}_{W_j}$, we have a further commutative diagram
    \[\begin{tikzcd}
    \msh_{\mathfrak{c}_{W_1\times W_2,\partial L_{12}}}(\mathfrak{c}_{W_{12}}) \otimes \msh_{\mathfrak{c}_{W_2\times W_3,\partial L_{23}}}(\mathfrak{c}_{W_2 \times W_3}) \ar[r, "\sHom^\circ_g"] \ar[d, "\iota_{\partial L_{12}}^* \otimes \iota_{\partial L_{\bar 23}}^*" left] & \msh_{\mathfrak{c}_{W_1^-\times W_3,\partial L_{13}}}(\mathfrak{c}_{W_1^- \times W_3}) \ar[d, "\iota_{\partial L_{\bar 13}}^*"] \\
    \msh_{\mathfrak{c}_{W_1\times W_2}}(\mathfrak{c}_{W_{12}}) \otimes \msh_{\mathfrak{c}_{W_2\times W_3}}(\mathfrak{c}_{W_2 \times W_3}) \ar[r, "\sHom^\circ_a"] & \msh_{\mathfrak{c}_{W_1^-\times W_3}}(\mathfrak{c}_{W_1^- \times W_3}).
    \end{tikzcd}\]
\end{theorem}

    Finally, we explain the necessity to develop the relative version of nearby cycle functors (as opposed to applying the gapped nearby cycle functor on each fiber in \cite[Section 8]{Nadler-Shende}).

\begin{remark}\label{rem: necessity not restrict to fiber} 
    Instead of using Theorem \ref{thm:gapped composition microsheaf} (based on Theorem \ref{thm:gapped composition}) to obtain the isomorphism of sheaves, one may consider restricting to each fiber, use the Liouville flows only on the second factor $X_2 \,\widehat\times\, X_2 \subset S^*\bR^N \,\widehat\times\, S^*\bR^N$, and then try to apply Theorem \ref{thm: gapped specialization} (based on Theorem \ref{nearby fully faithful}) to each fiber and then use Proposition \ref{prop: nearby commute micro} (based on Proposition \ref{nearby commute with restrict}) to deduce the compatibility between the global and fiberwise nearby cycle functors. 
    
    This approach fails for at least two reasons. First, the gapped condition does not hold fiberwise after passing to the contact reductions as explained in Remark \ref{rem: necessity not restrict to fiber 1}. Second, the Lagrangian correspondence are not conic with respect to the Liouville flows on the first and third factors $X_2 \,\widehat\times\, X_2 \subset S^*\bR^N \,\widehat\times\, S^*\bR^N$. Indeed, Lemma \ref{lem: liouville flow pdff} and Example \ref{ex: limit liouville flow not pdff} show that the limit of the Lagrangian correspondence under the Liouville flow of only the second factor is not a Lagrangian subset in general, and so we cannot invoke Theorems \ref{nearby fully faithful} and \ref{nearby commute with restrict}.
\end{remark}

\subsection{Maslov data and  gapped composition}

    Given a Maslov data for a contact manifold $X$, we can consider the secondary Maslov data for a Legendrian submanifold $\Lambda \subset X$ as follows \cite[Remark 11.20]{Nadler-Shende}.

    Let $X$ be a contact manifold with Maslov data $\sigma$. Then for a Legendrian $L \subset X$, we have a natural inclusion
    $$\msh_{L_\sigma} \rightarrow \msh_{X_\sigma}|_{L}.$$
    Consider the standard Weinstein neighborhood of $L \subset X$ which is contactomorphic to $J^1L$. There is a natural fiber polarization which determines a Maslov data $\tau$ on $J^1L$. We have a natural isomorphism between sheaves of categories
    $$\msh_{L_\tau} \simeq \Loc_L.$$
    Then the secondary Maslov obstruction of $L \subset X$ is the map $\sigma - \tau : L \to BPic(\SC)$. Since $\msh_{L_\tau} \simeq \Loc_L$ is the constant sheaf of categories, secondary Maslov data are classified by global simple objects
    $\SF \in \msh_{L_\sigma}(L)$, where an object $\SF$ is simple if the stalk $\SF_p = 1_\SC$. In fact, a global simple object defines an equivalence by
    $$\msh_{L_\sigma}(L) \xrightarrow{\sim} \Loc(L), \quad \SG \mapsto \mu hom(\SF, \SG)$$
    where $\mu hom$ is the internal hom of the sheaf of microsheaves as in \cite[Definition 4.1]{KS}.

    We will now show that secondary Maslov data on given Lagrangians naturally induce  secondary Maslov data in their composition under certain non-degeneracy condition on the Lagrangian intersections.

\begin{theorem}
    Let $W_1, W_2, W_3$ be Weinstein manifolds with Maslov data. Let $L_{12} \subset W_1^- \times W_2 \times \bR$ and $L_{23} \subset W_2^- \times W_3 \times \bR$ be eventually conical smooth Legendrians. Suppose 
    $$L_{12} \times L_{23} \pitchfork W_1^- \times \Delta_{W_2} \times W_3.$$
    Then for any simple objects $\SF_{12} \in \msh_{L_{12}}(L_{12})$ and $\SF_{23} \in \msh_{L_{23}}(L_{23})$, $\SF_{12} \circ \SF_{23} \in \msh_{L_{13}}(L_{13})$ is also simple. In particular, under the identification $\msh_{L_{ij}}(L_{ij}) \simeq \Loc(L_{ij})$, the composition
    $$\Loc(L_{12}) \otimes \Loc(L_{23}) \longrightarrow \Loc(L_{13})$$
    is defined by the pull-back functor via $L_{13} \xrightarrow{\sim} (L_{12} \times L_{23}) \cap (W_1^- \times \Delta_{W_2} \times W_3)$.
\end{theorem}
\begin{proof}
    Let $X_i = W_i \times \bR$ be the contactization and $L_{ij,ex} \subset X_i \,\widehat\times\, X_j$ be the Reeb extension of $L_{ij} \subset W_i \times W_j$ as in Definition \ref{def: reeb extension}. Since $L_{12} \times L_{23} \pitchfork W_1^- \times \Delta_{W_2} \times W_3$, we know by Lemma \ref{reeb extension intersection} that for the Legendrian thickening of the Reeb extension $L_{12,\sigma_{12}}$ and $L_{23,\sigma_{23}} \subset S^*\bR^{N} \,\widehat\times\, S^*\bR^N$, we have
    $$L_{12,ex} \,\widehat\times\, L_{23,ex} \pitchfork X_1^- \,\widehat\times \Delta_{X_2} \,\widehat\times \, X_3, \quad L_{12,ex,\sigma_{12}} \,\widehat\times\, L_{23,ex,\sigma_{23}} \pitchfork S^*\bR^N \,\widehat\times\, S^*_\Delta\bR^{2N} \,\widehat\times\, S^*\bR^N.$$
    By Lemma \ref{ss of composition for doubling 1} and \ref{ss of composition for doubling 2}, we then know that 
    $$L_{12,ex,\sigma_{12}} \,\widehat\times\, L_{23,ex,\sigma_{23}} \pitchfork T^*\bR^N \times \dot T^*_\Delta\bR^{2N} \times T^*\bR^N.$$
    Then, for simple sheaves $\SF_{12} \in \msh_{L_{12,\sigma_{12}}}(L_{12})$ and $\SF_{23} \in \msh_{L_{23,\sigma_{23}}}(L_{23})$, it follows from \cite[Theorem 7.5.11]{KS} that $\SF_{23} \circ \SF_{12} \in \msh_{L_{13,\sigma_{13}}}(L_{13})$ is a simple object. We know that the internal hom with the simple objects define equivalences between microsheaves on these Legendrians and local systems on the Legendrians:
    $$\msh_{L_{ij,\sigma_{ij}}}(L_{ij}) \xrightarrow{\sim} \Loc(L_{ij}), \quad \SG_{ij} \mapsto \mu hom(\SF_{ij}, \SG_{ij}).$$
    Then, considering the natural map $\pi_{13!}\Delta_2^*(\sHom(\SF_{12}, \SG_{12}) \boxtimes \sHom(\SF_{23}, \SG_{23})) \to \sHom(\SF_{12} \circ \SF_{23}, \SG_{12} \circ \SG_{23})$, by microlocalization, there is a natural map as in \cite[Proposition 4.4.11]{KS}:
    $$\pi_{13!}\Delta_2^*(\mu hom(\SF_{12}, \SG_{12}) \boxtimes \mu hom(\SF_{23}, \SG_{23})) \to \mu hom(\SF_{12} \circ \SF_{23}, \SG_{12} \circ \SG_{23}).$$
    Similar to the above discussion, \cite[Theorem 7.5.11]{KS} shows that the natural map induces isomorphisms on stalks. This identifies the composition $\msh_{L_{12,\sigma_{12}}}(L_{12}) \otimes \msh_{L_{23,\sigma_{23}}}(L_{23}) \to \msh_{L_{12} \circ L_{23,\sigma_{13}}}(L_{12} \circ L_{23})$ with the functor
    $$\Loc(L_{12}) \otimes \Loc(L_{23}) \longrightarrow \Loc(L_{13})$$
    defined by the pull-back functor via $L_{13} \xrightarrow{\sim} (L_{12} \times L_{23}) \cap (W_1^- \times \Delta_{W_2} \times W_3)$.
\end{proof}

    More generally, for a contact manifold $X$ and a Legendrian $L \subset X$, using the embedding via the universal polarization as in Section \ref{ssec: weinstein}, $\msh$ can be defined as a locally constant sheaf of categories on $LGr(X)|_L$ is the pull-back of a locally constant sheaf of categories on $BPic(\SC)(X)|_L$. When $X$ has Maslov data, we get a locally constant sheaves on $L$. 

    Consider the contact embedding of the universal polarization $LGr(X)|_L$ as a Lagrangian submanifold whose Lagrangian Gauss map $LGr(X)|_L \to LGr$ defines a universal $\bZ \times BO$-bundle $\mathscr{P}(LGr(X)|_L) \to LGr(X)|_L$, and with respect to the universal classifying map $\chi: \bZ \times BO \to Pic(\SC)$, we have
    $$\msh_{LGr(X)|_L} = (\Loc_{\mathscr{P}(LGr(X)|_L)})^\chi,$$
    as shown by Jin \cite[Proposition 1.3]{Jin2} (following Guillermou \cite[Proposition 10.5.3 \& 10.6.2]{Guillermou-survey}). Therefore, setting the pull-back bundle to $L$ to be $\mathscr{P}_L$ and applying the descent \cite[Theorem 11.14]{Nadler-Shende} \cite[Proposition 3.14]{CKNS}, we can identify microsheaves on $L$ as twisted local systems
    $$\msh_{L} = (\Loc_{\mathscr{P}(L)})^\chi.$$
    Under this identification, the composition can be identified as
    \begin{equation}
    (\Loc_{\mathscr{P}(L_{12})})^\chi \otimes (\Loc_{\mathscr{P}(L_{23})})^\chi \to (\Loc_{\mathscr{P}(L_{13})})^\chi,
    \end{equation}
    induced by the natural map $\mathscr{P}(L_{13}) \hookrightarrow \mathscr{P}(L_{12} \times L_{23}) \xleftarrow{\circ} \pi_{12}^{-1}\mathscr{P}(L_{12}) \times \pi_{23}^*\mathscr{P}(L_{23})$, where the second map is by fiberwise multiplication, which preserves $\chi$-invariance by Formula \eqref{eq: LGr to BPic}.

\subsection{Gapped compositions specialize to filtered morphisms} 
    While, for sheaves, the main result on gapped composition and nearby cycles (Theorem \ref{thm:gapped composition}) directly specializes to our previous result on full faithfulness of gapped specialization (Theorem \ref{fullfaithful nearby}). 
    
    It is somewhat more subtle in what sense the corresponding microsheaf assertion on gapped composition (Theorem \ref{thm: main composition}) recovers Theorem \ref{thm: lagrangian immersion tensor}; explaining this is the purpose of the present subsection.

    Let us first see what Theorem \ref{thm: main composition} asserts when $W_1 = W_3 = pt$. We take $W_2 = W$ and $L \subset W^- \times \bR$, $K \subset W \times \bR$ to be eventually conical Legendrians. Their composition is the image of the map from the Lagrangian intersection points $\ol{L} \cap \ol{K}$ to $\bR$ sending a Lagrangian intersection point in $\ol{L} \cap \ol{K}$ to its action, by Lemma \ref{reeb chord Lagrangian intersection}.

\begin{corollary}\label{cor: recover tensor}
    Let $W$ be a Weinstein manifold with Maslov data, and $L \subset W^- \times \bR, K \subset W \times \bR$ be $\epsilon$-thin eventually conical Legendrians. There is a commutative diagram
    \[\begin{tikzcd}
    \Sh_{\underline L_{\cup,\epsilon}^+}(\bR^{2N}) \otimes \Sh_{\underline K_{\cup,\epsilon}^+}(\bR^{2N}) \ar[d, "\psi \otimes \psi" left] \ar[r, "\circ"] & \Sh_{(\underline{\ol{L} \cap \ol{K}})_{\cup,\epsilon}^+}(\bR^{2N}) \ar[d, "\psi"] \\
    \msh_{\mathfrak{c}_{W^-}^{L}}(\mathfrak{c}_{W^-}^L) \otimes \msh_{\mathfrak{c}_{W}^K}(\mathfrak{c}_{W}^K) \ar[r, "\circ"] \ar[d, "i_\infty^* \otimes i_\infty^*" left] & \msh_{\mathfrak{c}_{pt}^{\ol{L} \cap \ol{K}}}(\mathfrak{c}_{pt}^{L \cap K}) \ar[d, "i_\infty^*"] \\
    \msh_{\mathfrak{c}_{W^-,\partial L}}(\mathfrak{c}_{W^-}) \otimes \msh_{\mathfrak{c}_{W,\partial K}}(\mathfrak{c}_{W}) \ar[r, "\circ "] & \SC.
    \end{tikzcd}\]
\end{corollary}

\begin{corollary}\label{cor: recover hom}
    Let $W$ be a Weinstein manifold with Maslov data, and $L \subset W \times \bR, K \subset W \times \bR$ be $\epsilon$-thin eventually conical Legendrians. There is a commutative diagram
    \[\begin{tikzcd}
    \Sh_{\underline L_{\cup,\epsilon}^+}(\bR^{2N}) \otimes \Sh_{\underline K_{\cup,\epsilon}^+}(\bR^{2N}) \ar[d, "\psi \otimes \psi" left] \ar[r, "{\sHom^\circ(-,-)}"] & \Sh_{(\underline{\ol{L} \cap \ol{K}})_{\cup,\epsilon}^+}(\bR^{2N}). \ar[d, "\psi"] \\
    \msh_{\mathfrak{c}_{W^-}^{L}}(\mathfrak{c}_{W^-}^L) \otimes \msh_{\mathfrak{c}_{W}^K}(\mathfrak{c}_{W}^K) \ar[r, "{\sHom^\circ(-,-)}"] \ar[d, "i_\infty^* \otimes i_\infty^*" left] & \msh_{\mathfrak{c}_{pt}^{\ol{L} \cap \ol{K}}}(\mathfrak{c}_{pt}^{L \cap K}) \ar[d, "i_\infty^*"] \\
    \msh_{\mathfrak{c}_{W^-,\partial L}}(\mathfrak{c}_{W^-}) \otimes \msh_{\mathfrak{c}_{W,\partial K}}(\mathfrak{c}_{W}) \ar[r, "{\Hom(-, -)}"] & \SC.
    \end{tikzcd}\]
\end{corollary}

    Observe that  $\msh_{\mathfrak{c}_{pt}^{\ol{L} \cap \ol{K}}}(\mathfrak{c}_{pt}^{\ol{L} \cap \ol{K}}) \simeq \Sh_{\ol{L} \cap \ol{K}}(\bR)_0 \subset \Sh_{\tau\geq 0}(\bR)_0$ consists of $\bR$-filtered objects in $\SC$, such that the jumps of the filtration
    occur at the actions of the Lagrangian intersection $\ol{L} \cap \ol{K}$ by Lemma \ref{reeb chord Lagrangian intersection}.

    From the lower commutative square of Corollary \ref{cor: recover tensor}, we can recover the lower commutative square of Theorem \ref{thm: lagrangian immersion tensor}. 
    To recover the lower square:

\begin{proposition}\label{prop: recover tensor full faithful}
    Let $W$ be a Weinstein manifold with Maslov data, and $L^- \subset W^- \times \bR$, $K \subset W \times \bR$ be $\epsilon$-thin eventually conic Legendrians with $-L \ll K$ where $-L$ is the image of $L$ under the involution along the $\bR$-factor,  there is a commutative diagram
    \[\begin{tikzcd}
    \Sh_{(\underline{\ol{L} \cap \ol{K}})_{\cup,\epsilon}^+}(\bR^{2N}) \ar[d, "\psi" left] \ar[r, "\pi_!\Delta^*"] & \SC \ar[d, "\rotatebox{90}{=}" right] \\
    \msh_{\mathfrak{c}_{pt}^{\ol{L} \cap \ol{K}}}(\mathfrak{c}_{pt}^{L \cap K}) \ar[r, "i_\infty^*"] & \SC.
    \end{tikzcd}\]
\end{proposition}
\begin{proof}
    Let $\SH \in \Sh_{(\underline{\ol{L} \cap \ol{K}})^+_{\cup,\epsilon}}(\bR^{2N})$ and consider $\pi_{!}\Delta^*\SH \in \SC$. Since $-L \ll K$, we can use Lemma \ref{reeb chord Lagrangian intersection} to show that $(\ol{L} \cap \ol{K} \times \bR^{2N-1})^+_{\pm\epsilon}$ and $S^*_{\Delta}(\bR^N \times \bR^N)$ are positively gapped. Then using the gapped base change formula Proposition \ref{gapped basechange}, we have a commutative diagram
    \[\begin{tikzcd}
    \Sh_{(\underline{\ol{L} \cap \ol{K}})_{\cup,\epsilon}^+}(\bR^{2N}) \ar[d, "\psi" left] \ar[r, "\pi_!\Delta^*"] & \SC \ar[d, "\rotatebox{90}{=}" right]. \\
    \Sh_{(\underline{\mathfrak{c}}_{pt}^{\ol{L} \cap \ol{K}})_{\cup,\epsilon}^+}(\bR^{2N}) \ar[r, "\pi_!\Delta^*"] & \SC.
    \end{tikzcd}\]
    Finally, we consider the nearby cycle functor $\psi_{pt}$ that sends the Legendrian lift of $\ol{L} \cap \ol{K} \to \bR$ to a single point via the flow of the contact vector field $Z = t\partial/\partial t$. By Theorem \ref{thm: nearby = restriction at infty}, we know that $\psi_{pt} \simeq i_\infty^*$, Moreover, under the contact flow, by Lemma \ref{reeb chord Lagrangian intersection}, $(\ol{L} \cap \ol{K} \times \bR^{2N-1})_{\pm\epsilon}$ and $S^*_{\Delta}(\bR^N \times \bR^N)$ remain positively gapped. Thus, we know there is a commutative diagram
    \[\begin{tikzcd}
    \Sh_{(\underline{\mathfrak{c}}_{pt}^{\ol{L} \cap \ol{K}})_{\cup,\epsilon}^+}(\bR^{2N}) \ar[d, "\psi_{pt} \otimes \psi_{pt}" left]  \ar[r, "\pi_!\Delta^*"] & \SC. \ar[d, "\rotatebox{90}{=}" right] \\    \Sh_{(\underline{\mathfrak{c}}_{pt}^+)_{\cup,\epsilon}}(\bR^{2N}) \ar[r, "\pi_!\Delta^*"] & \SC.
    \end{tikzcd}\]
    Since the upper horizontal functor is clearly an equivalence, we can conclude the commutativity of the diagram.
\end{proof}

    Finally, we can now conclude that Theorem \ref{thm: lagrangian immersion tensor} can be recovered from Theorem \ref{thm: main composition}, via Corollary \ref{cor: recover tensor} and Proposition \ref{prop: recover tensor full faithful}. 
    The same results holds for Theorem \ref{thm: lagrangian immersion hom} and Theorem \ref{thm: main composition hom}, and we will omit the statements.

%% file: oh-tanaka-comparison.tex
\section{Comparison of $G$-actions} \label{comparison of G-actions}

Oh and Tanaka have shown that the (infinite dimensional) group of all exact symplectomorphisms acts topologically on the Fukaya category of a Liouville manifold \cite[Theorem 1.0.3]{Oh-Tanaka1}.  Their approach does not use Lagrangian correspondences: instead, they construct a symplectic fibration over (an appropriate simplicial set that realizes) $BG$ with fiber $W$, and `take Fukaya categories' to  construct a local system of categories over $BG$. In \cite[Section 6]{McbreenShendeZhou}, this approach was adapted to the microsheaf setting for $G$ actions which preserve the skeleton (i.e. are conic as opposed to just eventually conic). 

Here we show how to compare our construction of $G$-actions by Lagrangian correspondences (Theorem \ref{thm: group action}) with this sort of approach.

\begin{lemma}
    Let $W$ be a Weinstein domain. Let $G$ be a (finite dimensional) Lie group with an eventually conic Hamiltonian $G$-action $\mu$ on $W$. Then there is a Liouville bundle $E_{BG} \to BG$. For any smooth map $\Delta \to BG$ (where $BG$ is considered as a diffeological space), the pull-back bundle is a Liouville bundle $E_{\Delta} \to \Delta$.
\end{lemma}
\begin{proof}
    We use the Milnor join construction \cite{Milnor-classify-ii} (equipped with the diffeological structure \cite{MagnotWatts}) in place of the Segal simplicial construction for $BG$ \cite{Segal} (they are known to be weak equivalent):
    $$\cdots \,\substack{\longrightarrow \\[-0.8em] \longrightarrow \\[-0.8em] \longrightarrow \\[-0.8em] \longrightarrow }\, G \times G \times G  \times \Delta^3 \,\substack{\longrightarrow \\[-0.8em] \longrightarrow \\[-0.8em] \longrightarrow}\, G \times G \times \Delta^2 \,\substack{\longrightarrow \\[-0.8em] \longrightarrow}\, G \times \Delta^1 \rightarrow *.$$
    Consider the trivial Liouville bundle structure on $W \times G^{j} \times \Delta^j$ by the pull-back of Liouville 1-form on $W$. Define the face maps by the group action and multiplications, and define the degeneracy maps by units. Equivalently, we can describe the universal bundle via the following local trivilizations. Consider a tubular neighbourhood of a face $U(G^{j-1} \times \Delta^{j-1}) \subset G^{j} \times \Delta^j$ which is homeomorphic to $G^{j-1} \times \Delta^{j-1} \times (G \times \Delta^1/G \times 0)$. We can extend the fiberwise Liouville structure trivially from $W \times G^{j-1} \times \Delta^{j-1}$ to $W \times U(G^{j-1} \times \Delta^j_i$. Then define transition maps $\varphi_{ij}: W \times U(G^{j-1} \times \Delta^{j-1}) \to W \times U(G^{j-1} \times \Delta^{j-1})$ of the Liouville bundles between neighbourhoods of different faces $U(G^{j-1} \times \Delta^{j-1})$ by
    \begin{gather*}
        \varphi_{01}(x, g_1, \dots, g_j) = (\varphi_\mu^{g_1}(x), g_1, \dots, g_j), \\
        \varphi_{i,i+1}(x, \dots, g_{i-1}, g_i, \dots, g_j) = (x, \dots, g_{i-1}g_i^{-1}, g_ig_{i+1}, \dots, g_j),
    \end{gather*}
    and all the maps are trivial on the $\Delta^j$ factors. Since $\mu$ is an eventually conic Hamiltonian action, the transition maps determined by $\varphi_\mu$ are eventually conic exact symplectomorphisms. Then by taking the geometric realization, this defines a Liouville bundle $E_{BG} \to BG$. Since the transition maps identifies the fiberwise Liouville forms up to a family of exact 1-forms, $E_{BG} \to BG$ is equipped with the 1-form determined by gluing fiberwise Liouville 1-forms via partition of unity.
\end{proof}

\begin{remark}
    When we truncate the join construction at the first step, the above construction is simply the clutching construction, where on the neighbourhood of each vertex $U(\star)$ homeomorphic to $G \times \Delta^1/G \times 0$, we consider the trivial Liouville bundle $W \times U(\star)$, and define the transition map by
    $\varphi_{12}(x, g) = (\varphi_\mu^g(x), g)$; see \cite{Seidel-automorphism} for the case when $G = S^1$.
\end{remark}

\begin{lemma}
    Let $W$ be a Weinstein domain. Let $G$ be a (finite dimensional) Lie group with an eventually conic Hamiltonian $G$-action $\mu$ on $W$. Then there is a Liouville bundle $W_{BG} \to T^*BG$ where $T^*BG$ is the Milnor construction. For any smooth map $\Delta \to BG$ (where $BG$ is considered as a diffeological space), the pull-back determines a Liouville bundle with Liouville total space $W_{\Delta} \to T^*\Delta$, and $W_{\Delta} = E_\Delta \times_\Delta T^*\Delta$.
\end{lemma}
\begin{proof}
    We again use the Milnor join construction \cite{Milnor-classify-ii}. Consider a tubular neighbourhood of a face $U(G^{j-1} \times \Delta^{j-1}) \subset G^{j} \times \Delta^j$ which is homeomorphic to $G^{j-1} \times \Delta^{j-1} \times (G \times \Delta^1/G \times 0)$. We extend the Liouville structure from $W \times T^*(G^{j-1} \times \Delta^{j-1})$ to $W \times T^*U(G^{j-1} \times \Delta^j_i)$ (with respect to the diffeology structure \cite{MagnotWatts})
    $$\lambda_{W \times T^*U(G^{j-1} \times \Delta^{j-1})} = \lambda_{W \times T^*(G^{j-1} \times \Delta^{j-1})} \times \lambda_{T^*(G \times \Delta^1)}.$$
    Then we define transition maps $\varphi_{ij}$ of the Liouville bundles between neighbourhoods of different faces $U(G^{j-1} \times \Delta^{j-1})$ by
    \begin{gather*}
        \varphi_{01}(x, g_1, h_1, \dots, g_j, h_j) = (\varphi_\mu^{g_1}(x), g_1, h_1 - \mu(\varphi_\mu^{g_1}(x)), \dots, g_j, h_j), \\
        \varphi_{i,i+1}(x, \dots, g_{i-1}, h_{i-1}, g_i, h_i, \dots) = (x, \dots, g_{i-1}g_i^{-1}, h_{i-1} - Ad_{g_{i-1}}h_i, g_ig_{i+1}, h_i + Ad_{g_i}h_{i+1}, \dots),
    \end{gather*}
    and all the maps are trivial on the $\Delta^j$ factors. Since $\mu$ is an eventually conic Hamiltonian action, the transition maps determined by $\varphi_\mu$ are eventually conic exact symplectomorphisms. by taking the geometric realization, this defines a Liouville bundle $W_{BG} \to T^*BG$, and by construction, it is straight forward to see that
    $$W_{BG} = E_{BG} \times_{BG} T^*BG.$$
    Since the transition maps identifies the global Liouville forms up to an exact 1-form, $W_{BG}$ is equipped with a Liouville 1-form determined on the total space by gluing the Liouville 1-forms via partition of unity.
\end{proof}

\begin{corollary} \label{cor: G action comparison}
    Let $W$ be a Weinstein domain. Let $G$ be a (finite dimensional) Lie group with an eventually conic Hamiltonian $G$-action $\mu$ on $W$.  Fix $G$-equivariant   Maslov data. Then Theorem \ref{thm: group action} determines a local system of categories in $\Fun(BG, \PrLst)$ such that for any smooth map $\Delta \to BG$, the sections on the manifold are $\msh_{\mathfrak{c}_{W_\Delta}}(\mathfrak{c}_{W_\Delta})$.
\end{corollary}
\begin{proof}
    We will construct a local system of categories from the universal Liouville bundle $W_{BG} \to BG$ and compare it with the local system of categories constructed from the Lagrangian correspondence in Theorem \ref{thm: group action}. Consider the Milnor join construction. For any tubular neighbourhood of a face $U(G^{j-1} \times \Delta^{j-1}) \subset G^{j} \times \Delta^j$ which is homeomorphic to $G^{j-1} \times \Delta^{j-1} \times (G \times \Delta^1/G \times 0)$, we define the local system of categories to be the constant sheaf of categories $\msh_{\mathfrak{c}_W}(\mathfrak{c}_W)$. On the intersection between neighbourhoods of different faces $U(G^{j-1} \times \Delta^{j-1})$, by Theorem \ref{cor: GKS kernel invertible}, the Liouville transition map induces the automorphism
    $$\pi_1^*\SK_\mu \circ -: \msh_{\mathfrak{c}_{W}}(\mathfrak{c}_W) \otimes \Loc(G^{\times j}) \to \msh_{\mathfrak{c}_{W}}(\mathfrak{c}_W) \otimes \Loc(G^{\times j}).$$
    This defines a local system of categories over $BG$ such that for any smooth map $\Delta \to BG$, the sections on the manifold are $\msh_{\mathfrak{c}_{W_\Delta}}(\mathfrak{c}_{W_\Delta})$.
    
    Then we compare the above construction through the universal Liouville bundle with the construction in Theorem \ref{thm: group action}. We know that the the group action in Theorem \ref{thm: group action} is also induced by functor defined by the Lagrangian correspondence
    $$\pi_1^*\SK_\mu \circ -: \msh_{\mathfrak{c}_{W}}(\mathfrak{c}_W) \otimes \Loc(G)^{\otimes j} \to \msh_{\mathfrak{c}_{W}}(\mathfrak{c}_W) \otimes \Loc(G)^{\otimes j-1}.$$
    Writing the above functor as the following composition implies that it factors as the automorphism induced by the Liouville transition map and the inclusion of trivial local systems:
    $$\pi_1^*\SK_\mu \circ -: \msh_{\mathfrak{c}_{W}}(\mathfrak{c}_W) \otimes \Loc(G^{\times j}) \to \msh_{\mathfrak{c}_{W}}(\mathfrak{c}_W) \otimes \Loc(G^{\times j}) \to \msh_{\mathfrak{c}_{W}}(\mathfrak{c}_W) \otimes \Loc(U(G^{\times j-1})).$$
    This therefore identifies the two constructions of local systems of categories.
\end{proof}

To give a direct comparison with the theorem of Oh and Tanaka, it would remain only to ``functorially apply the sheaf/Fukaya dictionary \cite{Ganatra-Pardon-Shende3} to Corollary \ref{cor: G action comparison}'' and also to show that resulting covariant maps on Fukaya categories from sector inclusions agree with the ``parallel transport'' maps that constructed by hand in \cite{Oh-Tanaka2}.

%% file: extras.tex

\section{Isotopies of microsheaves in Weinstein manifolds}
    Here we develop two additional results: first, an interpretation of sheaf quantization in Weinstein manifolds as a form of doubling, and second, a continuation map for the sheaf quantization of contact isotopies in Weinstein manifolds, and the corresponding analogue of Kuo's wrapping formula.  These results are not used in body of the article. 

\subsection{More on noncharacteristic propagation}

    For a sheaf $\SF \in \Sh(M \times I)$, when the microsupport is contained in $S^*M \times T^*_{\tau\leq 0}I$, we can use non-characteristic propagation to define continuation maps between sheaves $i_0^*\SF \to i_1^*\SF$ (see Proposition \ref{prop: positive continue shv}):

\begin{proposition}[{\cite[Proposition 4.8]{Guillermou-Kashiwara-Schapira} \cite[Proposition 3.2]{Kuo-wrapped-sheaves}}]\label{prop: positive continue shv app} 
    Fix $\Lambda \subset S^*M \times T^*_{\tau \leq 0}I$.  We write $\Lambda_0$ and $\Lambda_1$ for the contact reductions of $\Lambda$ onto $S^*M \times 0$ and $S^*M \times 1$. Then for any $\SF \in \Sh_\Lambda(M \times I)$, there is a natural continuation morphism $i_{0}^*\SF \to i_{1}^*\SF$ such that for any $U \subset M$, its induced map on sections is 
    $$\Gamma(U, i_0^*\SF) \xleftarrow{\sim} \Gamma(U \times I, \SF) \rightarrow \Gamma(U, i_1^*\SF).$$
\end{proposition}

\begin{proposition}[{\cite[Proposition 3.22]{Kuo-wrapped-sheaves}}]\label{prop: positive canonical shv}
    Fix $\Lambda \subset S^*M \times T^*_{\tau \leq 0}I \times T^*J$.  Assume the the contact reduction of $\Lambda$ onto $S^*(M \times J) \times 0$ is $\Lambda_0 \times 0_J$ and the reduction onto $S^*(M \times J) \times 1$ is $\Lambda_1 \times 0_J$. Then for any $\SF \in \Sh_\Lambda(M \times I \times J)$ and any $s, s' \in J$, the continuation morphisms fit into a commutative diagram
    \[\begin{tikzcd}
    i_{s,0}^*\SF \ar[r] \ar[d, "\rotatebox{90}{$\sim$}" left] & i_{s,1}^*\SF \ar[d, "\rotatebox{90}{$\sim$}"] \\
    i_{s',0}^*\SF \ar[r] & i_{s',1}^*\SF.
    \end{tikzcd}\]
\end{proposition}

    Using doubling functor (Theorem \ref{thm: relative-doubling}), we descend these notions to microsheaves with sufficiently Legendrian supports (in other words, pdff, perturbable to finite positions and self displaceable):

\begin{proposition}\label{prop: positive continue}
    Let $\Lambda \subset S^*M \times T^*_{\tau \leq 0}I$ be a compact sufficiently Legendrian subset. Then for any $\SF \in \msh_\Lambda(\Lambda)$, there is a canonical continuation morphism $i_0^*\SF \to i_1^*\SF$ in $\msh(S^*M)$ (compatible with the continuation of the doubling sheaves).
\end{proposition}
\begin{proof}
    Since $\Lambda$ is compact and sufficiently Legendrian, by Theorem \ref{thm: relative-doubling}, it suffices to show that there exists a continuation morphism
    $$i_{0}^*w_\Lambda(\SF) \to i_{1}^*w_\Lambda(\SF).$$
    Since $\Lambda \subset S^*M \times T^*_{\tau \leq 0}I$, we know that $\underline\Lambda_{\cup,s} \subset T^*M \times T^*_{\tau \leq 0}I$. Therefore, the continuation map $i_{0}^*w_\Lambda(\SF) \to i_{1}^*w_\Lambda(\SF)$ is constructed in Proposition \ref{prop: positive continue shv}.
\end{proof}

\begin{proposition}\label{prop: positive canonical}
    Let $\Lambda \subset S^*M \times T^*_{\tau \leq 0}I \times T^*J$ be a relative compact sufficiently Legendrian subsets. Assume the the contact reduction of $\Lambda$ onto $S^*(M \times J) \times 0$ is $\Lambda_0 \times 0_J$ and the reduction onto $S^*(M \times J) \times 1$ is $\Lambda_1 \times 0_J$. Then for any $\SF \in \msh_\Lambda(\Lambda)$ and any $s \in J$, there is a commutative diagram of continuation morphisms in $S^*M$ 
    \[\begin{tikzcd}
    i_{s,0}^*\SF \ar[r] \ar[d, "\rotatebox{90}{$\sim$}" left] & i_{s,1}^*\SF \ar[d, "\rotatebox{90}{$\sim$}"] \\
    i_{s',0}^*\SF \ar[r] & i_{s',1}^*\SF.
    \end{tikzcd}\]
\end{proposition}
\begin{proof}
    Since $\Lambda$ is compact sufficiently Legendrian, by Theorem \ref{thm: relative-doubling}, it suffices to show that for $s \in J$, the continuation morphisms
    $$i_{s,0}^*w_\Lambda(\SF) \to i_{s,1}^*w_\Lambda(\SF)$$
    are all homotopic. However, since $\underline\Lambda \subset S^*M \times T^*_{\tau \leq 0}I \times T^*J$, we know that $\Lambda_{\cup,s} \subset T^*M \times T^*_{\tau \leq 0}I \times T^*J$. Therefore, the result follows from the fact that the continuation map $i_{s,0}^*w_\Lambda(\SF) \to i_{s,1}^*w_\Lambda(\SF)$ is induced by by Proposition \ref{prop: positive continue shv} which commutes with the restriction maps to any $s \in J$ by Proposition \ref{prop: positive canonical shv}.
\end{proof}

\begin{proposition}\label{prop: hom deformation}
    Let $\Lambda, \Lambda' \subset S^*M \times T^*I$ be compact sufficiently Legendrian subsets. Suppose $\Lambda \cap \Lambda' \subset \Sigma$ and the complement $(-\Lambda + \Lambda') \setminus \Sigma \subset S^*M \times T^*_{\tau\leq 0}I$ and $(-\Lambda_0 + \Lambda'_0) \setminus \Sigma_0 \subset S^*M \times 0$. Then for any $\SF \in \msh_\Lambda(\Lambda)$ and $\SG \in \msh_{\Lambda'}(\Lambda')$, 
    $$\Hom_{\msh}(\SF, \SG) \simeq \Hom_{\msh}(i_0^*\SF, i_0^*\SG).$$
\end{proposition}
\begin{proof}
    Since $\Lambda$ is compact and sufficiently Legendrian, by Theorem \ref{thm: relative-doubling}, it suffices to show that
    $$\Hom(w_\Lambda(\SF), w_\Lambda(\SG)) \simeq \Hom(i_1^*w_\Lambda(\SF), i_1^*w_\Lambda(\SG)).$$
    Since $\underline\Lambda_{\cup,s}, \underline\Lambda'_{\cup,s}$ are compact and sufficiently Legendrian, we know that when $\epsilon > 0$ is sufficiently small, by Proposition \ref{prop: perturb-lim},
    $$\Hom(w_\Lambda(\SF), w_\Lambda(\SG)) = \Hom(w_\Lambda(\SF), T_\epsilon w_\Lambda(\SG)).$$
    Since $\Lambda \cap \Lambda' \subset \Lambda_0$ and the complement $(-\Lambda + \Lambda') \setminus \Lambda_0 \subset S^*M \times T^*_{\tau\leq 0}I$, we know that after the small pushoff, $\Lambda_{\pm s} \cap T_\epsilon\Lambda'_{\pm s} = \varnothing$, and $-\underline\Lambda_{\cup,s} + T_\epsilon\underline\Lambda'_{\cup,s} \subset T^*M \times T^*_{\tau\leq 0}I$. Then by the singular support estimate Formula \eqref{lem: ss-hom} we know that
    $$\sHom(w_\Lambda(\SF), T_\epsilon w_\Lambda(\SG)) \subset T^*M \times T^*_{\tau\leq 0}I,$$
    and the result follows from non-characteristic deformation Proposition \ref{prop: positive continue}.
\end{proof}

\subsection{Continuation and wrapping in Weinstein manifolds}

    We construct continuation morphisms for sheaf quantizations of contact isotopies, generalizing  \cite[Section 4.3]{Guillermou-Kashiwara-Schapira} and \cite[Section 3.1]{Kuo-wrapped-sheaves} from cotangent bundles to Weinstein manifolds. 

\begin{proposition}
    Let $W$ be a Weinstein domain with contact boundary $\partial W$ with Maslov data, $\varphi_{H_\infty}^t, \varphi_{H'_\infty}^t: \partial W \to \partial W$ be contact Hamiltonian isotopies such that $H_\infty \leq H'_\infty$. Then there is a canonical continuation morphism in $\msh_{W^-\times W \times T^*I}(\mathfrak{c}_{W^-\times W \times T^*I})$:
    $$\SK_{H_\infty} \to \SK_{H'_\infty}.$$
\end{proposition}
\begin{proof}
    Consider any family of increasing family of contact Hamiltonians $G_\infty^s$ such that $G_\infty^0 = H_\infty$ and $G_\infty^1 = H'_\infty$. By Proposition \ref{prop: GKS}, there is a canonical sheaf quantization $\SK_{G_\infty} \in \msh_{\mathfrak{c}_{W^-\times W \times T^*I \times T^*J,\Gamma_{G_\infty}}}(\mathfrak{c}_{W^-\times W \times T^*I \times T^*J,\Gamma_{G_\infty}})$ such that 
    $$i_0^*\SK_{G_\infty} \simeq \SK_{H_\infty}, \quad i_1^*\SK_{G_\infty} \simeq \SK_{H'_\infty}.$$
    For the increasing family of contact Hamiltonians $G_\infty^s$, by the non-characteristic deformation Proposition \ref{prop: positive continue}, there is a continuation morphism
    $$i_0^*\SK_{G_\infty} \rightarrow i_1^*\SK_{G_\infty}.$$
    Moreover, by considering higher parametric families of contact Hamiltonians, we know by Proposition \ref{prop: positive canonical} that the continuation morphism is canonically determined by the homotopy class of the path of increasing Hamiltonians, which completes the proof.
\end{proof}

\begin{proposition}\label{prop: continuation hom microsheaf}
    Let $W$ be a Weinstein domain with contact boundary $\partial W$ with Maslov data, $\Lambda \subset \partial W$ a compact sufficiently Legendrian subset, and $\varphi_{H_\infty}^t: \partial W \to \partial W$ be a positive contact Hamiltonian isotopy. Then for any $\SF, \SG \in \msh_{\mathfrak{c}_{W,\Lambda}}(\mathfrak{c}_{W,\Lambda})$ and sufficiently small $\epsilon > 0$,
    $$\Hom(\SF, \SG) \simeq \lim_{t\to 0^+}\Hom(\SF, \SK^t_{H_\infty} \circ \SG) \simeq \Hom(\SF, \SK^\epsilon_{H_\infty} \circ \SG).$$
\end{proposition}
\begin{proof}
    For the first identify, it follows from the fact that $\SG \simeq \lim_{t \to 0^+}\SK^t_{H_\infty} \circ \SG$. For the second identity, since $\Lambda$ is compact and $\Lambda$ and $\Lambda_t$ are disjoint for any $\epsilon/2 \leq t \leq \epsilon$, by considering a cut-off $H'_\infty$ of the Hamiltonian $H_\infty$ that is supported in an open neighborhood of $\bigcup_{\epsilon/2 \leq t \leq \epsilon}\Lambda_t$, we have
    $$\varphi_{H'_\infty}^s(\Lambda) = \Lambda, \quad \varphi_{H'_\infty}^s(\Lambda_t) = \Lambda_{s+t}, \quad \epsilon/2 \leq t \leq s+t \leq \epsilon.$$
    By Proposition \ref{prop: GKS}, we have the microsheaf quantization $\SK_{H'_\infty}$. Consider the microsheaves $\SK_{H'_\infty} \circ \SF \in \msh_{\mathfrak{c}_{W \times T^*I,\Lambda \times I}}(\mathfrak{c}_{W \times T^*I, \Lambda \times I})$ and $\SK_{H_\infty} \circ (\SK^t_{H_\infty} \circ \SF) \in \msh_{\mathfrak{c}_{W \times T^*I,\Lambda_{H_\infty}}}(\mathfrak{c}_{W \times T^*I, \Lambda_{H_\infty}})$. By Lemma \ref{lem:contact-transform-main}, we have
    $$\SK_{H'_\infty}^s \circ \SF = \SF, \quad \SK_{H'_\infty}^s \circ (\SK_{H_\infty}^t \circ \SF) = \SK_{H_\infty}^{s+t} \circ \SF, \quad \epsilon/2 \leq t \leq s+t \leq \epsilon.$$
    By Corollary \ref{cor: GKS kernel invertible}, this shows that 
    $$\Hom(\SF, \SK^t_{H_\infty} \circ \SG) \simeq \Hom(\SF, \SK^{s+t}_{H_\infty} \circ \SG), \quad \epsilon/2 \leq t \leq s+t \leq \epsilon,$$
    and thus completes the proof.
\end{proof}

    We recall the definition of the wrapping category, which will be used to define the wrapping functors (see also Definition \ref{def: wrapping}):

\begin{definition}[{\cite[Definition 3.14]{Kuo-wrapped-sheaves}}]
    Let $\Lambda \subset \partial W$ be a closed subset. The category $W(\Lambda)$ consists of non-negative compactly supported contact isotopies on $\partial W \setminus \Lambda$ as objects, whose $j$-morphisms are $\Delta^j$-families of positive contact isotopies.
\end{definition}

    The wrapping category is filtered and there is a simple criterion on cofinality:

\begin{lemma}[{\cite[Lemma 3.27]{Ganatra-Pardon-Shende1}}]\label{lem: wrap cofinal} 
    The category $W(\Lambda)$ is $\bR$-filtered. Let $\Lambda' \subset \partial W$ be a closed subset disjoint from $\Lambda$. Then there exists a contact isotopy $\Lambda'_t$ such that the assoicated contact vector field $X_t$ satisfies
    $$\int_{0}^\infty \inf_{p \in \Lambda'_t}\alpha(X_t)(p) dt = \infty,$$
    and such a contact isotopy $\{\Lambda'_t\}_{t\geq 0}$ is cofinal.
\end{lemma}

    Given the above results, we give a description of the left (resp.~right) adjoint of the tautological inclusion functor via wrapping by positive (resp.~negative) contact Hamiltonian isotopies, which generalizes the result of Kuo \cite[Theorem 1.2]{Kuo-wrapped-sheaves} (see Theorem \ref{rem:wrapping}):

\begin{theorem}
    Let $W$ be a Weinstein domain and $\Lambda \subset \Lambda' \subset \partial W$ be universally sufficient Legendrian subsets. Let
    $\iota_{\Lambda*}: \msh_{\mathfrak{c}_{W,\Lambda}}(\mathfrak{c}_{W,\Lambda}) \hookrightarrow \msh_{\mathfrak{c}_{W,\Lambda'}}(\mathfrak{c}_{W,\Lambda'})$
    be the tautological inclusion. Then the inclusion functor admits a left and right adjoint given by
    $$\iota_\Lambda^*(\SF) = \operatornamewithlimits{colim}_{H_\infty \in W(\Lambda)}\SK_{H_\infty}^1 \circ \SF, \quad \iota_\Lambda^!(\SF) = \lim_{H_\infty \in W(\Lambda)}\SK_{H_\infty}^{-1} \circ \SF.$$
\end{theorem}
\begin{proof}
    We only show that the left adjoint is $\iota_\Lambda^*(\SF) = \colim_{H_\infty \in W(\Lambda)}\SK_{H_\infty}^1 \circ \SF$. The result for the right adjoint is similar. First, we show that $\colim_{H_\infty \in W(\Lambda)}\SK_{H_\infty}^1 \circ \SF$ is well defined as a microsheaf in $\msh_{\mathfrak{c}_{W,\Lambda}}(\mathfrak{c}_{W,\Lambda})$. Since we do not know the cocompleteness of microsheaves on the Weinstein domain $\msh_{W}(W)$ (without sufficiently Legendrian support condition), we need to specify where the colimit takes place. In fact, we will consider
    $$\operatornamewithlimits{colim}_{H_\infty \in W(\Lambda)}\SK_{H_\infty}^1 \circ \SF := m_{W}\Big(\operatornamewithlimits{colim}_{H_\infty \in W(\Lambda)}w_{\mathfrak{c}_{W,{\Lambda'}_{H_\infty}^1}}^+(\SK_{H_\infty}^1 \circ \SF)\Big) \in \msh_{W}(W).$$
    We will show that the (micro)support 
    \begin{equation}\label{eq: wrapping microsupport}
    \SS(\colim_{H_\infty \in W(\Lambda)}\SK_{H_\infty}^1 \circ \SF) \subset \mathfrak{c}_{W,\Lambda}.
    \end{equation}
    Suppose this is the case. Then we can write the colimit as microlocalization along $\mathfrak{c}_{W,\Lambda}$ (instead of along $W$). Since $\mathfrak{c}_{W,\Lambda}$ is sufficiently Legendrian, we can conclude that
    \begin{align*}
    \Hom_{\msh} &\Big(m_{\mathfrak{c}_{W,\Lambda}}\big(\operatornamewithlimits{colim}_{H_\infty \in W(\Lambda)}w_{\mathfrak{c}_{W,{\Lambda'}_{H_\infty}^1}}^+(\SK^1_{H_\infty} \circ \SF)\big), \SG\Big) \\
    &= \Hom_{\Sh}\Big(\operatornamewithlimits{colim}_{H_\infty \in W(\Lambda)}w_{\mathfrak{c}_{W,{\Lambda'}_{H_\infty}^1}}^+(\SK^1_{H_\infty} \circ \SF), w_{\mathfrak{c}_{W,\Lambda}}^+(\SG)\Big) \\
    &= \lim_{H_\infty \in W(\Lambda)}\Hom_{\msh}(w_{\mathfrak{c}_{W,{\Lambda'}_{H_\infty}^1}}^+(\SK_{H_\infty}^1 \circ \SF), w_{\mathfrak{c}_{W,\Lambda}}^+(\SG)) \\
    &= \lim_{H_\infty \in W(\Lambda)}\Hom_{\msh}(\SK_{H_\infty}^1 \circ \SF, \SG).
    \end{align*}
    Then $\colim_{H_\infty \in W(\Lambda)}\SK_{H_\infty}^1 \circ \SF$ co-represents the functor $\lim_{H_\infty \in W(\Lambda)}\Hom_{\msh}(\SK_{H_\infty}^1 \circ \SF, -)$.

    Now we show that the singular support Formula \eqref{eq: wrapping microsupport} holds: $\SS(\colim_{H_\infty \in W(\Lambda)}\SK_{H_\infty}^1 \circ \SF) \subset \mathfrak{c}_{W,\Lambda}$. By Lemma \ref{lem: wrap cofinal}, we choose a cofinal positive contact isotopy $H_{\infty,t}$ for $\Lambda'$ such that
    $$\int_0^\infty \min_{p\in \Lambda'}\alpha(X_{H_{\infty,t}})(p) dt = \infty.$$
    Using cut-off functions, we extend it to a positive contact isotopy $H_s$ on the Weinstein domain $W$ such that $H_t(p) \to H_{\infty,t}(p)$ for any $p \notin \mathfrak{c}_{W,\Lambda}$. Under the contact embedding, we can extend the positive contact isotopy $H_t$ to the cosphere bundle $S^*\bR^N$. Since $H_t$ is compatible with the contact collar $\partial W \times T^*\bR$, we know that on the doubling of $\mathfrak{c}_{W,\Lambda'}$,
    $$\int_0^\infty \min_{p\in (\mathfrak{c}_{W,\Lambda'})_{\pm\epsilon}}\alpha(X_{H_t})(p) dt = \infty.$$
    In other words, $H_t$ defines a cofinal wrapping of $(\mathfrak{c}_{W,\Lambda'})_{\pm \epsilon}$ in the wrapping category of $(\mathfrak{c}_{W,\Lambda})_{\pm \epsilon}$. Therefore, for any $\SF \in \msh_{\mathfrak{c}_{W,\Lambda'}}(\mathfrak{c}_{W,\Lambda'})$ and $\SG \in \msh_{\mathfrak{c}_{W,\Lambda}}(\mathfrak{c}_{W,\Lambda})$, by Theorem \ref{rem:wrapping}, for the sheaf quantization $\SK_H^s$ in Theorem \ref{thm: GKS sheaf}, as in Proposition \ref{prop: GKS}, we know that we have sheaf quantizations $\SK_H^s \in \msh_{\Gamma_H^s}(\Gamma_H^s)$, such that
    $\SK_H^s \circ w_{\mathfrak{c}_{W,\Lambda'}}^+(\SF) = w_{\mathfrak{c}_{W,\Lambda'_s}}^+(\SK_H^s \circ \SF)$. By Proposition \ref{prop: GKS}, for the nearby cycle functor $\psi$ defined by the Liouville flow on $W$, we have $w_{\mathfrak{c}_{W,\Lambda'_s}}^+(\SK_{H_\infty}^s \circ \SF) = \psi( w_{\mathfrak{c}_{W,\Lambda'_s}}^+(\SK_{H}^s \circ \SF))$. By Proposition \ref{prop:wrapping nearby}, the colimit of wrappings by $H_s$ can be realized by a nearby cycle functor. Then Lemma \ref{nearby commute Nadler} implies that
    \begin{align*}
    \operatornamewithlimits{colim}_{s \to \infty}w_{\mathfrak{c}_{W,\Lambda'_s}}^+(\SK_{H_\infty}^s \circ \SF) = \operatornamewithlimits{colim}_{s \to \infty}w_{\varphi_{H}^s(\mathfrak{c}_{W,\Lambda'})}^+(\SK_H^s \circ \SF) = \operatornamewithlimits{colim}_{s \to \infty}\SK_H^s \circ w_{\mathfrak{c}_{W,\Lambda'}}^+(\SF).
    \end{align*}
    Therefore, by Theorem \ref{rem:wrapping}, we know that $\dot\SS(\operatornamewithlimits{colim}_{s \to \infty}\SK_H^s \circ \SF) \subset \mathfrak{c}_{W,\Lambda}$. Then, given the singular support estimation, we can show that the functor $\SF \to \colim_{H_\infty \in W(\Lambda)}\SK_{H_\infty}^1 \circ \SF$ is the left adjoint of the tautological inclusion. We can compute that
    \begin{align*}
    \Hom\Big(\operatornamewithlimits{colim}_{s \to \infty}\SK_{H_\infty}^s \circ \SF, \SG \Big) &= \Hom\Big(m_{\mathfrak{c}_{W,\Lambda}}\big(\operatornamewithlimits{colim}_{s \to \infty}w_{\mathfrak{c}_{W,\Lambda'_s}}^+(\SK_{H_\infty}^s \circ \SF), \SG \Big) \\
    &= \Hom\Big(\operatornamewithlimits{colim}_{s \to \infty}\SK_H^s \circ w_{\mathfrak{c}_{W,\Lambda'}}^+(\SF), w_{\mathfrak{c}_{W,\Lambda}}^+\SG\Big) \\
    &= \Hom\big(\iota_{(\underline{\mathfrak{c}}_{W,\Lambda})_{\cup,\epsilon}^+}^*w_{\mathfrak{c}_{W,\Lambda'}}^+(\SF), w_{\mathfrak{c}_{W,\Lambda}}^+\SG \big) = \Hom(\SF, \SG ).
    \end{align*}
    This shows that the left adjoint is given by $\iota_\Lambda^*(\SF) = \colim_{H_\infty \in W(\Lambda)}\SK_{H_\infty}^1 \circ \SF$.
\end{proof}


\subsection{Sheaf quantization by doubling in Weinstein manifolds}
    Let $W$ be a Weinstein manifold. Given an eventually conic sufficiently Legendrian $L \hookrightarrow W \times \bR$, we define the doubling movie of $L$ inside $W \times T^*\bR \times T^*\bR$ to be 
    $$L_{\cup_{W}}^\prec = L_{-} \cup L_+ \cup \bigcup_{s\geq 0,-s\leq r\leq s}\psi_{\mathfrak{c}_{W \times T^*\bR}}(L_r) \times s.$$
    Here $\Lambda_{\pm s}$ is the Hamiltonian pushoff of $\Lambda$ along the Reeb flow by time $\pm s^{3/2}$, $\Lambda_\pm \subset W \times \bR$ is the movie of the Hamiltonian pushoff along the Reeb flow under time $\pm s^{3/2}$, and $\psi_{\mathfrak{c}_{W \times T^*\bR}}(L_r)$ is the limit set under the negative Liouville flow on $W \times T^*\bR$. Note that for $t \gg 0$ there is an inclusion $i_{-,t}: \mathfrak{c}_W^L \hookrightarrow L_{\cup_W}$ by
    $$i_{-,s}: \mathfrak{c}_W^L \hookrightarrow L_{\cup_W}, \quad (x, t, \tau) \mapsto (x, t - s^{3/2}, s, \tau, 3s^{1/2}\tau/2).$$

    Recall that using the contact transformation Lemma \ref{lem:contact-transform-main}, we can construct $T_-$ and $T_+$ to be the functors induced by the negative and positive push-off by the Reeb flow in $W \times \bR$ with time $\pm s^{3/2}$, with a natural transformation $T_- \to T_+$ induced by Proposition \ref{prop: positive continue}. We set
    $$T_\prec = \mathrm{Fib}(T_- \to T_+).$$

\begin{proposition}\label{prop: fiber sequence weinstein}
    Let $W$ be a Weinstein manifold 
    with Maslov data and ${L} \subset W \times \bR$ be an $\epsilon$-thin eventually conical universally sufficiently Legendrian. Then for $\SF \in \msh_{\mathfrak{c}_{W}^{L}}(\mathfrak{c}_{W}^{L})$ and $\SG \in \msh_{L_{\cup_W}^\prec}(L_{\cup_W}^\prec)$, then for the inclusion $i_{-,t}: \mathfrak{c}_W^L \hookrightarrow L_{\cup_W}$, there is an adjunction when $t \geq 2 \epsilon$:
    $$\Hom(T_\prec \SF, \SG) = \Hom(\SF, T_{2t}(i_{-,t}^*\SG)).$$
\end{proposition}
\begin{proof}
    Let $T_\prec = \mathrm{Fib}(T_- \to T_+)$ where the continuation map is induced by Proposition \ref{prop: positive continue}. We consider the internal hom sheaf $\mu hom$ and apply the non-characteristic deformation lemma. First, considering a contact push-off $T_t: W \times \bR \times T^*\bR \to W \times \bR \times T^*\bR$ for any $t \in \bR_{>0}$, we know by Proposition \ref{prop: positive continue} and \ref{prop: continuation hom microsheaf} that
    \begin{align*}
    \Hom(T_\prec \SF, \SG) = \lim_{t \to 0^+}\Hom(T_\prec\SF, T_t \SG) = \Hom(T_\prec\SF, T_t \SG).
    \end{align*}
    Then, consider the microsupport. Since $\SS(T_\prec\SF) \cap \SS(T_t \SG) = L_{\cup_W}^\prec \cap T_t(L_{\cup_W}^\prec) \subset \mathfrak{c}_{W \times T^*\bR^2}$. Then it follows by the non-characteristic deformation Proposition \ref{prop: hom deformation} for the inclusion $W \times T^*\bR \times s \hookrightarrow W \times T^*\bR \times \bR$ that for $t \geq 2\epsilon$, we have
    $$\Hom(T_\prec\SF, T_{t} \SG) = \Hom(\mathrm{Fib}(T_{-t}\SF \to T_{t}\SF), T_{2t}(i_{-,t}^* \SG)).$$
    We also have $-\SS(\mathrm{Fib}(T_{-t}\SF \to T_{t} \SF)) \cap \SS(T_{2t} (i_{-,t}^*\SG)) \subset -\mathfrak{c}_{W \times T^*\bR}$. Then, by Proposition \ref{prop: hom deformation} for the inclusion $W \times T^*(-\infty, 0) \hookrightarrow W \times T^*\bR$, we know that 
    $$\Hom(T_\prec \SF, \SG) = \Hom(T_\prec \SF, T_{2t}\SG) = \Hom(\SF, T_{2t}(i_{-,t}^*\SG)).$$
    This now completes the proof.
\end{proof}

\begin{theorem}\label{thm: quantization double movie}
    Let $W$ be a Weinstein manifold with Maslov data and ${L} \subset W \times \bR$ be  eventually conic and sufficiently Legendrian. Then there a natural bijection between the isomorphism classes of objects:
    $$i_{-, t}^*: \msh_{L_{\cup_W}}(L_{\cup_W}) \xrightarrow{\sim} \msh_{\mathfrak{c}_{W}^{L}}(\mathfrak{c}_W^{L}).$$
    Moreover, for any $\SF, \SG \in \msh_{L_{\cup_W}^\prec}(L_{\cup_W}^\prec)$, when $L \ll T_tL$,
    $$\Hom(\SF, \SG) \simeq \Hom(i_{-,t}^*\SF, T_{2t}(i_{-,t}^*\SG)) \simeq \Hom(i_\infty^*\SF, i_\infty^*\SG).$$
    In other words, the following composition is fully faithful
    $$\msh_{L_{\cup_W}^\prec}(L_{\cup_W}^\prec) \xrightarrow{i_{-,t}^*} \msh_{\mathfrak{c}_{W}^L}(\mathfrak{c}_W^L) \xrightarrow{i_\infty^*} \msh_{\mathfrak{c}_{W,\partial L}}(\mathfrak{c}_W).$$
\end{theorem}
\begin{proof}
    First, we show that the restriction functor $i_{-,t}^*$ is a conservative functor. Suppose $i_{-,t}^*\SF = 0$. Then we know that the microstalk of $\SF$ along $L_{-t}$ is trivial, and thus the microstalk along the doubled Legendrian $L_- \cup L_+$ is trivial. Thus $\SF$ is supported on the Lagrangian skeleton $\mathfrak{c}_{W \times T^*\bR}$ whose microstalk at $-\infty$ is trivial. This implies that $\SF = 0$ by Lemma \ref{lem:contact-transform-main}. Then, we notice that $i_{-,t}^* \circ T_\prec = \id$. This shows that $i_{-,t}^*$ induces a bijection on the isomorphism classes of objects. Since $i_{-,t}^* \circ T_\prec = \id$, by Proposition \ref{prop: fiber sequence weinstein}, we have
    \begin{align*}
    \Hom(\SF, \SG) &= \Hom(T_\prec \circ i_{-,t}^*\SF, \SG) = \Hom(i_{-,t}^*\SF, T_{2t}(i_{-,t}^*\SG)).
    \end{align*}
    The rest of the result follows formally from Theorem \ref{thm: lagrangian immersion hom}.
\end{proof}

\begin{remark}
    In general, when $\ol{L} \looparrowright W$ is not an embedded exact Lagrangian, the functor $i_{-,t}^*$ is usually not fully faithful, though it is conservative and surjective on objects.
\end{remark}

\begin{example}\label{ex: ike kuwagaki 1} 
    Consider the case $W = T^*M$ with the polarization $\sigma$ by cotangent fibers and $L \hookrightarrow T^*M \times \bR$ a Legendrian submanifold. Another (equivalent) way to define the category of sheaf quantizations for $L$ is as
    $$\msh_{L_{\cup_{T^*M}}^\prec}(L_{\cup_{T^*M}}^\prec) = \Sh_{L^\prec}(M \times \bR \times \bR)_0.$$
    This variant was studied by Ike and Kuwagaki \cite[Definition 3.5]{IkeKuwagaki} (see also earlier works \cite[Definition 11.4.3]{Guillermou-survey} and \cite[Section 4.2]{Asano-Ike2}) to define sheaf quantizations of immersed Lagrangians in Weinstein manifolds. Moreover, when the coefficient is $\SC = \Mod(k)$ for a discrete field $k$, Ike and Kuwagaki defined the notion of sheaf theoretic bounding cochains \cite[Definition 6.37]{IkeKuwagaki} and they showed that sheaf theoretic bounding cochains give rise to rank 1 sheaf quantizations in $\Sh_{L^\prec}(M \times \bR \times \bR)_0$ (and rank 1 sheaf quantizations also give rise to sheaf theoretic bounding cochains) in \cite[Theorem 6.38]{IkeKuwagaki}.

    As a result, our main Theorem \ref{sheaf quantization commutes with composition}, combined with Theorem \ref{thm: quantization double movie} implies that,  given sheaf theoretic bounding cochains for exact immersed Lagrangians, their composition is also naturally equipped with sheaf theoretic bounding cochain in the sense of \cite[Definition 6.37]{IkeKuwagaki}.  Moreover, the same holds for the corresponding assertion in Weinstein manifolds. 
\end{example}

%% file: tamarkin.tex

\section{Reformulation via Tamarkin category}\label{sec:tamarkin}

For $W$ an exact symplectic manifold, equipped with Maslov data, we define the \emph{Tamarkin category} by the Dwyer--Kan lozalization (following Tamarkin \cite{Tamarkin1}; see also \cite{Guillermou-Schapira})\footnote{When $W$ is not the cotangent bundle, it is not clear to the authors whether this is a Bousfield localization in the sense of \cite[Definition 5.2.7.2]{Lurie-HTT}, i.e.~whether the quotient functor admits a fully faithful right adjoint.}
\begin{equation} \label{tamarkin category} 
\ST(W) := \msh_{W \times T^* \R}(W \times T^*\R) / \msh_{W \times T^*_{\le 0} \R}(W \times T^*\R). \end{equation}
In our setup where $W$ is a Weinstein manifold, $L \subset W \times \R$ is an eventually conical Legendrian, the category $\msh_{{\mathfrak{c}}_W^L}(\mathfrak{c}_W^L) \subset \msh_{W \times T^* \R}(W \times T^*\R)$ embeds in $\ST(W)$ (as will be recalled in Lemma \ref{lem: quantization embeds into tamarkin}). Our results can thus be formulated, sometimes more compellingly, in terms of $\ST(W)$.  Here we record the formulation. 

\subsection{Recollections on the Tamarkin category}

The category $\ST := \ST(\mathrm{pt})$ can be identified with the full subcategory of $\Sh_{T^*_{\ge 0} \R}(\R)$
on objects with vanishing compactly supported global sections \cite[Proposition 4.2]{Kuo-Shende-Zhang-Hochschild-Tamarkin}. Such sheaves can be understood as $\R$-filtered complexes; in particular, the microsupport of such an object gives the nontrivial steps of the filtration  \cite[Lemma 4.9]{Kuo-Shende-Zhang-Hochschild-Tamarkin}. 
Convolution in the $\R$ factor gives $\ST$ a symmetric monoidal structure, and there is a corresponding internal Hom \cite[Section 3]{Guillermou-Schapira} \cite[Section 4.1]{Kuo-Shende-Zhang-Hochschild-Tamarkin}.

Let us write $\vec{\R}$ for the symmetric monoidal category whose objects are elements of $\R$, morphisms are $a \to b$ when $a < b$, and the monoidal structure is given by addition.  There is a symmetric monoidal functor $\vec{\R} \to \ST$ sending $a \mapsto 1_{[a, \infty)} \in \Sh_{T^*_{\ge 0} \R}(\R)$ \cite[Lemma 4.6]{Kuo-Shende-Zhang-Hochschild-Tamarkin} \cite[Proposition 4.22]{EfimovK}. 
We write $T_a \in \ST$ for the image of $a \in \R$.  In particular, for any object $\SF$ of a $\ST$-linear category $\mathscr{X}$ (a stable category with a $\ST$-action), we may form another object $T_a \SF$, and if $a \ge b$, there is a map $T_b \SF \to T_a \SF$.  

We say an object is {\em torsion} if $\SF \to T_a \SF$ is the zero map for $a \gg 0$ \cite[Section 2.2.3]{Tamarkin1} \cite[Definition 5.1]{Guillermou-Schapira}.  For a $\ST$-linear category $\mathscr{X}$, we write $i_\infty^* : \mathscr{X} \to \mathscr{X}_\infty$ for the quotient of $\mathscr{X}$ by torsion objects.   
One has by a formal argument \cite[Proposition 5.7]{Guillermou-Schapira} \cite[Proposition 5.21]{Kuwagaki-Zhang}
\begin{equation}\label{eq: hom mod torsion}
\Hom_{\mathscr{X}_\infty}(i_\infty^*\SF, i_\infty^*\SG) = \operatornamewithlimits{colim}_{t\to \infty}\Hom_{\mathscr{X}}(\SF, T_t \SG).
\end{equation}
Dually, when $\mathscr{X}$ is a $\ST$-linear category with pairing $\circ_{\mathscr{X}}$, we can show that the pairing descends to the the quotient 
\begin{equation}\label{eq: pairing mod torsion}
i_\infty^*\SF \circ_{\mathscr{X}_\infty} i_\infty^*\SG = \operatornamewithlimits{colim}_{t\to \infty}\SF \circ_{\mathscr{X}} T_t\SG.
\end{equation}

We can also consider a full subcategory $\ST_c \subset \ST$ as the localization $\Sh(\bR)_c/\Sh_{T^*_{\leq 0}\bR}(\bR)_c$ where $\Sh(\bR)_c$ and $\Sh_{T^*_{\leq 0}\bR}(\bR)_c$ are the full subcategories of constructible sheaves in $\Sh(\bR)$ and $\Sh_{T^*_{\leq 0}\bR}(\bR)$. $\ST_c$ can be identified with the full subcategory of $\Sh_{T^*_{\geq 0}\bR}(\bR)$ of constructible sheaves with vanishing compactly supported global sections. Then all the above results for $\ST$ also holds for $\ST_c$: convolutions on $\bR$ gives $\ST$ a closed symmetric monoidal structure, and $\vec{\bR} \to \ST_c, a \mapsto 1_{[a,\infty)}$ defines a symmetric monoidal functor.


When $W =T^*M$ (with Maslov data from the fiber polarization), it follows from K\"unneth Formula \eqref{bingyu kunneth} that $\ST(T^*M) = \Sh(M, \ST)$; see also \cite[Proposition 5.5]{Kuo-Shende-Zhang-Hochschild-Tamarkin}. Thus $\ST(T^*M)$ is enriched over $\ST$. The closed symmetric monoidal structure on $\Sh(M, \ST)$ induced from that on $\ST$ can  be described in terms of (and in early references was defined using) convolution in the $\R$ factor; we refer to \cite{Guillermou-Schapira, Kuo-Shende-Zhang-Hochschild-Tamarkin} for details. 

When $W$ is a Weinstein manifold, we do not know whether the categories $\ST(W)$ (or for that matter, $\mu sh_W(W)$) are presentable, and relatedly we do not know a K\"unneth formula in this context.  As a consequence we cannot formally deduce the $\ST$-linearity of $\ST(W)$.

As elsewhere the situation improves after imposing constructibility hypotheses. 
We write $ \msh_{W\times T^*\bR}(W \times T^*\bR)_c$ for the full subcategory of $\msh_{W\times T^*\bR}(W \times T^*\bR)$ of objects with universally sufficiently Legendrian supports.
Consider the Dwyer--Kan localization
\begin{equation}\label{eq: tamarkin constructible def}
\ST_c(W) := \msh_{W\times T^*\bR}(W \times T^*\bR)_c / \msh_{W\times T^*_{\leq 0}\bR}(W \times T^*_{\leq 0}\bR)_c,
\end{equation}
We show: 

\begin{lemma}\label{lem: tamarkin linear}
    Let $W$ be a Weinstein manifold with Maslov data. Then $\ST_c(W)$ admits an action by the symmetric monoidal category $\ST_c$ such that the action map has a right adjoint. 
\end{lemma}
\begin{proof}
    Consider a contact embedding $W \times \bR \hookrightarrow S^*\bR^N$ and the corresponding product contact embedding $W \times T^*\bR \times \bR \hookrightarrow S^*\bR^N \times T^*\bR$ such that the map $T^*\bR \to T^*\bR$ is the identity. Let $s: \bR^N \times \bR \times \bR \to \bR^N \times \bR$ be the addition map $(x, t_1, t_2) \mapsto (x, t_1+t_2)$. Then we define
    $$\circ_\ST: \ST_c \times \ST_c(W) \to \ST_c(W), \quad (\SK, \SF) \mapsto \pi_{\bR}^*\SK \star \SF := m_{\SS(\SK \circ \SF)}(\pi_{\bR}^*\SK \star w_{\SS(\SF)}^+\SF).$$
    Since $\SK \in \ST_c$, we know that the projection of $\SS_{\tau>0}(\SK)$ is a discrete subset $K \subset \bR$. Therefore, by the standard singular support estimation, we know $\SS_{\tau>0}(\SK \circ_\ST w_{\SS(\SF)}^+\SF)$ is the subset contained in the subset $\SS(\SF) + K \subset W \times T^*\bR$. Since $\SS(\SF)$ is sufficiently Legendrian, it is still sufficiently Legendrian. The above action is well defined. We note that for any $\SG \in \msh_{W \times T^*_{\leq 0}\bR}(W \times T^*_{\leq 0}\bR)$, we have $\SS(\SK \circ_\ST \SG) \subset W \times T^*_{\leq 0}\bR$.
    Finally, we can show that the above action fits into the diagram
    $$\cdots 
    \,\substack{\longrightarrow \\[-0.8em] \longrightarrow \\[-0.8em] \longrightarrow \\[-0.8em] \longrightarrow \\[-0.8em] \longrightarrow}\, \ST_c^{\times 3} \times \ST_c(W) 
    \,\substack{\longrightarrow \\[-0.8em] \longrightarrow \\[-0.8em] \longrightarrow \\[-0.8em] \longrightarrow}\, \ST_c^{\times 2} \times \ST_c(W) \,\substack{\longrightarrow \\[-0.8em] \longrightarrow \\[-0.8em] \longrightarrow}\, \ST_c \times \ST_c(W) \,\substack{\longrightarrow \\[-0.8em] \longrightarrow}\, \ST_c(W).$$
    For example, using the adjunctions between microlocalization and doubling in Theorem \ref{thm: doubling adjoint}, we can show that
    \begin{align*}
    m_{\SS(\SK_1 \circ \SK_2 \circ \SF)}(\SK_1 \circ_\ST \SK_2 \circ_\ST w_{\SS(\SF)}^+\SF) 
    \simeq m_{\SS(\SK_1 \circ\SK_2 \circ \SF)}(\SK_1 \circ_\ST w_{\SS(\SK_2 \circ \SF)}^+m_{\SS(\SK_2 \circ \SF)}(\SK_2 \circ_\ST w_{\SS(\SF)}^+\SF)).
    \end{align*}
    Considering the higher compositions and using the functoriality of the counits of adjunctions in Theorem \ref{thm: doubling adjoint}, this shows that $\ST_c(W)$ is a stable category with a $\ST_c$-action. The action is closed because $\SK \circ_\ST -$ has a right adjoint 
    $$\pi_{\bR,*}\sHom^\star(\pi_{\bR}^*\SK, -) := \pi_{\bR,*} s_*\sHom(\pi_\bR^*\SK, \pi^*w_{\SS(-)}^+(-)).$$
    This then completes the proof.
\end{proof}

\begin{corollary}\label{cor: R action}
    Let $W$ be a Weinstein manifold with Maslov data. Then $\ST_c(W)$ is admits an
    action by the symmetric monoidal category $\vec{\R}$ such that the action map has a right adjoint.
\end{corollary}

\begin{lemma}\label{def: enriched hom}
    Let $W$ be a Weinstein manifold with Maslov data. Then $\ST(W)_c$ is a $\ST_c$-enriched category with the morphism object described as follows: Consider a contact embedding $W \times \bR \hookrightarrow S^*\bR^N$ and the corresponding product contact embedding $W \times T^*\bR \times \bR \hookrightarrow S^*\bR^N \times T^*\bR$ such that the map $T^*\bR \to T^*\bR$ is the identity. Let $s: \bR^N \times \bR \times \bR \to \bR^N \times \bR$ be the addition map $(x, t_1, t_2) \mapsto (x, t_1+t_2)$. The $\ST_c$-enriched Hom is
    $$\Hom_{\ST(W)_c/\ST_c}(\SF, \SG) = m_{-\SS(\SF) \circ \SS(\SG)}\big(\pi_{\bR,!}\sHom^\star(w_{\SS(\SF)}^+\SF, w_{\SS(\SF)}^+\SG)\big) \in \ST_c.$$
\end{lemma}
\begin{proof}
    We follow the definition of enriched category in \cite[Definition 4.2.1.28]{Lurie-HA} and show that there is an adjunction for any $\SK \in \ST_c, \SF, \SG \in \ST(W)_c$:
    $$\Hom_{\ST_c}(\SK, \Hom_{\ST(W)_c/\ST_c}(\SF, \SG)) = \Hom_{\ST(W)_c}(\SK \circ_{\ST_c} \SF, \SG).$$
    The above identity follows from adjunctions of six-functors for the doublings of $\SF$ and $\SG$.
\end{proof}

\begin{lemma}    
    Let $W$ be a Weinstein manifold with Maslov data. Then there is a $\ST_c$-enriched pairing described as follows, under the same contact embeddings as above:
    $$\SF \circ_{\ST(W)_c/\ST_c} \SG = m_{\SS(\SF) \circ \SS(\SG)}\big(\pi_{\bR,!}(w_{\SS(\SF)}^+\SF \star w_{\SS(\SF)}^+\SG)\big) \in \ST_c,$$
    such that for any $\SK \in \ST_c$, we have $\Hom_{\ST_c}(\SF \circ_{\ST(W)_c/\ST_c} \SG, \SK) = \Hom_{\ST(W)_c}(\SF, \sHom^\star(\SG, \pi_\bR^*\SK))$.
\end{lemma}

Directly applying the standard definition of microsupport to Tamarkin categories does not result in a useful notion, due to the failure of the non-characteristic deformation lemma for the non-compactly-generated category $\ST$ \cite[Remark 4.24]{EfimovK}.
Instead, from the right hand side of Formula \eqref{tamarkin category} or \eqref{eq: tamarkin constructible def}, we see that objects in $\SF \in \ST(W)$ or $\ST(W)_c$ have a naturally defined microsupport in $W \times T^*_{>0} \R$.  We write $\ss_{\tau > 0}(\SF)$ to mean the restriction of this microsupport to the locus with positive unit $\R$-covector, $W \times \R \subset W \times T^*_{>0} \R$; by conicity, this encodes the microsupport in 
$W \times T^*_{>0} \R$.  We write $\underline{\ss}(\SF)$ for the (nonconic) projection of this locus to 
$W$ (sometimes called reduced microsupport in the literature).

For $L \subset W \times \R$, we write $\ST_L(W) \subset \ST(W)_c$ for the full subcategory on objects with $\ss_{\tau > 0}(\SF) \subset L$; we refer to such objects as sheaf quantizations of $L$. The following lemma justifies the name by showing the consistency of the above definition with Definition \ref{def: quantization immerse}:

\begin{lemma}\label{lem: quantization embeds into tamarkin}
    Let $W$ be a Weinstein manifold with Maslov data and $L \subset W \times \R$ be a universally sufficiently Legendrian. Then there is a natural embedding $\msh_{\mathfrak{c}_W^L}(\mathfrak{c}_W^L) \hookrightarrow \ST_L(W)$.
\end{lemma}
\begin{proof}
    Consider the inclusion $\msh_{\mathfrak{c}_W^L}(\mathfrak{c}_W^L) \subset \msh_{W \times T^*\bR}(W \times T^*\bR)_c$. We show that the subcategory is right orthogonal to $\msh_{W \times T^*_{\leq 0}\bR}(W \times T^*_{\leq 0}\bR)_c$. Similar to Lemma \ref{lem: tamarkin linear}, consider a product contact embedding $W \times T^*\bR \times \bR \hookrightarrow S^*\bR^N \times T^*\bR$ such that the map $T^*\bR \to T^*\bR$ is the identity. Then for any $\SF \in \msh_{\mathfrak{c}_W^L}(\mathfrak{c}_W^L)$ and $\SG \in \msh_{W \times T^*_{\leq 0}\bR}(W \times T^*_{\leq 0}\bR)_c$, by Proposition \ref{prop: hom deformation}, for $i_-: (-\infty, -N) \hookrightarrow \bR$ for $N \gg 0$, we have
    $$\Hom_{\msh}(\SF, \SG) = \Hom(i_{-}^*(w_{\SS(\SF)}^+\SF), i_-^*(w_{\SS(\SG)}^+\SG)) = 0.$$
    Therefore, by the formal property of Dwyer--Kan localization \cite[Theorem I.3.3]{NikolausScholze}, we conclude that the functor from $\msh_{\mathfrak{c}_W^L}(\mathfrak{c}_W^L)$ to the quotient $\ST(W)_c$ is also fully faithful.
\end{proof}

\subsection{Reformulations of the main theorems}
    We define the composition functor over Tamarkin categories and reformulate our main theorems in terms of this composition (recall that eventually conical sufficiently Legendrian subsets are always composable in the sense by Lemma \ref{lem: weinstein composable}):

\begin{definition}\label{def: enriched composition}
    Let $W_1, W_2, W_3$ be Weinstein manifolds with Maslov data and $L_{12} \subset W_1^- \times W_2 \times \bR$, $L_{23} \subset W_2^- \times W_3 \times \bR$ be eventually conic sufficiently Legendrians and $L_{13} = L_{23} \circ L_{12} \subset W_1^- \times W_3 \times \bR$. Consider the product contact embeddings 
    $$W_i^- \times W_j \times T^*\bR \times \bR \hookrightarrow (S^*\bR^{N} \,\widehat\times\, S^*\bR^N) \times T^*\bR,$$
    induced by exact embeddings $W_i, W_j \hookrightarrow S^*\bR^N$ such that second factor $T^*\bR \to T^*\bR$ is the identity map. Write $\mathfrak{c}_{W_i^-\times W_j}^{L_{ij}} = \mathfrak{c}_{ij}$ for simplicity. Let $s: \bR \times \bR \to \bR$ be the addition map and $\pi_{ij}: M_1 \times M_2 \times M_3 \times \bR^2 \to M_i \times M_j \times \bR$ be the projections. 
    We define the composition on $\ST_{L_{12}}(W_1^- \times W_2) \otimes \ST_{L_{23}}(W_2^- \times W_3)$ by
    $$\SF_{23} \circ_\ST \SF_{12} := m_{\mathfrak{c}_{13}} \big(\pi_{13!}\big(\pi_{12}^*(w_{\mathfrak{c}_{23}}^+\SF_{23} ) \star \pi_{23}^*(w_{\mathfrak{c}_{12}}^+\SF_{12})\big) \big),$$
    where $w_{\mathfrak{c}_{12}}^+$ and $w_{\mathfrak{c}_{23}}^+$ are the doublings of the Legendrian thickenings of $\mathfrak{c}_{12}$ and $\mathfrak{c}_{23}$ with respect to the contact embeddings in Theorem \ref{thm: relative-doubling}.

    Similarly, let $L_{12} \subset W_1 \times W_2 \times \bR$, $L_{23} \subset W_2 \times W_3 \times \bR$ be eventually conic sufficiently Legendrians and $L_{\bar 13} = L_{23} \circ (-L_{12}) \subset W_1^- \times W_3 \times \bR$. We define the hom composition on $\ST_{L_{12}}(W_1 \times W_2) \otimes \ST_{L_{23}}(W_2 \times W_3)$ by
    $$\sHom_\ST^\circ(\SF_{12}, \SF_{23}) := m_{\mathfrak{c}_{13}} \big(\pi_{13*}\sHom^\star\big(\pi_{12}^*(w_{\mathfrak{c}_{12}}^+\SF_{12}), \pi_{23}^!(w_{\mathfrak{c}_{23}}^+\SF_{23})\big) \big).$$
\end{definition}

    The following observation shows that compositions in Tamarkin categories agree with compositions in Definition \ref{def: gap composition weinstein}:

\begin{proposition}\label{prop: composition tamarkin wein}
    Let $W_1, W_2, W_3$ be Weinstein manifolds with Maslov data and $L_{12} \subset W_1^- \times W_2 \times \bR$, $L_{23} \subset W_2^- \times W_3 \times \bR$ be eventually conic sufficiently Legendrians and $L_{13} = L_{23} \circ L_{12} \subset W_1^- \times W_3 \times \bR$. Then there is a commutative diagram
    \[\begin{tikzcd}
        \msh_{\mathfrak{c}_{W_1^- \times W_2}^{L_{12}}}(\mathfrak{c}_{W_1^- \times W_2}^{L_{12}}) \otimes \msh_{\mathfrak{c}_{W_2^- \times W_3}^{L_{23}}}(\mathfrak{c}_{W_2^- \times W_3}^{L_{23}}) \ar[r, "\circ"] \ar[d] & \msh_{\mathfrak{c}_{W_1^- \times W_3}^{L_{13}}}(\mathfrak{c}_{W_1^- \times W_3}^{L_{13}}) \ar[d] \\
        \ST_{L_{12}}(W_1^- \times W_2) \otimes \ST_{L_{23}}(W_2^- \times W_3) \ar[r, "\circ_\ST"] & \ST_{L_{13}}(W_1^-\times W_3).
    \end{tikzcd}\]
    Similarly, let $L_{12} \subset W_1 \times W_2 \times \bR$, $L_{23} \subset W_2 \times W_3 \times \bR$ be eventually conic sufficiently Legendrians and $L_{\bar 13} = L_{23} \circ (-L_{12}) \subset W_1^- \times W_3 \times \bR$. There is also a commutative diagram
    \[\begin{tikzcd}[column sep=50pt]
        \msh_{\mathfrak{c}_{W_1 \times W_2}^{L_{12}}}(\mathfrak{c}_{W_1 \times W_2}^{L_{12}}) \otimes \msh_{\mathfrak{c}_{W_2 \times W_3}^{L_{23}}}(\mathfrak{c}_{W_2 \times W_3}^{L_{23}}) \ar[r, "{\sHom^\circ(-,-)}"] \ar[d] & \msh_{\mathfrak{c}_{W_1^- \times W_3}^{L_{\bar 13}}}(\mathfrak{c}_{W_1^- \times W_3}^{L_{\bar 13}}) \ar[d] \\
        \ST_{L_{12}}(W_1 \times W_2) \otimes \ST_{L_{23}}(W_2 \times W_3) \ar[r, "{\sHom^\circ_{\ST}(-,-)}"] & \ST_{L_{\bar 13}}(W_1^-\times W_3).
    \end{tikzcd}\]
\end{proposition}
\begin{proof}
    We only show the first isomorphism. The second one will be similar.
    For this, we need to consider specific contact embeddings 
    $$W_i^- \times W_j \times T^*\bR \hookrightarrow (S^*\bR^{N} \,\widehat\times\, S^*\bR^N) \times T^*\bR$$
    induced by exact embeddings $W_i, W_j \hookrightarrow S^*\bR^N$ and the second factor $T^*\bR \to T^*\bR$ is the identity map. Write $\mathfrak{c}_{W_i^-\times W_j}^{L_{ij}} = \mathfrak{c}_{ij}$ for simplicity. Under such choice of contact embeddings, the addition map $s: \bR \times \bR \to \bR$ is well defined on the level of sheaves is compatible with the $\bR$ factors on the level of microsheaves. Let $\pi_{ij}: M_1 \times M_2 \times M_3 \times \bR^2 \to M_i \times M_j \times \bR$ and $\ol{\pi}_{ij} : M_1 \times M_2 \times M_3 \times \bR^2 \times \bR^2 \to M_i \times M_j \times \bR^2$ be the projections. Using the doubling Theorem \ref{thm: relative-doubling}, we define the composition on $\ST_{L_{12}}(W_1^- \times W_2) \otimes \ST_{L_{23}}(W_2^- \times W_3)$ by
    $$\SF_{23} \circ_\ST \SF_{12} = m_{\mathfrak{c}_{13}} \big(\pi_{13!}\big(\pi_{12}^*(w_{\mathfrak{c}_{12}}^+\SF_{12} ) \star \pi_{23}^*(w_{\mathfrak{c}_{23}}^+\SF_{23})\big) \big).$$
    Let $\Delta_{2}: M_1 \times M_2 \times M_3 \times \bR^3 \to M_1 \times M_2^2 \times M_3 \times \bR^3$ be induced by the diagonal embedding $\bR \hookrightarrow \bR^2$ and $a: M_i \times M_j \times \bR \hookrightarrow M_i \times M_j \times \bR^2$ be the anti-diagonal embedding. It follows from the proper base change formula that
    \begin{align*}
    a^*\pi_{13!}s_!\Delta_2^*\big(\pi_{12}^*(w_{\mathfrak{c}_{12}}^+\SF_{12}) \otimes \pi_{23}^*\big(w_{\mathfrak{c}_{23}}^+\SF_{23}\big)\big) \simeq \ol{\pi}_{13!}\Delta_2^*\big(\ol{\pi}_{12}^*a^*\big(w_{\mathfrak{c}_{12}}^+\SF_{12}\big) \otimes \ol{\pi}_{23}^*a^*\big(w_{\mathfrak{c}_{23}}^+\SF_{23}\big)\big).
    \end{align*}
    This completes the proof that $\SF_{23} \circ_\ST \SF_{12} \simeq \SF_{23} \circ \SF_{12}.$
\end{proof}

    We can now reformulate our main Theorems \ref{thm: main composition} and \ref{thm: main composition hom} in terms of compositions over the Tamarkin categories:

\begin{theorem}[Reformulation of  Theorems \ref{thm: main composition} and \ref{thm: main composition hom}]\label{thm: main tamarkin}
    Let $W_1, W_2, W_3$ be Weinstein manifolds with Maslov data, $L_{12} \subset W_1^- \times W_2 \times \bR$, $L_{23} \subset W_2^- \times W_3 \times \bR$ be eventually conic sufficiently Legendrians and $L_{13} = L_{23} \circ L_{12} \subset W_1^- \times W_3 \times \bR$. Then there is a commutative diagram
    \[\begin{tikzcd}
    \ST_{L_{12}}(W_1^- \times W_2) \otimes \ST_{L_{23}}(W_2 \times W_3) \ar[r, "{(-) \circ_\ST (-)}"] \ar[d, "i_\infty^* \otimes i_\infty^*" left] & \ST_{L_{13}}(W_1^- \times W_3) \ar[d, "i_\infty^*"] \\
    \msh_{\mathfrak{c}_{W_1^-\times W_2,\partial L_{12}}}(\mathfrak{c}_{W_1^-\times W_2}) \otimes \msh_{\mathfrak{c}_{W_2^-\times W_3,\partial L_{23}}}(\mathfrak{c}_{W_2^-\times W_3}) \ar[r, "{- \circ -}"] & \msh_{\mathfrak{c}_{W_1^-\times W_3,\partial L_{13}}}(\mathfrak{c}_{W_1^-\times W_3}).
    \end{tikzcd}\]
    Similarly, for $L_{12} \subset W_1 \times W_2 \times \bR$, $L_{23} \subset W_2 \times W_3 \times \bR$ be eventually conic sufficiently Legendrians and $L_{\bar 13} = L_{23} \circ (-L_{12}) \subset W_1^- \times W_3 \times \bR$, there is a commutative diagram
    \[\begin{tikzcd}
    \ST_{L_{12}}(W_1 \times W_2) \otimes \ST_{L_{23}}(W_2 \times W_3) \ar[r, "{\sHom^\circ_\ST(-,-)}"] \ar[d, "i_\infty^* \otimes i_\infty^*" left] & \ST_{L_{\bar 13}}(W_1^- \times W_3) \ar[d, "i_\infty^*"] \\
    \msh_{\mathfrak{c}_{W_1\times W_2,\partial L_{12}}}(\mathfrak{c}_{W_1\times W_2}) \otimes \msh_{\mathfrak{c}_{W_2\times W_3,\partial L_{23}}}(\mathfrak{c}_{W_2\times W_3}) \ar[r, "{\sHom^\circ(-,-)}"] & \msh_{\mathfrak{c}_{W_1^-\times W_3,\partial L_{\bar 13}}}(\mathfrak{c}_{W_1^-\times W_3}).
    \end{tikzcd}\]
\end{theorem}

    For the reformulation of Theorem \ref{thm: main tamarkin} (of the main Theorems \ref{thm: main composition} and \ref{thm: main composition hom}), when we specialize to the case of $W_1 = W_3 = pt$ and $W_2 = W$, we can show that the result recovers the filtered morphisms of microsheaves supported on $L$ and $K$. Comparing Definitions \ref{def: enriched hom} and \ref{def: enriched composition}, one can see that the composition exactly reduces to the enriched Hom over $\ST_c$:

\begin{theorem}[Reformulation of Corollaries \ref{cor: recover tensor} and \ref{cor: recover hom}]
    Let $L, K \subset W \times \bR$ be eventually conic sufficiently Legendrians with conical ends contained in the sufficiently Legendrian $\Lambda \subset \partial W$. Then we have a commutative diagram 
     \[\begin{tikzcd}[column sep=80pt]
        \ST_L(W) \times \ST_K(W) \ar[d, "i_\infty^* \otimes i_\infty^*" left] \ar[r, "{\Hom_{\ST(W)}(-,-)}"] & \ST_c \ar[d, "i_\infty^*"]\\
        \msh_{\mathfrak{c}_{W,\Lambda}}(\mathfrak{c}_{W,\Lambda}) \times \msh_{\mathfrak{c}_{W,\Lambda}}(\mathfrak{c}_{W,\Lambda}) \ar[r, "{\Hom_{\msh}(-,-)}"] & \SC.
    \end{tikzcd}\]
    Similarly, for $L \subset W^- \times \bR$ and $K \subset W \times \bR$ eventually conic sufficiently Legendrians with conical ends contained in the sufficiently Legendrian $\Lambda \subset \partial W$, we have a commutative diagram
    \[\begin{tikzcd}[column sep=80 pt]
        \ST_L(W) \times \ST_K(W) \ar[d, "i_\infty^* \otimes i_\infty^*" left] \ar[r, "{(-) \circ_{\ST(W)}(-)}"] & \ST \ar[d, "i_\infty^*"]\\
        \msh_{\mathfrak{c}_{W,\Lambda}}(\mathfrak{c}_{W,\Lambda}) \times \msh_{\mathfrak{c}_{W,\Lambda}}(\mathfrak{c}_{W,\Lambda}) \ar[r, "{(-)\circ(-)}"] & \SC.
    \end{tikzcd}\]
\end{theorem}

    Finally, we note that Theorems  \ref{thm: lagrangian immersion tensor} and \ref{thm: lagrangian immersion hom} on the (unfiltered) morphisms of microsheaves supported on $L \ll K$ can be reformulated as (unfiltered) morphisms in $\ST(W)_\infty$ in terms of Formlae \eqref{eq: hom mod torsion} and \eqref{eq: pairing mod torsion}.

\begin{proposition}
[Reformulation of Theorems  \ref{thm: lagrangian immersion tensor} and \ref{thm: lagrangian immersion hom}]
    Let $L, K \subset W \times \bR$ be eventually conic sufficiently Legendrians with conical ends contained in the sufficiently Legendrian $\Lambda \subset \partial W$. Then we have a commutative diagram
    \[\begin{tikzcd}[column sep=80pt]
        \ST_L(W)_{\infty} \times \ST_K(W)_{\infty} \ar[d, "i_\infty^* \otimes i_\infty^*" left] \ar[r, "{\Hom_{\ST(W)_\infty}(-,-)}"] & \SC \ar[d, "\rotatebox{90}{$=$}"]\\
        \msh_{\mathfrak{c}_{W,\Lambda}}(\mathfrak{c}_{W,\Lambda}) \times \msh_{\mathfrak{c}_{W,\Lambda}}(\mathfrak{c}_{W,\Lambda}) \ar[r, "{\Hom_{\msh}(-,-)}"] & \SC.
    \end{tikzcd}\]
    Similarly, for $L \subset W^- \times \bR$ and $K \subset W \times \bR$ eventually conic sufficiently Legendrians   with conical ends contained in the sufficiently Legendrian $\Lambda \subset \partial W$, we have a commutative diagram
    \[\begin{tikzcd}[column sep=80 pt]
        \ST_L(W)_{\infty} \times \ST_K(W)_{\infty} \ar[d, "i_\infty^* \otimes i_\infty^*" left] \ar[r, "{(-) \circ_{\ST(W)_\infty}(-)}"] & \SC \ar[d, "\rotatebox{90}{$=$}"]\\
        \msh_{\mathfrak{c}_{W,\Lambda}}(\mathfrak{c}_{W,\Lambda}) \times \msh_{\mathfrak{c}_{W,\Lambda}}(\mathfrak{c}_{W,\Lambda}) \ar[r, "{(-)\circ(-)}"] & \SC.
    \end{tikzcd}\]
\end{proposition}

We can also reformulate Theorem \ref{thm: quantization double movie} using the language of Tamarkin categories:

\begin{proposition}[Reformulation of Theorem \ref{thm: quantization double movie}]
    Let $W$ be a Weinstein manifold with Maslov data and ${L} \subset W \times \bR$ be an $\epsilon$-thin eventually conical sufficiently Legendrian. Then for $t > \epsilon$, there a fully faithful embedding
    $$i_{-,t}^*: \msh_{L_{\cup_W}}(L_{\cup_W}) \hookrightarrow \ST_L(W)_\infty \hookrightarrow \msh_{\mathfrak{c}_{W,\partial L}}(\mathfrak{c}_W).$$
\end{proposition}

In conclusion, this shows that the recovery of the unfiltered morphism of microsheaves in Theorems \ref{thm: lagrangian immersion tensor}, \ref{thm: lagrangian immersion hom} (and \ref{thm: quantization double movie}) from the filtered morphisms, as explained in Corollaries \ref{cor: recover tensor} and \ref{cor: recover hom}, is exactly realized through quotient by torsion objects in the Tamarkin category.